\documentclass{amsart}
\usepackage[T1]{fontenc}
\usepackage{amsmath,a4wide}
\usepackage{amssymb}
\usepackage{amscd}
\usepackage{epsfig}
\usepackage[T1]{fontenc}
\usepackage{epsfig}
\usepackage{amsmath}
\usepackage{graphicx}
\usepackage{subfigure}
\usepackage{mathrsfs}
\usepackage{shuffle}
\usepackage{epstopdf}
\usepackage{color}
\usepackage{booktabs}
\usepackage{ifpdf}
\usepackage{cite}
\usepackage{amsthm}
\usepackage{amsmath}
\usepackage{dsfont}
\definecolor{myblue}{rgb}{0.2,0,0.9}
\definecolor{blue-violet}{rgb}{0.54, 0.17, 0.89}
\usepackage{enumitem}
\setlist[enumerate]{label={\upshape(\roman*)}}
\usepackage{algorithm,algorithmic}

\usepackage{pgfplots}
\usepgfplotslibrary{groupplots}
\pgfplotsset{compat=1.12}

\usepackage{mathtools}
\mathtoolsset{showonlyrefs}

\usepackage{tikz}
\usepackage{tikz-3dplot}
\usetikzlibrary{arrows,intersections,positioning}
\usepackage{boldline}
\usepackage{multirow}
\usepackage[margin=1.32in]{geometry}

\usepackage{varioref}
\usepackage{xr-hyper}
\usepackage{hyperref}
\hypersetup{colorlinks,linkcolor=blue,citecolor=blue,filecolor=red,urlcolor=blue}

\definecolor{myblue}{rgb}{0.2,0,0.9}
\definecolor{blue_violet}{rgb}{0.54, 0.17, 0.89}
\definecolor{darkgreen}{rgb}{0,0.35,0}

\allowdisplaybreaks

\makeatletter
\DeclareRobustCommand*\cal{\@fontswitch\relax\mathcal}
\makeatother

\makeatletter
\newcommand{\labeltext}[3][]{%
\@bsphack%
\csname phantomsection\endcsname
\def\tst{#1}%
\def\labelmarkup{\emph}
\def\refmarkup{}%
\ifx\tst\empty\def\@currentlabel{\refmarkup{#2}}{\label{#3}}%
\else\def\@currentlabel{\refmarkup{#1}}{\label{#3}}\fi%
\@esphack%
\labelmarkup{#2}
}
\makeatother

\newtheorem{theorem}{Theorem}[section]
\newtheorem{proposition}[theorem]{Proposition}

\newtheorem{corollary}[theorem]{Corollary}
\newtheorem{lemma}[theorem]{Lemma}
\numberwithin{equation}{section}

\theoremstyle{definition}
\newtheorem{remark}[theorem]{Remark}
\newtheorem{example}[theorem]{Example}
\newtheorem{definition}[theorem]{Definition}

\DeclareMathOperator{\Lip}{Lip}

\DeclareMathOperator{\Pol}{Pol}
\DeclareMathOperator*{\linspan}{span}
\DeclareMathOperator{\id}{id}
\DeclareMathOperator{\supp}{supp}
\DeclareMathOperator{\imag}{Im}

\DeclareMathOperator{\ad}{ad}

\usepackage{xparse}
\makeatletter
\RenewDocumentCommand{\title}{om}{%
\IfNoValueTF{#1}
{\gdef\shorttitle{}}
{\gdef\shorttitle{#1}}%
\gdef\@title{#2}%
}
\makeatother

\usepackage{scalerel,stackengine}
\stackMath
\newcommand\reallywidehat[1]{%
\savestack{\tmpbox}{\stretchto{%
		\scaleto{%
			\scalerel*[\widthof{\ensuremath{#1}}]{\kern-.6pt\bigwedge\kern-.6pt}%
			{\rule[-\textheight/2]{1ex}{\textheight}}
		}{\textheight}%
	}{0.5ex}}%
\stackon[1pt]{#1}{\tmpbox}%
}

\makeatletter
\def\@tocline#1#2#3#4#5#6#7{\relax
\ifnum #1>\c@tocdepth 
\else
\par \addpenalty\@secpenalty\addvspace{#2}%
\begingroup \hyphenpenalty\@M
\@ifempty{#4}{%
	\@tempdima\csname r@tocindent\number#1\endcsname\relax
}{%
	\@tempdima#4\relax
}%
\parindent\z@ \leftskip#3\relax \advance\leftskip\@tempdima\relax
\rightskip\@pnumwidth plus4em \parfillskip-\@pnumwidth
#5\leavevmode\hskip-\@tempdima
\ifcase #1
\or\or \hskip 2em \or \hskip 2em \else \hskip 3em \fi%
#6\nobreak\relax
\hfill\hbox to\@pnumwidth{\@tocpagenum{#7}}\par
\nobreak
\endgroup
\fi}
\makeatother

\title[Weighted universal approximation of differentiable maps]{Weighted universal approximation of differentiable maps on infinite-dimensional manifolds}

\author[P.~Schmocker]{Philipp Schmocker}
\address{Department of Mathematics, ETH Zurich, Switzerland}
\email{philipp.schmocker@math.ethz.ch}

\author[J.~Teichmann]{Josef Teichmann}
\address{Department of Mathematics, ETH Zurich, Switzerland}
\email{jteichma@math.ethz.ch}

\thanks{\textit{Key words:} Functional input neural networks, weighted universal approximation, infinite-dimensional manifold, Stone-Weierstrass theorem, Nachbin theorem, non-anticipative functional, signature, path space}
\thanks{\textit{MSC2020 Subject Classification:} 26F16, 41A65, 41A81, 46E15, 58C20, 60L10, 68T07}

\date{\today}

\begin{document}

\begin{abstract}
	We generalize the universal approximation theorem for \emph{functional input neural networks (FNN)} to differentiable maps by including the approximation of the derivatives. A FNN maps the input from a possibly infinite-dimensional weighted manifold to the real-valued hidden layer, on which a non-linear scalar activation function is applied, and then returns the output into a Banach space via some linear readouts. By proving a weighted Nachbin theorem, we establish a universal approximation theorem for differentiable maps, which goes beyond the usual formulation on compact sets and also includes the approximation of the derivatives. This leads us to approximation results for non-anticipative functionals including the horizontal and vertical derivatives. As a further application, we show that linear functions of the signature are able to approximate path space functionals including their directional derivatives.
\end{abstract}

\vspace{-1.em}
\maketitle

\vspace{-1.em}
{
	\hypersetup{linkcolor=black}
	\tableofcontents
}

\section{Introduction}
\label{intro}

In recent years, machine learning has transformed a wide range of scientific domains with major breakthroughs in image classification \cite{krizhevsky12}, speech recognition \cite{hinton12}, and computer games \cite{silver16}. Along these advances, one of the oldest branches of mathematical analysis -- approximation theory -- has again attracted more attention: Given a target function, can a model class approximate it to arbitrary accuracy? In this paper, we consider \emph{functional input neural networks (FNNs)} introduced in \cite{cuchiero23}, which extend classical neural networks between Euclidean spaces to infinite-dimensional spaces. In particular, we are interested whether such neural networks can also include the approximation of the directional derivatives. This contributes to the rigorous mathematical understanding of supervised machine learning methods in artificial intelligence (see \cite{turing50,mitchell97,montavon12,gbc16}).

Neural networks between Euclidean spaces were discovered in the seminal work~\cite{mcculloch43} of W.~McCulloch and W.~Pitts. They mimic the functionality of a human brain consisting of connections between neurons, i.e., the data is fed into the network, sent along various connections, transformed in the neurons, and then finally returned as output. In mathematical terms, a neural network can be described by a composition of affine and non-linear maps, where the affine maps describe the connections between neurons and the non-linear map describes the transformation of the data inside a neuron. Neural networks enjoy the so-called universal approximation property, meaning that they can approximate any continuous function uniformly on compact subsets of the Euclidean space. This fundamental result goes back to G.~Cybenko~\cite{cybenko89} and K.~Hornik~\cite{hornik89} who established in so-called universal approximation theorems (UATs) denseness of the set of neural networks in suitable function spaces. These UATs were then also extended to differentiable functions by taking into account the simultaneous approximation of the derivatives (see \cite{hsw90,hornik91}). Subsequently, other works \cite{barron93,candes98,bgkp17} related the approximation error to the network complexity by proving quantitative approximation rates under more restrictive assumptions on the target function.

The main objective of this article is to generalize the universal approximation theorem (UAT) for functional input neural networks in \cite{cuchiero23} to differentiable maps, in the sense that not only the values of a given map are approximated but also its directional derivatives. To this end, we extend the weighted framework of \cite{cuchiero23} by introducing additional weight functions on the higher-order tangent spaces of the input manifold, which in turn requires a slight adaptation of Bastiani calculus~\cite{bastiani64} to our $\sigma$-compact setting. The weights control the functions and their derivatives outside of large compact subsets, which allows us to formulate UATs including the derivatives beyond the usual approximation on compacta. This is relevant for the approximation of stochastic processes as their realizations usually do not stay in a compact path space almost surely. In particular, the weighted setting has been applied as a theoretical framework for generalized Feller processes and their semigroups (see, e.g., \cite{roeckner06,doersek10,cuchiero20,blessing22}), which is important for the approximation of solutions of stochastic (partial) differential equations (see, e.g., \cite{gierjatowicz20,kidger21,cohen21,salvi22,gazzani22}).

In order to establish the universal approximation property of functional input neural networks (FNNs), we first prove a Nachbin theorem in our weighted setting. The original Nachbin theorem, established by L.~Nachbin in \cite{nachbin49} over finite-dimensional manifolds, generalizes the classical Stone-Weierstrass theorem by including the approximation of the derivatives. This result was later extended by J.B.~Prolla and C.S.~Guerreiro in \cite{prolla76} as well as R.M.~Aron and J.B.~Prolla in \cite{aron80} to infinite-dimensional Banach spaces by using either the compact-open topology or the topology of compact convergence, both yielding approximation results over compact subsets of the input space. Only \cite{nachbin91} proved a weighted approximation result including the derivatives for polynomials on the Euclidean space. In contrast, our weighted Nachbin theorem is able to approximate a given function and its derivatives globally over an entire chart of an infinite-dimensional manifold.

By applying the weighted Nachbin theorem, we can lift the universal approximation theorem (UAT) of neural networks over the real line to a UAT for functional input neural networks (FNNs) defined on infinite-dimensional weighted manifolds. However, even on the real line, the weighted setting requires a global UAT, which is fundamentally different from classical UATs over compact subsets (see, e.g., \cite{cybenko89,hornik91,chen95}). Indeed, the weighted UATs in \cite[Proposition~4.4~(A3)]{cuchiero23} and \cite[Theorem~2.7]{schmocker26} rely on J.~Korevaar's distributional extension \cite{korevaar65} of N.~Wiener’s Tauberian theorem \cite{wiener32} to obtain sufficient conditions on the Fourier transform of the activation function, ensuring that it is discriminatory for the corresponding linear functionals (in the sense of \cite{cybenko89}). To include the approximation of the derivatives, \cite[Theorem~2.7]{schmocker26} followed \cite{hsw90,hornik91} and mollified the linear functionals, which allows us to apply integration by parts to eliminate the derivatives from the linear functionals. Finally, we assume the bounded approximation property to lift the UAT from finite-dimensional spaces to infinite-dimensional input and output spaces.

Let us remark that there is of course an extensive literature on infinite-dimensional generalizations of neural networks. Early contributions \cite{chen95,mhaskar97b,stinchcombe99,rossi05} studied the approximation of nonlinear functionals. More recent developments address approximation on non-Euclidean domains~\cite{kratsios20,galimberti22,kratsios23}, on Fréchet spaces~\cite{benth21}, on topological spaces~\cite{galimberti25,ismailov26}, and approximation rates for nonlinear functionals on $L^p$ spaces~\cite{song23}. Moreover, in the setting of adapted maps between suitably defined discrete-time path spaces, echo-state network architectures were shown to be universal \cite{grigoryeva18,gonon23}, while so-called metric hypertransformers were introduced in \cite{acciaio22}. In the context of learning solution operators for partial differential equations, we further refer to the works on the deep Galerkin method~\cite{sirignano18}, physics-informed neural networks~\cite{raissi19}, Fourier neural operators~\cite{li20}, neural integral operators~\cite{stuart21}, DeepONets~\cite{lu21,lanthaler22}, and generative equilibrium operators~\cite{kratsios25}.

Apart from neural networks, there are many other families serving as universal approximators on function spaces. By using the weighted Nachbin theorem, we show that linear functions of the signature are able to approximate a given path space functional including its directional derivatives. The signature plays a central role in rough path theory, introduced by T.~Lyons in \cite{lyons07} (see also the textbooks \cite{friz10,friz20}), and can be interpreted as polynomials on path space. More precisely, we prove that a path space functional can be approximated with linear functions of the signature on the whole path space, which extends the global universal approximation theorem (UAT) in \cite[Theorem~5.4]{cuchiero23} by including the approximation of the derivatives.

The remainder of this article is structured as follows. In Section~\ref{sec:w_dom}, we introduce weighted domains and manifolds, and characterize maps defined thereon. In Section~\ref{sec:nachbin}, we prove weighted Nachbin theorems, which are used to show universal approximation theorems for functional input neural networks in Section~\ref{sec:w_uat}. Subsequently, we apply these weighted approximation results to non-anticipative functionals in Section~\ref{sec:naf} and to linear functions of the signature in Section~\ref{sec:sig}. Finally, we provide two numerical examples in Section~\ref{sec:numerics}. Some proofs are given in Appendices~\ref{app:skorokhod}--\ref{app:naf}.

\subsection{Notation}
\label{sec:notation}

As usual, we denote by $\mathbb{N} := \lbrace 1, 2, 3, \ldots \rbrace$ and $\mathbb{N}_0 := \mathbb{N} \cup \lbrace 0 \rbrace$ the sets of natural numbers. For $n \in \mathbb{N}$, we define $\mathcal{S}_n$ as the set of permutations $\sigma: \lbrace 1,\ldots,n \rbrace \rightarrow \lbrace 1,\ldots,n \rbrace$, whereas $\mathscr{P}_n$ denotes the set of partitions $\pi := \lbrace \pi_1,\ldots,\pi_{\vert\pi\vert} \rbrace$ of $\lbrace 1,\ldots,n \rbrace$, consisting of disjoint subsets $\pi_1,\ldots,\pi_{\vert\pi\vert} \subseteq \lbrace 1,\ldots,n \rbrace$ with $\pi_1 \cup \ldots \cup \pi_{\vert\pi\vert} = \lbrace 1,\ldots,n \rbrace$. Moreover, we introduce the set of multi-indices as $\mathbb{N}^d_{0,n} := \left\lbrace \alpha := (\alpha_1,\ldots,\alpha_d) \in \mathbb{N}_0^d: \vert \alpha \vert \leq n \right\rbrace$ with $\vert \alpha \vert := \alpha_1 + \ldots + \alpha_d$. In addition, $\mathbb{R}$ and $\mathbb{C}$ (with imaginary unit $\mathbf{i} = \sqrt{-1} \in \mathbb{C}$) represent the sets of real and complex numbers, respectively. Furthermore, for $d,m \in \mathbb{N}$, we denote by $\mathbb{R}^d$ the $d$-dimensional Euclidean space equipped with the norm $\Vert x \Vert = \big(\sum_{i=1}^d x_i^2 \big)^{1/2}$, while $\mathbb{R}^{d \times m}$ denotes the vector space of matrices $A := (a_{i,j})_{i=1,\ldots,d}^{j=1,\ldots,m} \in \mathbb{R}^{d \times m}$ equipped with the Frobenius norm $\Vert A \Vert := \big( \sum_{i=1}^d \sum_{j=1}^m \vert a_{i,j} \vert^2 \big)^{1/2}$.

Moreover, a topological space $(X,\tau_X)$ is called \emph{Hausdorff} if for distinct points $x,y \in X$ there exist open sets $U,V \in \tau_X$ with $x \in U$ and $y \in V$ such that $U \cap V = \emptyset$. In addition, a \emph{topological vector space} $(X,\tau_X)$ is a vector space $X$ equipped with a topology $\tau_X$ such that addition $X \times X \ni (x_1,x_2) \mapsto x_1+x_2 \in X$ and scalar multiplication $\mathbb{R} \times X \ni (\lambda,x) \mapsto \lambda x \in X$ are both continuous. Let us remark that only vector spaces over $\mathbb{R}$ are considered in this paper. Furthermore, for topological spaces $(X,\tau_X)$ and $(Y,\tau_Y)$, we denote by $\mathcal{F}_X := \sigma(\tau_X)$ the Borel $\sigma$-algebra of $(X,\tau_X)$, and define $C^0(X;Y)$ as the vector space of continuous maps $f: X \rightarrow Y$.

In addition, a \emph{locally convex topological vector space} $(X,\tau_X)$ is a Hausdorff topological vector space such that $\tau_X$ admits a $0$-neighborhood basis consisting of balanced and convex sets. In this case, the topology $\tau_X$ is equivalently generated by a fundamental system of seminorms $\mathfrak{P}_{(X,\tau_X)}$, i.e., by sets of the form $\lbrace x \in X: p(x) < \varepsilon \rbrace$, for $\varepsilon > 0$ and $p \in \mathfrak{P}_{(X,\tau_X)}$ (see \cite[Section~II.4]{schaefer99}). A seminorm is a map $p: X \rightarrow [0,\infty)$ such that for every $\lambda \in \mathbb{R}$ and $x_1,x_2 \in X$ it holds that $p(\lambda x_1) = \vert \lambda \vert p(x_1)$ and $p(x_1+x_2) \leq p(x_1)+p(x_2)$, while $\mathfrak{P}_{(X,\tau_X)}$ is fundamental if for every $p_1,p_2 \in \mathfrak{P}_{(X,\tau_X)}$ there exist $C > 0$ and $p_3 \in \mathfrak{P}_{(X,\tau_X)}$ such that for every $x \in X$ we have $\max(p_1(x),p_2(x)) \leq C p_3(x)$. If $\mathfrak{P}_{(X,\tau_X)}$ is countable (i.e., $(X,\tau_X)$ is metrizable) and $(X,\tau_X)$ is complete, then $(X,\tau_X)$ is called a \emph{Fr\'echet space}. Moreover, if $\mathfrak{P}_{(X,\tau_X)}$ consists only of one norm $\Vert \cdot \Vert_X$ and $(X,\tau_X)$ is complete, then $(X,\Vert \cdot \Vert_X)$ is called a \emph{Banach space}. In this case, $B^X_r(x) := \lbrace y \in X: \Vert y-x \Vert_X < r \rbrace$ and $\overline{B}^X_r(x) := \lbrace y \in X: \Vert y-x \Vert_X \leq r \rbrace$ denote the open and closed ball of radius $r > 0$ around $x \in X$. When $x = 0$, we set $B^X_r := B^X_r(0)$ and $\overline{B}^X_r := \overline{B}^X_r(0)$.

Furthermore, for a family of locally convex topological vector spaces $(X_i,\tau_{X_i})_{i \in I}$, we consider the Cartesian product $\prod_{i \in I} X_i := \left\lbrace (x_i)_{i \in I}: x_i \in X_i \right\rbrace$, which is equipped with the product topology $\prod_{i \in I} \tau_{X_i}$ defined as the initial topology with respect to the projections
\begin{equation}
	\label{eq:def:product_top}
	\pi_i:
	\begin{cases}
		\prod_{i \in I} X_i \quad & \rightarrow \quad X_i \\
		(x_i)_{i \in I} \quad & \mapsto \quad x_i
	\end{cases},
	\quad\quad i \in I,
\end{equation}
i.e., the weakest topology on $\prod_{i \in I} X_i$ such that the mappings \eqref{eq:def:product_top} are continuous. Then, $(\prod_{i \in I} X_i,\prod_{i \in I} \tau_{X_i})$ is again a locally convex topological vector space (see \cite[p.~52]{schaefer99}). For example, if $(X_n,\Vert \cdot \Vert_{X_n})_{n=1,\ldots,N}$ are Banach spaces, with finite $N \in \mathbb{N}$, then the product topology $\prod_{n=1}^N \tau_{X_n}$ on $\prod_{n=1}^N X_n$ is generated by the norm $\Vert (x_n)_{n=1,\ldots,N} \Vert_{\prod_{n=1}^N X_n} := \sum_{n=1}^N \Vert x_n \Vert_{X_n}$.

Moreover, for two locally convex topological vector spaces $(X,\tau_X)$ and $(Y,\tau_Y)$, we denote by $L(X;Y)$ the vector space of continuous linear maps $T: X \rightarrow Y$, which is equipped (unless otherwise specified) with the topology of uniform convergence on bounded subsets of $X$. In particular, if $(X,\Vert \cdot \Vert_X)$ and $(Y,\Vert \cdot \Vert_Y)$ are normed vector spaces (resp.~Banach spaces), then $L(X;Y)$ is under the norm $\Vert T \Vert_{L(X;Y)} := \sup_{x \in X, \, \Vert x \Vert_X \leq 1} \Vert T(x) \Vert_Y$ again a normed vector space (resp.~Banach space). On the other hand, if $Y = \mathbb{R}$, the space $X^* := L(X;\mathbb{R})$ is the dual space of $(X,\tau_X)$ consisting of continuous linear functionals $l: X \rightarrow \mathbb{R}$.

In addition, for two locally convex topological vector spaces $(X,\tau_X)$ and $(Y,\tau_Y)$, a continuous linear map $T: X \rightarrow Y$ is called compact if there exists a $0$-neighborhood $U$ of $(X,\tau_X)$ such that $T(U)$ is relatively compact in $(Y,\tau_Y)$ (see, e.g., \cite[p.~98]{schaefer99}). If $(X,\tau_X)$ is a normed vector space $(X,\Vert \cdot \Vert_X)$, this is equivalent to the condition that for every $\Vert \cdot \Vert_X$-bounded subset $B \subseteq X$ the image $T(B)$ is relatively compact in $(Y,\tau_Y)$. Indeed, the latter implies that $T(B^X_1)$ is relatively compact in $(Y,\tau_Y)$. Conversely, if there exists a $0$-neighborhood $U$ of $(X,\Vert \cdot \Vert_X)$ with $B^X_r \subseteq U$, for some $r > 0$, then for every $\Vert \cdot \Vert_X$-bounded subset $B \subseteq B^X_R \subset X$, with some $R > 0$, it holds that $T(B) \subseteq T(B^X_R) \subseteq \frac{R}{r} T(B^X_r) \subseteq \frac{R}{r} T(U)$, where $\frac{R}{r} T(B^X_r)$ is relatively compact in $(Y,\tau_Y)$.

Furthermore, a Banach space $(X,\Vert \cdot \Vert_X)$ is called a \emph{dual Banach space} if there exists an isometric isomorphism $\mathfrak{I}: X \rightarrow E^*$ into the dual of another Banach space $(E,\Vert \cdot \Vert_E)$, called a \emph{predual}. Then, the dual pairing $E \times X \ni (e,x) \mapsto \langle e,x \rangle_{E \times X} := \mathfrak{I}(x)(e) \in \mathbb{R}$ is continuous. Hence, $X$ can be equipped with a weak-$*$-topology generated by sets of the form $\left\lbrace x \in X: \langle e,x \rangle_{E \times X} \in U \right\rbrace$, for $e \in E$ and $U \subseteq \mathbb{R}$ open (see also \cite[Appendix~A]{cuchiero23}).

Moreover, for locally convex topological vector spaces $(X,\tau_X)$ and $(Y,\tau_Y)$, we denote by $X^* \otimes Y := (X,\tau_X)^* \otimes Y := \linspan\left\lbrace X \ni x \mapsto \ell(x) y \in Y: \ell \in X^*, \, y \in Y \right\rbrace \subseteq L(X;Y)$ the vector subspace of finite rank operators, i.e., continuous linear maps $T: X \rightarrow Y$ with finite-dimensional range. Then, $(X,\tau_X)$ is said to have the \emph{approximation property (AP)} if the identity $\id_X: X \rightarrow X$ belongs to the closure of $X^* \otimes X$ with respect to uniform convergence on relatively compact subsets of $(X,\tau_X)$, i.e., there exists a net $(T_\gamma)_\gamma \subseteq X^* \otimes X$ approximating the identity $\id_X: X \rightarrow X$ uniformly on relatively compact subsets of $(X,\tau_X)$ (see \cite[Section~III.9]{schaefer99}). Furthermore, we say that $(X,\tau_X)$ has the \emph{($\mathfrak{Q}_X$-) bounded approximation property (BAP)} if there exists a set of seminorms $\mathfrak{Q}_X$ generating the topology $\tau_{\mathfrak{Q}_X}$ on $X$ with $\tau_{\mathfrak{Q}_X} \supseteq \tau_X$ such that $(X,\tau_X)$ has AP with finite rank operators $(T_\gamma)_\gamma \subseteq (X,\tau_X)^* \otimes X$ and for every $p \in \mathfrak{P}_{(X,\tau_X)}$ there exist $q \in \mathfrak{Q}_X$ and $\lambda \geq 0$ with $p(T_\gamma(x)) \leq \lambda q(x)$ for all $\gamma$ and $x \in X$. If $(X,\tau_X)$ is a Banach space, $(X,\Vert \cdot \Vert_X)$ has AP if and only if for every $\varepsilon > 0$ and relatively compact subset $K \subset X$ there exists some $T \in X^* \otimes X$ such that $\sup_{x \in K} \Vert x-T(x) \Vert_X < \varepsilon$. Moreover, $(X,\Vert \cdot \Vert_X)$ has BAP if and only if there exists a constant $\lambda \geq 1$ such that $(X,\Vert \cdot \Vert_X)$ has AP with finite rank operators $(T_\gamma)_\gamma \subseteq X^* \otimes X$ satisfying $\Vert T_\gamma \Vert_{L(X;X)} \leq \lambda$ (see also \cite[Section~1.e]{lindenstrauss96}). In addition, for $U \subseteq X$, we say that $(U,\tau_X)$ has (B)AP if $(X,\tau_X)$ has (B)AP with net of finite rank operators $(T_\gamma)_\gamma \subseteq X^* \otimes X$ satisfying $T_\gamma(U) \subseteq U$.

In addition, for $U \subseteq \mathbb{R}$ open, we denote by $C^\infty_c(U;\mathbb{C})$ the vector space of smooth functions $g: U \rightarrow \mathbb{C}$ with compact support $\supp(g) := \overline{\left\lbrace s \in U: g(s) \neq 0 \right\rbrace}$ contained in $U$. Furthermore, $\mathscr{S}(\mathbb{R};\mathbb{C})$ represents the Schwartz space consisting of smooth functions $g: \mathbb{R} \rightarrow \mathbb{C}$ with finite seminorms $\max_{j=0,\ldots,n} \sup_{s \in \mathbb{R}} \big( (1+\vert s \vert^2)^n \vert g^{(j)}(s) \vert \big)$, for all $n \in \mathbb{N}_0$, which generate the topology of $\mathscr{S}(\mathbb{R};\mathbb{C})$. Then, its dual space $\mathscr{S}'(\mathbb{R};\mathbb{C})$ consists of linear functionals $T: \mathscr{S}(\mathbb{R};\mathbb{C}) \rightarrow \mathbb{C}$ called tempered distributions. For example, $\rho \in C^0(\mathbb{R};\mathbb{C})$ with $\sup_{s \in \mathbb{R}} \frac{\vert \rho(s) \vert}{(1+\vert s \vert^2)^n} < \infty$, for some $n \in \mathbb{N}_0$, induces the tempered distribution $g \mapsto T_\rho(g) := \int_{\mathbb{R}} \rho(s) g(s) ds \in \mathscr{S}'(\mathbb{R};\mathbb{C})$. Moreover, the support of any $T \in \mathscr{S}'(\mathbb{R};\mathbb{C})$ is defined as the complement of the largest open set $U \subseteq \mathbb{R}$ on which $T \in \mathscr{S}'(\mathbb{R};\mathbb{C})$ vanishes, i.e., $T(g) = 0$ for all $g \in C^\infty_c(U;\mathbb{C})$. In addition, the Fourier transform of any $g \in L^1(\mathbb{R};\mathbb{C})$ is defined as $\mathbb{R} \ni \xi \mapsto \widehat{g}(\xi) = \frac{1}{\sqrt{2\pi}} \int_{\mathbb{R}} e^{-i \xi s} g(s) ds \in \mathbb{C}$, while the Fourier transform of any $T \in \mathscr{S}'(\mathbb{R};\mathbb{C})$ is defined by $\big( g \mapsto \widehat{T}(g) := T(\widehat{g}) \big) \in \mathscr{S}'(\mathbb{R};\mathbb{C})$. For more details, we refer to \cite[Chapters~7 and 9]{folland92}.

In addition, if the functions are real-valued, we use the abbreviations $C^k(U) := C^k(U;\mathbb{R})$, $C^\infty_c(U) := C^\infty_c(U;\mathbb{R})$, $L^p(\Omega) := L^p(\Omega;\mathbb{R})$, $C^\alpha(S) := C^\alpha(S;\mathbb{R})$, $D^{\alpha,1}([0,T]) := D^{\alpha,1}([0,T];\mathbb{R})$, etc. Most of the function spaces are introduced in the following sections.

\subsection{Bastiani calculus on \texorpdfstring{$\sigma$}{sigma}-compact spaces}

In this section, we first recall the notion of Bastiani calculus (also known as Keller's $C^1_c$-theory, see \cite{bastiani64,keller74} and also \cite{berger77,gloeckner02,walter12,schmeding23}) and then introduce a slight generalization onto $\sigma$-compact spaces. For an open subset $U \subseteq X$ of a locally convex topological vector space $(X,\tau_X)$ as input space and a locally convex topological vector space $(Y,\tau_Y)$ as output space, we define the directional derivative of a map $f: U \rightarrow Y$ at the point $u \in U$ in direction $v \in X$ (if it exists) as
\begin{equation}
	\label{eq:def:dir_der}
	df(u;v) := D_v f(u) := \lim_{h \rightarrow 0} \frac{f(u+hv) - f(u)}{h}.
\end{equation}
For $j \geq 2$, we define the $j$-th order directional derivatives of a map $f: U \rightarrow Y$ at the point $u \in U$ in directions $v_1,\ldots,v_j \in X$ (if they exist) as
\begin{equation}
	\label{eq:def:dir_der2}
	d^j f(u;v_1,\ldots,v_j) := D_{v_j} \cdots D_{v_1} f(u).
\end{equation}
Then, for $k \in \mathbb{N}_0$, the $C^k$-space $C^k(U;Y)$ in the sense of Bastiani is defined as the vector space of maps $f: U \rightarrow Y$ whose $j$-th order directional derivatives exist and the mappings $U \times X^j \ni (u,v_1,\ldots,v_j) \mapsto d^j f(u;v_1,\ldots,v_j) \in Y$ are continuous, for all $j = 0,\ldots,k$, with $d^0 f := f$.

Moreover, if the input space $(X,\tau_X)$ is $\sigma$-compact, we define $C^k_{loc}(U;Y)$ as the vector space of maps $f: U \rightarrow Y$ whose $j$-th order directional derivatives exist and the mappings $d^j f\vert_K: K \rightarrow Y$ are continuous, for any compact subset $K \subseteq U \times X^j$ and $j = 0,\ldots,k$. Compared to Bastiani calculus with globally continuous mappings $d^j f: U \times X^j \rightarrow Y$, $j = 0,\ldots,k$, we only require them to be continuous on compact subsets, implying that $C^k(U;Y) \subseteq C^k_{loc}(U;Y)$. However, if $(X,\tau_X)$ is locally compact or first countable (ensuring that $(X,\tau_X)$ is compactly generated, see \cite[Lemma~46.3]{munkres14}), every mapping $U \times X^j \ni (u,v_1,\ldots,v_j) \mapsto d^j f(u;v_1,\ldots,v_j) \in Y$ that is continuous on compacta, is also globally continuous (see \cite[Lemma~46.4]{munkres14}), whence the two notions are equivalent. Therefore, our notion of $C^k_{loc}$-maps is stronger than G\^{a}teaux differentiability (except on finite-dimensional spaces), but weaker than Fr\'echet differentiability (except on finite-dimensional spaces). 

In order to establish some properties of our $C^k_{loc}$-differential calculus that are known for Bastiani calculus (see, e.g., \cite{bastiani64,keller74,gloeckner02,schmeding23}), we first prove the following auxiliary lemmas. For an open interval $I \subseteq \mathbb{R}$, $a,b \in \mathbb{R}$, and $f \in C^0(I;Y)$, we say that the \emph{weak integral} $\int_a^b f(t) dt$ exists if there is a point $y \in Y$ such that for every $\ell \in Y^*$ it holds that $\ell(y) = \int_a^b \ell(f(t)) dt$.

\begin{lemma}[Fundamental theorem of calculus]
	\label{lem:foc}
	Let $0 \in I \subseteq \mathbb{R}$ be an open interval and let $c \in C^1_{loc}(I;Y)$. Then, for every $h \in I$, the weak integral $\int_0^h c'(t) dt$ exists and satisfies
	\begin{equation}
		c(h) - c(0) = \int_0^h c'(t) dt.
	\end{equation}
\end{lemma}
\begin{proof}
	Since $c \in C^1_{loc}(I;Y)$, we have $c' \in C^0(I;Y)$. Hence, we can apply the fundamental theorem of calculus for real-valued functions to conclude for every $\ell \in Y^*$ and $h \in I$ that
	\begin{equation}
		\ell(c(h)-c(0)) = \ell(c(h)) - \ell(c(0)) = \int_0^h (\ell \circ c)'(t) dt = \int_0^h \ell(c'(t)) dt.
	\end{equation}
	Hence, $y := c(h)-c(0) \in Y$ satisfies the defining properties of the weak integral.
\end{proof}

Note that this fundamental theorem of calculus for $C^1_{loc}$-curves with values in a locally convex topological vector space $(Y,\tau_Y)$ holds irrespective of completeness of $Y$. Moreover, as an application of the bipolar theorem (see, e.g., \cite[Theorem~IV.1.5]{schaefer99}), we obtain the following result.

\begin{lemma}
	\label{lem:wint_bound}
	Let $a,b \in \mathbb{R}$ and $f \in C^0([a,b];Y)$ such that the weak integral $\int_a^b f(t) dt$ exists. Then, for every $p_Y \in \mathfrak{P}_{(Y,\tau_Y)}$, it holds that
	\begin{equation}
		p_Y\left( \int_a^b f(t) dt \right) \leq \vert b-a \vert \sup_{t \in [a,b]} p_Y(f(t)).
	\end{equation}
\end{lemma}

Next, we prove the following properties of our $C^k_{loc}$-differential calculus including the linearity of the differential and the chain rule, which are known for Bastiani calculus (see, e.g., \cite{bastiani64,keller74,gloeckner02,schmeding23}).

\begin{proposition}
	\label{prop:db_prop}
	Let $f \in C^1_{loc}(U;Y)$. Then, the following holds true:
	\begin{enumerate}
		\item\label{prop:db_prop:1} For every $u \in U$ the map $X \ni v \mapsto df(u;v) \in Y$ is linear and in $C^0_{loc}(X;Y)$.
		\item\label{prop:db_prop:2} Let $V \subseteq Y$ be open with $f(U) \subseteq V$, let $(Y,\tau_Y)$ be $\sigma$-compact, let $(Z,\tau_Z)$ be another locally convex topological vector space, and let $g \in C^1_{loc}(V;Z)$. Then, $g \circ f \in C^1_{loc}(U;Z)$ and for every $u \in U$ and $v \in X$ we have $d(g \circ f)(u;v) = dg(f(u);df(u;v))$.
		\item\label{prop:db_prop:3} If $f \in C^k_{loc}(U;Y)$ with $k \geq 2$, then for every $u \in U$ and $v_1,\ldots,v_k \in X$ it holds that $d\big(d^{k-1} f(\cdot;v_1,\ldots,v_{k-1})\big)(u;v_k) = d^k f(u;v_1,\ldots,v_k)$.
		\item\label{prop:db_prop:4} If $f \in C^k_{loc}(U;Y)$, then for every $u \in U$, $v_1,\ldots,v_k \in X$, and every permutation $\sigma \in \mathcal{S}_k$ we have $d^k f(u;v_{\sigma(1)},\ldots,v_{\sigma(k)}) = d^k f(u;v_1,\ldots,v_k)$.
	\end{enumerate}
\end{proposition}
\begin{proof}
	For \ref{prop:db_prop:1} we fix some $u \in U$. Then, for every $v \in X$ and $\lambda \in \mathbb{R}$, the homogeneity $df(u;\lambda v) = \lambda df(u;v)$ follows from \eqref{eq:def:dir_der}. For linearity of $X \ni v \mapsto df(u;v) \in Y$, we fix some $\varepsilon > 0$, $p_Y \in \mathfrak{P}_{(Y,\tau_Y)}$, $u \in U$, $v_1,v_2 \in X$, and $\delta > 0$ such that $u+rv_1+sv_2 \in U$ for all $r,s \in [-\delta,\delta]$. Then, by applying Lemma~\ref{lem:foc} twice, it follows for every $h \in (-\delta,\delta)$ that
	\begin{equation}
		\label{eq:prop:db_prop:proof1}
		\begin{aligned}
			f(u+h(v_1+v_2)) - f(u) & = f(u+hv_1) - f(u) + \int_0^1 df(u+hv_1+shv_2;hv_2) ds \\
			& = \int_0^1 df(u+shv_1;hv_1) ds + \int_0^1 df(u+hv_1+shv_2;hv_2) ds \\
			& = h(df(u;v_1) + df(u;v_2)) + \int_0^1 \left( df(u+shv_1;hv_1) - df(u;hv_1) \right) ds \\
			& \quad\quad + \int_0^1 \left( df(u+hv_1+shv_2;hv_2) - df(u;hv_2) \right) ds,
		\end{aligned}
	\end{equation}
	where all integrals exist as weak integrals. Moreover, by using that $f \in C^1_{loc}(U;Y)$ and the image of $[-\delta,\delta] \ni s \mapsto u+sv_1 \in U$ is compact in $U$, we conclude that $[-\delta,\delta] \ni s \mapsto df(u+sv_1;v_1) - df(u;v_1) \in Y$ is continuous, thus uniformly continuous, whence there exists some $\delta_1 \in (0,\delta)$ such that for every $h \in (-\delta_1,\delta_1)$ it holds that
	\begin{equation}
		\label{eq:prop:db_prop:proof2}
		\begin{aligned}
			& p_Y\left( \frac{1}{h} \int_0^1 \left( df(u+shv_1;hv_1) - df(u;hv_1) \right) ds \right) \\
			& = p_Y\left( \int_0^1 \left( df(u+shv_1;v_1) - df(u;v_1) \right) ds \right) \\
			& \leq \sup_{s \in [0,1]} p_Y\left( df(u+shv_1;v_1) - df(u;v_1) \right) < \frac{\varepsilon}{2},
		\end{aligned}
	\end{equation}
	where we have applied Lemma~\ref{lem:wint_bound} for the first inequality. Similarly, by using that $[-\delta,\delta]^2 \ni (r,s) \mapsto df(u+rv_1+sv_2;v_1) - df(u;v_1) \in Y$ is continuous, thus uniformly continuous, there exists some $\delta_2 \in (0,\delta)$ such that for every $h \in (-\delta_2,\delta_2)$ we have
	\begin{equation}
		\label{eq:prop:db_prop:proof3}
		\begin{aligned}
			& p_Y\left( \frac{1}{h} \int_0^1 \left( df(u+hv_1+shv_2;hv_2) - df(u;hv_2) \right) ds \right) \\
			& = p_Y\left( \int_0^1 \left( df(u+hv_1+shv_2;v_2) - df(u;v_2) \right) ds \right) \\
			& \leq \sup_{s \in [0,1]} p_Y\left( df(u+hv_1+shv_2;v_2) - df(u;v_2) \right) < \frac{\varepsilon}{2}.
		\end{aligned}
	\end{equation}
	Hence, by inserting \eqref{eq:prop:db_prop:proof2}--\eqref{eq:prop:db_prop:proof3} into \eqref{eq:prop:db_prop:proof1} and defining $\delta_0 := \min(\delta_1,\delta_2) > 0$, it follows for every $h \in (-\delta_0,\delta_0)$ that
	\begin{equation}
		\begin{aligned}
			& p_Y\left( \frac{f(u+h(v_1+v_2)) - f(u)}{h} - \left( df(u;v_1) + df(u;v_2) \right) \right) \\
			& \leq p_Y\left( \frac{1}{h} \int_0^1 \left( df(u+shv_1;hv_1) - df(u;hv_1) \right) ds \right) \\
			& \quad\quad + p_Y\left( \frac{1}{h} \int_0^1 \left( df(u+hv_1+shv_2;hv_2) - df(u;hv_2) \right) ds \right) \\
			& < \frac{\varepsilon}{2} + \frac{\varepsilon}{2} = \varepsilon.
		\end{aligned}
	\end{equation}
	Since $\varepsilon > 0$ was chosen arbitrarily, this and the homogeneity show that $X \ni v \mapsto df(u;v) \in Y$ is linear. Finally, we use that $f \in C^1_{loc}(U;Y)$ to see that $(v \mapsto df(u;v)) \in C^0_{loc}(X;Y)$.
	
	For \ref{prop:db_prop:2}, we fix some $u \in U$, $v \in X$, and $\delta > 0$ such that $u+sv \in U$ for all $s \in [-\delta,\delta]$. Then, by using that $df(u;v)$ exists and is continuous on compacta, there exists a continuous function $r: [-\delta,\delta] \rightarrow Y$ with $r(0) = 0$ such that for every $h \in [-\delta,\delta]$ we have
	\begin{equation}
		(g \circ f)(u+hv) = g\left( f(u) + h\left( df(u;v) + r(h) \right) \right).
	\end{equation}
	Defining $w(h) := df(u;v) + r(h)$ and shrinking $\delta > 0$ if necessary, we can assume that $f(u) + sh w(h) \in V$ for all $h \in [-\delta,\delta]$ and $s \in I$, where $I$ is an open interval containing $[0,1]$. Moreover, since $f \in C^1_{loc}(U;Y)$ and $g \in C^1_{loc}(V;Z)$, the curve $I \ni s \mapsto c(s) := g(f(u)+hsw(h)) \in Z$ is continuously differentiable with $c'(s) = dg(f(u)+hsw(h);hw(h)) = hdg(f(u)+hsw(h);w(h))$, whence Lemma~\ref{lem:foc} implies for every $h \in [-\delta,\delta]$ that		
	\begin{equation}
		\label{eq:prop:db_prop:proof4}
		\begin{aligned}
			& (g \circ f)(u+hv) = c(1) = c(0) + \int_0^1 c'(s) ds \\
			& = g(f(u)) + h \, dg(f(u);df(u;v)) \\
			& \quad\quad + h \int_0^1 \big( dg(f(u)+hsw(h);df(u;v)+r(h)) - dg(f(u);df(u;v)) \big) ds \\
			& = g(f(u)) + h \, dg(f(u);df(u;v)) \\
			& \quad\quad + h \int_0^1 \big( dg(f(u)+hsw(h);df(u;v)) + dg(f(u)+hsw(h);r(h)) - dg(f(u);df(u;v)) \big) ds,
		\end{aligned}
	\end{equation}
	where all integrals exist as weak integrals. Now, for every fixed $\varepsilon > 0$ and $p_Z \in \mathfrak{P}_{(Z,\tau_Z)}$, we use that $f \in C^1_{loc}(U;Y)$ and $g \in C^1_{loc}(V;Z)$ to conclude that
	\begin{equation}
		\begin{aligned}
			[-\delta,\delta] \times [0,1] \ni (h,s) \mapsto \Phi(h,s) & := dg(f(u)+hsw(h);r(h)) \in Z \\
			[-\delta,\delta] \times [0,1] \ni (h,s) \mapsto \Psi(h,s) & := dg(f(u)+hsw(h);df(u;v)) - dg(f(u);df(u;v)) \in Z
		\end{aligned}
	\end{equation}
	are continuous, thus uniformly continuous, with $\Phi(0,s) = \Psi(0,s) = 0$ for all $s \in [0,1]$. Hence, there exists some $\delta_0 \in (0,\delta)$ such that for every $(h,s) \in (-\delta_0,\delta_0) \times [0,1]$ it holds that $p_Z(\Phi(h,s)) < \varepsilon/2$ and $p_Z(\Psi(h,s)) < \varepsilon/2$, which implies by Lemma~\ref{lem:wint_bound} for every $h \in (-\delta_0,\delta_0)$ that
	\begin{equation}
		\label{eq:prop:db_prop:proof5}
		\begin{aligned}
			& p_Z\left( \int_0^1 \left( dg(f(u)+hsw(h);df(u;v)) + dg(f(u)+hsw(h);r(h)) - dg(f(u);df(u;v)) \right) ds \right) \\
			& \quad\quad \leq \sup_{(s,h) \in [0,1] \times (-\delta_0,\delta_0)} p_Z(\Phi(h,s) + \Psi(h,s)) < \frac{\varepsilon}{2} + \frac{\varepsilon}{2} = \varepsilon.
		\end{aligned}
	\end{equation}
	Hence, by inserting \eqref{eq:prop:db_prop:proof5} into \eqref{eq:prop:db_prop:proof4}, it follows for every $h \in (-\delta_0,\delta_0)$ that
	\begin{equation}
		p_Z\left( \frac{(g \circ f)(u+hv) - g(f(u))}{h} - dg(f(u);df(u;v)) \right) < \varepsilon.
	\end{equation}
	Since $\varepsilon > 0$ was chosen arbitrarily, this shows that $d(g \circ f)(u;v) = dg(f(u);df(u;v))$.
	
	While \ref{prop:db_prop:3} follows from the definition \eqref{eq:def:dir_der2}, we fix for \ref{prop:db_prop:4} some $\delta > 0$ such that $u + h_1v_1 + \ldots + h_kv_k \in U$ for all $h_1,\ldots,h_k \in [-\delta,\delta]$. Then, by using that the mapping
	\begin{equation}
		[-\delta,\delta]^k \ni (h_1,\ldots,h_k) \quad \mapsto \quad F(h_1,\ldots,h_k) := f(u+h_1v_1+\ldots+h_kv_k) \in Y
	\end{equation}
	is $k$-times continuously differentiable on $(-\delta,\delta)^k$, we can apply the finite-dimensional Schwarz theorem on $(-\delta,\delta)^k \subseteq \mathbb{R}^k$ to conclude that
	\begin{equation}
		d^k f(u;v_{\sigma(1)},\ldots,v_{\sigma(k)}) = \frac{d^k F}{dh_{\sigma(1)} \cdots dh_{\sigma(k)}}(0,\ldots,0) = \frac{d^k F}{dh_1 \cdots dh_k}(0,\ldots,0) = d^k f(u;v_1,\ldots,v_k),
	\end{equation}
	which completes the proof.
\end{proof}

\subsection{Manifolds over \texorpdfstring{$\sigma$}{sigma}-compact model spaces}
\label{sec:manifold}

In this section, we introduce the notion of manifolds that are modelled over $\sigma$-compact locally convex topological vector spaces (see also \cite{berger77,kriegl97,schmeding23} for more details). To this end, we shall fix some $k \in \mathbb{N} \cup \lbrace \infty \rbrace$, a topological space $(M,\tau_M)$, and a family of $\sigma$-compact locally convex topological vector spaces $(X_i,\tau_{X_i})_{i \in I}$, where $I$ is an arbitrary index set. Then, a $C^k_{loc}$-atlas $(U_i,\phi_i)_{i \in I}$ for $M$ consists of an open cover $(U_i)_{i \in I}$ of $M$, i.e., $\bigcup_{i \in I} U_i = M$, and homeomorphisms $\phi_i: U_i \rightarrow \phi_i(U_i) \subseteq X_i$ called charts such that the transition maps $\phi_{i_1} \circ \phi_{i_2}\vert_{\phi_{i_2}(U_{i_1} \cap U_{i_2})}^{-1}: \phi_{i_1}(U_{i_1} \cap U_{i_2}) \rightarrow \phi_{i_2}(U_{i_1} \cap U_{i_2})$ are $C^k_{loc}$-maps, for all $i_1,i_2 \in I$. If such a $C^k_{loc}$-atlas $(U_i,\phi_i)_{i \in I}$ exists for $M$, then we call $(M,\tau_M)$ a \emph{$C^k_{loc}$-manifold} (with atlas $(U_i,\phi_i)_{i \in I}$ over model spaces $(X_i,\tau_{X_i})_{i \in I}$). For example, if $U \subseteq X$ is an open subset of a $\sigma$-compact locally convex topological vector space $(X,\tau_X)$, then $M := U$ is a $C^\infty_{loc}$-manifold with global chart given by the smooth inclusion $U \hookrightarrow X$.

Moreover, we follow \cite{miron97,suri16} and define for every $j = 0,\ldots,k$ the \emph{tangent space of order $j$} at $x \in M$ as the set $T^j_x M$ of equivalence classes $[c]^j_x$ of $C^j$-curves $c: (-\varepsilon,\varepsilon) \rightarrow M$ with $c(0) = x$ whose accelerations agree up to order $j$, i.e., $c \sim_j \widetilde{c}$ if and only if $c^{(\ell)}(0) = \widetilde{c}^{(\ell)}(0)$, for all $\ell=0,\ldots,j$, where $T^0_x M := \lbrace 0 \rbrace$. If $x \in U_i$, then $T^j_x M$ is topologically isomorphic to $X_i^j$ with isomorphism
\begin{equation}
	\Phi^j_{i,x}:
	\begin{cases}
		T^j_x M & \rightarrow \quad X_i^j \\
		[c]^j_x & \mapsto \quad \left( (\phi_i \circ c)^{(\ell)}(0) \right)_{\ell=1,\ldots,j}
	\end{cases}
\end{equation}
and inverse
\begin{equation}
	\Phi^{-j}_{i,x}:
	\begin{cases}
		X_i^j & \rightarrow \quad T^j_x M \\
		(v_1,\ldots,v_j) & \mapsto \quad \left[ t \mapsto \phi_i^{-1}\left( \phi_i(x) + \frac{t}{1!} v_1 + \ldots + \frac{t^j}{j!} v_j \right) \right]^j_x
	\end{cases}.
\end{equation}
Furthermore, we define for every $j = 0,\ldots,k$ the \emph{tangent bundle of order $j$} as $T^j M := \bigcup_{x \in M} T^j_x M := \left\lbrace (x,[c]^j_x): x \in M, \, [c]^j_x \in T^j_x M \right\rbrace$, with $T^0 M := M$, which we equip with the final topology $\tau_{T^j M}$ with respect to the family of mappings
\begin{equation}
	\label{eq:def:TjM_finaltop}
	\Phi^{-j}_i:
	\begin{cases}
		\phi_i(U_i) \times X_i^j & \rightarrow \quad T^j M \\
		(u,v_1,\ldots,v_j) & \mapsto \quad \left( \phi_i^{-1}(u), \Phi^{-j}_{i,\phi_i^{-1}(u)}(v_1,\ldots,v_j) \right)
	\end{cases}, \quad\quad i \in I,
\end{equation}
i.e., the finest topology on $T^j M$ such that the mappings \eqref{eq:def:TjM_finaltop} are continuous. Then, by following the proof of \cite[Theorem~2.1]{suri16} (with manifolds over locally convex topological vector spaces instead of Banach manifolds), one can show that $T^j M$ is, as a fibre bundle, a $C^0_{loc}$-manifold with atlas $(\pi_{T^j M}^{-1}(U_i),\Phi^j_i)_{i \in I}$ over model spaces $(X_i \times X_i^j,\tau_{X_i} \times \tau_{X_i}^j)_{i \in I}$, where
\begin{equation}
	\label{eq:def:TjM_charts}
	\Phi^j_i:
	\begin{cases}
		\pi_{T^j M}^{-1}(U_i) & \rightarrow \quad \phi_i(U_i) \times X_i^j \\
		(x,[c]^j_x) & \mapsto \quad \left( \phi_i(x), \Phi^j_{i,x}([c]^j_x) \right)
	\end{cases}, \quad\quad i \in I,
\end{equation}
are the charts, and where $T^j M \ni (x,[c]^j_x) \mapsto \pi_{T^j M}(x,[c]^j_x) := x \in M$ is the bundle projection. For example, if $M := U \subseteq X$ is an open subset of a locally convex topological vector space $(X,\tau_X)$, then it holds that $T^j M \cong U \times X^j$, for all $j \in \mathbb{N}_0$.

In addition, for $k \in \mathbb{N}_0$ and a given $C^k_{loc}$-manifold $(M,\tau_M)$, we denote by $C^k_{loc}(M;Y)$ the vector space of maps $f: M \rightarrow Y$ such that $f \circ \phi_i^{-1} \in C^k_{loc}(\phi_i(U_i);Y)$ for all $i \in I$.

\subsection{Examples of \texorpdfstring{$\sigma$}{sigma}-compact model spaces}
\label{sec:model_sp}

In this section, we present some examples of $\sigma$-compact locally convex topological vector spaces used as model spaces for manifolds. 

For $\alpha \in (0,\infty)$, a compact metric space $(S,d_S)$ with designated origin $0 \in S$, and a dual Banach space $(Z,\Vert \cdot \Vert_Z)$ with predual $(E,\Vert \cdot \Vert_E)$, we denote by $C^\alpha(S;Z)$ the space of $\alpha$-H\"older continuous functions $x: (S,d_S) \rightarrow (Z,\Vert \cdot \Vert_Z)$ satisfying
\begin{equation}
	\label{eq:def:hol_norm}
	\Vert x \Vert_\alpha := \Vert x(0) \Vert_Z + \vert x \vert_\alpha < \infty.
\end{equation}
Here, $\vert x \vert_\alpha$ denotes the $\alpha$-H\"older seminorm of $x: S \rightarrow Z$ defined as
\begin{equation}
	\label{eq:def:hol_seminorm}
	\vert x \vert_\alpha := \sup_{s,t \in S, \, s \neq t} \frac{\Vert x(s) - x(t) \Vert_Z}{d_S(s,t)^\alpha}.
\end{equation}
Then, the norm $\Vert \cdot \Vert_\alpha$ turns $C^\alpha(S;Z)$ into a Banach space (see \cite[Theorem~5.25]{friz10} and \cite[Proposition~2.3(b)]{weaver99}). Moreover, for $\alpha' \in [0,\alpha]$, we equip $C^\alpha(S;Z)$ also with the weaker $w^*$-$C^{\alpha'}$-topology $\tau_{\alpha'}$ generated by seminorms of the form
\begin{equation}
	\Vert x \Vert_{\alpha',e} := \left\vert \langle x(0), e \rangle_{Z \times E} \right\vert + \vert x \vert_{\alpha',e},
\end{equation}
for $e \in E$, where $\vert x \vert_{\alpha',e}$ denotes the $(\alpha',e)$-H\"older seminorm of $x: S \rightarrow Z$ defined as
\begin{equation}
	\vert x \vert_{\alpha',e} := \sup_{s,t \in S \atop s \neq t} \frac{\vert \langle x(s) - x(t), e \rangle_{Z \times E} \vert}{d_S(s,t)^{\alpha'}}.
\end{equation}
Hence, $(C^\alpha(S;Z),\tau_{\alpha'})$ forms a locally convex topological vector space. Note that for $\alpha' = 0$ the $w^*$-$C^0$-topology $\tau_0$ is equivalent to the $w^*$-uniform topology $\tau_\infty$ generated by seminorms of the form $\Vert x \Vert_{\infty,e} := \sup_{t \in S} \vert \langle x(t), e \rangle_{Z \times E} \vert$, for $e \in E$ (see \cite[Lemma~A.1]{cuchiero23}). In addition, for $\alpha' \in [0,\alpha)$, the embedding $(C^\alpha(S;Z),\Vert \cdot \Vert_\alpha) \hookrightarrow (C^\alpha(S;Z),\tau_{\alpha'})$ is by \cite[Theorem~A.4]{cuchiero23} compact, whence $(C^\alpha(S;Z),\tau_{\alpha'})$ is as the image of countably many $\Vert \cdot \Vert_\alpha$-balls $\sigma$-compact. Furthermore, $C^\alpha(S;Z)$ is a dual Banach space (see \cite[Theorem~A.5]{cuchiero23}), which is by the Banach-Alaoglu theorem also $\sigma$-compact with respect to its weak-$*$-topology $\tau_{w^*}$. Furthermore, we denote by $C^\alpha_0(S;Z) \subseteq C^\alpha(S;Z)$ the vector subspace of $\alpha$-H\"older continuous functions $x \in C^\alpha(S;Z)$ with $x(0) = 0 \in Z$.

Moreover, for $T > 0$ and a dual Banach space $(Z,\Vert \cdot \Vert_Z)$, we denote by $D^0([0,T];Z)$ the vector space of c\`adl\`ag paths $x: [0,T] \rightarrow (Z,\Vert \cdot \Vert_Z)$, whose left limits $x(t-) := \lim_{s \rightarrow t^-} x(s)$ exist, for all $t \in (0,T]$ and the right limits satisfy $x(t+) := \lim_{s \rightarrow t^+} x(s) = x(t)$, for all $t \in [0,T)$. Then, the norm $\Vert x \Vert_\infty := \sup_{t \in [0,T]} \Vert x(t) \Vert_Z$ turns $D^0([0,T];Z)$ into a Banach space (see, e.g., \cite[Section~12]{billingsley99}, \cite[Section~3.5]{ethier05}, and \cite[p.~1]{jakubowski07}). In addition, for $\alpha \in [0,1)$, we define $D^{\alpha,1}([0,T];Z) \subseteq D^0([0,T];Z)$ as the vector subspace of c\`adl\`ag paths $x \in D^0([0,T];Z)$ satisfying
\begin{equation}
	\label{eq:def:sk_norm}
	\Vert x \Vert_{\alpha,\ell^1} := \max\left( \Vert x^c \Vert_\alpha, \Vert \Delta x \Vert_{\ell^1} \right) < \infty,
\end{equation}
where the continuous part $[0,T] \ni t \mapsto x^c(t) := x(t) - \sum_{s \in (0,t]} \Delta x(s) \in Z$ is $\alpha$-H\"older continuous and the jump part $(0,T] \ni t \mapsto \Delta x(t) := x(t) - x(t-) \in Z$ is summable, i.e.,
\begin{equation}
	\label{eq:def:sk_jumps}
	\Vert \Delta x \Vert_{\ell^1} := \sum_{t \in (0,T]} \Vert \Delta x(t) \Vert_Z < \infty.
\end{equation}
Since every c\`adl\`ag path has at most countably many jumps (see \cite[Lemma~5.1]{ethier05}), the condition \eqref{eq:def:sk_jumps} is only an assumption on the jump sizes. Then, $(D^{\alpha,1}([0,T];Z),\Vert \cdot \Vert_{\alpha,\ell^1})$ is a Banach space, which is isometrically isomorphic to the direct sum of the Banach spaces $(C^\alpha([0,T];Z),\Vert \cdot \Vert_\alpha)$ and $(\ell^1((0,T];Z),\Vert \cdot \Vert_{\ell^1})$, where the latter consists of $Z$-valued sequences $(z_t)_{t \in (0,T]}$ with $\Vert z \Vert_{\ell^1} := \sum_{t \in (0,T]} \Vert z_t \Vert_Z < \infty$ (see Theorem~\ref{thm:sk_banach}). Furthermore, if $\alpha \in (0,1)$, then $(D^{\alpha,1}([0,T];Z),\Vert \cdot \Vert_{\alpha,\ell^1})$ is a dual Banach space (see Theorem~\ref{thm:sk_dual}), which is by the Banach-Alaoglu theorem also $\sigma$-compact with respect to its weak-$*$-topology $\tau_{w^*}$. Note that $\tau_{w^*}$ coincides on $\Vert \cdot \Vert_{\alpha,\ell^1}$-bounded subsets of $C^\alpha([0,T];Z) \subseteq D^{\alpha,1}([0,T];Z)$ with the $w^*$-uniform topology $\tau_\infty$.

In addition, for $p \in [1,\infty]$, a $\sigma$-finite measure space $(\Omega,\mathcal{F},\mu)$, and a dual Banach space $(Z,\Vert \cdot \Vert_Z)$ with predual $(E,\Vert \cdot \Vert_E)$, we denote by $L^p(\Omega;Z) := L^p(\Omega,\mathcal{F},\mu;Z)$ the Bochner space of (equivalence classes of) strongly $\mu$-measurable maps $x: \Omega \rightarrow Z$ with finite norm
\begin{equation}
	\Vert x \Vert_{L^p(\Omega;Z)} := 
	\begin{cases}
		\left( \int_{\Omega} \Vert x(\omega) \Vert_Z^p \mu(d\omega) \right)^{1/p}, & p \in [1,\infty), \\
		\inf\left\lbrace c > 0: \mu\left(\left\lbrace \omega \in \Omega: \Vert x(\omega) \Vert_Z > c \right\rbrace\right) = 0 \right\rbrace, & p = \infty.
	\end{cases}
\end{equation}
Then, the norm $\Vert \cdot \Vert_{L^p(\Omega;Z)}$ turns $L^p(\Omega;Z)$ into a Banach space (see \cite[Section~1.2b]{hytoenen16}). In particular, for $p \in (1,\infty]$ and $p' \in [1,\infty)$ with $1/p+1/p' = 1$, and if $(Z,\Vert \cdot \Vert_Z)$ has the Radon-Nikodym property with respect to $\mu$ (see \cite[Definition~1.3.9]{hytoenen16}), then the Bochner space $L^p(\Omega;Z) \cong L^{p'}(\Omega;E)^*$ is a dual Banach space, which is by the Banach-Alaoglu theorem $\sigma$-compact with respect to its weak-$*$-topology $\tau_{w^*}$.

Furthermore, for a weighted space $(\Omega,\psi_\Omega)$ (see \cite[Definition~2.1]{cuchiero23}), we denote by $\mathcal{M}_{\psi_\Omega}(\Omega)$ the vector space of signed Radon measures $x: \mathcal{F}_\Omega \rightarrow \mathbb{R}$ with $\int_\Omega \psi_\Omega(\omega) \vert x \vert(d\omega) < \infty$. Then,
\begin{equation}
	\Vert x \Vert_{\mathcal{M}_{\psi_\Omega}(\Omega)} := \sup\left\lbrace \left\vert \int_\Omega f(\omega) x(d\omega) \right\vert: f \in \mathcal{B}_{\psi_\Omega}(\Omega), \, \Vert f \Vert_{\mathcal{B}_{\psi_\Omega}(\Omega)} \leq 1 \right\rbrace
\end{equation}
turns $\mathcal{M}_{\psi_\Omega}(\Omega)$ into a Banach space, where the weighted function space $\mathcal{B}_{\psi_\Omega}(\Omega)$ is defined in \cite[Definition~2.5]{cuchiero23}. Then, $\mathcal{M}_{\psi_\Omega}(\Omega) = \mathcal{B}_{\psi_\Omega}(\Omega)^*$ is by the Riesz representation theorem in \cite[Theorem~2.8]{doersek10} a dual Banach space, which is by the Banach-Alaoglu theorem $\sigma$-compact with respect to its weak-$*$-topology $\tau_{w^*}$.

\section{Weighted spaces and differentiable maps}
\label{sec:w_dom}

For the approximation results on infinite-dimensional manifolds, we endow the input space with a weight function and assume that the output space is a Banach space. This weighted setting is in particular inspired by the works on Kolmogorov equations, splitting schemes of (stochastic) partial differential equations, and generalized Feller processes (see, e.g., \cite{roeckner06,doersek10,cuchiero20}).

In the following, we first introduce our weighted setting on domains given as open subsets of locally convex topological vector spaces, followed by weighted infinite-dimensional manifolds. Later on, we introduce the weighted $\mathcal{B}^k_\Psi$-function space that was under slightly different conditions also studied in \cite{bernstein24,nachbin65,summers68,bierstedt71,prolla71b,nachbin91,triebel92,triebel06,triebel10,cuchiero23}.

\subsection{Weighted domains}

In the following, we shall fix some $k \in \mathbb{N}$ and consider an open subset $U \subseteq X$ of a locally convex topological vector space $(X,\tau_X)$. We refer to Section~\ref{sec:notation} for the mathematical background of locally convex topological vector spaces.

\begin{definition}
	\label{def:w_dom}
	A collection $\Psi := (\psi_j)_{j=0,\ldots,k}$ of weight functions $\psi_j: U \times X^j \rightarrow (0,\infty)$ is called \emph{admissible (on $U$)} if
	\begin{enumerate}
		\item\label{def:w_dom1} for every $j = 0,\ldots,k$ and $R > 0$ the pre-image
		\begin{equation}
			K_{j,R} := \psi_j^{-1}((0,R]) := \left\lbrace (u,v_1,\ldots,v_j) \in U \times X^j: \psi_j(u,v_1,\ldots,v_j) \leq R \right\rbrace
		\end{equation}
		is compact with respect to $\tau_X \times \tau_X^j$, and
		\item\label{def:w_dom2} $\Psi := (\psi_j)_{j=0,\ldots,k}$ is monotone, i.e., there exists a constant $C_\Psi \geq 1$ such that for every $j = 1,\ldots,k$, $\ell = 1,\ldots,j$, $\sigma \in \mathcal{S}_j$, and $(u,v_1,\ldots,v_j) \in U \times X^j$ it holds that
		\begin{equation}
			\psi_\ell(u,v_{\sigma(1)},\ldots,v_{\sigma(\ell)}) \psi_{j-\ell}(u,v_{\sigma(\ell+1)},\ldots,v_{\sigma(j)}) \leq C_\Psi \psi_j(u,v_1,\ldots,v_j).
		\end{equation}
	\end{enumerate}
	In this case, we call $(U,\Psi)$ a weighted domain.
\end{definition}

\begin{remark}
	\label{rem:w_dom}
	If $(U,\Psi)$ is a weighted domain, then the following holds true:
	\begin{enumerate}
		\item\label{rem:w_dom:1} The weight functions $\psi_j: U \times X^j \rightarrow (0,\infty)$ are necessarily lower semicontinuous and bounded from below by a strictly positive constant (see \cite[Remark~2.2~(i)]{cuchiero23}).
		
		\item\label{rem:w_dom:2} The domain $U$ is $\sigma$-compact with respect to $\tau_X$ as $U = \bigcup_{R \in \mathbb{N}} K_{0,R}$.
		
		\item\label{rem:w_dom:3} The locally convex topological vector space $(X,\tau_X)$ is also $\sigma$-compact because of $X = \bigcup_{R \in \mathbb{N}} \pi_1(K_{1,R})$, where $\pi_1(K_{1,R})$ is compact as continuous image of the compact set $K_{1,R}$ under the projection $U \times X \ni (u,v_1) \mapsto \pi_1(u,v_1) := v_1 \in X$.
		
		\item\label{rem:w_dom:4} If $(X,\tau_X)$ is complete, then $X$ is finite-dimensional. Indeed, this follows from Baire's category theorem and \ref{rem:w_dom:3} (see also \cite[Remark~2.2~(iii)]{cuchiero23}). Hence, for an infinite-dimensional domain $U$, we need to consider an incomplete locally convex topological vector space $(X,\tau_X)$ instead of a Banach space or a Fr\'echet space.
		
		\item\label{rem:w_dom:5} If $U = X$ is separable and $\psi_j: X \times X^j \rightarrow (0,\infty)$ is convex, then $\psi_j: X \times X^j \rightarrow (0,\infty)$ is continuous on a convex subset $E \subseteq X \times X^j$ if and only if $E$ is locally compact (see also \cite[Remark~2.2]{cuchiero20}).
	\end{enumerate}
\end{remark}

In the following, we present various examples of weighted domains $(U,\Psi)$, where $U \in \tau_X$ is a subset of a Banach space $(X,\Vert \cdot \Vert_X)$ that is equipped with a weaker topology $\tau_X$ than the norm topology (except $X$ is finite-dimensional).

\begin{lemma}
	\label{lem:w_dom}
	Let $U \subseteq X$ be an open subset of one of the following two types of locally convex topological vector spaces $(X,\tau_X)$:
	\begin{enumerate}
		\item\label{lem:w_dom_cpt} $(X,\Vert \cdot \Vert_X)$ is a Banach space equipped with the initial topology $\tau_X := \tau_{\mathrm{init}}$ of a compact embedding $\Gamma: (X,\Vert \cdot \Vert_X) \hookrightarrow (X_0,\tau_0)$ into another locally convex topological vector space $(X_0,\tau_0)$ such that $\overline{B}^X_r$ is closed with respect to $\tau_{\mathrm{init}}$, for all $r > 0$. Here, $\tau_{\mathrm{init}}$ is the weakest locally convex topology on $X$ such that $\Gamma: (X,\tau_{\mathrm{init}}) \hookrightarrow (X_0,\tau_0)$ is continuous.
		
		\item\label{lem:w_dom_dual} $(X,\Vert \cdot \Vert_X)$ is a dual Banach space equipped with the weak-$*$-topology $\tau_X := \tau_{w^*}$.
	\end{enumerate}
	Moreover, let $\Psi := (\psi_j)_{j=0,\ldots,k}$ be a collection of weight functions of the form
	\begin{equation}
		\label{eq:lem:w_dom}
		U \times X^j \ni (u,v_1,\ldots,v_j) \mapsto \psi_j(u,v_1,\ldots,v_j) = \eta\left( \max(j,1) \left( \delta_{U^c}(u)^{-1} \!+\! \Vert u \Vert_X \right) + \sum_{\ell=1}^j \Vert v_\ell \Vert_X \right) \!\in\! (0,\infty),
	\end{equation}
	where $(U,\tau_{\mathrm{init}}) \ni u \mapsto \delta_{U^c}^{-1}(u) := 1/\inf_{v \in X \setminus U} \Vert u-v \Vert_X \in [0,\infty)$ is assumed to be lower semicontinuous, and where $\eta: [0,\infty) \rightarrow (0,\infty)$ is a continuous and increasing function with $\lim_{r \rightarrow \infty} \eta(r) = \infty$. Then, $(U,\Psi)$ is a weighted domain.
\end{lemma}
\begin{proof}
	For \ref{lem:w_dom_cpt}, we fix some $j = 0,\ldots,k$, $R > 0$, and consider the pre-image $K_{j,R} := \psi_j^{-1}((0,R])$. Then, for every $(u,v_1,\ldots,v_j) \in K_{j,R}$, it holds that
	\begin{equation}
		\max(j,1) \left( \delta_{U^c}(u)^{-1} + \Vert u \Vert_X \right) + \sum_{\ell=1}^j \Vert v_\ell \Vert_X \leq \eta^{-1}_R := \sup\lbrace r \geq 0: \eta(r) \leq R \rbrace < \infty,
	\end{equation}
	which ensures that $K_{j,R} \subseteq \overline{B}^X_{\eta^{-1}_R} \times \big( \overline{B}^X_{\eta^{-1}_R} \big)^j$. Since the product topology $\tau_{\mathrm{init}} \times \tau_{\mathrm{init}}^j$ on $X \times X^j$ coincides with the initial topology induced by the mapping $X \times X^j \ni (u,v_1,\ldots,v_j) \mapsto (\Gamma(u),\Gamma(v_1),\ldots,\Gamma(v_j)) \in (X_0 \times X_0^j,\tau_0 \times \tau_0^j)$, it follows that $K_{j,R}$ is a relatively compact subset of $(U \times X^j,\tau_{\mathrm{init}} \times \tau_{\mathrm{init}}^j)$. In order to show that $K_{j,R}$ is also closed with respect to $\tau_{\mathrm{init}} \times \tau_{\mathrm{init}}^j$, we fix a net $\big( u^{(\gamma)}, v^{(\gamma)}_1, \ldots, v^{(\gamma)}_j \big)_\gamma \subseteq K_{j,R}$ converging to some $(u,v_1,\ldots,v_j) \in U \times X^j$ with respect to $\tau_{\mathrm{init}} \times \tau_{\mathrm{init}}^j$. Then, by using that $\delta_{U^c}^{-1}: (U,\tau_{\mathrm{init}}) \rightarrow [0,\infty)$ as well as $\Vert \cdot \Vert_X: (X,\tau_{\mathrm{init}}) \rightarrow [0,\infty)$ are lower semicontinuous and that $\eta: [0,\infty) \rightarrow (0,\infty)$ is continuous, we conclude that
	\begin{equation}
		\begin{aligned}
			\psi_j(u,v_1,\ldots,v_j) & = \eta\left( \max(j,1) \left( \delta_{U^c}(u)^{-1} + \Vert u \Vert_X \right) + \sum_{\ell=1}^j \Vert v_\ell \Vert_X \right) \\
			& \leq \liminf_\gamma \eta\left( \max(j,1) \big( \delta_{U^c}(u^{(\gamma)})^{-1} + \Vert u^{(\gamma)} \Vert_X \big) + \sum_{\ell=1}^j \Vert v^{(\gamma)}_\ell \Vert_X \right) \\
			& = \liminf_\gamma \psi_j\left( u^{(\gamma)},v^{(\gamma)}_1,\ldots,v^{(\gamma)}_j \right) \leq R,
		\end{aligned}
	\end{equation}
	which shows that $K_{j,R}$ is closed and therefore compact with respect to $\tau_{\mathrm{init}} \times \tau_{\mathrm{init}}^j$.
	
	For \ref{lem:w_dom_dual}, let $(E,\Vert \cdot \Vert_E)$ be a predual for $(X,\Vert \cdot \Vert_X)$. Then, for every fixed $j = 0,\ldots,k$, we use that $(X \times X^j,\Vert \cdot \Vert_{X \times X^j})$ is a dual Banach space with predual $(E \times E^j,\Vert \cdot \Vert_{E \times E^j})$, whose weak-$*$-topology coincides with the product topology $\tau_{w^*} \times \tau_{w^*}^j$ on $X \times X^j$. Thus, for every $R > 0$, we use that $K_{j,R} := \psi_j^{-1}((0,R])$ is bounded with respect to $\Vert \cdot \Vert_{X \times X^j}$ to conclude that $K_{j,R}$ is by the Banach-Alaoglu theorem a compact subset of $(U \times X^j,\tau_{w^*} \times \tau_{w^*}^j)$.
\end{proof}

\begin{remark}
	\label{rem:dual_banach}
	By \cite[Theorem~1]{kaijser77} of S.~Kaijser (formally generalizing the Dixmier-Ng theorem in \cite{dixmier48,ng71}), a Banach space $(X,\Vert \cdot \Vert_X)$ is a dual Banach space if there exists a point separating subset $\mathcal{L} \subseteq X^*$ such that the unit ball $\overline{B}^X_1$ is compact with respect to the weak topology on $X$ induced by $\mathcal{L} \subseteq X^*$. Thus, a compactly embedded Banach space $(X,\Vert \cdot \Vert_X)$ as in \ref{lem:w_dom_cpt} can be turned into a dual Banach space (see also \cite[Appendix~A]{cuchiero23}).
\end{remark}

In the following, we give some examples of weighted domains $(U,\Psi)$. We refer to Section~\ref{sec:model_sp} for the precise definition of some of the vector spaces that appear below.

\begin{example}
	\label{ex:w_dom}
	The following are examples of weighted domains $(U,\Psi)$, where $U \subseteq X$ is an open subset of a locally convex topological vector space $(X,\tau_X)$, and where $\Psi := (\psi_j)_{j=0,\ldots,k}$ is a collection of weight functions of the form \eqref{eq:lem:w_dom}.
	\begin{enumerate}
		\item\label{ex:w_dom:cpt} First, we consider an open subset $U \in \tau_{\mathrm{init}}$ of a Banach space $(X,\Vert \cdot \Vert_X)$ equipped with the initial topology $\tau_{\mathrm{init}}$ of a compact embedding as in Lemma~\ref{lem:w_dom}~\ref{lem:w_dom_cpt}:
		\begin{enumerate}
			\item[\labeltext{(a)}{ex:w_dom:cpt:Rd}] Euclidean space $X := \mathbb{R}^d$ with $\tau_{\mathrm{init}}$ generated by the Euclidean norm $\Vert \cdot \Vert$.
			\item[\labeltext{(b)}{ex:w_dom:cpt:hol}] $\alpha$-H\"older space $X := C^\alpha(S;Z)$ with $\tau_{\mathrm{init}}$ generated by the compact embedding \allowbreak $(C^\alpha(S;Z),\Vert \cdot \Vert_\alpha) \hookrightarrow (C^\alpha(S;Z),\tau_{\alpha'})$ (see \cite[Theorem~A.4]{cuchiero23}), where $0 \leq \alpha' < \alpha < 1$, $(S,d_S)$ is a compact metric space, and $(Z,\Vert \cdot \Vert_Z)$ is a dual Banach space.
			\item[\labeltext{(c)}{ex:w_dom:cpt:rellich}] Sobolev space $X \!:=\! W^{1,p}(\Omega)$ with $\tau_{\mathrm{init}}$ induced by the compact embedding $(W^{1,p}(\Omega),\Vert \cdot \Vert_{W^{1,p}(\Omega)}) \hookrightarrow (L^q(\Omega),\Vert \cdot \Vert_{L^q(\Omega)})$ (see \cite[Theorem~9.16]{brezis11}), where $p \in [1,d)$, $q \in [1,\frac{dp}{d-p})$, and $\Omega \subset \mathbb{R}^d$ is an open bounded Lipschitz domain.
			\item[\labeltext{(d)}{ex:w_dom:cpt:besov}] Besov space $X := B^s_{p,q}(\Omega)$ with $\tau_{\mathrm{init}}$ generated by the compact embedding $(B^s_{p,q}(\Omega),\Vert \cdot \Vert_{B^s_{p,q}(\Omega)}) \hookrightarrow (B^{s'}_{p',q'}(\Omega),\Vert \cdot \Vert_{B^{s'}_{p',q'}(\Omega)})$ (see \cite[Theorem~1.97]{triebel06}), where $\Omega \subset \mathbb{R}^d$ is open and bounded, $p,q \in (1,\infty]$ (with dual exponents $p',q' \in [1,\infty)$), and $s' \in (-\infty,s)$ with $s-\frac{d}{p} > s'-\frac{d}{p'}$.
		\end{enumerate}
		\item\label{ex:w_dom:dual} Second, we consider an open subset $U \in \tau_{w^*}$ of a dual Banach space $(X,\Vert \cdot \Vert_X)$ equipped with the weak-$*$-topology $\tau_{w^*}$ as in Lemma~\ref{lem:w_dom}~\ref{lem:w_dom_dual}:
		\begin{enumerate}
			\item[\labeltext{(e)}{ex:w_dom:dual:Rd}] Euclidean space $X := \mathbb{R}^d \cong (\mathbb{R}^d)^*$ is a dual Banach space.
			\item[\labeltext{(f)}{ex:w_dom:dual:hol}] $\alpha$-H\"older space $X := C^\alpha(S;Z)$ is a dual Banach space (see \cite[Theorem~A.4]{cuchiero23}), where $\alpha \in (0,1]$, $(S,d_S)$ is a compact metric space, and $(Z,\Vert \cdot \Vert_Z)$ is a dual Banach space.
			\item[\labeltext{(g)}{ex:w_dom:dual:sk}] $\alpha$-H\"older Skorokhod space $X := D^{\alpha,1}([0,T];Z)$ is a dual Banach space (see Theorem~\ref{thm:sk_dual}), where $\alpha \in (0,1)$, $T > 0$, and $(Z,\Vert \cdot \Vert_Z)$ is a dual Banach space.
			\item[\labeltext{(h)}{ex:w_dom:dual:Lp}] $L^p$-space $X := L^p(\Omega;Z) \cong L^{p'}(\Omega;E)^*$ is a dual Banach space (see \cite[Theorem~1.3.10]{hytoenen16}), where $p \in (1,\infty]$ (with dual exponent $p' \in [1,\infty)$), $(\Omega,\mathcal{F},\mu)$ is a $\sigma$-finite measure space, $(Z,\Vert \cdot \Vert_Z)$ is a dual Banach space having the Radon-Nikodym property with respect to $\mu$ (see \cite[Definition~1.3.9]{hytoenen16}), and $(E,\Vert \cdot \Vert_E)$ is a predual for $(Z,\Vert \cdot \Vert_Z)$.
			\item[\labeltext{(i)}{ex:w_dom:dual:bv}] Space $X := BV(\Omega)$ of integrable functions $x: \Omega \rightarrow \mathbb{R}$ with bounded variation is a dual Banach space (see \cite[Remark~3.12]{ambrosio00}), where $\Omega \subseteq \mathbb{R}^d$ is an open subset.
			\item[\labeltext{(j)}{ex:w_dom:dual:besov}] Besov space $X := B^s_{p,q}(\mathbb{R}^d) \cong B^{-s}_{p',q'}(\mathbb{R}^d)^*$ is a dual Banach space (see \cite[Theorem~2.11.2~(i)]{triebel10}), where $p,q \in (1,\infty]$ (with dual exponents $p',q' \in [1,\infty)$) and $s \in \mathbb{R}$.
			\item[\labeltext{(k)}{ex:w_dom:dual:measure}] Weighted measure space $X := \mathcal{M}_{\psi_\Omega}(\Omega) \cong \mathcal{B}_{\psi_\Omega}(\Omega)^*$ is a dual Banach space (see \cite[Theorem~2.4]{doersek10}), where $(\Omega,\psi_\Omega)$ is a weighted space in the sense of \cite[Definition~2.1]{cuchiero23} and $\mathcal{B}_{\psi_\Omega}(\Omega)$ consists of weighted functions defined on $(\Omega,\psi_\Omega)$ (see \cite[Definition~2.5]{cuchiero23} for the precise definition).
		\end{enumerate}
	\end{enumerate}
\end{example}

\subsection{\texorpdfstring{$\mathcal{B}^k_\Psi$}{Bkpsi}-maps over weighted domains}
\label{sec:BPsik_maps}

In this section, we introduce differentiable maps on a weighted domain $(U,\Psi)$ taking values in a Banach space $(Y,\Vert \cdot \Vert_Y)$. Let $U \subseteq X$ be an open subset of a $\sigma$-compact locally convex topological vector space $(X,\tau_X)$. Moreover, for a given set $\mathfrak{Q}_X$ of seminorms on $X$, we assume that the admissible collection of weight functions $\Psi := (\psi_j)_{j=0,\ldots,k}$ grows fast enough such that for every $j = 0,\ldots,k$ and $p_X \in \mathfrak{P}_{(X,\tau_X)} \cup \mathfrak{Q}_X$ it holds that
\begin{equation}
	\label{eq:ass:w_growth}
	\lim_{R \rightarrow \infty} \sup_{(u,v_1,\ldots,v_j) \in (U \times X^j) \setminus K_{j,R}} \frac{p_X(v_1) \cdots p_X(v_j)}{\psi_j(u,v_1,\ldots,v_j)} = 0.
\end{equation}
Then, we define $C^k_b(U;Y) \subseteq C^k(U;Y)$ as the vector subspace of maps $f \in C^k(U;Y)$ such that $(d^j f(u;\cdot))_{u \in U} \subseteq L(X^j;Y)$ is equicontinuous, for all $j = 0,\ldots,k$, i.e., there exists a constant $C_f \geq 0$ and a seminorm $p_X \in \mathfrak{P}_{(X,\tau_X)}$ such that for every $j = 0,\ldots,k$, $u \in U$, and $v_1,\ldots,v_j \in X$ we have
\begin{equation}
	\label{eq:Ckb:glob_bd}
	\Vert d^j f(u;v_1,\ldots,v_j) \Vert_Y \leq C_f p_X(v_1) \cdots p_X(v_j),
\end{equation}
with $d^0 f(u) := f(u)$. Furthermore, we define the weighted norm
\begin{equation}
	\label{eq:def:w_norm}
	\Vert f \Vert_{\mathcal{B}^k_\Psi(U;Y)} = \max_{j=0,\ldots,k} \sup_{(u,v_1,\ldots,v_j) \in U \times X^j} \frac{\Vert d^j f(u;v_1,\ldots,v_j) \Vert_Y}{\psi_j(u,v_1,\ldots,v_j)},
\end{equation}
for $f \in C^k_b(U;Y)$, which is well-defined by \eqref{eq:ass:w_growth}--\eqref{eq:Ckb:glob_bd} and Remark~\ref{rem:w_dom}~\ref{rem:w_dom:1}. Now, we can introduce the weighted function space $\mathcal{B}^k_\Psi(U;Y)$.

\begin{definition}
	\label{def:BkPsi}
	Let $(U,\Psi)$ be a weighted domain. Then, we define $\mathcal{B}^k_\Psi(U;Y)$ as the closure of $C^k_b(U;Y)$ with respect to $\Vert \cdot \Vert_{\mathcal{B}^k_\Psi(U;Y)}$, which is a Banach space under the weighted norm defined in \eqref{eq:def:w_norm}. If $Y = \mathbb{R}$, we shall only write $\mathcal{B}^k_\Psi(U)$.
\end{definition}

Since the weight functions $\psi_j: U \times X^j \rightarrow (0,\infty)$, $j = 0,\ldots,k$, grow on the compact pre-images $\psi_j^{-1}((0,R])$, the derivatives of a map $f \in \mathcal{B}^k_\Psi(U;Y)$ are typically unbounded. However, the growth of $d^j f: U \times X^j \rightarrow Y$ is controlled by $\psi_j: U \times X^j \rightarrow (0,\infty)$.

\begin{remark}
	\label{rem:banach_out}
	For simplicity, we always assume that the output space is a Banach space $(Y,\Vert \cdot \Vert_Y)$. However, the following results can be generalized to locally convex topological vector spaces $(Y,\tau_Y)$ as output space.
\end{remark}

In order to characterize maps in $\mathcal{B}^k_\Psi(U;Y)$ in Proposition~\ref{prop:BPsik_equiv} below, we first provide some examples of weighted domains, which have the (bounded) approximation property ((B)AP). To this end, we assume that the Banach space $(X,\Vert \cdot \Vert_X)$ is equipped with a weaker topology $\tau_X$ than the norm topology (see Lemma~\ref{lem:w_dom}). For more background on (B)AP, we refer to Section~\ref{sec:notation}.

\begin{lemma}
	\label{lem:ap_cpt}
	Let $(X,\Vert \cdot \Vert_X)$ be a Banach space equipped with the initial topology $\tau_{\mathrm{init}}$ of a compact embedding $\Gamma: (X,\Vert \cdot \Vert_X) \rightarrow (X_0,\tau_{X_0})$ as in Lemma~\ref{lem:w_dom}~\ref{lem:w_dom_cpt}. Moreover, let $U \in \tau_{\mathrm{init}}$ be an open subset and assume that $(X_0,\tau_{X_0})$ has AP (resp., $\mathfrak{P}_{(X_0,\tau_{X_0})}$-BAP) with finite rank operators $(T_{0,\gamma})_\gamma \in (X_0,\tau_{X_0})^* \otimes \Gamma(X)$ satisfying $T_{0,\gamma}(\Gamma(U)) \subseteq \Gamma(U)$. In addition, let $\Psi = (\psi_j)_{j=0,\ldots,k}$ be a collection of weight functions of the form \eqref{eq:lem:w_dom} with $\eta: [0,\infty) \rightarrow (0,\infty)$ additionally satisfying $\lim_{r \rightarrow \infty} \frac{r^k}{\eta(r)} = 0$. Then, $(U,\tau_{\mathrm{init}})$ has AP (resp., $\mathfrak{P}_{(X,\tau_{\mathrm{init}})}$-BAP) and \eqref{eq:ass:w_growth} is satisfied.
\end{lemma}
\begin{proof}
	First, we observe that the image $\Gamma(K)$ of any relatively compact subset $K$ of $(X,\tau_{\mathrm{init}})$ under the continuous embedding $\Gamma: (X,\Vert \cdot \Vert_X) \rightarrow (X_0,\tau_{X_0})$ is relatively compact in $(X_0,\tau_{X_0})$. Now, since $(X_0,\tau_{X_0})$ has AP, there exists a net of finite rank operators $(T_{0,\gamma})_\gamma \in (X_0,\tau_{X_0})^* \otimes \Gamma(X)$ such that for every relatively compact subset $K$ of $(X_0,\tau_{X_0})$ and $p_{X_0} \in \mathfrak{P}_{(X_0,\tau_{X_0})}$ we have
	\begin{equation}
		\label{eq:lem:ap_cpt:proof1}
		\lim_\gamma \sup_{x_0 \in \Gamma(K)} p_{X_0}(x_0 - T_{0,\gamma}(x_0)) = 0.
	\end{equation}
	Then, for every $\gamma$, there exists some $T_\gamma \in (X,\tau_{\mathrm{init}})^* \otimes X$ with $\Gamma \circ T_\gamma = T_{0,\gamma} \circ \Gamma$ and $\Gamma(T_\gamma(U)) = T_{0,\gamma}(\Gamma(U)) \subseteq \Gamma(U)$ implying that $T_\gamma(U) \subseteq U$. Hence, \eqref{eq:lem:ap_cpt:proof1} ensures for every relatively compact subset $K$ of $(X,\tau_{\mathrm{init}})$ and $\big( x \mapsto p_X(x) := p_{X_0}(\Gamma(x)) \big) \in \mathfrak{P}_{(X,\tau_{\mathrm{init}})}$ that
	\begin{equation}
		\begin{aligned}
			\lim_\gamma \sup_{x \in K} p_X(x - T_\gamma(x)) & = \lim_\gamma \sup_{x \in K} p_{X_0}(\Gamma(x - T_\gamma(x))) \\
			& = \lim_\gamma \sup_{x \in K} p_{X_0}(\Gamma(x) - \Gamma(T_\gamma(x))) \\
			& = \lim_\gamma \sup_{x_0 \in \Gamma(K)} p_{X_0}(x_0 - T_{0,\gamma}(x_0)) = 0,
		\end{aligned}
	\end{equation}
	which shows that $(U,\tau_{\mathrm{init}})$ has AP. 
	
	Moreover, if $(X_0,\tau_{X_0})$ additionally has $\mathfrak{P}_{(X_0,\tau_{X_0})}$-BAP, then for every $p_{X_0} \in \mathfrak{P}_{(X_0,\tau_{X_0})}$ there exists some $q_{X_0} \in \mathfrak{P}_{(X_0,\tau_{X_0})}$ such that for every $\gamma$ and $x_0 \in X_0$ it holds that $p_{X_0}(T_{0,\gamma}(x_0)) \leq q_{X_0}(x_0)$. Hence, for every $\big( x \mapsto p_X(x) := p_{X_0}(\Gamma(x)) \big) \in \mathfrak{P}_{(X,\tau_{\mathrm{init}})}$, we use $\big( x \mapsto q_X(x) := q_{X_0}(\Gamma(x)) \big) \in \mathfrak{P}_{(X,\tau_{\mathrm{init}})}$ to conclude for every $\gamma$ and $x \in X$ that
	\begin{equation}
		p_X(T_\gamma(x)) = p_{X_0}(\Gamma(T_{\gamma}(x))) = p_{X_0}(T_{0,\gamma}(\Gamma(x))) \leq q_{X_0}(\Gamma(x)) = q_X(x),
	\end{equation}
	which shows that $(U,\tau_{\mathrm{init}})$ has $\mathfrak{P}_{(X,\tau_{\mathrm{init}})}$-BAP.
	
	Finally, by using that $\Gamma: (X,\Vert \cdot \Vert_X) \rightarrow (X_0,\tau_{X_0})$ is continuous, i.e., that for every $\big( x \mapsto p_X(x) := p_{X_0}(\Gamma(x)) \big) \in \mathfrak{P}_{(X,\tau_{\mathrm{init}})}$ there exists a constant $C_{\Gamma,p_X} \geq 1$ such that for every $v \in X$ it holds that $p_X(v) := p_{X_0}(\Gamma(v)) \leq C_{\Gamma,p_X} \Vert v \Vert_X$, and the assumption $\lim_{r \rightarrow \infty} \frac{r^k}{\eta(r)} = 0$, we obtain for every $j = 0,\ldots,k$ and $p_X \in \mathfrak{P}_{(X,\tau_{\mathrm{init}})}$ that
	\begin{equation}
		\begin{aligned}
			& \lim_{R \rightarrow \infty} \max_{j=0,\ldots,k} \sup_{(u,v_1,\ldots,v_j) \in (U \times X^j) \setminus K_{j,R}} \frac{p_X(v_1) \cdots p_X(v_j)}{\psi_j(u,v_1,\ldots,v_j)} \\
			& \quad\quad \leq C_{\Gamma,p_X}^j \lim_{R \rightarrow \infty} \max_{j=0,\ldots,k} \sup_{(u,v_1,\ldots,v_j) \in (U \times X^j) \setminus K_{j,R}} \frac{\Vert v_1 \Vert_X \cdots \Vert v_j \Vert_X}{\eta\left( \max(j,1) \left( \delta_{U^c}(u)^{-1} + \Vert u \Vert_X \right) + \sum_{\ell=1}^j \Vert v_\ell \Vert_X \right)} \\
			& \quad\quad \leq C_{\Gamma,p_X}^k \lim_{R \rightarrow \infty} \max_{j=0,\ldots,k} \sup_{(u,v_1,\ldots,v_j) \in (U \times X^j) \setminus K_{j,R}} \frac{\left( 1 + \Vert u \Vert_X + \sum_{\ell=1}^j \Vert v_\ell \Vert_X \right)^j}{\eta\left( \max(j,1) \Vert u \Vert_X + \sum_{\ell=1}^j \Vert v_\ell \Vert_X \right)} = 0,
		\end{aligned}
	\end{equation}
	which shows that \eqref{eq:ass:w_growth} is satisfied.
\end{proof}

\begin{lemma}
	\label{lem:ap_dual}
	Let $(X,\Vert \cdot \Vert_X)$ be a dual Banach space equipped with the weak-$*$-topology $\tau_{w^*}$. Moreover, let $U \in \tau_{w^*}$ be an open subset with $\pi(U) \subseteq U$, for all projections $\pi \in X^* \otimes X$. In addition, let $\Psi = (\psi_j)_{j=0,\ldots,k}$ be a collection of weights of the form \eqref{eq:lem:w_dom} with $\eta: [0,\infty) \rightarrow (0,\infty)$ additionally satisfying $\lim_{r \rightarrow \infty} \frac{r^k}{\eta(r)} = 0$. Then, $(U,\tau_{w^*})$ has AP and \eqref{eq:ass:w_growth} is satisfied. Furthermore, if the predual $(E,\Vert \cdot \Vert_E)$ of $(X,\Vert \cdot \Vert_X)$ has BAP with finite rank operators $(Q_\gamma)_\gamma$ satisfying $Q_\gamma^*(U) \subseteq U$, then $(U,\tau_{w^*})$ has $\Vert \cdot \Vert_X$-BAP.
\end{lemma}
\begin{proof}
	For fixed linearly independent $e_1,\ldots,e_N \in E$, we consider the seminorm $\big( x \mapsto p_X(x) := \max_{n=1,\ldots,N} \vert \langle x, e_n \rangle_{X \times E} \vert \big) \in \mathfrak{P}_{(X,\tau_{w^*})}$. Then, by using that $E^* \cong X$ is due to the Hahn-Banach theorem point separating on $E$, there exist some linearly independent $x_1,\ldots,x_N \in X$ such that 
	\begin{align}
		\langle x_m, e_n \rangle_{X \times E} & = \delta_{m,n} & & \text{for all } m,n = 1,\ldots,N, \text{ and } \\
		\langle x, e_n \rangle_{X \times E} & = 0 & & \text{for all } x \in X_{1:N}^\perp \text{ and } n = 1,\ldots,N,
	\end{align}
	where $X_{1:N} := \linspan\lbrace x_1,\ldots,x_N \rbrace$ and $X_{1:N}^\perp$ satisfy $X_{1:N} \oplus X_{1:N}^\perp = X$ (see also \cite[Corollary~4.2.2]{schaefer99}). From this, we define the finite rank operator $\big( x \mapsto T_{e_{1:N}}(x) := \sum_{n=1}^N \langle x, e_n \rangle_{X \times E} x_n \big) \in (X,\tau_{w^*})^* \otimes X$ satisfying $T_{e_{1:N}}(x) = x$ for any $x \in X_{1:N}$ and therefore $T_{e_{1:N}}(U) \subseteq U$ (as $T_{e_{1:N}}$ is the projection onto $X_{1:N}$). Thus, for every relatively compact subset $K$ of $(X,\tau_{w^*})$, it holds that
	\begin{equation}
		\begin{aligned}
			\sup_{x \in K} p_X\left( x - T_{e_{1:N}}(x) \right) & = \sup_{x \in K} \max_{n=1,\ldots,N} \left\vert \Big\langle x - \sum_{m=1}^N \langle x, e_m \rangle_{X \times E} x_m, e_n \Big\rangle_{X \times E} \right\vert \\
			& = \sup_{x \in K} \max_{n=1,\ldots,N} \Bigg\vert \langle x, e_n \rangle_{X \times E} - \sum_{m=1}^N \langle x, e_m \rangle_{X \times E} \underbrace{\langle x_m, e_n \rangle_{X \times E}}_{\delta_{m,n}} \Bigg\vert = 0,
		\end{aligned}
	\end{equation}
	which shows that the net $(T_{e_{1:N}})_{e_{1:N}} \subseteq (X,\tau_{w^*})^* \otimes X$ converges to $\id_X: X \rightarrow X$ uniformly on each relatively compact subset of $(X,\tau_{w^*})$, whence $(U,\tau_{w^*})$ has AP. 
	
	Moreover, if $(E,\Vert \cdot \Vert_E)$ has BAP, there exists some $\lambda \geq 1$ and a net of finite rank operators $(Q_\gamma)_\gamma \subseteq E^* \otimes E$ with $\Vert Q_\gamma \Vert_{L(E;E)} \leq \lambda$, for all $\gamma$, such that for every relatively compact subset $L$ of $(E,\Vert \cdot \Vert_E)$, it holds that
	\begin{equation}
		\lim_\gamma \sup_{e \in L} \Vert e - Q_\gamma(e) \Vert_E = 0.
	\end{equation}
	Hence, by defining $T_\gamma := Q_\gamma^* \in (X,\tau_{w^*})^* \otimes X$, we conclude for every seminorm $\big( x \mapsto p_X(x) := \max_{n=1,\ldots,N} \vert \langle x, e_n \rangle_{X \times E} \vert \big) \in \mathfrak{P}_{(X,\tau_{w^*})}$ and relatively compact subset $K$ of $(X,\tau_{w^*})$ that
	\begin{equation}
		\begin{aligned}
			\lim_\gamma \sup_{x \in K} p_X(x - T_\gamma(x)) & = \lim_\gamma \sup_{x \in K} \max_{n=1,\ldots,N} \left\vert \langle (\id_X-T_\gamma)(x), e_n \rangle_{X \times E} \right\vert \\
			& = \lim_\gamma \sup_{x \in K} \max_{n=1,\ldots,N} \left\vert \langle x, (\id_E - Q_\gamma) e_n \rangle_{X \times E} \right\vert \\
			& \leq \left( \sup_{x \in K} \Vert x \Vert_X \right) \lim_\gamma \max_{n=1,\ldots,N} \Vert e_n - Q_\gamma(e_n) \Vert_E = 0.
		\end{aligned}
	\end{equation}
	In addition, for every $\big( x \mapsto \widetilde{p}_X(x) := \max_{m=1,\ldots,M} \vert \langle x, \widetilde{e}_m \rangle_{X \times E} \vert \big) \in \mathfrak{P}_{(X,\tau_{w^*})}$, we have
	\begin{equation}
		\begin{aligned}
			\widetilde{p}_X\left( T_\gamma(x) \right) & = \max_{m=1,\ldots,M} \left\vert \langle T_\gamma(x), \widetilde{e}_m \rangle_{X \times E} \right\vert = \max_{m=1,\ldots,M} \left\vert \langle x, Q_\gamma(\widetilde{e}_m) \rangle_{X \times E} \right\vert \\
			& \leq \Vert x \Vert_X \Vert Q_\gamma \Vert_{L(E;E)} \max_{m=1,\ldots,M} \Vert \widetilde{e}_m \Vert_E \leq \left( \lambda \max_{m=1,\ldots,M} \Vert \widetilde{e}_m \Vert_E \right) \Vert x \Vert_X,
		\end{aligned}
	\end{equation}
	which shows that $(U,\tau_{w^*})$ has $\Vert \cdot \Vert_X$-BAP.	
	
	Finally, by using that for every $\big( x \mapsto p_X(x) := \max_{n=1,\ldots,N} \vert \langle x, e_n \rangle_{X \times E} \vert \big) \in \mathfrak{P}_{(X,\tau_{w^*})}$ there exists a constant $C_{p_X} > 0$ such that for every $x \in X$ it holds that $p_X(x) \leq C_{p_X} \Vert x \Vert_X$ and that $\lim_{r \rightarrow \infty} \frac{r^k}{\eta(r)} = 0$, we obtain for every $j = 0,\ldots,k$ and $p_X \in \mathfrak{P}_{(X,\tau_{w^*})}$ that
	\begin{equation}
		\begin{aligned}
			& \lim_{R \rightarrow \infty} \max_{j=0,\ldots,k} \sup_{(u,v_1,\ldots,v_j) \in (U \times X^j) \setminus K_{j,R}} \frac{p_X(v_1) \cdots p_X(v_j)}{\psi_j(u,v_1,\ldots,v_j)} \\
			& \leq C_{p_X}^j \lim_{R \rightarrow \infty} \max_{j=0,\ldots,k} \sup_{(u,v_1,\ldots,v_j) \in (U \times X^j) \setminus K_{j,R}} \frac{\Vert v_1 \Vert_X \cdots \Vert v_j \Vert_X}{\eta\left( \max(j,1) \left( \delta_{U^c}(u)^{-1} + \Vert u \Vert_X \right) + \sum_{\ell=1}^j \Vert v_\ell \Vert_X \right)} \\
			& \leq C_{p_X}^j \lim_{R \rightarrow \infty} \max_{j=0,\ldots,k} \sup_{(u,v_1,\ldots,v_j) \in (U \times X^j) \setminus K_{j,R}} \frac{\left( \max(j,1) \Vert u \Vert_X + \sum_{\ell=1}^j \Vert v_\ell \Vert_X \right)^j}{\eta\left( \max(j,1) \Vert u \Vert_X + \sum_{\ell=1}^j \Vert v_\ell \Vert_X \right)} = 0,
		\end{aligned}
	\end{equation}
	which shows that \eqref{eq:ass:w_growth} is satisfied.
\end{proof}

In the following, we characterize maps in $\mathcal{B}^k_\Psi(U;Y)$, which extends \cite[Theorem~2.7]{doersek10} and \cite[Lemma~2.3]{cuchiero23} to differentiable maps. The proof is given in Appendix~\ref{app:BPsik_equiv}.

\begin{proposition}
	\label{prop:BPsik_equiv}
	Let $(U,\Psi)$ be a weighted domain satisfying \eqref{eq:ass:w_growth}, where $K_{j,R} := \psi_j^{-1}((0,R])$ denotes the compact pre-image of the admissible collection $\Psi := (\psi_j)_{j=0,\ldots,k}$ of weight functions, $j = 0,\ldots,k$ and $R > 0$. Then, the following holds true:
	\begin{enumerate}
		\item\label{prop:BPsik_equiv:1} If $f \in \mathcal{B}^k_\Psi(U;Y)$, then $f \in C^k_{loc}(U;Y)$ and it holds that
		\vspace{-0.03cm}
		\begin{equation}
			\label{eq:prop:BPsik_equiv:1}
			\lim_{R \rightarrow \infty} \max_{j=0,\ldots,k} \sup_{(u,v_1,\ldots,v_j) \in (U \times X^j) \setminus K_{j,R}} \frac{\left\Vert d^j f(u;v_1,\ldots,v_j) \right\Vert_Y}{\psi_j(u,v_1,\ldots,v_j)} = 0.
			\vspace{-0.03cm}
		\end{equation}
		\item\label{prop:BPsik_equiv:2} Let $f \in C^k_{loc}(U;Y)$ satisfy
		\vspace{-0.03cm}
		\begin{equation}
			\label{eq:prop:BPsik_equiv:2}
			\lim_{R \rightarrow \infty} \max_{j=0,\ldots,k} \sup_{(u,v_1,\ldots,v_j) \in (U \times X^j) \setminus K_{j,R}} \frac{\left\Vert d^j f(u;v_1,\ldots,v_j) \right\Vert_Y}{\psi_j(u,v_1,\ldots,v_j)} = 0.
			\vspace{-0.03cm}
		\end{equation}
		Moreover, if $(X,\tau_X)$ is not locally compact, we assume additionally that $(U,\tau_X)$ has AP with net of finite rank operators $(T_\gamma)_\gamma \subseteq X^* \otimes X$ satisfying
		\vspace{-0.03cm}
		\begin{equation}
			\label{eq:prop:BPsik_equiv:3}
			\lim_{R \rightarrow \infty} \sup_\gamma \max_{j=0,\ldots,k \atop \mathcal{L} \subseteq \lbrace 1,\ldots,j \rbrace} \sup_{(u,v_1,\ldots,v_j) \in (U \times X^j) \setminus K_{j,R}} \frac{\Vert d^{\vert\mathcal{L}\vert} f(T_\gamma(u);(T_\gamma(v_\ell))_{\ell \in \mathcal{L}}) \Vert_Y}{\psi_\mathcal{L}(u,v_\mathcal{L})} = 0.
			\vspace{-0.03cm}
		\end{equation}
		Then, $f \in \mathcal{B}^k_\Psi(U;Y)$.
	\end{enumerate}
\end{proposition}

Note that \ref{prop:BPsik_equiv:1} is a straightforward generalization of \cite[Theorem~2.7]{doersek10} to differentiable maps. However, for \ref{prop:BPsik_equiv:2}, we need to assume the approximation property (AP) and \eqref{eq:prop:BPsik_equiv:3} (if $(X,\tau_X)$ is not locally compact), which is more restrictive than the original result in \cite[Theorem~2.7]{doersek10} for $\mathcal{B}^0_\Psi$-maps. There, the Tietze extension theorem is applied to extend a $C^0_{loc}$-map beyond compacta.

\subsection{Weighted manifolds and \texorpdfstring{$\mathcal{B}^k_\Psi$}{Bkpsi}-maps thereon}

For some $k \in \mathbb{N}_0 \cup \lbrace \infty \rbrace$, we now consider a $C^k_{loc}$-manifold $(M,\tau_M)$, which we endow similarly as in Definition~\ref{def:w_dom} with a collection of weight functions $\Psi := (\psi_j)_{j=0,\ldots,k}$. For more details on the notion of manifolds, we refer to Section~\ref{sec:manifold}.

\begin{definition}
	\label{def:w_mfd}
	Let $(M,\tau_M)$ be a $C^k_{loc}$-manifold with atlas $(U_i,\phi_i)_{i \in I}$ over the model spaces $(X_i,\tau_{X_i})_{i \in I}$. Then, a collection $\Psi := (\psi_j)_{j=0,\ldots,k}$ of weight functions $\psi_j: T^j M \rightarrow (0,\infty)$ is called \emph{admissible (on $M$)} if for every $i \in I$ the collection $\Psi_i := (\psi_{i,j})_{j=0,\ldots,k}$ of push-forward weight functions defined by
	\begin{equation}
		\label{eq:def:w_mfd:1}
		\psi_{i,j} := \psi_j \circ \Phi_i^{-j}: \phi_i(U_i) \times X_i^j \rightarrow (0,\infty)
	\end{equation}
	is admissible on $\phi_i(U_i)$, i.e., if for every $i \in I$ the pair $(\phi_i(U_i),\Psi_i)$ is a weighted domain. In this case, we call $(M,\Psi)$ a \emph{weighted $C^k_{loc}$-manifold}.
\end{definition}

Note that the admissibility of $\Psi := (\psi_j)_{j=0,\ldots,k}$ is atlas-dependent because an intrinsic (global) version is not suitable for our approximation results (see also Remark~\ref{rem:atlas} below).

\begin{remark}
	\label{rem:w_mfd}
	If $(M,\Psi)$ is a weighted $C^k_{loc}$-manifold, then it holds for every $i \in I$ that:
	\begin{enumerate}
		\item\label{rem:w_mfd:1} $\phi_i(U_i)$ is $\sigma$-compact with respect to $\tau_{X_i}$ (see Remark~\ref{rem:w_dom}~\ref{rem:w_dom:2}). Hence, by using the continuous function $\phi_i^{-1}: \phi_i(U_i) \rightarrow U_i$, the set $U_i$ is $\sigma$-compact with respect to $\tau_M$.
		\item\label{rem:w_mfd:2} $(X_i,\tau_{X_i})$ is also $\sigma$-compact (see Remark~\ref{rem:w_dom}~\ref{rem:w_dom:3}).
		\item\label{rem:w_mfd:3} If $(X_i,\tau_{X_i})$ is complete, then $X_i$ is finite-dimensional (see Remark~\ref{rem:w_dom}~\ref{rem:w_dom:4}).
	\end{enumerate}
	Hence, for a weighted $C^k_{loc}$-manifold $(M,\Psi)$, the following holds true:
	\begin{enumerate}[resume]
		\item\label{rem:w_mfd:4} If $\vert I \vert < \infty$, then $(M,\tau_M)$ is $\sigma$-compact as $M = \bigcup_{i \in I} U_i$ with $\sigma$-compact $U_i$ (see \ref{rem:w_dom:1}).
		\item\label{rem:w_mfd:5} If $(M,\tau_M)$ is a Banach manifold or a Fr\'echet manifold, i.e., the model spaces $(X_i,\tau_{X_i})_{i \in I}$ are complete, then $(M,\tau_M)$ is finite-dimensional (see \ref{rem:w_mfd:3}). Hence, for an infinite-di- mensional manifold $(M,\tau_M)$, we necessarily have to consider \emph{incomplete} locally convex topological vector spaces $(X_i,\tau_{X_i})_{i \in I}$ as model spaces.
	\end{enumerate}
\end{remark}

Now, we relate the pre-images of the weight functions in $\Psi := (\psi_j)_{j=0,\ldots,k}$ to the pre-images of the push-forward weights in $\Psi_i := (\psi_{i,j})_{j=0,\ldots,k}$, for $i \in I$ (see \eqref{eq:def:w_mfd:1}).

\begin{lemma}
	\label{lem:w_pfwd}
	Let $(M,\tau_M)$ be a $C^k_{loc}$-manifold over model spaces $(X_i,\tau_{X_i})_{i \in I}$ and let $\Psi := (\psi_j)_{j=0,\ldots,k}$ be a collection of weight functions $\psi_j: T^j M \rightarrow (0,\infty)$, $j = 0,\ldots,k$. Then:
	\begin{enumerate}
		\item\label{lem:w_pfwd:1} If for every $j = 0,\ldots,k$ and $R > 0$ the pre-image
		\begin{equation}
			\label{eq:lem:w_pfwd:1}
			K_{j,R} := \psi_j^{-1}((0,R]) = \left\lbrace (x,[c]^j_x) \in T^j M: \psi_j(x,[c]^j_x) \leq R \right\rbrace
		\end{equation}
		is compact with respect to $\tau_{T^j M}$, then $(M,\Psi)$ is a weighted $C^k_{loc}$-manifold.
		\item\label{lem:w_pfwd2} If $(M,\Psi)$ is a weighted $C^k_{loc}$-manifold with $\vert I \vert < \infty$, then for every $j = 0,\ldots,k$ and $R > 0$ the pre-image $K_{j,R}$ defined in \eqref{eq:lem:w_pfwd:1} is compact with respect to $\tau_{T^j M}$.
	\end{enumerate}
\end{lemma}
\begin{proof}
	For \ref{lem:w_pfwd:1}, we fix some $i \in I$, $j = 0,\ldots,k$, and $R > 0$. Then, $K_{i,j,R} := \psi_{i,j}^{-1}((0,R])$ is a compact subset of $(\phi_i(U_i) \times X_i^j,\tau_{X_i} \times \tau_{X_i}^j)$ as image of the compact set $K_{j,R}$ under the continuous chart $\Phi^j_i: T^j M \rightarrow \phi_i(U_i) \times X_i^j$ (see \eqref{eq:def:TjM_charts}), whence $(\phi_i(U_i),\Psi_i)$ is a weighted domain. Since $i \in I$ was chosen arbitrarily, $(M,\Psi)$ is a weighted $C^k_{loc}$-manifold.
	
	For \ref{lem:w_pfwd2}, we assume that $\vert I \vert < \infty$ and fix some $i \in I$, $j = 0,\ldots,k$, and $R > 0$. Then, $K_{i,j,R} := \psi_{i,j}^{-1}((0,R])$ is by definition compact subset of $(\phi_i(U_i) \times X_i^j,\tau_{X_i} \times \tau_{X_i}^j)$. Hence, by using that $\Phi^{-j}_i: \phi_i(U_i) \times X_i^j \rightarrow T^j M$ is continuous (see \eqref{eq:def:TjM_finaltop}), the set $\Phi^{-j}_i(K_{i,j,R})$ is compact in $(T^j M, \tau_{T^j M})$ as continuous image of the compact set $K_{i,j,R}$. Hence, $K_{j,R} := \psi_j^{-1}((0,R]) = \bigcup_{i \in I} \Phi^{-j}_i(K_{i,j,R})$ is compact with respect to $\tau_{T^j M}$ as finite union of compact sets.
\end{proof}

Let us give an example of a weighted manifold $(M,\Psi)$ in the following.

\begin{example}
	\label{ex:w_mfd}
	Let $(X,\tau_X)$ be a locally convex topological vector space and let $\Psi := (\psi_j)_{j=0,\ldots,k}$ be a collection of admissible weight functions on $X$. Moreover, let $s \in C^k(X;\mathbb{R}^d)$ have constant rank, i.e., $\dim\left(\left\lbrace ds(x;v): v \in X \right\rbrace\right) = r$, for all $x \in X$ and some $r \in \mathbb{N}$. Then, for any $y \in \mathbb{R}^d$, the pre-image $M := s^{-1}(\lbrace y \rbrace)$ is by \cite[Theorem~F]{glockner15} a (split) $C^k_{loc}$-submanifold of $(X,\tau_X)$. Moreover, the collection $\Psi\vert_M := (\psi_j\vert_M)_{j=0,\ldots,k}$ of restricted weight functions is admissible on $M$. 
	
	For example, the space $M := \mathcal{P}_{\psi_\Omega}(\Omega)$ of probability measures $x: \mathcal{F}_\Omega \rightarrow [0,1]$ over a weighted space $(\Omega,\psi_\Omega)$ satisfying $\int_\Omega \psi_\Omega(\omega) x(d\omega) < \infty$ is a $C^\infty_{loc}$-manifold over the model space $X := \mathcal{M}_{\psi_\Omega}(\Omega)$ equipped with the weak-$*$-topology (see also Example~\ref{ex:w_dom:dual}~\ref{ex:w_dom:dual:measure}), where $\Psi := (\psi_j)_{j=0,\ldots,k}$ is of the form \eqref{eq:lem:w_dom}. Indeed, $M = s^{-1}(\lbrace 1 \rbrace)$ is the pre-image of the $C^\infty_{loc}$-map $\mathcal{M}_{\psi_\Omega}(\Omega) \ni x \mapsto s(x) := x(\Omega) \in \mathbb{R}$ having constant rank equal to one.
	
	For further examples of weighted manifolds, we refer to Section~\ref{sec:nachbin}.
\end{example}

In order to introduce maps on weighted manifolds, we assume that the input space $(M,\Psi)$ is a weighted $C^k_{loc}$-manifold and that the output space $(Y,\Vert \cdot \Vert_Y)$ is a Banach space.

\begin{definition}
	\label{def:BkPsi_mfd}
	Let $(M,\Psi)$ be a weighted $C^k_{loc}$-manifold with atlas $(U_i,\phi_i)_{i \in I}$ over the model spaces $(X_i,\tau_{X_i})_{i \in I}$ and let $\Psi_i := (\psi_{i,j})_{j=0,\ldots,k}$ be the collection of push-forward weight functions introduced in \eqref{eq:def:w_mfd:1}. Then, we define $\mathcal{B}^k_\Psi(M;Y)$ as the vector space of functions $f: M \rightarrow Y$ such that $f \circ \phi_i^{-1} \in \mathcal{B}^k_{\Psi_i}(\phi_i(U_i);Y)$, for all $i \in I$. We equip $\mathcal{B}^k_\Psi(M;Y)$ with the initial topology $\tau_{\mathcal{B}^k_\Psi(M;Y)}$ with respect to the family of mappings
	\begin{equation}
		\label{eq:def:BkPsi_mfd}
		\mathcal{B}^k_\Psi(M;Y) \ni f \quad \mapsto \quad f \circ \phi_i^{-1} \in \mathcal{B}^k_{\Psi_i}(\phi_i(U_i);Y), \quad\quad i \in I,
	\end{equation}
	i.e., the weakest topology such that the mappings \eqref{eq:def:BkPsi_mfd} are continuous.
\end{definition}

\begin{remark}
	\label{rem:atlas}
	As in Definition~\ref{def:w_mfd}, the space $\mathcal{B}^k_\Psi(M;Y)$ introduced in Definition~\ref{def:BkPsi_mfd} depends on the choice of atlas $(U_i,\phi_i)_{i \in I}$ for the manifold $M$. This dependence cannot be avoided for the infinite-dimensional approximation results in Section~\ref{sec:w_uat}--\ref{sec:sig} below, since the lack of partitions of unity on infinite-dimensional model spaces excludes the gluing of finite-dimensional local approximations into a global (atlas-independent) construction.
\end{remark}

For simplicity, we shall always assume that the output space is a Banach space $(Y,\Vert \cdot \Vert_Y)$. However, Definition~\ref{def:BkPsi_mfd} could also be extended to a locally convex topological vector space $(Y,\tau_Y)$ as output space (see also Remark~\ref{rem:banach_out}).

\section{Weighted Nachbin theorems}
\label{sec:nachbin}

In this section, we extend the Nachbin theorem to weighted (possibly infinite-dimensional) manifolds. Originally established by L.~Nachbin in \cite{nachbin49} over finite-dimensional manifolds, the theorem generalizes the classical Stone-Weierstrass theorem by including the approximation of the derivatives. Subsequently, the Nachbin theorem was extended in \cite{prolla76,aron80} to infinite-dimensional Banach spaces as input and output spaces, using the compact-open topology (of higher order) or the topology of compact convergence (of higher order), and in \cite{nachbin91} to a weighted approximation result for polynomials over the Euclidean space. First, we recall the classical Nachbin theorems.

\subsection{Classical formulation}
\label{sec:nachbin:cl}

Let us denote by $\Pol(\mathbb{R}^d) \subseteq C^\infty(\mathbb{R}^d)$ the vector space of polynomials of the form $\mathbb{R}^d \ni x := (x_1,\ldots,x_d)^\top \mapsto \sum_{\alpha \in \mathbb{N}^d_{0,n}} c_\alpha \prod_{i=1}^d x_i^{\alpha_i} \in \mathbb{R}$, with $n \in \mathbb{N}_0$ and $c_\alpha \in \mathbb{R}$.

\begin{theorem}[Weierstrass, {\cite[Theorem~1.1.2]{llavona86}}]
	\label{thm:wstrass}
	$\Pol(\mathbb{R}^d)$ is a dense subset of $C^k(\mathbb{R}^d)$ with respect to the compact-open topology\footnote{\label{footnote:co_top}For Banach spaces $(X,\Vert \cdot \Vert_X)$, $(Y,\Vert \cdot \Vert_Y)$, and $U \subseteq X$ open, the compact-open topology of order $k$ on $C^k(U;Y)$ is generated by seminorms $p_{K,L}(f) := \max_{j=0,\ldots,k} \sup_{(u,v) \in K \times L} \Vert d^j f(u;v,...,v) \Vert_Y$, for compact $K \subseteq U$ and $L \subseteq X$.} of order $k$.
\end{theorem}

Subsequently, the Weierstrass theorem (Theorem~\ref{thm:wstrass}) was generalized by L.~Nachbin in \cite{nachbin49} to the notion of subalgebras. Hereby, a vector space $\mathcal{G}$ of maps $g: X \rightarrow \mathbb{R}$ is called a \emph{subalgebra} if $\mathcal{G}$ is closed under multiplication, i.e., $g_1 \cdot g_2 \in \mathcal{G}$, for all $g_1, g_2 \in \mathcal{G}$.

\begin{theorem}[Nachbin on $C^k(M)$, {\cite[p.~1550]{nachbin49}}]
	\label{thm:nachbin_R}
	Let $(M,\tau_M)$ be a $C^\infty_{loc}$-manifold over finite-dimensional vector spaces. Moreover, let $\mathcal{G} \subseteq C^k(M)$ be a subalgebra such that
	\begin{enumerate}
		\item $\mathcal{G}$ is point separating on $M$, i.e., for any distinct points $x_1,x_2 \in M$ there exists some $g \in \mathcal{G}$ such that $g(x_1) \neq g(x_2)$,
		
		\item $\mathcal{G}$ vanishes nowhere on $M$, i.e., for every $x \in M$ there exists $g \in \mathcal{G}$ with $g(x) \neq 0$,
		
		\item $\mathcal{G}$ has nowhere vanishing derivatives on $M$, i.e., for every $(x,[c]^1_x) \in T^1 M$ with $c'(0) \neq 0$ there exists $g \in \mathcal{G}$ such that $(g \circ c)'(0) \neq 0$.
	\end{enumerate}
	Then, $\mathcal{G}$ is a dense subset of $C^k(M)$ with respect to the compact-open topology\footref{footnote:co_top} of order $k$.
\end{theorem}

Later, the Nachbin theorem was generalized by J.B.~Prolla and C.S.~Guerreiro in \cite{prolla76} to the following infinite-dimensional setting. For two locally convex topological vector spaces $(X,\tau_X)$ and $(Y,\tau_Y)$, the space of continuous homogeneous polynomials of finite type is defined as $P_f(X;Y) = \linspan\left\lbrace X \ni x \mapsto \ell(x)^n y \in Y: n \in \mathbb{N}_0, \, \ell \in X^*, \, y \in Y \right\rbrace \subseteq C^0(X;Y)$. Then, a subset $\mathcal{G} \subseteq C^0(X;Y)$ is called a \emph{polynomial algebra} if $r \circ g \in \mathcal{G}$ for all $g \in \mathcal{G}$ and $r \in P_f(Y;Y)$. This is the case if and only if $\mathcal{G}' := \lbrace \ell \circ g: \ell \in Y^*, \, g \in \mathcal{G} \rbrace$ is a subalgebra with $\mathcal{G}' \otimes Y \subseteq \mathcal{G}$ (see \cite[Lemma~4.6]{prolla77}).

\begin{theorem}[Nachbin on $C^k(U;Y)$, {\cite[Theorem~3.3]{prolla76}}]
	\label{thm:nachbin}
	Let $(X,\Vert \cdot \Vert_X)$ be a Banach space having AP and let $U \subseteq X$ be open. Moreover, let $\mathcal{G} \subseteq C^k(U;Y)$ be a polynomial subalgebra such that
	\begin{enumerate}
		\item $\mathcal{G}' := \lbrace \ell \circ g: \ell \in Y^*, \, g \in \mathcal{G} \rbrace$ is point separating on $U$,
		
		\item $\mathcal{G}'$ vanishes nowhere on $U$,
		
		\item $\mathcal{G}$ has nowhere vanishing derivatives on $U$, and
		
		\item for every $g \in \mathcal{G}$, $T \in X^* \otimes X$, and open subset $V \subseteq U$ with $T(V) \subseteq U$ the composition $g \circ T\vert_V \in C^k(V;Y)$ belongs to the closure of $\mathcal{G}\vert_V$ in $C^k(V;Y)$.
	\end{enumerate}
	Then, $\mathcal{G}$ is a dense subset of $C^k(U;Y)$ with respect to the compact-open topology$^{\text{\footref{footnote:co_top}}}$ of order $k$.
\end{theorem}

The Nachbin theorem was later extended to the topology\footnote{For Banach spaces $(X,\Vert \cdot \Vert_X)$, $(Y,\Vert \cdot \Vert_Y)$, and $U \subseteq X$ open, the topology of compact convergence of order $k$ on $C^k(U;Y)$ is generated by seminorms $p_K(f) := \max_{j=0,\ldots,k} \sup_{u \in K} \Vert d^j f(u;\cdot,\ldots,\cdot) \Vert_{L(X^j;Y)}$, for compact $K \subseteq U$. Note that the topology of compact convergence of order $k$ is stronger than the compact-open topology of order $k$ (except when $(X,\Vert \cdot \Vert_X)$ is finite-dimensional; then they are equivalent).} of compact convergence of order $k$ by R.M.~Aron and J.B.~Prolla in \cite{aron80}, and into other directions (see, e.g., \cite{nachbin78,gomez82,llavona86}).

\subsection{Subalgebras of \texorpdfstring{$\Psi$}{Psi}-moderate growth}

For the weighted Nachbin theorems, we impose the following conditions on a given subalgebra $\mathcal{G} \subseteq \mathcal{B}^k_\Psi(M)$. This is analogous to the concept of point separating and nowhere vanishing subalgebras of $\Psi$-moderate growth, which was introduced in \cite[Definition~3.4]{cuchiero23} for the non-differentiable case and is inspired by Nachbin's definition of localisability (see \cite[Definition~4]{nachbin65}). For $j = 0,\ldots,k$ and a partition $\pi \in \mathscr{P}_j$, we use the notation $d^\pi a(u;v_\pi) = \prod_{r=1}^{\vert\pi\vert} d^{\vert\pi_r\vert} a(u;v_{\pi_r})$ with $d^{\vert\pi_r\vert} a(u;v_{\pi_r}) := d^{\vert\pi_r\vert} a(u;(v_\ell)_{\ell \in \pi_r})$.

\begin{definition}
	Let $(M,\Psi)$ be a weighted $C^k_{loc}$-manifold with atlas $(U_i,\phi_i)_{i \in I}$ over finite-dimen-sional model spaces $(X_i,\tau_{X_i})_{i \in I}$. Then, a subalgebra $\mathcal{G} \subseteq \mathcal{B}^k_\Psi(M)$ is called \emph{strongly point separating and nowhere vanishing of $\Psi$-moderate growth} if there exists a vector subspace $\widetilde{\mathcal{G}} \subseteq \mathcal{G}$ such that
	\begin{enumerate}
		\item[\labeltext{(M1)}{M1}] $\widetilde{\mathcal{G}}$ is point separating on $M$, i.e., for any distinct points $x_1,x_2 \in M$ there exists some $\widetilde{g} \in \widetilde{\mathcal{G}}$ such that $\widetilde{g}(x_1) \neq \widetilde{g}(x_2)$,
		\item[\labeltext{(M2)}{M2}] $\widetilde{\mathcal{G}}$ is nowhere vanishing on $M$, i.e., for every $x \in M$ there exists $\widetilde{g} \in \widetilde{\mathcal{G}}$ with $\widetilde{g}(x) \neq 0$,
		\item[\labeltext{(M3)}{M3}] $\widetilde{\mathcal{G}}$ has nowhere vanishing derivatives on $M$, i.e., for every $(x,[c]^1_x) \in T^1 M$ with $c'(0) \neq 0$ there exists some $\widetilde{g} \in \widetilde{\mathcal{G}}$ such that $(\widetilde{g} \circ c)'(0) \neq 0$, 
		\item[\labeltext{(M4)}{M4}] for every $i \in I$ there exist $\widetilde{g}_1,\ldots,\widetilde{g}_m \in \widetilde{\mathcal{G}}$ such that $\eta_i := (\widetilde{g}_1 \circ \phi_i^{-1}, \ldots, \widetilde{g}_m \circ \phi_i^{-1})^\top: \phi_i(U_i) \rightarrow \mathbb{R}^m$ is an embedding and there exist cutoff functions $(h_{i,R})_{R > 0} \subseteq C^\infty_c(\eta_i(\phi_i(U_i)))$ with $0 \leq h_{i,R} \leq 1$ and $h_{i,R}\vert_{\eta_i(K_{i,R})} = 1$ such that
		\begin{equation}
			\quad\quad\quad \lim_{R \rightarrow \infty} \max_{1 \leq \ell \leq j \leq k} \sup_{(u,v_1,\ldots,v_j) \in (\phi_i(U_i) \times X_i^j) \setminus K_{i,j,R}} \frac{\Vert d^\ell (h_{i,R} \circ \eta_i)(u) \Vert_{L^\ell(X_i;\mathbb{R})} \Vert v_1 \Vert \cdots \Vert v_j \Vert}{\psi_{i,j}(u,v_1,\ldots,v_j)} = 0,
		\end{equation}
		where $K_{i,R} := \bigcup_{j=0}^k \pi_{i,0}(K_{i,j,R})$ with $\phi_i(U_i) \times X_i^j \ni (u,v_{1:j}) \mapsto \pi_{i,0}(u,v_{1:j}) := u \in \phi_i(U_i)$,
		\item[\labeltext{(M5)}{M5}] and for every $i \in I$ the vector space $\widetilde{\mathcal{G}}_i := \big\lbrace \widetilde{g} \circ \phi_i^{-1}: \widetilde{g} \in \widetilde{\mathcal{G}} \big\rbrace$ is of $\Psi_i$-moderate growth, i.e., for every $\widetilde{g}_i \in \widetilde{\mathcal{G}}_i$ there exists some $\lambda > 0$ such that
		\begin{equation}
			\quad\quad\quad \lim_{R \rightarrow \infty} \max_{j=0,\ldots,k \atop \pi \in \mathscr{P}_j} \sup_{(u,v_1,\ldots,v_j) \in (\phi_i(U_i) \times X_i^j) \setminus K_{i,j,R}} \frac{\exp\left( \lambda \left\vert \widetilde{g}_i(u) \right\vert \right) \left\vert d^\pi \widetilde{g}_i(u;v_\pi) \right\vert}{\psi_{i,j}(u,v_1,\ldots,v_j)} = 0,
		\end{equation}
		where $K_{i,j,R} := \psi^{-1}_{i,j}((0,R])$ denotes the compact pre-image of the push-forward weight functions $\Psi_i := (\psi_{i,j})_{j=0,\ldots,k}$ defined in \eqref{eq:def:w_mfd:1}.
	\end{enumerate}
	If $(M,\Psi)$ is a weighted $C^k_{loc}$-manifold with atlas $(U_i,\phi_i)_{i \in I}$ over infinite-dimensional model spaces $(X_i,\tau_{X_i})_{i \in I}$ each having BAP with finite rank operators $(T_{i,\gamma})_\gamma$, we replace \ref{M4} and \ref{M5} by
	\begin{enumerate}
		\item[\labeltext{(M4')}{M4'}] for every $i \in I$ and $\gamma$ the set $\widetilde{\mathcal{G}}_{i,\gamma} := \big\lbrace \widetilde{g} \circ \phi_i^{-1}\vert_{T_{i,\gamma}(\phi_i(U_i))}: \widetilde{g} \in \widetilde{\mathcal{G}} \big\rbrace$ satisfies \ref{M4}, and
		\item[\labeltext{(M5')}{M5'}] for every $i \in I$ and $\gamma$ the set $\widetilde{\mathcal{G}}_{i,\gamma} := \big\lbrace \widetilde{g} \circ \phi_i^{-1}\vert_{T_{i,\gamma}(\phi_i(U_i))}: \widetilde{g} \in \widetilde{\mathcal{G}} \big\rbrace$ is of $\Psi_{i,\gamma}$-moderate growth, i.e., for every $\widetilde{g}_{i,\gamma} \in \widetilde{\mathcal{G}}_{i,\gamma}$ there exists some $\lambda > 0$ such that
		\begin{equation}
			\quad\quad\quad \lim_{R \rightarrow \infty} \max_{j=0,\ldots,k \atop \pi \in \mathscr{P}_j} \sup_{(u,v_1,\ldots,v_j) \in (T_{i,\gamma}(\phi_i(U_i)) \times T_{i,\gamma}(X_i)^j) \setminus K_{i,\gamma,j,R}} \frac{\exp\left( \lambda \left\vert \widetilde{g}_{i,\gamma}(u) \right\vert \right) \left\vert d^\pi \widetilde{g}_{i,\gamma}(u;v_\pi) \right\vert}{\psi_{i,\gamma,j}(u,v_1,\ldots,v_j)} = 0,
		\end{equation}
		where $K_{i,\gamma,j,R} := \psi^{-1}_{i,\gamma,j}((0,R])$ is the pre-image of the weights $\Psi_{i,\gamma} := (\psi_{i,\gamma,j})_{j=0,\ldots,k}$ defined by $\psi_{\gamma,j}(\widetilde{u},\widetilde{v}_{1:j}) := \inf_{(u,v_{1:j}) \in U \times X^j, \, T_{i,\gamma}^{j+1}(u,v_{1:j}) = (\widetilde{u},\widetilde{v}_{1:j})} \psi_j(u,v_{1:j})$.
	\end{enumerate}
\end{definition}

\begin{remark}
	\label{rem:psi_mod_growth}
	The conditions~\ref{M1}--\ref{M3} ensure that the classical Nachbin theorem on compacta (Theorem~\ref{thm:nachbin_R}) can be applied. Moreover, \ref{M4} is needed to localize compact subsets of $M$, where the corresponding limit is zero if, e.g., $\eta_i$ has uniformly bounded derivatives, i.e., $\max_{\ell=1,\ldots,k} \sup_{u \in U} \Vert d^\ell \eta_i(u) \Vert_{L^\ell(X_i;\mathbb{R})} < \infty$ (see \eqref{eq:ass:w_growth}). In addition, if $\mathcal{G}$ consists of bounded maps, then the exponential part in \ref{M5} is bounded, whence \ref{M5} is by monotonicity of $\Psi$ satisfied.
\end{remark}

While \ref{M1}, \ref{M2}, and \ref{M5} are similar to the non-differentiable case in \cite[Definition~3.4]{cuchiero23}, the conditions \ref{M3} and \ref{M4} are needed to include the approximation of the derivatives in our weighted setting. Moreover, \ref{M5} is an analogue of the exponential moment condition for the uniqueness of the moment problem. The proof can be found in Appendix~\ref{app:real_anal}.

\begin{lemma}
	\label{lem:real_anal}
	For an open subset $U \subseteq \mathbb{R}^d$, let $\widetilde{g} \in \mathcal{B}^k_\Psi(U)$ satisfy \ref{M5} with $U_i := U$ and $\phi_i = \id_{U_i}$. Then, $\mathbb{R} \ni t \mapsto \cos(t \widetilde{g}(\cdot)) \in \mathcal{B}^k_\Psi(U)$ and $\mathbb{R} \ni t \mapsto \sin(t \widetilde{g}(\cdot)) \in \mathcal{B}^k_\Psi(U)$ are real-analytic.
\end{lemma}

\subsection{Weighted Nachbin theorems over finite-dimensional manifolds}

Now, we formulate a generalized version of the Nachbin theorem in our weighted setting, which extends the weighted approximation results in \cite{zapata73,nachbin91} for polynomials over $\mathbb{R}^d$ to the notion of subalgebras over finite-dimensional manifolds.

\begin{theorem}[Nachbin on $\mathcal{B}^k_\Psi(M)$]
	\label{thm:w_nachbin_findim}
	Let $(M,\Psi)$ be a weighted $C^k_{loc}$-manifold over finite-dimen-sional vector spaces. Moreover, let $\mathcal{G} \subseteq \mathcal{B}^k_\Psi(M)$ be a subalgebra such that $\mathcal{G}$ is strongly point separating and nowhere vanishing of $\Psi$-moderate growth. Then, $\mathcal{G}$ is dense in $\mathcal{B}^k_\Psi(M)$.
\end{theorem}
\begin{proof}
	First, we show the conclusion for a subalgebra $\mathcal{G} \subseteq \mathcal{B}^k_\Psi(U)$ consisting of bounded maps over a weighted domain $(U,\Psi)$, where we can choose $\widetilde{\mathcal{G}} := \mathcal{G}$ as a strongly point separating and nowhere vanishing separating vector subspace (see Remark~\ref{rem:psi_mod_growth}). Since $\mathcal{B}^k_\Psi(U)$ is defined as the closure of $C^k_b(U)$ with respect to $\Vert \cdot \Vert_{\mathcal{B}^k_\Psi(U)}$, it suffices to approximate any given $f \in C^k_b(U)$ by an element of $\mathcal{G}$. To this end, we fix some $f \in C^k_b(U)$ and $\varepsilon > 0$. Moreover, by defining $C_f := 1+$ $\max_{j=0,\ldots,k} \Vert d^j f \Vert_{L^j(\mathbb{R}^d;\mathbb{R})} \!>\! 0$, it holds for every $j = 0,\ldots,k$ and $(u,v_1,\ldots,v_j) \in U \times (\mathbb{R}^d)^j$ that
	\begin{equation}
		\vert d^j f(u;v_1,\ldots,v_j) \vert \leq C_f \Vert v_1 \Vert \cdots \Vert v_j \Vert.
	\end{equation}
	In addition, by \eqref{eq:ass:w_growth} and \ref{M4}, there exists some $R_2 > R > 0$ such that
	\begin{equation}
		\begin{aligned}
			\max_{j=0,\ldots,k} \sup_{(u,v_1,\ldots,v_j) \in (U \times (\mathbb{R}^d)^j) \setminus K_{j,R}} \frac{\Vert v_1 \Vert \cdots \Vert v_j \Vert}{\psi_j(u,v_1,\ldots,v_j)} & \leq \frac{\varepsilon}{3 \cdot 2^k C_f}, \\
			\max_{1 \leq \ell \leq j \leq k} \sup_{(u,v_1,\ldots,v_j) \in (U \times (\mathbb{R}^d)^j) \setminus K_{j,R}} \frac{\Vert d^\ell (h \circ \eta)(u) \Vert_{L^\ell(\mathbb{R}^d;\mathbb{R})} \Vert v_1 \Vert \cdots \Vert v_j \Vert}{\psi_j(u,v_1,\ldots,v_j)} & \leq \frac{\varepsilon}{3 \cdot 2^k C_f},
		\end{aligned}
	\end{equation}
	where $\eta := (\widetilde{g}_1,\ldots,\widetilde{g}_m)^\top: U \rightarrow \mathbb{R}^m$ is an embedding (with $\widetilde{g}_1,\ldots,\widetilde{g}_m \in \widetilde{\mathcal{G}}$) and $h := h_R \in C^\infty_c(\eta(U))$ is a cutoff function with $0 \leq h \leq 1$, $h\vert_{\eta(K_R)} = 1$, and $h\vert_{\mathbb{R}^m \setminus \eta(K_{R_2})} = 0$. Then, by using the Leibniz product rule (if $j \geq 1$), we conclude that
	\begin{equation}
		\label{eq:thm:w_nachbin_findim:proof1}
		\begin{aligned}
			& \max_{j=0,\ldots,k} \sup_{(u,v_1,\ldots,v_j) \in (U \times X^j) \setminus K_{j,R}} \frac{\vert d^j \big( (1 - h \circ \eta) \cdot f \big)(u;v_1,\ldots,v_j) \vert}{\psi_j(u,v_1,\ldots,v_j)} \\
			& \leq \max_{j=0,\ldots,k} \sup_{(u,v_1,\ldots,v_j) \in (U \times X^j) \setminus K_{j,R}} \frac{\sum_{\mathcal{L} \subseteq \lbrace 1,\ldots,j \rbrace} \vert d^{\vert\mathcal{L}\vert}(1 - h \circ \eta)(u;v_\mathcal{L}) \vert \vert d^{j-\vert\mathcal{L}\vert} f(u;v_{\mathcal{L}^c}) \vert}{\psi_j(u,v_1,\ldots,v_j)} \\
			& \leq 2^k \max_{j=0,\ldots,k \atop \mathcal{L} \subseteq \lbrace 1,\ldots,j \rbrace} \sup_{(U \times X^j) \setminus K_{j,R}} \frac{\Vert d^{\vert\mathcal{L}\vert} (1-h \circ \eta)(u) \Vert_{L^{\vert\mathcal{L}\vert}(\mathbb{R}^d;\mathbb{R})} \Vert d^{j-\vert\mathcal{L}\vert} f(u) \Vert_{L^{j-\vert\mathcal{L}\vert}(\mathbb{R}^d;\mathbb{R})} \Vert v_1 \Vert \cdots \Vert v_j \Vert}{\psi_j(u,v_1,\ldots,v_j)} \\
			& \leq 2^k C_f \max_{1 \leq \ell \leq j \leq k} \sup_{(u,v_1,\ldots,v_j) \in (U \times (\mathbb{R}^d)^j) \setminus K_{j,R}} \frac{\Vert d^\ell (h \circ \eta)(u) \Vert_{L^\ell(\mathbb{R}^d;\mathbb{R})} \Vert v_1 \Vert \cdots \Vert v_j \Vert}{\psi_j(u,v_1,\ldots,v_j)} \\
			& \leq 2^k C_f \frac{\varepsilon}{3 \cdot 2^k C_f} = \frac{\varepsilon}{3},
		\end{aligned}
	\end{equation}
	where we used that $\Vert d^0 (h \circ \eta)(u) \Vert_{L^0(\mathbb{R}^d;\mathbb{R})} = \vert h(\eta(u)) \vert \leq 1$. Now, on the set $K := \bigcup_{j=0}^k \pi_0(K_{j,R_2})$ (being compact as continuous image of the compact pre-images $K_{j,R_2} := \psi_j^{-1}((0,R_2])$, $j = 0,\ldots,k$), we can apply the classical Nachbin theorem (Theorem~\ref{thm:nachbin_R}) to obtain some $g \in \mathcal{G}$ satisfying
	\begin{equation}
		\label{eq:thm:w_nachbin_findim:proof2}
		\max_{j=0,\ldots,k} \sup_{u \in K} \Vert d^j f(u) - d^j g(u) \Vert_{L^j(\mathbb{R}^d;\mathbb{R})} < \frac{\varepsilon}{3 \cdot 2^k k! C_\Psi C_\eta C_{\inf}},
	\end{equation}
	where the constant $C_{\inf} := \max_{j=0,\ldots,k} \sup_{(u,v_1,\ldots,v_j)} \frac{\Vert v_1 \Vert \cdots \Vert v_j \Vert}{\psi_j(u,v_1,\ldots,v_j)} > 0$ is by \eqref{eq:ass:w_growth} finite, and the constant $C_\eta := \max_{j=0,\ldots,k} \sup_{(u,v_1,\ldots,v_j) \in U \times (\mathbb{R}^d)^j} \frac{\vert d^j (h \circ \eta)(u;v_1,\ldots,v_j) \vert}{\psi_j(u,v_1,\ldots,v_j)} > 0$ is by \ref{M4} finite. Thus, by using again the Leibniz product rule (if $j \geq 1$), that $\supp(h \circ \eta) \subseteq K$, the monotonicity of $\Psi := (\psi_j)_{j=0,\ldots,k}$, and \eqref{eq:thm:w_nachbin_findim:proof2}, it follows that 
	\begin{equation}
		\label{eq:thm:w_nachbin_findim:proof3}
		\begin{aligned}
			& \Vert (h \circ \eta) \cdot (f-g) \Vert_{\mathcal{B}^k_\Psi(U)} = \max_{j=0,\ldots,k} \sup_{(u,v_1,\ldots,v_j) \in U \times (\mathbb{R}^d)^j} \frac{\left\vert d^j \big( (h \circ \eta) \cdot (f-g) \big)(u;v_1,\ldots,v_j) \right\vert}{\psi_j(u,v_1,\ldots,v_j)} \\
			& \leq C_\Psi \max_{j=0,\ldots,k} \sup_{(u,v_1,\ldots,v_j) \in K \times (\mathbb{R}^d)^j} \frac{\sum_{\mathcal{L} \subseteq \lbrace 1,\ldots,j \rbrace} \vert d^{\vert\mathcal{L}\vert} (h \circ \eta)(u;v_\mathcal{L}) \vert \vert d^{j-\vert\mathcal{L}\vert} (f-g)(u;v_{\mathcal{L}^c}) \vert}{\psi_{\vert\mathcal{L}\vert}(u,v_\mathcal{L}) \psi_{j-\vert\mathcal{L}\vert}(u,v_{\mathcal{L}^c})} \\
			& \leq 2^k C_\Psi C_\eta \max_{j=0,\ldots,k} \sup_{(u,v_1,\ldots,v_j) \in K \times (\mathbb{R}^d)^j} \frac{\Vert d^j f(u) - d^j g(u) \Vert_{L^j(\mathbb{R}^d;\mathbb{R})} \Vert v_1 \Vert \cdots \Vert v_j \Vert}{\psi_j(u,v_1,\ldots,v_j)} \\
			& \leq 2^k C_\Psi C_\eta C_{\inf} \sup_{u \in K} \max_{j=0,\ldots,k} \Vert d^j f(u) - d^j g(u) \Vert_{L^j(\mathbb{R}^d;\mathbb{R})} \\
			& < 2^k k! C_\Psi C_\eta C_{\inf} \frac{\varepsilon}{3 \cdot 2^k k! C_\Psi C_\eta C_{\inf}} = \frac{\varepsilon}{3}.
		\end{aligned}
	\end{equation}
	Therefore, by defining the function $\mathbb{R}^m \times \mathbb{R} \ni (y,z) \mapsto H(y,z) := h(y) z \in \mathbb{R}$, we conclude from \eqref{eq:thm:w_nachbin_findim:proof1} and \eqref{eq:thm:w_nachbin_findim:proof3} that
	\begin{equation}
		\label{eq:thm:w_nachbin_findim:proof4}
		\begin{aligned}
			\Vert f - H \circ (\eta \oplus g) \Vert_{\mathcal{B}^k_\Psi(U)} & \leq \Vert f - (h \circ \eta) \cdot f + (h \circ \eta) \cdot (f-g) \Vert_{\mathcal{B}^k_\Psi(U)} \\
			& \leq \Vert f - (h \circ \eta) \cdot f \Vert_{\mathcal{B}^k_\Psi(U)} + \Vert (h \circ \eta) \cdot (f-g) \Vert_{\mathcal{B}^k_\Psi(U)} \\
			& < \frac{\varepsilon}{3} + \frac{\varepsilon}{3} = \frac{2\varepsilon}{3}.
		\end{aligned}
	\end{equation}
	Next, we define the set $\widetilde{K} := \overline{\eta(U) \times g(U)} \subseteq \mathbb{R}^m \times \mathbb{R} \cong \mathbb{R}^{m+1}$, which is compact as $\mathcal{G}$ consists of bounded maps. Then, by applying the Weierstrass theorem (Theorem~\ref{thm:wstrass}), there exists some $p_n \in \Pol(\mathbb{R}^{m+1}) \cong \Pol(\mathbb{R}^m \times \mathbb{R})$ satisfying
	\begin{equation}
		\label{eq:thm:w_nachbin_findim:proof5}
		\max_{j=0,\ldots,k} \sup_{(y,z) \in \widetilde{K}} \Vert d^j H(y,z) - d^j p_n(y,z) \Vert_{L^j(\mathbb{R}^m \times \mathbb{R};\mathbb{R})} < \frac{\varepsilon}{3 k! C_{\eta,g}}.
	\end{equation}
	Hence, by using the Fa\`a di Bruno formula (if $j \geq 1$), that $\vert \mathscr{P}_j \vert \leq j! \leq k!$, and \eqref{eq:thm:w_nachbin_findim:proof5}, we have
	\begin{equation}
		\label{eq:thm:w_nachbin_findim:proof6}
		\begin{aligned}
			& \Vert H \circ (\eta \oplus g) - p_n \circ (\eta \oplus g) \Vert_{\mathcal{B}^k_\Psi(U)} \\
			& = \max_{j=0,\ldots,k} \sup_{(u,v_1,\ldots,v_j) \in U \times (\mathbb{R}^d)^j} \frac{\big\vert \sum_{\pi \in \mathscr{P}_j} d^\pi (H - p_n)\big( (\eta(u),g(u)); d^\pi (\eta \oplus g)(u;v_\pi) \big) \big\vert}{\psi_j(u,v_1,\ldots,v_j)} \\
			& \leq k! \! \max_{j=0,\ldots,k \atop \pi \in \mathscr{P}_j} \sup_{U \times (\mathbb{R}^d)^j} \! \frac{\Vert d^{\vert\pi\vert} H(\eta(u),g(u)) \!-\! d^{\vert\pi\vert} p_n(\eta(u),g(u)) \Vert_{L^{\vert\pi\vert}(\mathbb{R}^m \times \mathbb{R};\mathbb{R})} \prod_{r=1}^{\vert\pi\vert} \! \Vert d^{\vert\pi_r\vert} (\eta \oplus g)(u;v_{\pi_r}) \Vert}{\psi_j(u,v_1,\ldots,v_j)} \\
			& \leq k! C_{\eta,g} \max_{j=0,\ldots,k} \sup_{(y,z) \in \widetilde{K}} \Vert d^j H(y,z) - d^j p_n(y,z) \Vert_{L^j(\mathbb{R}^m \times \mathbb{R};\mathbb{R})} \\
			& < k! C_{\eta,g} \frac{\varepsilon}{3 k! C_{\eta,g}} = \frac{\varepsilon}{3}.
		\end{aligned}
	\end{equation}
	Finally, by combining \eqref{eq:thm:w_nachbin_findim:proof4} with \eqref{eq:thm:w_nachbin_findim:proof6}, it follows for $p_n \circ (\eta \oplus g) \in \mathcal{G}$ (as $\mathcal{G}$ is a subalgebra) that
	\begin{equation}
		\begin{aligned}
			\Vert f - p_n \circ (\eta \oplus g) \Vert_{\mathcal{B}^k_\Psi(U)} & \leq \Vert f - H \circ (\eta \oplus g) \Vert_{\mathcal{B}^k_\Psi(U)} + \Vert H \circ (\eta \oplus g) - p_n \circ (\eta \oplus g) \Vert_{\mathcal{B}^k_\Psi(U)} \\
			& < \frac{2\varepsilon}{3} + \frac{\varepsilon}{3} = \varepsilon. 
		\end{aligned}
	\end{equation}
	Since $f \in C^k_b(U)$ and $\varepsilon > 0$ was chosen arbitrarily, this shows that $\mathcal{G}$ is dense in $\mathcal{B}^k_\Psi(U)$.
	
	Now, we show the conclusion for a general subalgebra $\mathcal{G} \subseteq \mathcal{B}^k_\Psi(U)$ over a weighted domain $(U,\Psi)$, where $\widetilde{\mathcal{G}} \subseteq \mathcal{G}$ denotes the strongly point separating and nowhere vanishing vector subspace of $\Psi$-moderate growth. By using $\mathbb{R} \ni s \mapsto \cos_0(s) := \cos(s) - 1 \in \mathbb{R}$, we introduce the set
		\begin{equation}
		\label{eq:thm:w_nachbin_findim:proof7}
		\mathcal{G}_{\mathrm{trig}} := \linspan\left( \left\lbrace \cos_0\! \circ \widetilde{g}: \widetilde{g} \in \widetilde{\mathcal{G}} \right\rbrace \cup \left\lbrace \sin \circ \widetilde{g}: \widetilde{g} \in \widetilde{\mathcal{G}} \right\rbrace \right),
	\end{equation}
	which consists of bounded maps. In order to show that $\mathcal{G}_{\mathrm{trig}} \subseteq \mathcal{B}^k_\Psi(U)$, we fix some $\widetilde{g} \in \widetilde{\mathcal{G}}$ and $\varepsilon > 0$. Since $\widetilde{g} \in \mathcal{B}^k_\Psi(U)$, there exists by definition of $\mathcal{B}^k_\Psi(U)$ some $b \in C^k_b(U)$ such that
	\begin{equation}
		\label{eq:thm:w_nachbin_findim:proof8}
		\Vert \widetilde{g} - b \Vert_{\mathcal{B}^k_\Psi(U)} = \max_{j=0,\ldots,k} \sup_{(u,v_1,\ldots,v_j) \in U \times (\mathbb{R}^d)^j} \frac{\vert d^j \widetilde{g}(u;v_1,\ldots,v_j) - d^j b(u;v_1,\ldots,v_j) \vert}{\psi_j(u,v_1,\ldots,v_j)} < \frac{\varepsilon}{k \cdot k! (C_\Psi C_{\widetilde{g}})^k},
	\end{equation}
	where $C_{\widetilde{g}} := 1+\Vert \widetilde{g} \Vert_{\mathcal{B}^k_\Psi(U)}$. Thus, the Fa\`a di Bruno formula, $\vert \mathscr{P}_j \vert \leq j! \leq k!$, that $\vert \cos^{(j)}(s) \vert \leq 1$, a telescoping sum, and $\frac{\vert d^{\vert\pi_\ell\vert} b(u;v_{\pi_\ell}) \vert}{\psi_{\vert\pi_\ell\vert}(u,v_{\pi_\ell})} \leq \varepsilon + \frac{\vert d^{\vert\pi_\ell\vert} \widetilde{g}(u,v_{\pi_\ell}) \vert}{\psi_{\vert\pi_\ell\vert}(u,v_{\pi_\ell})} \leq 1 + \Vert \widetilde{g} \Vert_{\mathcal{B}^k_\Psi(U)} = C_{\widetilde{g}}$, and \eqref{eq:thm:w_nachbin_findim:proof8} imply that
	\begin{equation}
		\label{eq:thm:w_nachbin_findim:proof9}
		\begin{aligned}
			& \Vert \cos_0\! \circ \widetilde{g} - \cos_0\! \circ b \Vert_{\mathcal{B}^k_\Psi(U)} = \max_{j=0,\ldots,k} \sup_{(u,v_1,\ldots,v_j) \in U \times (\mathbb{R}^d)^j} \!\! \frac{\big\vert \sum_{\pi \in \mathscr{P}_j} \!\! \cos_0^{(\vert\pi\vert)}(\widetilde{g}(u)) \big( d^\pi \widetilde{g}(u;v_\pi) \!-\! d^\pi b(u;v_\pi) \big) \big\vert}{\psi_j(u,v_1,\ldots,v_j)} \\
			& \leq k! \max_{j=0,\ldots,k \atop \pi \in \mathscr{P}_j} \sup_{(u,v_1,\ldots,v_j) \in U \times (\mathbb{R}^d)^j} \frac{\vert \cos_0^{(\vert\pi\vert)}(\widetilde{g}(u)) \vert \big\vert \prod_{r=1}^{\vert\pi\vert} d^{\vert\pi_r\vert} \widetilde{g}(u;v_{\pi_r}) - \prod_{r=1}^{\vert\pi\vert} d^\pi b(u;v_{\pi_r}) \big\vert}{\psi_j(u,v_1,\ldots,v_j)} \\
			& \leq k! C_\Psi^k \! \max_{j=0,\ldots,k \atop \pi \in \mathscr{P}_j} \sup_{U \times (\mathbb{R}^d)^j} \!\! \frac{\sum_{r=1}^{\vert\pi\vert} \! \prod_{\ell=1}^r \! \vert d^{\vert\pi_\ell\vert} \widetilde{g}(u,v_{\pi_\ell}) \vert \vert d^{\vert\pi_r\vert} \widetilde{g}(u,v_{\pi_r}) \!-\! d^{\vert\pi_r\vert} b(u;v_{\pi_r}) \vert \prod_{\ell=r+1}^{\vert\pi\vert} \! \vert d^{\vert\pi_\ell\vert} b(u;v_{\pi_\ell}) \vert}{\prod_{\ell=1}^r \psi_{\vert\pi_\ell\vert}(u,v_{\pi_\ell}) \psi_{\vert\pi_r\vert}(u,v_{\pi_r}) \prod_{\ell=r+1}^{\vert\pi\vert} \psi_{\vert\pi_\ell\vert}(u,v_{\pi_\ell})} \\
			& \leq k \cdot k! (C_\Psi C_{\widetilde{g}})^k \max_{j=0,\ldots,k \atop \pi \in \mathscr{P}_j, \, r=1,\ldots,\vert\pi\vert} \sup_{(u,v_1,\ldots,v_j) \in U \times (\mathbb{R}^d)^j} \frac{\vert d^{\vert\pi_r\vert} \widetilde{g}(u,v_{\pi_r}) - d^{\vert\pi_r\vert} b(u;v_{\pi_r}) \vert}{\psi_{\vert\pi_r\vert}(u,v_{\pi_r})} \\
			& \leq k \cdot k! (C_\Psi C_{\widetilde{g}})^k \frac{\varepsilon}{k \cdot k! (C_\Psi C_{\widetilde{g}})^k} = \varepsilon.
		\end{aligned}
	\end{equation}
	Since $\varepsilon > 0$ was chosen arbitrarily and $\cos_0\! \circ b \in C^k_b(U)$, this shows that $\cos_0\! \circ \widetilde{g} \in \mathcal{B}^k_\Psi(U)$, which holds analogously for $\sin \circ \widetilde{g} \in \mathcal{B}^k_\Psi(U)$, thus $\mathcal{G}_{\mathrm{trig}} \subseteq \mathcal{B}^k_\Psi(U)$. Moreover, the trigonometric identities
	\begin{equation}
		\label{eq:thm:w_nachbin_findim:proof10}
		\begin{aligned}
			\cos_0(s) \cos_0(t) & = \cos(s) \cos(t) - \cos(s) - \cos(t) + 1 \\
			& = \frac{1}{2} \big( \cos(s-t) + \cos(s+t) \big) - \cos(s) - \cos(t) + 1 \\
			& = \frac{1}{2} \big( \cos_0(s-t) + \cos_0(s+t) \big) - \cos_0(s) - \cos_0(t), \\
			\cos_0(s) \sin(t) & = \cos(s) \sin(t) - \sin(t) \\
			& = \frac{1}{2} \big( \sin(s+t) - \sin(s-t) \big) - \sin(t), \\
			\sin(s) \sin(t) & = \frac{1}{2} \big( \cos(s-t) - \cos(s+t) \big) \\
			& = \frac{1}{2} \big( \cos_0(s-t) - \cos_0(s+t) \big),
		\end{aligned}
	\end{equation}
	ensure that $\mathcal{G}_{\mathrm{trig}}$ is a subalgebra. Now, we check that $\mathcal{G}_{\mathrm{trig}}$ satisfies \ref{M1}--\ref{M5}. For \ref{M1}, we find for any distinct points $u_1,u_2 \in U$ some $\widetilde{g} \in \widetilde{\mathcal{G}}$ with $\widetilde{g}(u_1) \neq \widetilde{g}(u_2)$. Thus, for $t \neq 0$ small enough, $t \widetilde{g}(u_1) \neq t \widetilde{g}(u_2)$ are distinct points in $(-\frac{\pi}{2},\frac{\pi}{2})$ and the map $\sin(t \widetilde{g}(\cdot)) \in \mathcal{G}_{\mathrm{trig}}$ separates them. For \ref{M2}, there exist for every $u \in U$ some $\widetilde{g} \in \widetilde{\mathcal{G}}$ with $\widetilde{g}(u) \neq 0$, whence there exists a suitable $t \neq 0$ such that the map $\sin(t \widetilde{g}(\cdot)) \in \mathcal{G}_{\mathrm{trig}}$ satisfies $\sin(t \widetilde{g}(u)) \neq 0$. For \ref{M3}, we find for any $u \in U$ and $v \in \mathbb{R}^d \setminus \lbrace 0 \rbrace$ some $\widetilde{g} \in \widetilde{\mathcal{G}}$ with $d\widetilde{g}(u;v) \neq 0$, thus either the map $\cos_0\,\circ \widetilde{g} \in \mathcal{G}_{\mathrm{trig}}$ or $\sin \circ \widetilde{g} \in \mathcal{G}_{\mathrm{trig}}$ satisfies $d(\cos_0\! \circ \widetilde{g})(u;v) = -\sin(\widetilde{g}(u)) d\widetilde{g}(u;v) \neq 0$ or $d(\sin \circ \widetilde{g})(u;v) = \cos(\widetilde{g}(u)) d\widetilde{g}(u;v) \neq 0$. For \ref{M4}, there exist some $\widetilde{g}_1,\ldots,\widetilde{g}_m \in \widetilde{\mathcal{G}}$ such that $\eta := (\widetilde{g}_1,\ldots,\widetilde{g}_m)^\top: U \rightarrow \mathbb{R}^m$ is an embedding. Hence, by using that $\mathbb{R} \ni s \mapsto (\cos_0(s), \sin(s), \cos_0(\pi s), \sin(\pi s))^\top \in \mathbb{R}^4$ is injective, the map
	\begin{equation}
		U \ni u \quad \mapsto \quad \eta_{\mathrm{trig}}(u) := \big( \cos_0(\widetilde{g}_\ell(u)), \sin(\widetilde{g}_\ell(u)), \cos_0(\pi \widetilde{g}_\ell(u)), \sin(\pi \widetilde{g}_\ell(u)) \big)_{\ell=1,\ldots,m}^\top \in \mathbb{R}^{4m}
	\end{equation}
	is an embedding with components from $\mathcal{G}_{\mathrm{trig}}$ satisfying \ref{M4}. For \ref{M5}, we use that $\mathcal{G}_{\mathrm{trig}}$ consists of bounded maps implying that \ref{M5} is already satisfied (see Remark~\ref{rem:psi_mod_growth}). Thus, we can now apply the previous step to conclude that $\mathcal{G}_{\mathrm{trig}}$ is dense in $\mathcal{B}^k_\Psi(U)$. 
	
	Next, we show that $\mathcal{G}_{\mathrm{trig}}$ is contained in the closure of $\mathcal{G}$ with respect to $\Vert \cdot \Vert_{\mathcal{B}^k_\Psi(U)}$. To this end, we fix some $\widetilde{g} \in \widetilde{\mathcal{G}}$ and $\varepsilon > 0$. Then, by \ref{M5}, there exists some $\lambda > 0$ and $R > 0$ such that
	\begin{equation}
		\label{eq:thm:w_nachbin_findim:proof11}
		\max_{j=0,\ldots,k \atop \pi \in \mathscr{P}_j} \sup_{(u,v_1,\ldots,v_j) \in (U \times X^j) \setminus K_{j,R}} \frac{\big( 1 + \exp(\lambda \vert \widetilde{g}(u) \vert) \big) \vert \lambda \vert^{\vert\pi\vert} \vert d^\pi \widetilde{g}(u;v_\pi) \vert}{\psi_j(u,v_1,\ldots,v_j)} < \frac{\varepsilon}{2k!}.
	\end{equation}
	From this, we define the compact set $K := \bigcup_{j=0}^k \pi_0(K_{j,R})$ and the constants $c := \lambda \sup_{u \in K} \vert \widetilde{g}(u) \vert$ as well as $C_{\widetilde{g}} := (1 + \Vert \lambda \widetilde{g} \Vert_{\mathcal{B}^k_\Psi(U)})^k$. Hence, by using the Taylor polynomial $p_n(s) := \sum_{\ell=1}^n \frac{(-1)^\ell}{(2\ell)!} s^{2\ell}$ of $\cos_0: \mathbb{R} \rightarrow \mathbb{R}$, there exists a large enough $n \in \mathbb{N}$ such that
	\begin{equation}
		\label{eq:thm:w_nachbin_findim:proof12}
		\max_{j=0,\ldots,k} \sup_{s \in [-c,c]} \vert \cos_0^{(j)}(s) - p_n^{(j)}(s) \vert \leq \frac{c^{2n-k+1}}{(2n-k+1)!} < \frac{\varepsilon}{2 k! C_\Psi^k C_{\widetilde{g}}}.
	\end{equation}
	Thus, by using the Fa\`a di Bruno formula, $\vert \mathscr{P}_j \vert \leq j! \leq k!$, the monotonicity of $\Psi := (\psi_j)_{j=0,\ldots,k}$, that $\vert \cos^{(\vert\pi\vert)}(s) \vert \leq 1$, that $\vert p_n^{(\vert\pi\vert)}(s) \vert \leq \exp(\vert s \vert)$, and \eqref{eq:thm:w_nachbin_findim:proof11} as well as \eqref{eq:thm:w_nachbin_findim:proof12}, it follows that
	\begin{equation}
		\begin{aligned}
			& \Vert \cos_0(\lambda \widetilde{g}(\cdot)) - p_n(\lambda \widetilde{g}(\cdot)) \Vert_{\mathcal{B}^k_\Psi(U)} \\
			& = \max_{j=0,\ldots,k} \sup_{(u,v_1,\ldots,v_j) \in U \times (\mathbb{R}^d)^j} \frac{\big\vert \sum_{\pi \in \mathscr{P}_j} \big( \cos_0^{(\vert\pi\vert)}(\lambda \widetilde{g}(u)) - p_n^{(\vert\pi\vert)}(\lambda \widetilde{g}(u)) \big) d^\pi (\lambda \widetilde{g})(u;v_\pi) \big\vert}{\psi_j(u,v_1,\ldots,v_j)} \\
			& \leq k! \max_{j=0,\ldots,k} \sup_{(u,v_1,\ldots,v_j) \in (U \times (\mathbb{R}^d)^j) \setminus K_{j,R}} \frac{\big\vert \cos_0^{(\vert\pi\vert)}(\lambda \widetilde{g}(u)) \big\vert \vert d^\pi (\lambda \widetilde{g})(u;v_\pi) \vert}{\psi_j(u,v_1,\ldots,v_j)} \\
			& \quad\quad + k! \max_{j=0,\ldots,k} \sup_{(u,v_1,\ldots,v_j) \in (U \times (\mathbb{R}^d)^j) \setminus K_{j,R}} \frac{\big\vert p_n^{(\vert\pi\vert)}(\lambda \widetilde{g}(u)) \big\vert \vert d^\pi (\lambda \widetilde{g})(u;v_\pi) \vert}{\psi_j(u,v_1,\ldots,v_j)} \\
			& \quad\quad + k! C_\Psi^k \max_{j=0,\ldots,k} \sup_{(u,v_1,\ldots,v_j) \in K_{j,R}} \frac{\big\vert \cos_0^{(\vert\pi\vert)}(\lambda \widetilde{g}(u)) - p_n^{(\vert\pi\vert)}(\lambda \widetilde{g}(u)) \big\vert \prod_{r=1}^{\vert\pi\vert} \vert d^{\vert\pi_r\vert} (\lambda \widetilde{g})(u;v_{\pi_r}) \vert}{\prod_{r=1}^{\vert\pi\vert} \psi_{\vert\pi_r\vert}(u,v_{\pi_r})} \\
			& \leq k! \max_{j=0,\ldots,k} \sup_{(u,v_1,\ldots,v_j) \in (U \times (\mathbb{R}^d)^j) \setminus K_{j,R}} \frac{\big( 1 + \exp(\lambda\vert\widetilde{g}(u)\vert) \big) \vert \lambda \vert^{\vert\pi\vert} \vert d^\pi \widetilde{g}(u;v_\pi) \vert}{\psi_j(u,v_1,\ldots,v_j)} \\
			& \quad\quad + k! C_\Psi^k C_{\widetilde{g}} \max_{j=0,\ldots,k} \sup_{s \in [-c,c]} \vert \cos_0^{(j)}(s) - p_n^{(j)}(s) \vert \\
			& \leq k! \frac{\varepsilon}{2 k!} + k! C_\Psi^k C_{\widetilde{g}} \frac{\varepsilon}{2 k! C_\Psi^k C_{\widetilde{g}}} = \varepsilon.
		\end{aligned}
	\end{equation}
	Since $\varepsilon > 0$ was chosen arbitrarily, the map $\cos_0(\lambda \widetilde{g}(\cdot))$ belongs to the closure of $\mathcal{G}$ with respect to $\Vert \cdot \Vert_{\mathcal{B}^k_\Psi(U)}$, which holds analogously true for $\sin(\lambda \widetilde{g}(\cdot))$. Hence, by using that $\mathbb{R} \ni t \mapsto \cos_0(t \widetilde{g}(\cdot)) \in \mathcal{B}^k_\Psi(U)$ and $\mathbb{R} \ni t \mapsto \sin(t \widetilde{g}(\cdot)) \in \mathcal{B}^k_\Psi(U)$ are real-analytic (see Lemma~\ref{lem:real_anal}), \cite[Lemma~3.7]{cuchiero23} ensures that $\cos_0(t \widetilde{g}(\cdot)) \in \mathcal{B}^k_\Psi(U)$ and $\sin(t \widetilde{g}(\cdot)) \in \mathcal{B}^k_\Psi(U)$, for all $t \in \mathbb{R}$, which shows by taking $t = 1$ that $\mathcal{G}_{\mathrm{trig}}$ is contained in the closure of $\mathcal{G}$ with respect to $\Vert \cdot \Vert_{\mathcal{B}^k_\Psi(U)}$. Combining this with the previous step, i.e., that $\mathcal{G}_{\mathrm{trig}}$ is dense in $\mathcal{B}^k_\Psi(U)$, it follows that $\mathcal{G}$ is also dense in $\mathcal{B}^k_\Psi(U)$.
	
	Finally, for a general subalgebra $\mathcal{G} \subseteq \mathcal{B}^k_\Psi(M)$ over a weighted $C^k_{loc}$-manifold $(M,\Psi)$, we observe that $\mathcal{G}_i := \big\lbrace g \circ \phi_i^{-1}: g \in \mathcal{G} \big\rbrace \subseteq \mathcal{B}^k_{\Psi_i}(\phi_i(U_i))$ is a strongly point separating and nowhere vanishing subalgebra of $\Psi_i$-moderate growth (as $\phi_i: U_i \rightarrow \phi_i(U_i)$ is a diffeomorphism), whence $\mathcal{G}_i$ is by the previous step dense in $\mathcal{B}^k_{\Psi_i}(\phi_i(U_i))$. Thus, by using that $\mathcal{B}^k_\Psi(M)$ is equipped with the initial topology with respect to \eqref{eq:def:BkPsi_mfd}, it follows that $\mathcal{G}$ is dense in $\mathcal{B}^k_\Psi(M)$.
\end{proof}

\begin{remark}
	Theorem~\ref{thm:w_nachbin_findim} generalizes the weighted approximation results in \cite[Proposition~4]{nachbin91} from polynomials over $\mathbb{R}^d$ to the notion of subalgebras on more general weighted $C^k_{loc}$-manifolds.
\end{remark}

Next, we extend the weighted Nachbin theorem to the vector-valued case. To this end, we first assume that the output space $(Y,\Vert \cdot \Vert_Y)$ has the bounded approximation property (BAP), which ensures that $\mathcal{B}^k_\Psi(U) \otimes Y$ is dense in $\mathcal{B}^k_\Psi(U;Y)$. Compared to the compact-open topology, for which $C^k(U) \otimes Y$ is by a compactness argument dense in $C^k(U;Y)$ (see \cite[Lemma~2.1]{prolla76}), the BAP is required in this weighted setting. The proof of the following lemma is given in Appendix~\ref{app:lem:out_bap}.

\begin{lemma}
	\label{lem:out_bap}
	For an open subset $U \subseteq \mathbb{R}^d$, let $(U,\Psi)$ be a weighted domain and assume that $(Y,\Vert \cdot \Vert_Y)$ is a Banach space having BAP. Then, 
	\begin{equation}
		\mathcal{B}^k_\Psi(U) \otimes Y := \linspan\left\lbrace U \ni u \mapsto f(u) y \in Y: f \in \mathcal{B}^k_\Psi(U), \, y \in Y \right\rbrace
	\end{equation}
	is a dense subset of $\mathcal{B}^k_\Psi(U;Y)$.
\end{lemma}

Now, we combine the property of polynomial algebras $\mathcal{G} \subseteq \mathcal{B}^k_\Psi(M;Y)$ (see Section~\ref{sec:nachbin:cl}) with Lemma~\ref{lem:out_bap} to obtain the following vector-valued weighted Nachbin theorem.

\begin{theorem}[Nachbin on $\mathcal{B}^k_\Psi(M;Y)$]
	\label{thm:w_nachbin_findim_vector}
	Let $(M,\Psi)$ be a weighted $C^k_{loc}$-manifold over some finite-dimensional vector spaces and assume that $(Y,\Vert \cdot \Vert_Y)$ is a Banach space having BAP. Moreover, let $\mathcal{G} \subseteq \mathcal{B}^k_\Psi(M;Y)$ be a polynomial subalgebra such that $\mathcal{G}' := \left\lbrace \ell \circ g: \ell \in Y^*, \, g \in \mathcal{G} \right\rbrace$ is strongly point separating and nowhere vanishing of $\Psi$-moderate growth. Then, $\mathcal{G}$ is dense in $\mathcal{B}^k_\Psi(M;Y)$.
\end{theorem}
\begin{proof}
	First, we show the conclusion for a polynomial subalgebra $\mathcal{G} \subseteq \mathcal{B}^k_\Psi(U;Y)$ over a finite-dimensional weighted domain $(U,\Psi)$. Since $\mathcal{B}^k_\Psi(U) \otimes Y$ is by Lemma~\ref{lem:out_bap} dense in $\mathcal{B}^k_\Psi(U;Y)$, it suffices to approximate any map $f \in \mathcal{B}^k_\Psi(U) \otimes Y$ by some element in $\mathcal{G}$. To this end, we fix some $f = \sum_{n=1}^N f_n(\cdot) y_n \in \mathcal{B}^k_\Psi(U) \otimes Y$ and $\varepsilon > 0$, where $N \in \mathbb{N}$, $f_1,\ldots,f_N \in \mathcal{B}^k_\Psi(U)$, and $y_1,\ldots,y_N \in Y$. Then, by applying Theorem~\ref{thm:w_nachbin_findim}, we conclude that $\mathcal{G}'$ is dense in $\mathcal{B}^k_\Psi(U)$, which implies the existence of some $g_1,\ldots,g_N \in \mathcal{G}'$ such that $\Vert f_n - g_n \Vert_{\mathcal{B}^k_\Psi(U)} < \frac{\varepsilon}{N (1+\Vert y_n \Vert_Y)}$. Hence, by using that $g := \sum_{n=1}^N g_n(\cdot) y_n \in \mathcal{G}' \otimes Y \subseteq \mathcal{G}$ (see \cite[Lemma~4.6]{prolla77}), it follows that
	\begin{equation}
		\begin{aligned}
			\Vert f - g \Vert_{\mathcal{B}^k_\Psi(U;Y)} & \leq \sum_{n=1}^N \left\Vert f_n(\cdot) y_n - g_n(\cdot) y_n \right\Vert_{\mathcal{B}^k_\Psi(U;Y)} \leq \sum_{n=1}^N \Vert f_n - g_n \Vert_{\mathcal{B}^k_\Psi(U)} \Vert y_n \Vert_Y \\
			& \leq \sum_{n=1}^N \frac{\varepsilon}{N (1+\Vert y_n \Vert_Y)} \Vert y_n \Vert_Y \leq \varepsilon.
		\end{aligned}
	\end{equation}
	Since $f \in \mathcal{B}^k_\Psi(U) \otimes Y$ and $\varepsilon > 0$ were chosen arbitrarily, and $\mathcal{B}^k_\Psi(U) \otimes Y$ is a dense subset of $\mathcal{B}^k_\Psi(U;Y)$, this shows that $\mathcal{G}$ is dense in $\mathcal{B}^k_\Psi(U;Y)$.
	
	Finally, for a general polynomial subalgebra $\mathcal{G} \subseteq \mathcal{B}^k_\Psi(M;Y)$ over a weighted $C^k_{loc}$-manifold $(M,\Psi)$, we observe that $\mathcal{G}_i := \left\lbrace g \circ \phi_i^{-1}: g \in \mathcal{G} \right\rbrace \subseteq \mathcal{B}^k_{\Psi_i}(\phi_i(U_i);Y)$ is a polynomial subalgebra such that $\mathcal{G}_i' := \left\lbrace \ell \circ g_i: \ell \in Y^*, \, g_i \in \mathcal{G}_i \right\rbrace$ is a strongly point separating and nowhere vanishing subalgebra of $\Psi_i$-moderate growth (as $\phi_i: U_i \rightarrow \phi_i(U_i)$ is a diffeomorphism). Hence, $\mathcal{G}_i$ is by the previous step dense in $\mathcal{B}^k_{\Psi_i}(\phi_i(U_i);Y)$. Since $\mathcal{B}^k_\Psi(M;Y)$ is equipped with the initial topology with respect to \eqref{eq:def:BkPsi_mfd}, $\mathcal{G}$ is dense in $\mathcal{B}^k_\Psi(M;Y)$.
\end{proof}

\subsection{Weighted Nachbin theorems over infinite-dimensional manifolds}
\label{sec:nachbinWeightedInfDim}

In this section, we generalize the weighted Nachbin theorem to infinite-dimensional input manifolds by assuming the bounded approximation property (BAP). Recall that a domain $(U,\tau_X)$ is said to have BAP if $(X,\tau_X)$ has BAP, i.e., $(X,\tau_X)$ has BAP with finite rank operators $(T_\gamma)_\gamma \subseteq (X,\tau_X)^* \otimes X$ satisfying $T_\gamma(U) \subseteq U$ (see also Sections~\ref{sec:notation} and \ref{sec:BPsik_maps}).

\begin{lemma}
	\label{lem:inp_bap}
	Let $(U,\Psi)$ be a weighted domain such that $(U,\tau_X)$ has BAP and assume that $(Y,\Vert \cdot \Vert_Y)$ is a Banach space. Then, for every $f \in C^k_b(U;Y)$ there exists a net of finite rank operators $(T_\gamma)_\gamma \subseteq X^* \otimes X$ with $T_\gamma(U) \subseteq U$ such that $\lim_\gamma \Vert f - f \circ T_\gamma \Vert_{\mathcal{B}^k_\Psi(U;Y)} = 0$.
\end{lemma}

We now first apply Lemma~\ref{lem:out_bap} to reduce the approximation problem to a finite-dimensional domain and then apply the weighted Nachbin theorem (Theorem~\ref{thm:w_nachbin_findim}).

\begin{theorem}[Nachbin on $\mathcal{B}^k_\Psi(M)$]
	\label{thm:w_nachbin_infdim}
	Let $(M,\Psi)$ be a weighted $C^k_{loc}$-manifold with atlas $(U_i,\phi_i)_{i \in I}$ over model spaces $(X_i,\tau_{X_i})_{i \in I}$ such that each domain $(\phi_i(U_i),\tau_{X_i})$ has BAP with finite rank operators $(T_{i,\gamma})_\gamma \subseteq (X_i,\tau_{X_i})^* \otimes X_i$. Moreover, let $\mathcal{G} \subseteq \mathcal{B}^k_\Psi(M)$ be a subalgebra such that
	\begin{enumerate}
		\item\label{thm:w_nachbin_infdim:1} $\mathcal{G}$ is strongly point separating and nowhere vanishing of $\Psi$-moderate growth, and
		\item\label{thm:w_nachbin_infdim:2} for every $g \in \mathcal{G}$, $i \in I$, and $\gamma$ the composition $g \circ \phi_i^{-1} \circ T_{i,\gamma}: \phi_i(U_i) \rightarrow \mathbb{R}$ belongs to the closure of $\mathcal{G}_i := \big\lbrace g \circ \phi_i^{-1}: g \in \mathcal{G} \big\rbrace$ with respect to $\Vert \cdot \Vert_{\mathcal{B}^k_{\Psi_i}(\phi_i(U_i))}$.
	\end{enumerate}
	Then, $\mathcal{G}$ is dense in $\mathcal{B}^k_\Psi(M)$.
\end{theorem}
\begin{proof}
	First, we show the conclusion for a subalgebra $\mathcal{G} \subseteq \mathcal{B}^k_\Psi(U)$ over a weighted domain $(U,\Psi)$. Since $\mathcal{B}^k_\Psi(U)$ is defined as the closure of $C^k_b(U)$ with respect to $\Vert \cdot \Vert_{\mathcal{B}^k_\Psi(U)}$, it suffices to approximate any given $f \in C^k_b(U)$ by an element of $\mathcal{G}$. Moreover, by using that $(U,\tau_X)$ has BAP with finite rank operators $(T_\gamma)_\gamma \subseteq (X,\tau_X)^* \otimes X$, there exists by Lemma~\ref{lem:inp_bap} some $\gamma$ such that
	\begin{equation}
		\label{eq:thm:w_nachbin_infdim:proof1}
		\Vert f - f \circ T_\gamma \Vert_{\mathcal{B}_\Psi^k(U)} < \frac{\varepsilon}{3}.
	\end{equation}
	Now, we define the collection $\Psi_\gamma := (\psi_{\gamma,j})_{j=0,\ldots,k}$ of weights $\psi_{\gamma,j}: T_\gamma(U) \times T_\gamma(X)^j \rightarrow (0,\infty)$ by
	\begin{equation}
		\label{eq:thm:w_nachbin_infdim:proof2}
		\psi_{\gamma,j}(\widetilde{u},\widetilde{v}_1,\ldots,\widetilde{v}_j) := \inf_{(u,v_1,\ldots,v_j) \in U \times X^j \atop T_\gamma^{j+1}(u,v_1,\ldots,v_j) = (\widetilde{u},\widetilde{v}_1,\ldots,\widetilde{v}_j)} \psi_j(u,v_1,\ldots,v_j),
	\end{equation}
	for $(\widetilde{u},\widetilde{v}_1,\ldots,\widetilde{v}_j) \in T_\gamma(U) \times T_\gamma(X)^j$, where $T_\gamma^{j+1}: U \times X^j \rightarrow T_\gamma(U) \times T_\gamma(X)^j$ is defined by $T_\gamma^{j+1}(u,v_1,\ldots,v_j) := (T_\gamma(u),T_\gamma(v_1),\ldots,T_\gamma(v_j))$, for $(u,v_1,\ldots,v_j) \in U \times X^j$. Then, for every $R > 0$, we claim that
	\begin{equation}
		\label{eq:thm:w_nachbin_infdim:proof3}
		\psi_{\gamma,j}^{-1}((0,R]) = T_\gamma^{j+1}(K_{j,R}).
	\end{equation}
	For $\psi_{\gamma,j}^{-1}((0,R]) \supseteq T_\gamma^{j+1}(K_{j,R})$, there exists for any $(\widetilde u,\widetilde v_1,\ldots,\widetilde v_j) \!\in\! T_\gamma^{j+1}(K_{j,R})$ some $(u,v_1,\ldots,v_j) \!\in\! K_{j,R}$ such that $T_\gamma^{j+1}(u,v_1,\ldots,v_j) = (\widetilde{u},\widetilde{v}_1,\ldots,\widetilde{v}_j)$. Hence, by the definition of $\psi_{\gamma,j}$, it holds that
	\begin{equation}
		\psi_{\gamma,j}(\widetilde{u},\widetilde{v}_1,\ldots,\widetilde{v}_j) \leq \psi_j(u,v_1,\ldots,v_j) \leq R,
	\end{equation}
	which shows that $(\widetilde{u},\widetilde{v}_1,\ldots,\widetilde{v}_j) \in \psi_{\gamma,j}^{-1}((0,R])$. Conversely, for $\psi_{\gamma,j}^{-1}((0,R]) \subseteq T_\gamma^{j+1}(K_{j,R})$, we fix some $(\widetilde{u},\widetilde{v}_1,\ldots,\widetilde{v}_j) \in \psi_{\gamma,j}^{-1}((0,R])$. Then, by definition of $\psi_{\gamma,j}$, there exists for every $n \in \mathbb{N}$ some $\big(u^{(n)},v^{(n)}_1,\ldots,v^{(n)}_j\big) \in U \times X^j$ with $T_\gamma^{j+1}\big(u^{(n)},v^{(n)}_1,\ldots,v^{(n)}_j\big) = (\widetilde{u},\widetilde{v}_1,\ldots,\widetilde{v}_j)$ such that $\psi_j\big(u^{(n)},v^{(n)}_1,\ldots,v^{(n)}_j\big) \leq R+\frac{1}{n}$, whence $\big(u^{(n)},v^{(n)}_1,\ldots,v^{(n)}_j\big) \in K_{j,R+1}$. Since $K_{j,R+1}$ is compact, there exists a subnet $\big(u^{(\gamma)},v^{(\gamma)}_1,\ldots,v^{(\gamma)}_j\big)_\gamma$, converging to some $(u,v_1,\ldots,v_j) \in K_{j,R+1}$, which together with the continuity of $T_\gamma^{j+1}$ implies that 
	\begin{equation}
		T_\gamma^{j+1}(u,v_1,\ldots,v_j) = \lim_\gamma T_\gamma^{j+1}\big(u^{(\gamma)},v^{(\gamma)}_1,\ldots,v^{(\gamma)}_j\big) = (\widetilde{u},\widetilde{v}_1,\ldots,\widetilde{v}_j).
	\end{equation}
	Moreover, by using that $\psi_j$ is lower semicontinuous, it follows that
	\begin{equation}
		\psi_j(u,v_1,\ldots,v_j) \leq \liminf_\gamma \psi_j\big(u^{(\gamma)},v^{(\gamma)}_1,\ldots,v^{(\gamma)}_j\big) \leq R.
	\end{equation}
	Hence, $(u,v_1,\ldots,v_j) \in K_{j,R}$ and therefore $(\widetilde{u},\widetilde{v}_1,\ldots,\widetilde{v}_j) = T_\gamma^{j+1}(u,v_1,\ldots,v_j) \in T_\gamma^{j+1}(K_{j,R})$. This proves \eqref{eq:thm:w_nachbin_infdim:proof3}, which ensures that $\psi_{\gamma,j}^{-1}((0,R]) = T_\gamma^{j+1}(K_{j,R})$ is compact as continuous image of the compact set $K_{j,R}$, showing that $\Psi_\gamma := (\psi_{\gamma,j})_{j=0,\ldots,k}$ is admissible on $T_\gamma(U) \subseteq T_\gamma(X)$. 
	
	Moreover, we claim that $\mathcal{G}_\gamma := \mathcal{G}\vert_{T_\gamma(U)}$ is a strongly point separating and nowhere vanishing of $\Psi_\gamma$-moderate growth. Indeed, while the conditions of point separation, nowhere vanishing, and nowhere vanishing derivatives in \ref{M1}--\ref{M3} are inherited to the sub-domain $T_\gamma(U) \subseteq U$, the other conditions \ref{M4'}--\ref{M5'} are defined such that $\mathcal{G}_\gamma = \mathcal{G}\vert_{T_\gamma(U)}$ satisfies \ref{M4}--\ref{M5} on $(T_\gamma(U),\Psi_\gamma)$. Hence, we can apply the weighted Nachbin theorem (Theorem~\ref{thm:w_nachbin_findim}) on the map $f\vert_{T_\gamma(U)} \in \mathcal{B}^k_{\Psi_\gamma}(T_\gamma(U))$ to obtain some $\widetilde{g} \in \mathcal{G}_\gamma$ such that
	\begin{equation}
		\begin{aligned}
			& \Vert f\vert_{T_\gamma(U)} - \widetilde{g} \Vert_{\mathcal{B}^k_{\Psi_\gamma}(T_\gamma(U))} \\
			& = \max_{j=0,\ldots,k} \sup_{(\widetilde{u},\widetilde{v}_1,\ldots,\widetilde{v}_j) \in T_\gamma(U) \times T_\gamma(X)^j} \frac{\left\vert d^j f(\widetilde{u};\widetilde{v}_1,\ldots,\widetilde{v}_j) - d^j \widetilde{g}(\widetilde{u};\widetilde{v}_1,\ldots,\widetilde{v}_j) \right\vert}{\psi_{\gamma,j}(\widetilde{u},\widetilde{v}_1,\ldots,\widetilde{v}_j)} < \frac{\varepsilon}{3}.
		\end{aligned}
	\end{equation}
	Thus, by using the chain rule, we conclude that
	\begin{equation}
		\label{eq:thm:w_nachbin_infdim:proof4}
		\begin{aligned}
			& \Vert f \circ T_\gamma - \widetilde{g} \circ T_\gamma \Vert_{\mathcal{B}_\Psi^k(U)} \\
			& = \max_{j=0,\ldots,k} \sup_{(u,v_1,\ldots,v_j) \in U \times X^j} \frac{\vert d^j (f \circ T_\gamma)(u;v_1,\ldots,v_j) - d^j (\widetilde{g} \circ T_\gamma)(u;v_1,\ldots,v_j) \vert}{\psi_j(u,v_1,\ldots,v_j)} \\
			& = \max_{j=0,\ldots,k} \sup_{(u,v_1,\ldots,v_j) \in U \times X^j} \frac{\vert d^j f(T_\gamma(u);T_\gamma(v_1),\ldots,T_\gamma(v_j)) - d^j \widetilde{g}(T_\gamma(u);T_\gamma(v_1),\ldots,T_\gamma(v_j)) \vert}{\psi_j(u,v_1,\ldots,v_j)} \\
			& \leq \max_{j=0,\ldots,k} \sup_{(\widetilde{u},\widetilde{v}_1,\ldots,\widetilde{v}_j) \in T_\gamma(U) \times T_\gamma(X)^j} \frac{\vert d^j f(\widetilde{u};\widetilde{v}_1,\ldots,\widetilde{v}_j) - d^j \widetilde{g}(\widetilde{u};\widetilde{v}_1,\ldots,\widetilde{v}_j) \vert}{\psi_{\gamma,j}(\widetilde{u},\widetilde{v}_1,\ldots,\widetilde{v}_j)} < \frac{\varepsilon}{3}.
		\end{aligned}
	\end{equation}
	Next, we use that $\widetilde{g} \circ T_\gamma$ belongs by \ref{thm:w_nachbin_infdim:2} to the closure of $\mathcal{G}$ with respect to $\Vert \cdot \Vert_{\mathcal{B}^k_\Psi(U)}$ to obtain some $g \in \mathcal{G}$ such that
	\begin{equation}
		\label{eq:thm:w_nachbin_infdim:proof5}
		\Vert \widetilde{g} \circ T_\gamma - g \Vert_{\mathcal{B}^k_\Psi(U)} < \frac{\varepsilon}{3}.
	\end{equation}
	Finally, by combining \eqref{eq:thm:w_nachbin_infdim:proof1}, \eqref{eq:thm:w_nachbin_infdim:proof4}, and \eqref{eq:thm:w_nachbin_infdim:proof5}, it follows that
	\begin{equation}
		\begin{aligned}
			\Vert f - g \Vert_{\mathcal{B}^k_\Psi(U)} & \leq \Vert f - f \circ T_\gamma \Vert_{\mathcal{B}^k_\Psi(U)} + \Vert f \circ T_\gamma - \widetilde{g} \circ T_\gamma \Vert_{\mathcal{B}^k_\Psi(U)} + \Vert \widetilde{g} \circ T_\gamma - g \Vert_{\mathcal{B}^k_\Psi(U)} \\
			& < \frac{\varepsilon}{3} + \frac{\varepsilon}{3} + \frac{\varepsilon}{3} = \varepsilon.
		\end{aligned}
	\end{equation}
	Since $f \in C^k_b(U)$ and $\varepsilon > 0$ were chosen arbitrarily, this shows that $\mathcal{G}$ is dense in $\mathcal{B}^k_\Psi(U)$.
	
	Finally, for a general subalgebra $\mathcal{G} \subseteq \mathcal{B}^k_\Psi(M)$ over a weighted $C^k_{loc}$-manifold $(M,\Psi)$, we observe that $\mathcal{G}_i := \big\lbrace g \circ \phi_i^{-1}: g \in \mathcal{G} \big\rbrace \subseteq \mathcal{B}^k_{\Psi_i}(\phi_i(U_i))$ is a strongly point separating and nowhere vanishing subalgebra of $\Psi_i$-moderate growth satisfying \ref{thm:w_nachbin_infdim:2}. Hence, we can apply the previous step to conclude that $\mathcal{G}_i$ is dense in $\mathcal{B}^k_{\Psi_i}(\phi_i(U_i))$. Thus, by using that $\mathcal{B}^k_\Psi(M)$ is equipped with the initial topology with respect to \eqref{eq:def:BkPsi_mfd}, it follows that $\mathcal{G}$ is dense in $\mathcal{B}^k_\Psi(M)$.
\end{proof}

Moreover, by following the arguments of Theorem~\ref{thm:w_nachbin_findim_vector}, we can derive the following vector-valued weighted Nachbin theorem over infinite-dimensional manifolds. 

\begin{corollary}[Nachbin on $\mathcal{B}^k_\Psi(M;Y)$]
	\label{cor:w_nachbin_infdim_vector}
	Let $(M,\Psi)$ be a weighted $C^k_{loc}$-manifold with atlas $(U_i,\phi_i)_{i \in I}$ over model spaces $(X_i,\tau_{X_i})_{i \in I}$ such that each domain $(\phi_i(U_i),\tau_{X_i})$ has BAP with finite rank operators $(T_{i,\gamma})_\gamma \subseteq (X_i,\tau_{X_i})^* \otimes X_i$. Moreover, let $(Y,\Vert \cdot \Vert_Y)$ be a Banach space having BAP and assume that $\mathcal{G} \subseteq \mathcal{B}^k_\Psi(M;Y)$ is a polynomial subalgebra such that
	\begin{enumerate}
		\item $\mathcal{G}' := \left\lbrace \ell \circ g: \ell \in Y^*, \, g \in \mathcal{G} \right\rbrace$ is strongly point separating and nowhere vanishing of $\Psi$-moderate growth, and
		\item for every $g \in \mathcal{G}$, $i \in I$, and $\gamma$ the composition $g \circ \phi_i^{-1} \circ T_{i,\gamma}: \phi_i(U_i) \rightarrow Y$ belongs to the closure of $\mathcal{G}_i := \big\lbrace g \circ \phi_i^{-1}: g \in \mathcal{G} \big\rbrace$ with respect to $\Vert \cdot \Vert_{\mathcal{B}^k_{\Psi_i}(\phi_i(U_i);Y)}$.
	\end{enumerate}
	Then, $\mathcal{G}$ is dense in $\mathcal{B}^k_\Psi(M;Y)$.
\end{corollary}
\begin{proof}
	The proof follows from the same arguments as in the proof of Theorem~\ref{thm:w_nachbin_findim_vector}, where we now apply Theorem~\ref{thm:w_nachbin_infdim} instead of Theorem~\ref{thm:w_nachbin_findim}.
\end{proof}

\section{Weighted universal approximation of functional input neural networks}
\label{sec:w_uat}

We now introduce a generalization of neural networks to infinite-dimensional spaces, called functional input neural networks (FNNs), and show different universal approximation theorems (UATs) for FNNs. To this end, we assume that the input space $(M,\Psi)$ is a weighted $C^k_{loc}$-manifold with admissible collection $\Psi = (\psi_j)_{j=0,\ldots,k}$ of weight functions $\psi_j: T^j M \rightarrow (0,\infty)$, $j = 0,\ldots,k$. Moreover, the output space $(Y,\Vert \cdot \Vert_Y)$ is supposed to be a Banach space.

\subsection{Functional input neural networks}

In this section, we define neural networks between infinite-dimensional spaces. To this end, we first introduce the infinite-dimensional analogue of weight matrices that connect adjacent layers in classical neural networks.

\begin{definition}
	\label{def:add_fam}
	Let $(M,\tau_M)$ be a $C^k_{loc}$-manifold over finite-dimensional model spaces $(X_i,\tau_{X_i})_{i \in I}$. A subset $\mathcal{A} \subseteq C^k_{loc}(M)$ is called an \emph{additive family (on $M$)} if
	\begin{enumerate}
		\item[\labeltext{(A1)}{A1}] $\mathcal{A}$ is closed under addition, i.e., for every $a_1, a_2 \in \mathcal{A}$ it holds that $a_1 + a_2 \in \mathcal{A}$,
		\item[\labeltext{(A2)}{A2}] $\mathcal{A}$ is point separating on $M$, i.e., for any distinct points $x_1, x_2 \in M$ there exists some $a \in \mathcal{A}$ such that $a(x_1) \neq a(x_2)$, and
		\item[\labeltext{(A3)}{A3}] $\mathcal{A}$ has nowhere vanishing derivatives on $M$, i.e., for every $(x,[c]^1_x) \in T^1 M$ with $c'(0) \neq 0$ there exists some $a \in \mathcal{A}$ such that $(a \circ c)'(0) \neq 0$.
		\item[\labeltext{(A4)}{A4}] for every $i \in I$ there exist $a_1,\ldots,a_m \in \mathcal{A}$ such that $\eta_i := (a_1 \circ \phi_i^{-1}, \ldots, a_m \circ \phi_i^{-1})^\top: \phi_i(U_i) \rightarrow \eta_i(\phi_i(U_i))$ is an embedding and there exist some cutoff functions $(h_{i,R})_{R > 0} \subseteq C^\infty_c(\eta_i(\phi_i(U_i)))$ with $0 \leq h_{i,R} \leq 1$ and $h_{i,R}\vert_{\eta_i(K_{i,R})} = 1$ such that
		\begin{equation}
			\quad\quad\quad \lim_{R \rightarrow \infty} \max_{1 \leq \ell \leq j \leq k} \sup_{(u,v_1,\ldots,v_j) \in (\phi_i(U_i) \times X_i^j) \setminus K_{i,j,R}} \frac{\Vert d^\ell (h_{i,R} \circ \eta_i)(u) \Vert_{L^\ell(X_i;\mathbb{R})} \Vert v_1 \Vert \cdots \Vert v_j \Vert}{\psi_{i,j}(u,v_1,\ldots,v_j)} = 0,
		\end{equation}
		where $K_{i,R} := \bigcup_{j=0}^k \pi_{i,0}(K_{i,j,R})$ with $\phi_i(U_i) \times X_i^j \ni (u,v_{1:j}) \mapsto \pi_{i,0}(u,v_{1:j}) := u \in \phi_i(U_i)$.
	\end{enumerate}
	If $(M,\tau_M)$ is a $C^k_{loc}$-manifold over infinite-dimensional model spaces $(X_i,\tau_{X_i})_{i \in I}$ each having BAP with finite rank operators $(T_{i,\gamma})_\gamma$, we replace condition~\ref{A4} by
	\begin{enumerate}
		\item[\labeltext{(A4')}{A4'}] for every $i \in I$ and $\gamma$ the set $\mathcal{A}_i\vert_{T_{i,\gamma}(\phi_i(U_i))} := \left\lbrace a \circ \phi_i^{-1}\vert_{T_{i,\gamma}(\phi_i(U_i))}: a \in \mathcal{A} \right\rbrace$ satisfies \ref{A4}.
	\end{enumerate}
\end{definition}

\begin{remark}
	In contrast to \cite[Definition~4.1]{cuchiero23}, we do not include the constants in the additive family. Under this consideration, the conditions \ref{A1}--\ref{A2} are the same as in the non-differentiable case of \cite[Definition~4.1]{cuchiero23}, whereas \ref{A3}--\ref{A4} are additionally required for the approximation of the derivatives in our weighted setting. In addition, if the embedding $\eta_i$ in \ref{A4} has uniformly bounded derivatives, then the corresponding limit is zero (see also Remark~\ref{rem:psi_mod_growth}).           
\end{remark}

For an open subset $U \subseteq X$ of the Euclidean space $X := \mathbb{R}^d$, we observe that the weight matrices in classical neural networks form an additive family.

\begin{example}
	\label{ex:add_fam:eucl}
	For an open subset $M := U \subseteq X$ of the Euclidean space $X := \mathbb{R}^d$, an additive family is given by $\mathcal{A} = \left\lbrace M \ni x \mapsto a^\top x \in \mathbb{R}: a \in \mathbb{R}^d \right\rbrace$. Note that $\mathcal{A} = \left\lbrace M \ni x \mapsto a^\top x \in \mathbb{R}: a \in \mathbb{N}_0^d \right\rbrace$ is an even smaller additive family.
\end{example}

\begin{definition}
	\label{def:fnn}
	For a given additive family $\mathcal{A} \subseteq C^k_{loc}(M)$, a function $\rho \in C^k(\mathbb{R})$, and a subset $\mathcal{L} \subseteq Y$, we define a \emph{functional input neural network (FNN)} $\varphi: M \rightarrow Y$ as
	\begin{equation}
		\label{eq:def:fnn:1}
		M \ni x \quad \mapsto \quad \varphi(x) = \sum_{n=1}^N y_n \rho(a_n(x)+b_n) \in Y,
	\end{equation}
	where $N \in \mathbb{N}$ denotes the \emph{number of neurons}, where $a_1,\ldots,a_N \in \mathcal{A}$ are the \emph{hidden layer maps}, where $b_1,\ldots,b_N \in \mathbb{R}$ represent the \emph{biases}, and where $y_1,\ldots,y_N \in \mathcal{L}$ are the \emph{linear readouts}. Moreover, we denote by $\mathcal{NN}^{\mathcal{A},\rho,\mathcal{L}}_{M,Y}$ the set of FNNs of the form \eqref{eq:def:fnn:1}.
\end{definition}

\vspace{-0.1cm}

\begin{figure}[ht]
	\centering
	\begin{tikzpicture}[
		inputnode/.style={circle, draw=green!60, fill=green!5, very thick, minimum size=5mm},
		hiddennode/.style={circle, draw=blue!60, fill=blue!5, very thick, minimum size=5mm},
		outputnode/.style={circle, draw=red!60, fill=red!5, very thick, minimum size=5mm},
		node distance=7mm,
		]
		\node[inputnode] (x1) {};
		\node[hiddennode] (y2) [right = 2cm of x1] {};
		\node[hiddennode] (y1) [above of = y2] {};
		\node[hiddennode] (y3) [below of = y2] {};
		\node[outputnode] (o1) [right = 2cm of y2] {};
		
		\draw[shorten >=0.1cm,shorten <=0.1cm, ->] (x1.east) -- (y1.west);	
		\draw[shorten >=0.1cm,shorten <=0.1cm, ->] (x1.east) -- (y2.west);
		\draw[shorten >=0.1cm,shorten <=0.1cm, ->] (x1.east) -- (y3.west);
		\draw[shorten >=0.1cm,shorten <=0.1cm, ->] (y1.east) -- (o1.west);
		\draw[shorten >=0.1cm,shorten <=0.1cm, ->] (y2.east) -- (o1.west);
		\draw[shorten >=0.1cm,shorten <=0.1cm, ->] (y3.east) -- (o1.west);
		\draw[shorten >=0.2cm,shorten <=0.2cm, ->] (y3) to[out=-120,in=-60,loop] ();
		
		\draw[] (0,1.0) node[anchor=center, align=center] {\footnotesize Input Layer};
		\draw[] (0,0.65) node[anchor=center, align=center] {\footnotesize $(M,\Psi)$};
		\draw[] (2.5,1.6) node[anchor=center, align=center] {\footnotesize Hidden Layer};
		\draw[] (2.5,1.25) node[anchor=center, align=center] {\footnotesize $\mathbb{R}$};
		\draw[] (5.1,1.0) node[anchor=center, align=center] {\footnotesize Output Layer};
		\draw[] (5.1,0.65) node[anchor=center, align=center] {\footnotesize $(Y,\Vert \cdot \Vert_Y)$};
		\draw[] (-0.9,0) node[anchor=center, align=center] {\footnotesize $M \ni x$};
		\draw[] (6.2,0) node[anchor=center, align=center] {\footnotesize $\varphi(x) \in Y$};
		\draw[] (1.0,-0.8) node[anchor=center, align=center] {\footnotesize $\mathcal{A}\oplus\mathbb{R}$};
		\draw[] (3.8,-0.8) node[anchor=center, align=center] {\footnotesize $\mathcal{L}$};
		\draw[] (2.53,-1.3) node[anchor=center, align=center] {\footnotesize $\rho$};
	\end{tikzpicture}
	\caption{A functional input neural network $\varphi: M \rightarrow Y$ with additive family $\mathcal{A}$, activation function $\rho \in C^k(\mathbb{R})$, linear readout $\mathcal{L} \subseteq Y$, and $N = 3$ neurons.}
	\label{fig:fnn}
\end{figure}
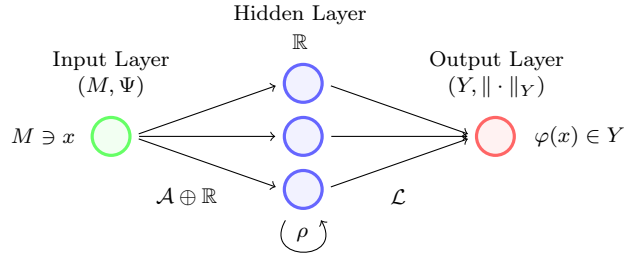

\vspace{-0.1cm}

\begin{remark}
	\label{ref:fnn:eucl}
	Definition~\ref{def:fnn} extends the notion of classical neural networks between Euclidean spaces. Indeed, let $\varphi: \mathbb{R}^d \rightarrow \mathbb{R}^m$ be a classical neural network of the form
	\begin{equation}
		\label{eq:ref:fnn:eucl:1}
		\mathbb{R}^d \ni x \quad \mapsto \quad \varphi(x) = W \rho(Ax + b) = \sum_{n=1}^N y_n \rho\left( a_n^\top x + b_n \right) \in \mathbb{R}^m,
	\end{equation}
	for some $W = (y_1,\ldots,y_N) \in \mathbb{R}^{m \times N}$, $A = (a_1,\ldots,a_N)^\top \in \mathbb{R}^{N \times d}$, and $b = (b_1,\ldots,b_N)^\top \in \mathbb{R}^N$, where $y_1,\ldots,y_N \in \mathbb{R}^m$ denote the columns of $W \in \mathbb{R}^{m \times N}$, and where $a_1,\ldots,a_N \in \mathbb{R}^d$ represent the rows of $A \in \mathbb{R}^{N \times d}$. Moreover, by a slight abuse of notation, $\rho \in C^k(\mathbb{R})$ is applied componentwise to $Ax + b \in \mathbb{R}^N$ after the first equality in \eqref{eq:ref:fnn:eucl:1}. If we choose $\mathcal{A}$ as in Example~\ref{ex:add_fam:eucl} and $\mathcal{L} = \mathbb{R}^m$, then $\varphi: \mathbb{R}^d \rightarrow \mathbb{R}^m$ is a functional input neural network in $\mathcal{NN}^{\mathcal{A},\rho,\mathcal{L}}_{\mathbb{R}^d,\mathbb{R}^m}$. 
\end{remark}

Moreover, we can construct deep functional input neural networks by concatenation. For an additive family $\mathcal{A} \subseteq C^k_{loc}(M)$ and two activation functions $\rho_1,\rho_2 \in C^k(\mathbb{R})$, we introduce a deep FNN with two hidden layers. Indeed, by assuming that $\rho_1 \in C^k(\mathbb{R})$ is (strongly) non-polynomial (see below), the set $\mathcal{NN}^{\mathcal{A},\rho_1,\mathbb{R}}_{M,\mathbb{R}}$ is another additive family on $M$. Hence, a functional input neural network with two hidden layers $\varphi: M \rightarrow Y$ is of the form
\begin{equation}
	\begin{aligned}
		M \ni x \quad \mapsto \quad \varphi(x) & = \sum_{n_2=1}^{N_2} y_{n_2} \rho_2 \left( \varphi_{n_2}(x) + b^{(2)}_{n_2} \right) \\
		& = \sum_{n_2=1}^{N_2} y_{n_2} \rho_2\left( \sum_{n_1=1}^{N_1} w_{n_2, n_1} \rho_1\left( a_{n_2, n_1}(x) + b^{(1)}_{n_1} \right) + b^{(2)}_{n_2} \right) \in Y,
	\end{aligned}
\end{equation}
where $y_1,\ldots,y_{N_2} \in \mathcal{L}$ are linear readouts, $b^{(2)}_1,\ldots,b^{(2)}_{N_2} \in \mathbb{R}$ are the biases of the second layer, $w_{1,1},\ldots,w_{N_2,N_1} \in \mathbb{R}$ are the connections between the layers, and where $a_{1,1},\ldots,a_{N_2,N_1} \in \mathcal{A}$ and $b^{(1)}_1,\ldots,b^{(1)}_{N_1} \in \mathbb{R}$ are the weights and biases of the first layer, respectively. Moreover, $\varphi_1,\ldots,\varphi_{N_2} \in \mathcal{NN}^{\mathcal{A},\rho_1,\mathbb{R}}_{M,\mathbb{R}}$ are FNNs of the form $\varphi_{n_2}(x) := \sum_{n_1=1}^{N_1} w_{n_2, n_1} \rho_1(a_{n_2, n_1}(x) + b^{(1)}_{n_1})$, for all $x \in M$ and $n_2 = 1,\ldots,N_2$. Hence, by an analogous concatenation, it is possible to construct deep functional input neural networks with finitely many hidden layers.

\subsection{Examples of additive families}

In this section, we give some examples of additive families on weighted manifolds having an atlas with one global chart. More precisely, for $k \in \mathbb{N} \cup \lbrace \infty \rbrace$, we consider a weighted $C^k_{loc}$-manifold $(M,\Psi)$ with global chart $\phi_i: M \rightarrow \phi_i(M)$ over a Banach space $(X,\Vert \cdot \Vert_X)$ that is equipped with a weaker topology $\tau_X$ than the norm topology (except $X$ is finite-dimensional). This applies in particular to every open subset $M := U_i$ of $(X,\tau_X)$, where the chart $\phi_i := \id_{U_i}: U_i \rightarrow U_i \subseteq X$ is equal to the identity.

\begin{lemma}
	\label{lem:add_fam:cpt}
	Let $(M,\tau_M)$ be a $C^k_{loc}$-manifold with global chart $\phi_i: M \rightarrow \phi_i(M)$ over a Banach space $(X,\Vert \cdot \Vert_X)$, which is equipped with the initial topology $\tau_{\mathrm{init}}$ of a compact embedding $\Gamma: (X,\Vert \cdot \Vert_X) \rightarrow (X_0,\Vert \cdot \Vert_{X_0})$ into another Banach space $(X_0,\Vert \cdot \Vert_{X_0})$. Moreover, assume that $(X,\tau_X)$ has BAP with finite rank operators $(T_\gamma)_\gamma$. Then,
	\begin{equation}
		\mathcal{A} := \left\lbrace M \ni x \mapsto \ell(\phi_i(x)) \in \mathbb{R}: \ell \in (X,\tau_{\mathrm{init}})^* \right\rbrace \subseteq C^k_{loc}(M)
	\end{equation}
	is an additive family on $M$.
\end{lemma}
\begin{proof}
	First, for every $\ell \in (X,\tau_{\mathrm{init}})^*$, we observe that $a := \ell(\phi_i(\cdot)) \in \mathcal{A}$ satisfies $a \circ \phi_i^{-1} = \ell\vert_{\phi_i(M)} \in C^k_{loc}(\phi_i(M))$, which ensures that $a \in C^k_{loc}(M)$. Now, we verify the conditions~\ref{A1}--\ref{A4}. For \ref{A1}, $\mathcal{A}$ is by definition closed under addition. For \ref{A2}, we fix some distinct points $x_1,x_2 \in M$, which also satisfy $\phi_i(x_1) \neq \phi_i(x_2)$ as $\phi_i: M \rightarrow \phi_i(M)$ is injective. Since $(X,\tau_{\mathrm{init}})^*$ is by the Hahn-Banach theorem point separating on $X$, there exists some $\ell \in (X,\tau_{\mathrm{init}})^*$ such that $a := \ell(\phi_i(\cdot)) \in \mathcal{A}$ satisfies $a(x_1) = \ell(\phi_i(x_1)) \neq \ell(\phi_i(x_2)) = a(x_2)$. For \ref{A3}, we fix some $(x,[c]^1_x) \in T^1 M$ with $c'(0) \neq 0$, which also satisfies $(\phi_i \circ c)'(0) \neq 0$ as $\phi_i: M \rightarrow \phi_i(M)$ is a diffeomorphism. Thus, by using again the Hahn-Banach theorem, there exists some $\ell \in (X,\tau_{\mathrm{init}})^*$ such that $\ell((\phi_i \circ c)'(0)) \neq 0$, whence $a := \ell(\phi_i(\cdot)) \in \mathcal{A}$ satisfies
	\begin{equation}
		(a \circ c)'(0) = (a \circ \phi_i^{-1} \circ \phi_i \circ c)'(0) = d(a \circ \phi_i^{-1})(\phi_i(x);(\phi_i \circ c)'(0)) = \ell((\phi_i \circ c)'(0)) \neq 0.
	\end{equation}
	For \ref{A4'}, we fix some $\gamma$ and consider the finite-dimensional vector subspace $T_\gamma(X) \subseteq X$. Then, there exists a basis $b_1,\ldots,b_m$ of $T_\gamma(X)$ and some $\ell_1,\ldots,\ell_m \in (T_\gamma(X),\tau_{\mathrm{init}})^*$ such that $\ell_n(b_{\widetilde{n}}) = \delta_{n,\widetilde{n}}$, for all $n,\widetilde{n} = 1,\ldots,m$. Hence, by the Hahn-Banach theorem, we can extend $\ell_1,\ldots,\ell_m \in (T_\gamma(X),\tau_{\mathrm{init}})^*$ to some $L_1,\ldots,L_m \in (X,\tau_{\mathrm{init}})^*$ with $L_n\vert_{T_\gamma(X)} = \ell_n$, for all $n = 1,\ldots,m$, which implies that $L := (L_1,\ldots,L_m)^\top\vert_{T_\gamma(X)}: T_\gamma(X) \rightarrow \mathbb{R}^m$ is a linear isomorphism. Hence, by defining $a_n := L_n \circ \phi_i \in \mathcal{A}$, we conclude that
	\begin{equation}
		\eta_{i,\gamma} := \big( a_1 \circ \phi_i^{-1}, \ldots, a_m \circ \phi_i^{-1} \big)^\top\big\vert_{T_\gamma(\phi_i(M))} = L\vert_{T_\gamma(\phi_i(M))}: T_\gamma(\phi_i(M)) \rightarrow \eta_{i,\gamma}(T_\gamma(\phi_i(M)))
	\end{equation}
	is the restriction of a linear isomorphism and therefore an embedding.
\end{proof}

With the additive family $\mathcal{A}$ from Lemma~\ref{lem:add_fam:cpt}, a corresponding FNN $\varphi: M \rightarrow Y$ is of the form
\begin{equation}
	M \ni x \quad \mapsto \quad \varphi(x) := \sum_{n=1}^N y_n \rho\left( \ell_n(\phi_i(x)) + b_n \right) \in Y,
\end{equation}
where $N \in \mathbb{N}$, $\ell_1,\ldots,\ell_N \in (X,\tau_{\mathrm{init}})^*$, $b_1,\ldots,b_N \in \mathbb{R}$, $y_1,\ldots,y_N \in \mathcal{L} \subseteq Y$, and $\rho \in C^k(\mathbb{R})$.

\begin{lemma}
	\label{lem:add_fam:dual}
	Let $(M,\tau_M)$ be a $C^k_{loc}$-manifold with global chart $\phi_i: M \rightarrow \phi_i(M)$ over a dual Banach space $(X,\Vert \cdot \Vert_X)$ which is equipped with the weak-$*$-topology $\tau_{w^*}$. Moreover, assume that $(X,\tau_{w^*})$ has BAP with finite rank operators $(T_\gamma)_\gamma$. Then,
	\begin{equation}
		\mathcal{A} := \left\lbrace M \ni x \mapsto \langle \phi_i(x), e \rangle_{X \times E} \in \mathbb{R}: e \in E \right\rbrace \subseteq C^k_{loc}(M)
	\end{equation}
	is an additive family on $M$.
\end{lemma}
\begin{proof}
	First, for every $e \in E$, we observe that $a := \langle \phi_i(\cdot), e \rangle_{X \times E} \in \mathcal{A}$ satisfies $a \circ \phi_i^{-1} = \langle \cdot, e \rangle_{X \times E} \in C^k_{loc}(\phi_i(M))$, which ensures that $a \in C^k_{loc}(M)$. For \ref{A1}, $\mathcal{A}$ is by definition closed under addition. For \ref{A2}, we fix some distinct points $x_1,x_2 \in M$, which also satisfy $\phi_i(x_1) \neq \phi_i(x_2)$ as $\phi_i: M \rightarrow \phi_i(M)$ is injective. Since $E^* \cong X$ is by the Hahn-Banach theorem point separating on $X$, there exists some $e \in E$ such that $a := \langle \phi_i(\cdot), e \rangle_{X \times E} \in \mathcal{A}$ satisfies $a(x_1) = \langle \phi_i(x_1), e \rangle_{X \times E} \neq \langle \phi_i(x_2), e \rangle_{X \times E} = a(x_2)$. For \ref{A3}, we fix some $(x,[c]^1_x) \in T^1 M$ with $c'(0) \neq 0$, as $\phi_i: M \rightarrow \phi_i(M)$ is a diffeomorphism. Thus, by using again the Hahn-Banach theorem, there exists some $e \in E$ such that $\langle (\phi_i \circ c)'(0), e \rangle_{X \times E} \neq 0$, whence $a := \langle \phi_i(\cdot), e \rangle_{X \times E} \in \mathcal{A}$ satisfies
	\begin{equation}
		(a \circ c)'(0) = (a \circ \phi_i^{-1} \circ \phi_i \circ c)'(0) = d(a \circ \phi_i^{-1})(\phi_i(x);(\phi_i \circ c)'(0)) = \langle (\phi_i \circ c)'(0), e \rangle_{X \times E} \neq 0.
	\end{equation}
	For \ref{A4'}, we fix some $\gamma$ and consider the finite-dimensional vector subspace $T_\gamma(X) \subseteq X$. Since $X \cong E^*$ is point separating on $E$, there exist some $e_1,\ldots,e_m \in E$ such that
	\begin{equation}
		T_\gamma(X) \ni u \quad \mapsto \quad L(u) := \big( L_1(u),\ldots,L_m(u) \big)^\top := \big( \langle u, e_1 \rangle_{X \times E}, \ldots, \langle u, e_m \rangle_{X \times E} \big)^\top \in \mathbb{R}^m
	\end{equation}
	is a linear isomorphism. Hence, by defining $a_n := L_n \circ \phi_i \in \mathcal{A}$, we conclude that
	\begin{equation}
		\eta_{i,\gamma} := \big( a_1 \circ \phi_i^{-1}, \ldots, a_m \circ \phi_i^{-1} \big)^\top\big\vert_{T_\gamma(\phi_i(M))} = L\vert_{T_\gamma(\phi_i(M))}: T_\gamma(\phi_i(M)) \rightarrow \eta_{i,\gamma}(T_\gamma(\phi_i(M)))
	\end{equation}
	is the restriction of a linear isomorphism and therefore an embedding.
\end{proof}

With the additive family $\mathcal{A}$ from Lemma~\ref{lem:add_fam:dual}, a corresponding FNN $\varphi: M \rightarrow Y$ is of the form
\begin{equation}
	M \ni x \quad \mapsto \quad \varphi(x) := \sum_{n=1}^N y_n \rho\left( \langle \phi_i(x), e_n \rangle_{X \times E} + b_n \right) \in Y,
\end{equation}
where $N \in \mathbb{N}$, $e_1,\ldots,e_N \in E$, $b_1,\ldots,b_N \in \mathbb{R}$, $y_1,\ldots,y_N \in \mathcal{L} \subseteq Y$, and $\rho \in C^k(\mathbb{R})$.

In the following, we now apply Lemmas~\ref{lem:add_fam:cpt}--\ref{lem:add_fam:dual} to construct different additive families.

\begin{example}
	\label{ex:add_fam}
	The following examples are additive families:
	\begin{enumerate}
		\item\label{ex:add_fam:1} Let $M := C^\alpha(S;Z)$ be as in Example~\ref{ex:w_dom}~\ref{ex:w_dom:cpt:hol} with admissible collection of weight functions $\Psi = (\psi_j)_{j=0,\ldots,k}$. Then, an additive family is given by
		\begin{equation}
			\quad\quad \mathcal{A} := \left\lbrace C^\alpha(S;Z) \ni x \mapsto \int_S \langle x(s), e \rangle_{Z \times E} \, \nu(ds) \in \mathbb{R}: \nu: \mathcal{F}_S \rightarrow \mathbb{R} \right\rbrace \subseteq \mathcal{B}^k_\Psi(C^\alpha(S;Z)),
		\end{equation}
		where $\nu: \mathcal{F}_S \rightarrow \mathbb{R}$ is a finite signed regular Borel measure.
		
		\item\label{ex:add_fam:2} Let $M := L^p(\Omega;Z)$ with $p \in (1,\infty]$ be as in Example~\ref{ex:w_dom}~\ref{ex:w_dom:dual:Lp} with $(Z,\Vert \cdot \Vert_Z)$ having predual $(E,\Vert \cdot \Vert_E)$ and with admissible collection of weight functions $\Psi = (\psi_j)_{j=0,\ldots,k}$. Then, an additive family is given by
		\begin{equation}
			\quad\quad \mathcal{A} := \left\lbrace L^p(\Omega;Z) \ni x \mapsto \int_\Omega \langle x(\omega), g(\omega) \rangle_{Z \times E} \, \mu(d\omega) \in \mathbb{R}: g \in L^{p'}(\Omega;E) \right\rbrace \subseteq \mathcal{B}^k_\Psi(L^p(\Omega;Z)),
		\end{equation}
		where $1/p+1/p' = 1$.
		
		\item\label{ex:add_fam:3} Let $M := \mathcal{P}_{\psi_\Omega}(\Omega) \subseteq \mathcal{M}_{\psi_\Omega}(\Omega)$ be the space of probability measures over a weighted space $(\Omega,\psi_\Omega)$ as in Example~\ref{ex:w_mfd} with admissible collection of weight functions $\Psi = (\psi_j)_{j=0,\ldots,k}$. Then, an additive family is given by
		\begin{equation}
			\quad\quad \mathcal{A} := \left\lbrace \mathcal{P}_{\psi_\Omega}(\Omega) \ni x \mapsto \int_\Omega f(\omega) x(d\omega) \in \mathbb{R}: f \in \mathcal{B}_{\psi_\Omega}(\Omega) \right\rbrace \subseteq \mathcal{B}^k_\Psi(\mathcal{P}_{\psi_\Omega}(\Omega)),
		\end{equation}
		where the function $f \in \mathcal{B}_{\psi_\Omega}(\Omega)$ could be replaced by a neural network $\varphi: \Omega \rightarrow \mathbb{R}$.
	\end{enumerate}
\end{example}
\begin{proof}
	For \ref{ex:add_fam:1}, we first apply \cite[Theorem~A.3]{cuchiero23} to obtain that $C^\alpha(S;Z) \hookrightarrow C^0(S;Z)$ is a compact embedding, where $(C^\alpha(S;Z),\tau_\infty)$ has $\Vert \cdot \Vert_\alpha$-BAP by Theorem~\ref{thm:C0_bap}. While \ref{A1} is satisfied, we use Dirac measures to see that $\mathcal{A}$ is point separating on $C^\alpha(S;Z)$. Hence, we can follow the proof of Lemma~\ref{lem:add_fam:cpt} to conclude that $\mathcal{A}$ is an additive family. 
	
	For \ref{ex:add_fam:2}, we use that $L^p(\Omega;Z) \cong L^{p'}(\Omega;E)^*$ is a dual Banach space (see also Example~\ref{ex:w_dom}~\ref{ex:w_dom:dual:Lp}). Moreover, since $(E,\Vert \cdot \Vert_E)$ has BAP, also $(L^{p'}(\Omega;E),\Vert \cdot \Vert_{L^{p'}(\Omega;E)})$ admits BAP. Hence, by using the adjoints of the finite rank operators on $L^{p'}(\Omega;E)$ as in Lemma~\ref{lem:ap_dual}, we conclude that $(L^p(\Omega;Z),\tau_{w^*})$ has $\Vert \cdot \Vert_{L^p(\Omega;Z)}$-BAP. Thus, we can apply Lemma~\ref{lem:add_fam:dual} to obtain the conclusion.
	
	For \ref{ex:add_fam:3}, we use that $\mathcal{P}_{\psi_\Omega}(\Omega) \subseteq \mathcal{M}_{\psi_\Omega}(\Omega) \cong \mathcal{B}_{\psi_\Omega}(\Omega)^*$ is a subset of a dual Banach space (see also Example~\ref{ex:w_dom}~\ref{ex:w_dom:dual:measure}). Moreover, since $(\mathcal{B}_{\psi_\Omega}(\Omega),\Vert \cdot \Vert_{\mathcal{B}_{\psi_\Omega}(\Omega)})$ has BAP, we can use again the adjoints of the finite rank operators on $\mathcal{B}_{\psi_\Omega}(\Omega)$ as in Lemma~\ref{lem:ap_dual} to conclude that $(\mathcal{M}_{\psi_\Omega}(\Omega),\tau_{w^*})$ has $\Vert \cdot \Vert_{\mathcal{M}_{\psi_\Omega}(\Omega)}$-BAP. Thus, we can apply Lemma~\ref{lem:add_fam:dual} to obtain the conclusion.
\end{proof}

\subsection{Weighted UAT over finite-dimensional manifolds}
\label{sec:w_uat_findim}

Neural networks between Euclidean spaces enjoy the universal approximation property, meaning that they can approximate any continuous function uniformly on compact subsets. This fundamental result was first proven by G.~Cybenko (see \cite{cybenko89}) and K.~Hornik (see \cite{hornik91}) in so-called universal approximation theorems (UATs), which establish denseness of neural networks in suitable function spaces. Subsequently, other works \cite{barron93,candes98,bgkp17} related the approximation error to the network complexity by proving quantitative approximation rates under more restrictive assumptions on the target function.

In this section, we now prove a UAT for functional input neural networks on finite-dimensional manifolds. To this end, we assume that the activation function $\rho: \mathbb{R} \rightarrow \mathbb{R}$ is non-polynomial, i.e., its Fourier transform $\widehat{T_\rho} \in \mathscr{S}'(\mathbb{R};\mathbb{C})$ in the sense of distributions has a non-zero point in its support. This is similar to the UATs with non-polynomial activation function of \cite{leshno93,chen95,pinkus99}. 

We now introduce a weighted function space that is similar to $\mathcal{B}^k_\Psi(\mathbb{R})$ with polynomial weights $\Psi$. For $c \in (0,\infty)$, we denote by $\mathscr{B}^k_c(\mathbb{R})$ the closure of $C^k_b(\mathbb{R})$ with respect to the weighted norm $\Vert f \Vert_{\mathscr{B}^k_c(\mathbb{R})} := \max_{j=0,\ldots,k} \sup_{s \in \mathbb{R}} \frac{\vert f^{(j)}(s) \vert}{(1+\vert s \vert)^c}$. Then, $\mathscr{B}^k_c(\mathbb{R})$ can be related to $\mathcal{B}^k_\Psi(\mathbb{R})$ by viewing the derivatives as differentials. Moreover, $\rho \in \mathscr{B}^k_c(\mathbb{R})$ if and only if $\rho \in C^k(\mathbb{R})$ with $\lim_{\vert s \vert \rightarrow \infty} \frac{\vert \rho^{(j)}(s) \vert}{(1+\vert s \vert)^c} = 0$, for all $j = 0,\ldots,k$ (see \cite[Notation~(v)]{schmocker26}). In addition, any $\rho \in \mathscr{B}^k_c(\mathbb{R})$ induces the tempered distribution $\big( g \mapsto T_\rho(g) := \int_{\mathbb{R}} \rho(s) g(s) ds \big) \in \mathscr{S}'(\mathbb{R};\mathbb{C})$ (see, e.g., \cite[Equation~9.26]{folland92}).

\begin{definition}
	For $c \in (0,\infty)$, we introduce the following:
	\begin{enumerate}
		\item $\rho \in \mathscr{B}^k_c(\mathbb{R})$ is called \emph{non-polynomial} if its Fourier transform $\widehat{T_\rho} \in \mathscr{S}'(\mathbb{R};\mathbb{C})$ has a non-zero point in its support.
		\item $\rho \in \mathscr{B}^k_c(\mathbb{R})$ is called \emph{strongly non-polynomial} if its Fourier transform $\widehat{T_\rho} \in \mathscr{S}'(\mathbb{R};\mathbb{C})$ has a support with $0 \in \mathbb{R}$ as inner point.
	\end{enumerate}
	For the definition of the support of $\widehat{T_\rho} \in \mathscr{S}'(\mathbb{R};\mathbb{C})$, we refer to Section~\ref{sec:notation}.
\end{definition}

First, we combine the weighted UATs of \cite[Theorem~2.7]{schmocker26} and \cite[Proposition~4.4~(A3)]{cuchiero23} for classical neural networks on the real line. They both rely on Korevaar's distributional extension~\cite{korevaar65} of Wiener's Tauberian theorem~\cite{wiener32}, which provides sufficient conditions on the Fourier transform of a function such that the linear span of its translations is dense in $L^1(\mathbb{R})$. More precisely, for $\widetilde{\rho}_a(s) := \rho(-as)$ and a finite signed measure $\mu$ on $\mathbb{R}$, the condition $(\widetilde{\rho}_a * \mu)(b) := \int_{\mathbb{R}} \rho(as-ab) \mu(ds) = 0$, for all $a,b \in \mathbb{R}$, implies by Korevaar's argument that $\mu = 0$, which means that the activation function $\rho \in \mathscr{B}^k_c(\mathbb{R})$ is discriminatory (cf., \cite{cybenko89,leshno93,chen95} for compactly supported measures $\mu$). However, in order to include the approximation of the derivatives, the weighted UAT of \cite[Theorem~2.7]{schmocker26} followed the proof ideas of \cite{hsw90,hornik91} and mollified the linear functionals, which allows the application of integration by parts to eliminate the derivatives.

\begin{proposition}
	\label{prop:w_uat_R}
	Let $c \in (0,\infty)$. Then, the following holds true:
	\begin{enumerate}
		\item\label{prop:w_uat_R:1} If $\rho \in \mathscr{B}^k_c(\mathbb{R})$ is strongly non-polynomial, then
		\begin{equation}
			\mathcal{NN}^{\mathbb{N}_0,\rho,\mathbb{R}}_{\mathbb{R},\mathbb{R}} := \linspan\left\lbrace \mathbb{R} \ni s \mapsto \rho(as+b) \in \mathbb{R}: a \in \mathbb{N}_0, \, b \in \mathbb{R} \right\rbrace
		\end{equation}
		is a dense subset of $\mathscr{B}^k_c(\mathbb{R})$.
		\item\label{prop:w_uat_R:2} If $\rho \in \mathscr{B}^k_c(\mathbb{R})$ is non-polynomial, then
		\begin{equation}
			\mathcal{NN}^{\mathbb{R},\rho,\mathbb{R}}_{\mathbb{R},\mathbb{R}} := \linspan\left\lbrace \mathbb{R} \ni s \mapsto \rho(as+b) \in \mathbb{R}: a, b \in \mathbb{R} \right\rbrace
		\end{equation}
		is a dense subset of $\mathscr{B}^k_c(\mathbb{R})$.
	\end{enumerate}
\end{proposition}
\begin{proof}
	Part~\ref{prop:w_uat_R:2} follows directly from \cite[Theorem~2.7]{schmocker26}. For \ref{prop:w_uat_R:1}, we follow the proof of \cite[Theorem~2.7]{schmocker26} and replace the auxiliary result~\cite[Proposition~4.3]{schmocker26} by the argument of \cite[Proposition~4.4~(A3)]{cuchiero23}. The latter uses the strongly non-polynomial assumption to conclude from $\int_\mathbb{R} \rho(as + b) \mu(ds) = 0$, for all $a \in \mathbb{N}_0$ and $b \in \mathbb{R}$, that $\mu = 0 \in \mathcal{M}_{(1+\vert\cdot\vert)^c}(\mathbb{R})$. Hence, by continuing the proof of \cite[Theorem~2.7]{schmocker26}, we also obtain denseness in \ref{prop:w_uat_R:1}.
\end{proof}

Next, we lift the UAT from neural networks on the real line (see Proposition~\ref{prop:w_uat_R}) to FNNs defined on a weighted $C^k_{loc}$-manifold over finite-dimensional vector spaces.

\begin{theorem}[Universal approximation on $\mathcal{B}^k_\Psi(M;Y)$]
	\label{thm:w_uat_findim}
	Let $(M,\Psi)$ be a weighted $C^k_{loc}$-manifold with atlas $(U_i,\phi_i)_{i \in I}$ over finite-dimensional vector spaces $(X_i,\tau_{X_i})_{i \in I}$ and let $(Y,\Vert \cdot \Vert_Y)$ be a Banach space having BAP. Moreover, for $c \in (0,\infty)$, let $\rho \in \mathscr{B}^k_c(\mathbb{R})$ be strongly non-polynomial and assume that $\mathcal{A} \subseteq \mathcal{B}^k_\Psi(M)$ is an additive family such that for every $a \in \mathcal{A}$ and $i \in I$ we have
	\begin{equation}
		\label{eq:thm:w_uat_findim:1}
		C_{a,i} := \max_{j=0,\ldots,k \atop \pi \in \mathscr{P}_j} \sup_{(u,v_1,\ldots,v_j) \in \phi_i(U_i) \times X_i^j} \frac{\left( 1 + \left\vert \left( a \circ \phi_i^{-1} \right)(u) \right\vert \right)^c \left\vert d^\pi \left( a \circ \phi_i^{-1} \right)(u;v_\pi) \right\vert}{\psi_{i,j}(u,v_1,\ldots,v_j)} < \infty.
	\end{equation}
	In addition, let $\mathcal{L} \subseteq Y$ be a dense vector subspace. Then, $\mathcal{NN}^{\mathcal{A},\rho,\mathcal{L}}_{M,Y}$ is a dense subset of $\mathcal{B}^k_\Psi(M;Y)$.
\end{theorem}
\begin{proof}
	First, we show the conclusion for functional input neural networks $\mathcal{NN}^{\mathcal{A},\rho,\mathcal{L}}_{U,Y}$ defined on a weighted domain $(U,\Psi)$, which satisfy $\mathcal{NN}^{\mathcal{A},\rho,\mathcal{L}}_{U,Y} \subseteq \mathcal{B}^k_\Psi(U;Y)$ by Lemma~\ref{lem:NN_well_def}. Now, we show that
	\begin{equation}
		\mathcal{G} := \linspan\left( \left\lbrace y \cos_0(a(\cdot)): a \in \mathcal{A}, \, y \in Y \right\rbrace \cup \left\lbrace y \sin(a(\cdot)): a \in \mathcal{A}, \, y \in Y \right\rbrace \right) \subseteq \mathcal{B}^k_\Psi(U;Y)
	\end{equation}
	is dense in $\mathcal{B}^k_\Psi(U;Y)$. To this end, we follow the same arguments as in \eqref{eq:thm:w_nachbin_findim:proof9} to obtain $\mathcal{G}' := \lbrace \ell \circ g: \ell \in Y^*, \, g \in \mathcal{G} \rbrace \subseteq \mathcal{B}^k_\Psi(U)$ and therefore $\mathcal{G} \subseteq \mathcal{B}^k_\Psi(U;Y)$. Moreover, since $\mathcal{G}'$ is by \eqref{eq:thm:w_nachbin_findim:proof10} a subalgebra and it holds that $\mathcal{G}' \otimes Y \subseteq \mathcal{G}$, \cite[Lemma~4.6]{prolla77} ensures that $\mathcal{G}$ is a polynomial subalgebra. In addition, we can follow the arguments below \eqref{eq:thm:w_nachbin_findim:proof10} to deduce that $\mathcal{G}'$ is point separating and nowhere vanishing of $\Psi$-moderate growth, where we use the constant map $g(\cdot) := \cos(0) \in \mathcal{G}'$ for the nowhere vanishing condition in \ref{M2}. Hence, we can now apply the weighted Nachbin theorem (Theorem~\ref{thm:w_nachbin_findim_vector}) to conclude that $\mathcal{G}$ is dense in $\mathcal{B}^k_\Psi(U;Y)$.
	
	Next, we prove that $\mathcal{G}$ is contained in the closure of $\mathcal{NN}^{\mathcal{A},\rho,\mathcal{L}}_{U,Y}$ with respect to $\Vert \cdot \Vert_{\mathcal{B}^k_\Psi(U;Y)}$ to conclude from denseness of $\mathcal{G}$ that $\mathcal{NN}^{\mathcal{A},\rho,\mathcal{L}}_{U,Y}$ is also dense in $\mathcal{B}^k_\Psi(U;Y)$. To this end, we fix some $a \in \mathcal{A}$, $y \in Y$, and $\varepsilon > 0$. Then, by using that $\mathcal{L}$ is dense in $Y$, there exists some $\widetilde{y} \in \mathcal{L}$ such that
	\begin{equation}
		\Vert y - \widetilde{y} \Vert_Y < \frac{\varepsilon}{2 \big( 1 + \Vert \cos_0(a(\cdot)) \Vert_{\mathcal{B}^k_\Psi(U)} \big)}.
	\end{equation}
	Moreover, by applying the UAT in Proposition~\ref{prop:w_uat_R}~\ref{prop:w_uat_R:1}, there exists $\varphi_0 \in \mathcal{NN}^{\mathbb{N}_0,\rho,\mathbb{R}}_{\mathbb{R},\mathbb{R}}$ satisfying
	\begin{equation}
		\label{eq:thm:w_uat_findim:proof1}
		\Vert \cos_0 - \varphi_0 \Vert_{\mathscr{B}^k_c(\mathbb{R})} = \max_{j=0,\ldots,k} \sup_{s \in \mathbb{R}} \frac{\vert \cos_0^{(j)}(s) - \varphi_0^{(j)}(s) \vert}{(1+\vert s \vert)^c} < \frac{\varepsilon}{2 C_a k! (1 + \Vert \widetilde{y} \Vert_Y)}.
	\end{equation}
	Hence, by using that $\mathcal{A}$ is closed under addition and that $\mathcal{L}$ is a vector space, we can define the FNN $\varphi := \widetilde{y} \varphi_0(a(\cdot)) \in \mathcal{NN}^{\mathcal{A},\rho,\mathcal{L}}_{U,Y}$. Thus, by using the Fa\`a di Bruno formula, that $\vert \mathscr{P}_j \vert \leq j! \leq k!$, the constant $C_a > 0$ defined in \eqref{eq:thm:w_uat_findim:1}, and \eqref{eq:thm:w_uat_findim:proof1}, it follows that
	\begin{equation}
		\begin{aligned}
			& \Vert y \cos_0(a(\cdot)) - \widetilde{y} \varphi_0(a(\cdot)) \Vert_{\mathcal{B}^k_\Psi(U;Y)} = \Vert y \cos_0(a(\cdot)) - \widetilde{y} \cos_0(a(\cdot)) + \widetilde{y} \cos_0(a(\cdot)) - \widetilde{y} \varphi_0(a(\cdot)) \Vert_{\mathcal{B}^k_\Psi(U;Y)} \\
			& \leq \Vert y - \widetilde{y} \Vert_Y \Vert \cos_0(a(\cdot)) \Vert_{\mathcal{B}^k_\Psi(U)} + \Vert \widetilde{y} \Vert_Y \Vert \cos_0(a(\cdot)) - \varphi_0(a(\cdot)) \Vert_{\mathcal{B}^k_\Psi(U)} \\
			& \leq \frac{\varepsilon}{2 \big( 1 + \Vert \cos_0(a(\cdot)) \Vert_{\mathcal{B}^k_\Psi(U)} \big)} \Vert \cos_0(a(\cdot)) \Vert_{\mathcal{B}^k_\Psi(U)} \\
			& \quad\quad + \Vert \widetilde{y} \Vert_Y \max_{j=0,\ldots,k} \sup_{(u,v_1,\ldots,v_j) \in U \times (\mathbb{R}^d)^j} \frac{\sum_{\pi \in \mathscr{P}_j} \big\vert \cos_0^{(\vert\pi\vert)}(a(u)) - \varphi_0^{(\vert\pi\vert)}(a(u)) \big\vert \vert d^\pi a(u;v_\pi) \vert}{\psi_j(u,v_1,\ldots,v_j)} \\
			& \leq \frac{\varepsilon}{2} + C_a k! \Vert \widetilde{y} \Vert_Y \max_{j=0,\ldots,k \atop \pi \in \mathscr{P}_j} \sup_{(u,v_1,\ldots,v_j) \in U \times (\mathbb{R}^d)^j} \frac{\big\vert \cos_0^{(\vert\pi\vert)}(a(u)) - \varphi_0^{(\vert\pi\vert)}(a(u)) \big\vert}{(1+\vert a(u) \vert)^c} \\
			& \leq \frac{\varepsilon}{2} + C_a k! \Vert \widetilde{y} \Vert_Y \max_{j=0,\ldots,k} \sup_{s \in \mathbb{R}} \frac{\big\vert \cos_0^{(j)}(s) - \varphi_0^{(j)}(s) \big\vert}{(1+\vert s \vert)^c} \\
			& < \frac{\varepsilon}{2} + C_a k! \Vert \widetilde{y} \Vert_Y \frac{\varepsilon}{2 C_a k! (1 + \Vert \widetilde{y} \Vert_Y)} \leq \varepsilon.
		\end{aligned}
	\end{equation}
	Since $\varepsilon > 0$ was chosen arbitrarily, this shows that $y \cos_0(a(\cdot))$ belongs to the closure of $\mathcal{NN}^{\mathcal{A},\rho,\mathcal{L}}_{U,Y}$ with respect to $\Vert \cdot \Vert_{\mathcal{B}^k_\Psi(U;Y)}$, which holds analogously true for the map $y \sin(a(\cdot))$. Thus, we conclude that the entire polynomial algebra $\mathcal{G}$ is contained in the closure of $\mathcal{NN}^{\mathcal{A},\rho,\mathcal{L}}_{U,Y}$ with respect to $\Vert \cdot \Vert_{\mathcal{B}^k_\Psi(U;Y)}$. Therefore, by combining this with the previous step, i.e., that $\mathcal{G}$ is dense in $\mathcal{B}^k_\Psi(U;Y)$, it follows that $\mathcal{NN}^{\mathcal{A},\rho,\mathcal{L}}_{U,Y}$ is also dense in $\mathcal{B}^k_\Psi(U;Y)$.
	
	Finally, for the set of FNNs $\mathcal{NN}^{\mathcal{A},\rho,\mathcal{L}}_{M,Y} \subseteq \mathcal{B}^k_\Psi(M;Y)$ over a weighted $C^k_{loc}$-manifold $(M,\Psi)$, we observe for every $i \in I$ that $\mathcal{A}_i := \big\lbrace a \circ \phi_i^{-1}: a \in \mathcal{A} \big\rbrace \subseteq \mathcal{B}^k_{\Psi_i}(\phi_i(U_i))$ is an additive family on $\phi_i(U_i)$. Hence, $\mathcal{NN}^{\mathcal{A}_i,\rho,\mathcal{L}}_{\phi_i(U_i),Y}$ is by the previous step dense in $\mathcal{B}^k_{\Psi_i}(\phi_i(U_i);Y)$. Thus, by using that $\mathcal{B}^k_\Psi(M;Y)$ is equipped with the initial topology with respect to \eqref{eq:def:BkPsi_mfd}, $\mathcal{NN}^{\mathcal{A},\rho,\mathcal{L}}_{M,Y}$ is dense in $\mathcal{B}^k_\Psi(M;Y)$.
\end{proof}

\begin{remark}
	\label{rem:non_poly1}
	If $\mathcal{A} \subseteq \mathcal{B}^k_\Psi(U)$ is a vector subspace, then it suffices to assume that the activation function $\rho \in \mathscr{B}^k_c(\mathbb{R})$ is non-polynomial (see Proposition~\ref{prop:w_uat_R}~\ref{prop:w_uat_R:2}).
\end{remark}

\subsection{Weighted UAT over infinite-dimensional manifolds}
\label{sec:w_uat_infdim}

In this section, we lift the UAT from finite-dimensional input spaces to infinite-dimensional input spaces by assuming the bounded approximation property (BAP). For more details on BAP, we refer to Section~\ref{sec:notation}.

\begin{theorem}[Universal approximation on $\mathcal{B}^k_\Psi(M;Y)$]
	\label{thm:w_uat_infdim}
	Let $(M,\Psi)$ be a weighted $C^k_{loc}$-manifold with atlas $(U_i,\phi_i)_{i \in I}$ over model spaces $(X_i,\tau_{X_i})_{i \in I}$ such that each domain $(\phi_i(U_i),\tau_{X_i})$ has BAP with finite rank operators $(T_{i,\gamma})_\gamma \subseteq (X_i,\tau_{X_i})^* \otimes X_i$. Moreover, let $(Y,\Vert \cdot \Vert_Y)$ be a Banach space having BAP. In addition, for $c \in (0,\infty)$, let $\rho \in \mathscr{B}^k_c(\mathbb{R})$ be strongly non-polynomial and assume that $\mathcal{A} \subseteq \mathcal{B}^k_\Psi(M)$ is an additive family such that for every $a \in \mathcal{A}$ and $i \in I$ it holds that
	\vspace{-0.05cm}
	\begin{equation}
		\label{eq:thm:w_uat_infdim1}
		C_{a,i} := \max_{j=0,\ldots,k \atop \pi \in \mathscr{P}_j} \sup_{(u,v_1,\ldots,v_j) \in \phi_i(U_i) \times X_i^j} \frac{\left( 1 + \left\vert \left( a \circ \phi_i^{-1} \right)(u) \right\vert \right)^c \left\vert d^\pi \left( a \circ \phi_i^{-1} \right)(u;v_\pi) \right\vert}{\psi_{i,j}(u,v_1,\ldots,v_j)} < \infty
		\vspace{-0.05cm}
	\end{equation}
	and for every $a \in \mathcal{A}$, $b \in \mathbb{R}$, $i \in I$, and $\gamma$ the composition
	\vspace{-0.05cm}
	\begin{equation}
		\label{eq:thm:w_uat_infdim2}
		\phi_i(U_i) \ni u \quad \mapsto \quad \rho\left( \left( a \circ \phi_i^{-1} \circ T_{i,\gamma} \right)(u) + b \right) \in \mathbb{R}
		\vspace{-0.05cm}
	\end{equation}
	belongs to the closure of $\mathcal{NN}^{\mathcal{A}_i,\rho,\mathbb{R}}_{\phi_i(U_i),\mathbb{R}}$ with respect to $\Vert \cdot \Vert_{\mathcal{B}^k_{\Psi_i}(\phi_i(U_i))}$. Furthermore, let $\mathcal{L} \subseteq Y$ be a dense vector subspace. Then, $\mathcal{NN}^{\mathcal{A},\rho,\mathcal{L}}_{M,Y}$ is a dense subset of $\mathcal{B}^k_\Psi(M;Y)$.
\end{theorem}
\begin{proof}
	First, we show the conclusion for functional input neural networks $\mathcal{NN}^{\mathcal{A},\rho,\mathcal{L}}_{U,Y}$ defined on a weighted domain $(U,\Psi)$ such that $(U,\tau_X)$ has BAP with finite rank operators $(T_\gamma)_\gamma \subseteq (X,\tau_X)^* \otimes X$. Note that Lemma~\ref{lem:NN_well_def} ensures that $\mathcal{NN}^{\mathcal{A},\rho,\mathcal{L}}_{U,Y} \subseteq \mathcal{B}^k_\Psi(U;Y)$. Since $\mathcal{B}^k_\Psi(U) \otimes Y$ is by Lemma~\ref{lem:out_bap} dense in $\mathcal{B}^k_\Psi(U;Y)$ and $\mathcal{B}^k_\Psi(U)$ is defined as the closure of $C^k_b(U)$ with respect to $\Vert \cdot \Vert_{\mathcal{B}^k_\Psi(U)}$, it suffices to approximate any given $f \in C^k_b(U) \otimes Y$ by an element of $\mathcal{NN}^{\mathcal{A},\rho,\mathcal{L}}_{U,Y}$. To this end, we fix some $f := \sum_{n=1}^N f_n(\cdot) y_n \in C^k_b(U) \otimes Y$ and $\varepsilon > 0$, where $N \in \mathbb{N}$, $f_1,\ldots,f_N \in C^k_b(U)$, and $y_1,\ldots,y_N \in Y$. Then, by using that $\mathcal{L}$ is dense in $Y$, there exists some $\widetilde{y}_1,\ldots,\widetilde{y}_N \in \mathcal{L}$ such that
	\begin{equation}
		\label{eq:thm:w_uat_infdim:proof1}
		\Vert y_n - \widetilde{y}_n \Vert_Y < \frac{\varepsilon}{4N \big( 1 + \Vert f_n \Vert_{\mathcal{B}^k_\Psi(U)} \big)}.
	\end{equation}
	Moreover, by using that $(U,\tau_X)$ has BAP with finite rank operators $(T_\gamma)_\gamma \subseteq (X,\tau_X)^* \otimes X$, there exists by Lemma~\ref{lem:inp_bap} some $\gamma$ such that for every $n = 1,\ldots,N$ it holds that
	\begin{equation}
		\label{eq:thm:w_uat_infdim:proof2}
		\Vert f_n - f_n \circ T_\gamma \Vert_{\mathcal{B}_\Psi^k(U)} < \frac{\varepsilon}{4N (1+\Vert \widetilde{y}_n \Vert_Y)}.
	\end{equation}
	Now, we define the collection $\Psi_\gamma := (\psi_{\gamma,j})_{j=0,\ldots,k}$ of weight functions $\psi_{\gamma,j}: T_\gamma(U) \times T_\gamma(X)^j \rightarrow (0,\infty)$ as in \eqref{eq:thm:w_nachbin_infdim:proof2} and observe that $\mathcal{A}\vert_{T_\gamma(U)} \subseteq \mathcal{B}^k_{\Psi_\gamma}(T_\gamma(U))$ is an additive family on $T_\gamma(U)$ satisfying \eqref{eq:thm:w_uat_findim:1}. Hence, for every fixed $n = 1,\ldots,N$, the UAT in Theorem~\ref{thm:w_uat_findim} applied to $f_n\vert_{T_\gamma(U)} \in \mathcal{B}^k_{\Psi_\gamma}(T_\gamma(U))$ ensures the existence of some $\varphi_n \in \mathcal{NN}^{\mathcal{A}\vert_{T_\gamma(U)},\rho,\mathbb{R}}_{T_\gamma(U),\mathbb{R}}$ satisfying
	\begin{equation}
		\Vert f_n\vert_{T_\gamma(U)} - \varphi_n \Vert_{\mathcal{B}^k_{\Psi_\gamma}(T_\gamma(U))} < \frac{\varepsilon}{4N (1+\Vert \widetilde{y}_n \Vert_Y)}.
	\end{equation}
	Thus, by applying the chain rule as in \eqref{eq:thm:w_nachbin_infdim:proof4}, it follows that
	\begin{equation}
		\label{eq:thm:w_uat_infdim:proof3}
		\Vert f_n \circ T_\gamma - \varphi_n \circ T_\gamma \Vert_{\mathcal{B}^k_\Psi(U)} < \frac{\varepsilon}{4N (1+\Vert \widetilde{y}_n \Vert_Y)}.
	\end{equation}
	Now, we use that $\varphi_n \circ T_\gamma: U \rightarrow \mathbb{R}$ belongs by assumption to the closure of $\mathcal{NN}^{\mathcal{A},\rho,\mathbb{R}}_{U,\mathbb{R}}$ with respect to $\Vert \cdot \Vert_{\mathcal{B}^k_\Psi(U)}$ to obtain some $\widetilde{\varphi}_n \in \mathcal{NN}^{\mathcal{A},\rho,\mathbb{R}}_{U,\mathbb{R}}$ such that
	\begin{equation}
		\label{eq:thm:w_uat_infdim:proof4}
		\Vert \varphi_n \circ T_\gamma - \widetilde{\varphi}_n \Vert_{\mathcal{B}^k_\Psi(U)} < \frac{\varepsilon}{4N (1+\Vert \widetilde{y}_n \Vert_Y)}.
	\end{equation}
	Hence, for $\widetilde{\varphi} := \sum_{n=1}^N \widetilde{y}_n \widetilde{\varphi}_n \in \mathcal{NN}^{\mathcal{A},\rho,\mathcal{L}}_{U,Y}$, we use \eqref{eq:thm:w_uat_infdim:proof1}, \eqref{eq:thm:w_uat_infdim:proof2}, \eqref{eq:thm:w_uat_infdim:proof3}, and \eqref{eq:thm:w_uat_infdim:proof4} to conclude that
	\begin{equation}
		\begin{aligned}
			& \Vert f - \widetilde{\varphi} \Vert_{\mathcal{B}^k_\Psi(U;Y)} = \left\Vert \sum_{n=1}^N y_n f_n(\cdot) - \sum_{n=1}^N \widetilde{y}_n \widetilde{\varphi}_n(\cdot) \right\Vert_{\mathcal{B}^k_\Psi(U;Y)} \\
			& < \sum_{n=1}^N \Vert y_n - \widetilde{y}_n \Vert_Y \Vert f_n \Vert_{\mathcal{B}^k_\Psi(U)} + \sum_{n=1}^N \Vert \widetilde{y}_n \Vert_Y \Vert f_n - \widetilde{\varphi}_n \Vert_{\mathcal{B}^k_\Psi(U)} \\
			& \leq \frac{\varepsilon}{4} + \sum_{n=1}^N \Vert \widetilde{y}_n \Vert_Y \left( \Vert f_n - f_n \circ T_\gamma \Vert_{\mathcal{B}^k_\Psi(U)} + \Vert f_n \circ T_\gamma - \varphi_n \circ T_\gamma \Vert_{\mathcal{B}^k_\Psi(U)} + \Vert \varphi_n \circ T_\gamma - \widetilde{\varphi}_n \Vert_{\mathcal{B}^k_\Psi(U)} \right) \\
			& \leq \frac{\varepsilon}{4} + \frac{3\varepsilon}{4} = \varepsilon.
		\end{aligned}
	\end{equation}
	Since $\varepsilon > 0$ and $f \in C^k_b(U) \otimes Y$ were chosen arbitrarily and $C^k_b(U) \otimes Y$ is dense in $\mathcal{B}^k_\Psi(U;Y)$, this shows that $\mathcal{NN}^{\mathcal{A},\rho,\mathcal{L}}_{U,Y}$ is dense in $\mathcal{B}^k_\Psi(U;Y)$.	
	
	Finally, for general functional input neural networks $\mathcal{NN}^{\mathcal{A},\rho,\mathcal{L}}_{M,Y}$ over a weighted $C^k_{loc}$-manifold $(M,\Psi)$, we observe for every fixed $i \in I$ that $\mathcal{A}_i := \big\lbrace a \circ \phi_i^{-1}: a \in \mathcal{A} \big\rbrace \subseteq \mathcal{B}^k_{\Psi_i}(\phi_i(U_i))$ is an additive family such that every composition of the form \eqref{eq:thm:w_uat_infdim2} belongs to the closure of $\mathcal{NN}^{\mathcal{A}_i,\rho,\mathbb{R}}_{\phi_i(U_i),\mathbb{R}}$ with respect to $\Vert \cdot \Vert_{\mathcal{B}^k_{\Psi_i}(\phi_i(U_i))}$. Hence, we can apply the previous step to conclude that $\mathcal{NN}^{\mathcal{A}_i,\rho,\mathcal{L}}_{\phi_i(U_i),Y}$ is dense in $\mathcal{B}^k_{\Psi_i}(\phi_i(U_i);Y)$. Thus, by using that $\mathcal{B}^k_\Psi(M;Y)$ is equipped with the initial topology with respect to \eqref{eq:def:BkPsi_mfd}, it follows that $\mathcal{NN}^{\mathcal{A},\rho,\mathcal{L}}_{M,Y}$ is dense in $\mathcal{B}^k_\Psi(M;Y)$.
\end{proof}

\begin{remark}
	\label{rem:non_poly:2}
	If $\mathcal{A} \subseteq \mathcal{B}^k_\Psi(M)$ is a vector subspace, then it suffices to assume that $\rho \in \mathscr{B}^k_c(\mathbb{R})$ is non-polynomial (see Proposition~\ref{prop:w_uat_R}~\ref{prop:w_uat_R:2}).
\end{remark}

\section{Weighted universal approximation of non-anticipative functionals}
\label{sec:naf}

In this section, we apply the weighted universal approximation theorem (UAT) in Theorem~\ref{thm:w_uat_infdim} to non-anticipative functionals, which extends the universal approximation result for continuous non-anticipative functionals in \cite[Corollary~4.17]{cuchiero23} by including the directional derivatives. Non-anticipative functional calculus was originally introduced in \cite{dupire09,cont10,contfournie13} to extend F\"ollmer's pathwise Ito calculus (see \cite{foellmer81}) to path-dependent functionals.

First, we recall some notions of non-anticipative functional calculus (see also \cite[Section~5.1]{cont16}). For a fixed terminal time $T \in (0,\infty)$ and a dual Banach space $(Z,\Vert \cdot \Vert_Z)$, we define the \emph{stopped path} of $x \in D^0([0,T];Z)$ at time $t \in [0,T]$ as $\left( s \mapsto x^t_s := x_{s \wedge t} \right) \in D^0([0,T];Z)$, where $s \wedge t := \min(s,t)$. Note that we adopt in this section the notation $x := (x_s)_{s \in [0,T]}$, where time $s \in [0,T]$ is now indicated as a subscript. Then, the space of stopped $Z$-valued c\`adl\`ag paths is defined as
\begin{equation}
	\Lambda^0_{T,Z} = \left\lbrace (t,x^t): (t,x) \in [0,T] \times D^0([0,T];Z) \right\rbrace \cong \left([0,T] \times D^0([0,T];Z) \right) / \sim,
\end{equation}
with $(t,x) \sim (s,y)$ if and only if $t = s$ and $x^t = y^s$. Moreover, $(\Lambda^0_{T,Z},d_\infty)$ with metric
\begin{equation}
	d_\infty((t,x),(s,y)) = \vert t-s \vert + \sup_{u \in [0,T]} \left\Vert x^t_u - y^s_u \right\Vert_Z
\end{equation}
is a complete metric space (see \cite[p.~131]{cont16}). In addition, for $\alpha \in (0,1)$, we denote by $\Lambda^{\alpha,1}_{T,Z} \subseteq \Lambda^0_{T,Z}$ the subspace of stopped $\alpha$-H\"older c\`adl\`ag paths with summable jumps.

Following the original definitions in \cite{dupire09,cont10,contfournie13}, the space of stopped paths $\Lambda^0_{T,Z}$ can also be seen as a vector bundle. More precisely, we define the \emph{space of stopped $Z$-valued $\alpha$-H\"older c\`adl\`ag paths with summable jumps} as the vector bundle
\begin{equation}
	\label{eq:naf_vectorbundle}
	\Lambda^{\alpha,1}_{T,Z} := \bigcup_{t \in (0,T)} D^{\alpha,1}([0,t];Z),
\end{equation}
over the base space $(0,T)$ with bundle projection $\Lambda^{\alpha,1}_{T,Z} \ni (t,x) \mapsto \pi_{\Lambda^{\alpha,1}_{T,Z}}(t,x) := t \in (0,T)$ and fibers $\pi^{-1}(\lbrace t \rbrace) = D^{\alpha,1}([0,t];Z)$. For technical reasons, we restrict ourselves here to an open interval $(0,T)$ instead of $[0,T]$, which therefore does not include any jump at terminal time $T$. In this case, $(\Lambda^{\alpha,1}_{T,Z},d_\infty)$ is a $C^\infty_{loc}$-manifold over the locally convex topological vector space $(\mathbb{R} \times D^{\alpha,1}([0,T];Z),\tau_\mathbb{R} \times \tau_{w^*})$ equipped with the product topology of the topology $\tau_\mathbb{R}$ on $\mathbb{R}$ and the weak-$*$-topology $\tau_{w^*}$ on $D^{\alpha,1}([0,T];Z)$. In addition, the global chart is given by
\begin{equation}
	\label{eq:naf_chart}
	\Lambda^{\alpha,1}_{T,Z} \ni (t,x) \quad \mapsto \quad \phi_i(t,x) := (t,\overline{x}^t) \in \mathbb{R} \times D^{\alpha,1}([0,T];Z),
\end{equation}
with extension $\overline{x}^t \in D^{\alpha,1}([0,T];Z)$ defined as $\overline{x}^t_s = x_s$ if $s \in [0,t]$, and $\overline{x}^t_s = x_t$ if $s \in (t,T]$. The inverse of the global chart~\eqref{eq:naf_chart} is equal to
\begin{equation}
	\label{eq:naf_chart_inv}
	\phi_i(\Lambda^{\alpha,1}_{T,Z}) \ni (t,x) \quad \mapsto \quad \phi_i^{-1}(t,x) := (t,x\vert_{[0,t]}) \in \Lambda^{\alpha,1}_{T,Z}.
\end{equation}
Furthermore, by using local trivializations of the vector bundle, one can show that the higher order tangent spaces at any point $(t,x) \in \Lambda^{\alpha,1}_{T,Z}$ are given by $T^j_{(t,x)} \Lambda^{\alpha,1}_{T,Z} \cong \big( \mathbb{R} \times D^{\alpha,1}([0,t];Z) \big)^j$, $j \in \mathbb{N}_0$, and the higher order tangent bundles are equal to
\begin{equation}
	T^j \Lambda^{\alpha,1}_{T,Z} = \left\lbrace \left( (t,x),[c]^j_{(t,x)} \right): (t,x) \in \Lambda^{\alpha,1}_{T,Z}, \, [c]^j_{(t,x)} \in T^j_{(t,x)} \Lambda^{\alpha,1}_{T,Z} \right\rbrace, \quad j \in \mathbb{N}_0.
\end{equation}
We shall use these higher order tangent bundles to define a weighted manifold.

\subsection{Non-anticipative path-neural networks}

In this section, we introduce a special type of functional input neural networks, so-called non-anticipative path-neural networks, to approximate a differentiable non-anticipative functional. To this end, we first introduce non-anticipative functionals as measurable maps from $\Lambda^{\alpha,1}_{T,Z}$ to a Banach space $(Y,\Vert \cdot \Vert_Y)$ as output space.

\begin{definition}
	A map $f: \Lambda^{\alpha,1}_{T,Z} \rightarrow Y$ is called a \emph{non-anticipative functional} if $f: \Lambda^{\alpha,1}_{T,Z} \rightarrow Y$ is a measurable map from $(\Lambda^{\alpha,1}_{T,Z},d_\infty)$ to $(Y,\Vert \cdot \Vert_Y)$.
\end{definition}

This notion of causality arises in many physical phenomena and in control theory (see, e.g., \cite{fliess81}). Moreover, a non-anticipative functional $f: \Lambda^{\alpha,1}_{T,Z} \rightarrow Y$ is said to be \emph{continuous} if $f: \Lambda^{\alpha,1}_{T,Z} \rightarrow Y$ is a continuous map from $(\Lambda^{\alpha,1}_{T,Z},d_\infty)$ to $(Y,\Vert \cdot \Vert_Y)$. In addition, we recall that $f \in C^k_{loc}(\Lambda^{\alpha,1}_{T,Z};Y)$ is a $C^k_{loc}$-map if and only if $f \circ \phi_i^{-1} \in C^k_{loc}(\phi_i(\Lambda^{\alpha,1}_{T,Z});Y)$ is a $C^k_{loc}$-map on the model space. While we consider in Section~\ref{sec:naf_full} the approximation of all directional derivatives, we shall restrict ourselves in Section~\ref{sec:naf_horver} to the horizontal and vertical derivatives.

Now, we introduce non-anticipative path-neural networks (path-NNs). To this end, we assume that $(E,\Vert \cdot \Vert_E)$ is a predual for the dual Banach space $(Z,\Vert \cdot \Vert_Z)$.

\begin{definition}
	\label{def:naf_nn}
	For $\rho,\widetilde{\rho} \in C^k(\mathbb{R})$ and a vector subspace $\mathcal{L} \subseteq Y$, we define a \emph{non-anticipative path-neural network (path-NN)} $\varphi: \Lambda^{\alpha,1}_{T,Z} \rightarrow Y$ as
	\begin{equation}
		\label{eq:def:naf_nn1}
		\Lambda^{\alpha,1}_{T,Z} \ni (t,x) \quad \mapsto \quad \varphi(t,x) = \sum_{n=1}^N y_n \rho\left( \lambda_n t + \int_0^T \langle \overline{x}^t_s, \widetilde{\varphi}_n(s) \rangle_{Z \times E} \, ds + b_n \right) \in Y,
	\end{equation}
	where $N \in \mathbb{N}$ denotes the \emph{number of neurons}, where $\lambda_1,\ldots,\lambda_N \in \mathbb{R}$ are the \emph{weights}, where $b_1,\ldots,b_N \in \mathbb{R}$ are the \emph{biases}, where $y_1,\ldots,y_N \in \mathcal{L}$ are the \emph{linear readouts}, where
	\begin{equation}
		\widetilde{\varphi}_1,\ldots,\widetilde{\varphi}_N \,\, \in \,\, \mathcal{NN}^{\mathbb{R},\widetilde{\rho},E}_{\mathbb{R},E} := \linspan\left\lbrace \mathbb{R} \ni s \mapsto e \, \widetilde{\rho}(as+b): a,b \in \mathbb{R}, \, e \in E \right\rbrace
	\end{equation}
	are $E$-valued neural networks, and where $\rho,\widetilde{\rho} \in C^k(\mathbb{R})$ represent the \emph{activation functions}. Moreover, we denote by $\mathcal{PN}^{\widetilde{\rho},\rho,\mathcal{L}}_{\Lambda^{\alpha,1}_{T,Z},Y}$ the set of path-NNs of the form \eqref{eq:def:naf_nn1}.
\end{definition}

\begin{remark}
	Let us point out the following remarks concerning Definition~\ref{def:naf_nn}:
	\begin{enumerate}
		\item In \eqref{eq:def:naf_nn1}, we can rewrite the integral as $\int_0^T \langle \overline{x}^t_s, \widetilde{\varphi}_n(s) \rangle_{Z \times E} \, ds = \int_0^t \langle x_s, \widetilde{\varphi}_n(s) \rangle_{Z \times E} \, ds + \big\langle x_t, \int_t^T \widetilde{\varphi}_n(s) ds \big\rangle_{Z \times E}$, which shows the non-anticipative behaviour of a path-NN.
		\item The set of linear readouts $E$ used for the $E$-valued neural networks $\mathcal{NN}^{\mathbb{R},\widetilde{\rho},E}_{\mathbb{R},E}$ could also be replaced by a dense vector subspace $\mathcal{L}_E \subseteq E$.
	\end{enumerate}
\end{remark}

\subsection{Weighted UAT for differentiable non-anticipative functionals}
\label{sec:naf_full}

We now apply the weighted universal approximation theorem (UAT) in Theorem~\ref{thm:w_uat_infdim} to establish a universal approximation result for non-anticipative path-neural networks (path-NNs) over the space of stopped $\alpha$-H\"older c\`adl\`ag paths, including the approximation of all possible directional derivatives. 

On the space of stopped $\alpha$-H\"older c\`adl\`ag paths $\Lambda^{\alpha,1}_{T,Z}$ introduced in \eqref{eq:naf_vectorbundle}, we define the collection $\Psi := (\psi_j)_{j=0,\ldots,k}$ of weight functions $\psi_j: T^j \Lambda^{\alpha,1}_{T,Z} \rightarrow (0,\infty)$ by
\begin{equation}
	\label{eq:naf_w_hoelder}
	\psi_j\big( (t,x), [c]^j_{(t,x)} \big) := \eta\left( \max(j,1) \Vert (t,x) \Vert_{\mathbb{R} \times D^{\alpha,1}([0,t];Z)} + \sum_{\ell=1}^j \Vert c^{(\ell)}(0) \Vert_{\mathbb{R} \times D^{\alpha,1}([0,t];Z)} \right),
\end{equation}
for $\big( (t,x), [c]^j_{(t,x)} \big) \in T^j \Lambda^{\alpha,1}_{T,Z}$ and some continuous non-decreasing function $\eta: [0,\infty) \rightarrow (0,\infty)$, where $\Vert (t,x) \Vert_{\mathbb{R} \times D^{\alpha,1}([0,t];Z)} := \vert t \vert + \Vert x \Vert_{\alpha,\ell^1}$. Hence, the collection $\Psi_i = (\psi_{i,j})_{j=0,\ldots,k}$ of push-forward weight functions $\psi_{i,j} := \psi_j \circ \Phi_i^{-j}: \phi_i(\Lambda^{\alpha,1}_{T,Z}) \times \big( \mathbb{R} \times D^{\alpha,1}([0,T];Z) \big)^j \rightarrow (0,\infty)$ is given by
\begin{equation}
	\label{eq:naf_w_hoelder_pwd}
	\psi_{i,j}\big( (t,\!\overline{x}^t), (s_1,\! v_1), \ldots, (s_j,\! v_j) \big) = \eta\left( \max(j,1) \Vert (t,\!\overline{x}^t) \Vert_{\mathbb{R} \times D^{\alpha,1}([0,T];Z)} \!+\! \sum_{\ell=1}^j \Vert (s_\ell,\! v_\ell) \Vert_{\mathbb{R} \times D^{\alpha,1}([0,T];Z)} \! \right),
\end{equation}
for $\big( (t,\overline{x}^t), (s_1,v_1), \ldots, (s_j,v_j) \big) \in \phi_i(\Lambda^{\alpha,1}_{T,Z}) \times \big( \mathbb{R} \times D^{\alpha,1}([0,T];Z) \big)^j$. Now, we show that $(\Lambda^{\alpha,1}_{T,Z},\Psi)$ is a weighted $C^k_{loc}$-manifold having BAP and that the additive family of the path-NNs introduced in Definition~\ref{def:naf_nn} indeed satisfies \ref{A1}-\ref{A4}. The proof is given in Appendix~\ref{app:lem:naf_mfd_addfam}.

\begin{lemma}
	\label{lem:naf_mfd_addfam}
	Let the predual $(E,\Vert \cdot \Vert_E)$ of $(Z,\Vert \cdot \Vert_Z)$ have BAP and assume that $\eta: [0,\infty) \rightarrow (0,\infty)$ is a continuous and increasing function with $\lim_{r \rightarrow \infty} \frac{r^k}{\eta(r)} = 0$. Then, $(\Lambda^{\alpha,1}_{T,Z},\Psi)$ is a weighted $C^k_{loc}$-manifold over $(\mathbb{R} \times D^{\alpha,1}([0,T];Z),\tau_\mathbb{R} \times \tau_\infty)$ having BAP, where $\Psi$ satisfies \eqref{eq:ass:w_growth}. Moreover, for non-polynomial $\widetilde{\rho} \in C^0(\mathbb{R})$, the set
	\begin{equation}
		\label{eq:lem:naf_mfd_addfam1}
		\mathcal{A} := \left\lbrace \Lambda^{\alpha,1}_{T,Z} \ni (t,x) \mapsto \lambda t + \int_0^T \langle \overline{x}^t_s, \widetilde{\varphi}(s) \rangle_{Z \times E} \, ds \in \mathbb{R}: \, \lambda \in \mathbb{R}, \, \widetilde{\varphi} \in \mathcal{NN}^{\mathbb{R},\widetilde{\rho},E}_{\mathbb{R},E} \right\rbrace
	\end{equation}
	is a vector subspace of $\mathcal{B}^k_\Psi(\Lambda^{\alpha,1}_{T,Z})$ and an additive family on $\Lambda^{\alpha,1}_{T,Z}$.
\end{lemma}

Now, we apply the weighted UAT for FNNs in Theorem~\ref{thm:w_uat_infdim} to obtain the following UAT for non-anticipative path-NNs on $\Lambda^{\alpha,1}_{T,Z}$. The proof can be found in Appendix~\ref{app:cor:naf_uat_full}.

\begin{corollary}[Universal Approximation on {$\mathcal{B}^k_\Psi(\Lambda^{\alpha,1}_{T,Z};Y)$}]
	\label{cor:naf_uat_full}
	Let the predual $(E,\Vert \cdot \Vert_E)$ of $(Z,\Vert \cdot \Vert_Z)$ have BAP, let $(Y,\Vert \cdot \Vert_Y)$ be a Banach space having BAP, and assume that $\mathcal{L} \subseteq Y$ is a dense vector subspace. Moreover, for $c \in (0,\infty)$, let $\widetilde{\rho}, \rho \in \mathscr{B}^{k+1}_c(\mathbb{R})$ be non-polynomial with bounded derivatives and assume that $\eta: [0,\infty) \rightarrow (0,\infty)$ is continuous and non-decreasing with $\lim_{r \rightarrow \infty} \frac{r^{k\max(1,c)}}{\eta(r)} = 0$. Then, $\mathcal{PN}^{\widetilde{\rho},\rho,\mathcal{L}}_{\Lambda^{\alpha,1}_{T,Z},Y}$ is a dense subset of $\mathcal{B}^k_\Psi(\Lambda^{\alpha,1}_{T,Z};Y)$, i.e., for every $f \in \mathcal{B}^k_\Psi(\Lambda^{\alpha,1}_{T,Z};Y)$ and $\varepsilon > 0$ there exists some $\varphi \in \mathcal{PN}^{\widetilde{\rho},\rho,\mathcal{L}}_{\Lambda^{\alpha,1}_{T,Z},Y}$ such that
	\begin{equation}
		\Vert f - \varphi \Vert_{\mathcal{B}^k_\Psi(\Lambda^{\alpha,1}_{T,Z};Y)} := \max_{j=0,\ldots,k} \sup_{((t,x),[c]^j_{(t,x)}) \in T^j \Lambda^{\alpha,1}_{T,Z}} \frac{\left\Vert (f \circ c)^{(j)}(0) - (\varphi \circ c)^{(j)}(0) \right\Vert_Y}{\psi_j\big( (t,x), [c]^j_{(t,x)} \big)} < \varepsilon.
	\end{equation}
\end{corollary}

\subsection{Weighted UAT for horizontal and vertical derivatives}
\label{sec:naf_horver}

In this section, we present a universal approximation theorem (UAT) for non-anticipative functionals, which only includes the approximation of the horizontal and vertical derivatives over $\mathbb{R}^d$.

For some fixed $k,\ell \in \mathbb{N}_0$ and $\alpha, T \in (0,\infty)$ as well as the Euclidean space $Z := \mathbb{R}^d$, we consider again the space of stopped $\alpha$-H\"older c\`adl\`ag paths in $\Lambda^{\alpha,1}_{T,\mathbb{R}^d}$. However, we equip $\Lambda^{\alpha,1}_{T,\mathbb{R}^d}$ with the single weight function of the form
\begin{equation}
	\label{eq:naf_w_horver}
	\Lambda^{\alpha,1}_{T,\mathbb{R}^d} \ni (t,x) \quad \mapsto \quad \psi(t,x) := \eta\left( \Vert x \Vert_{\alpha,\ell^1} \right) \in (0,\infty),
\end{equation}
for some continuous non-decreasing function $\eta: [0,\infty) \rightarrow (0,\infty)$. Compared to \eqref{eq:naf_w_hoelder} this weight function does no longer depend on the derivatives as we only consider some derivatives in particular directions, which have uniformly bounded norms.

\begin{definition}
	\label{def:naf_nn_horver}
	A non-anticipative functional $f: \Lambda^{\alpha,1}_{T,\mathbb{R}^d} \rightarrow Y$ is called
	\begin{enumerate}
		\item \emph{horizontally differentiable} if the limit $\mathcal{D} f(t,x) := \lim_{h \rightarrow 0^+} \frac{f(t+h,x^t)-f(t,x^t)}{h}$ exists in $(Y,\Vert \cdot \Vert_Y)$, for all $(t,x) \in \Lambda^{\alpha,1}_{T,\mathbb{R}^d}$ (see \cite[Definition~5.7]{cont16}).
		\item \emph{vertically differentiable} if the limit $\mathscr{D}_{e_i} f(t,x) := \lim_{h \rightarrow 0} \frac{f(t,x^t+h \mathds{1}_{[t,T]} e_i) - f(t,x^t)}{h}$ exists in $(Y,\Vert \cdot \Vert_Y)$, for all $i = 1,\ldots,d$ and $(t,x) \in \Lambda^{\alpha,1}_{T,\mathbb{R}^d}$ (see \cite[Definition~5.8]{cont16}).
	\end{enumerate}
	Moreover, the higher order horizontal derivatives $\mathcal{D}^j f(t,x)$, $j \in \mathbb{N}_0$, as well as the higher order vertical derivatives $\mathscr{D}_\beta f(t,x)$, $\beta \in \mathbb{N}^d_0$, are defined by iteration.
\end{definition}

For $k,\ell \in \mathbb{N}_0$, we denote by $\mathbb{C}^{k,\ell}_b(\Lambda^{\alpha,1}_{T,\mathbb{R}^d};Y)$ the vector space of bounded continuous non-anticipative functionals $f: \Lambda^{\alpha,1}_{T,\mathbb{R}^d} \rightarrow Y$ that are $k$-times horizontally differentiable with bounded continuous horizontal derivatives $\mathcal{D}^j f: \Lambda^{\alpha,1}_{T,\mathbb{R}^d} \rightarrow Y$, $j = 0,\ldots,k$, and $\ell$-times vertically differentiable with bounded continuous vertical derivatives $\mathscr{D}_\beta f: \Lambda^{\alpha,1}_{T,\mathbb{R}^d} \rightarrow Y$, $\beta \in \mathbb{N}^d_{0,\ell}$. Then, we define $\mathbb{B}^{k,\ell}_\psi(\Lambda^{\alpha,1}_{T,\mathbb{R}^d};Y)$ as the closure of $\mathbb{C}^{k,\ell}_b(\Lambda^{\alpha,1}_{T,\mathbb{R}^d};Y)$ with respect to the weighted norm
\begin{equation}
	\Vert f \Vert_{\mathbb{B}^{k,\ell}_\psi(\Lambda^{\alpha,1}_{T,\mathbb{R}^d};Y)} := \max\left( \max_{j=0,\ldots,k} \sup_{(t,x) \in \Lambda^{\alpha,1}_{T,\mathbb{R}^d}} \frac{\Vert \mathcal{D}^j f(t,x) \Vert_Y}{\psi(t,x)}, \max_{\beta \in \mathbb{N}^d_{0,\ell}} \sup_{(t,x) \in \Lambda^{\alpha,1}_{T,\mathbb{R}^d}} \frac{\Vert \mathscr{D}_\beta f(t,x) \Vert_Y}{\psi(t,x)} \right).
\end{equation}
Now, we show the following universal approximation theorem for non-anticipative functionals in $\mathbb{B}^{k,\ell}_\psi(\Lambda^{\alpha,1}_{T,\mathbb{R}^d};Y)$, where only the approximation of the horizontal and vertical derivatives is included. The proof can be found in Appendix~\ref{app:cor:naf_uat_horver}.

\begin{corollary}[Universal Approximation on {$\mathbb{B}^{k,\ell}_\psi(\Lambda^{\alpha,1}_{T,\mathbb{R}^d};Y)$}]
	\label{cor:naf_uat_horver}
	Let $(Y,\Vert \cdot \Vert_Y)$ be a Banach space having BAP and let $\mathcal{L} \subseteq Y$ be a dense vector subspace. Moreover, for $c \in (0,\infty)$, let $\widetilde{\rho}, \rho \in \mathscr{B}^{k+1}_c(\mathbb{R})$ be non-polynomial with bounded derivatives and assume that $\eta: [0,\infty) \rightarrow (0,\infty)$ is continuous and non-decreasing with $\lim_{r \rightarrow \infty} \frac{r^{\max(1,k,c)}}{\eta(r)} = 0$. Then, $\mathcal{PN}^{\widetilde{\rho},\rho,\mathcal{L}}_{\Lambda^{\alpha,1}_{T,\mathbb{R}^d},Y}$ is a dense subset of $\mathbb{B}^{k,\ell}_\psi(\Lambda^{\alpha,1}_{T,\mathbb{R}^d};Y)$, i.e., for every $f \in \mathbb{B}^{k,\ell}_\psi(\Lambda^{\alpha,1}_{T,\mathbb{R}^d};Y)$ and $\varepsilon > 0$ there exists $\varphi \in \mathcal{PN}^{\widetilde{\rho},\rho,\mathcal{L}}_{\Lambda^{\alpha,1}_{T,\mathbb{R}^d},Y}$ such that
	\begin{equation}
		\Vert f - \varphi \Vert_{\mathbb{B}^{k,\ell}_\psi(\Lambda^{\alpha,1}_{T,\mathbb{R}^d};Y)} < \varepsilon.
	\end{equation}
\end{corollary}

Corollary~\ref{cor:naf_uat_horver} is of particular interest for applications involving the path-dependent Ito formula (see, e.g., \cite{dupire09,cont10,contfournie13}). Indeed, in this case, only the horizontal and vertical derivatives appear, which can be approximated with non-anticipative path-neural networks (PNNs).

\section{Weighted universal approximation of linear functions of the signature}
\label{sec:sig}

In this section, we present an application of the weighted Nachbin theorem (Theorem~\ref{thm:w_nachbin_infdim}) to approximate path space functionals, which is similar to Section~\ref{sec:naf}, but using linear functions of the signature instead of non-anticipative functionals. The notion of the signature was introduced by K.-T.~Chen in \cite{chen57} and plays a central role in rough path theory developed by T.~Lyons in \cite{lyons07} (see also the textbooks \cite{friz10,friz20}).

Let us assume that the input data is sequentially ordered, representing a discretization of a path $X: [0,T] \rightarrow Z$ with values in a Banach space $(Z,\Vert \cdot \Vert_Z)$,  e.g., the motion of a plane in the airspace depending on time, the evolution of temperature or pressure measured by a sensor, or the stock prices in a financial market. Given a continuous path $X: [0,T] \rightarrow Z$ of finite variation, we define its signature (at terminal time) as the infinite collection of iterated integrals
\begin{equation}
	S(X)_T := \left( 1, \int_{0<u_1<T} dX_{u_1}, \cdots \int_{0<u_1<\ldots<u_N<T} dX_{u_1} \otimes \cdots \otimes dX_{u_N}, \cdots \right) \in T((Z)),
\end{equation}
where $T((Z)) := \prod_{n=0}^\infty Z^{\otimes n}$ denotes the extended tensor algebra (see Section~\ref{sec:sig_notion} below). For paths of lower regularity, e.g., $\alpha$-H\"older continuous paths $X: [0,T] \rightarrow Z$, one relies on the theory of rough paths to define their signature. In this case, linear functions of the signature (at terminal time $T$) are linear combinations of continuous linear functionals of the components of $S(X)_T$.

In the following, we show that (non-linear) path space functionals can be approximated by linear functions of the signatures over the whole path space, which extends the global universal approximation theorem (UAT) in \cite[Theorem~5.4]{cuchiero23} by including the approximation of the derivatives. This in turn generalizes the UATs (without derivatives) on compact subsets of the path space, e.g., for finite variation paths or for continuous functions of the whole signature (see \cite[Theorem~3.1]{levin13}, \cite[Theorem~1]{kiraly19}, and \cite[Section~3]{lyons20}) and for c\`adl\`ag paths (see \cite[Theorem~3.13]{primavera22}). More recently, UATs have been established on the entire path space in an $L^p$-sense (see \cite{bayer25,mihriban25}), extended to uniform approximation over the whole time interval rather than at a fixed terminal time (see \cite{bayer21,moeller22}), and further generalized to infinite-dimensional rough path settings (see \cite{cox26}).

To establish the universality of linear functions of the signature, we apply the weighted Nachbin theorem over infinite-dimensional manifolds (Theorem~\ref{thm:w_nachbin_infdim}), which relies on the following key features of the signature. First, linear functions of the signature are $C^k_{\mathrm{loc}}$-maps on the underlying rough path space with suitable growth conditions. Second, the signature (at terminal time) uniquely determines the path (up to so-called tree-like equivalences, see \cite{hambly10,boedihardjo16}), ensuring point separation. Third, using the shuffle product, any product of linear functions of the signature can be realized as another linear function of the signature, which asserts the algebra property.

\subsection{Notation related to the signature of (rough) paths}
\label{sec:sig_notion}

We now recall the most important notions. For a dual Banach space $(Z,\Vert \cdot \Vert_Z)$ with predual $(E,\Vert \cdot \Vert_E)$, we assume that $\Vert \cdot \Vert_{Z^{\otimes n}}$ is a norm on the $n$-th algebraic tensor product $Z^{\otimes_a n}$, $n \in \mathbb{N}_0$, with $Z^{\otimes_a 0} := \mathbb{R}$, satisfying
\begin{equation}
	\label{eq:sig_tensor1}
	\Vert \mathbf{a} \otimes \mathbf{b} \Vert_{Z^{\otimes (m+n)}} \leq \Vert \mathbf{a} \Vert_{Z^{\otimes m}} \Vert \mathbf{b} \Vert_{Z^{\otimes n}},
\end{equation}
for all $m,n \in \mathbb{N}_0$, $\mathbf{a} \in Z^{\otimes_a m}$, and $\mathbf{b} \in Z^{\otimes_a n}$, and that
\begin{equation}
	\label{eq:sig_tensor2}
	\Vert z_1 \otimes \cdots \otimes z_n \Vert_{Z^{\otimes n}} = \Vert z_{\sigma(1)} \otimes \cdots \otimes z_{\sigma(n)} \Vert_{Z^{\otimes n}}
\end{equation}
for all $n \in \mathbb{N}_0$, $\sigma \in \mathcal{S}_n$, and $z_1,\ldots,z_n \in Z$. Then, for any $m,n \in \mathbb{N}_0$, we define $Z^{\otimes m} \otimes Z^{\otimes n}$ as the completion of the algebraic tensor product $Z^{\otimes m} \otimes_a Z^{\otimes n}$ with respect to $\Vert \cdot \Vert_{Z^{\otimes (m+n)}}$, which ensures that $Z^{\otimes m} \otimes Z^{\otimes n} \cong Z^{\otimes (m+n)}$ are isomorphic as Banach spaces. For example, the injective tensor norm satisfies the two properties \eqref{eq:sig_tensor1}--\eqref{eq:sig_tensor2} (see, e.g., \cite{ryan02}). Moreover, we assume for every $n \in \mathbb{N}_0$ that $E^{\otimes n}$ is a predual for $Z^{\otimes n}$, which is, e.g., satisfied if $E^{\otimes n}$ is equipped with the projective tensor norm and $Z^{\otimes n}$ with the injective tensor norm (see, e.g., \cite[Theorem~2.9]{ryan02}).

Then, the \emph{extended tensor algebra (over $Z$)} is defined as
\vspace{-0.05cm}
\begin{equation}
	T((Z)) := \prod_{n=0}^\infty Z^{\otimes n},
	\vspace{-0.05cm}
\end{equation}
which is endowed with addition, tensor multiplication, and scalar multiplication defined by
\begin{equation}
	\mathbf{a} + \mathbf{b} := \left( \mathbf{a}^{(n)} + \mathbf{b}^{(n)} \right)_{n=0}^\infty, \quad\quad \mathbf{a} \otimes \mathbf{b} := \left( \sum_{k=0}^n \mathbf{a}^{(n-k)} \otimes \mathbf{b}^{(k)} \right)_{n=0}^\infty, \quad\quad \lambda \cdot \mathbf{a} := \left( \lambda \mathbf{a}^{(n)} \right)_{n=0}^\infty,
\end{equation}
for $\mathbf{a} := (\mathbf{a}^{(n)})_{n=0}^\infty \in T((Z))$, $\mathbf{b} := (\mathbf{b}^{(n)})_{n=0}^\infty \in T((Z))$, and $\lambda \in \mathbb{R}$. Moreover, for $N \in \mathbb{N}_0$, the \emph{truncated tensor algebra} is defined as
\vspace{-0.05cm}
\begin{equation}
	T^N(Z) := \prod_{n=0}^N Z^{\otimes n},
	\vspace{-0.05cm}
\end{equation}
where addition ``$+$'', tensor multiplication ``$\otimes$'', and scalar multiplication ``$\cdot$'' defined by
\begin{equation}
	\mathbf{a} + \mathbf{b} := \left( \mathbf{a}^{(n)} + \mathbf{b}^{(n)} \right)_{n=0}^N, \quad\quad \mathbf{a} \otimes \mathbf{b} := \left( \sum_{k=0}^n \mathbf{a}^{(n-k)} \otimes \mathbf{b}^{(k)} \right)_{n=0}^N, \quad\quad \lambda \cdot \mathbf{a} := \left( \lambda \mathbf{a}^{(n)} \right)_{n=0}^N,
\end{equation}
for $\mathbf{a} := (\mathbf{a}^{(n)})_{n=0}^N \in T^N(Z)$, $\mathbf{b} := (\mathbf{b}^{(n)})_{n=0}^N \in T^N(Z)$, and $\lambda \in \mathbb{R}$. We equip $T^N(Z)$ with the norm $\Vert \mathbf{a} \Vert_{T^N(Z)} := \max_{n=0,\ldots,N} \Vert \mathbf{a}^{(n)} \Vert_{Z^{\otimes n}}$, for $\mathbf{a} := (\mathbf{a}^{(n)})_{n=0}^N  \in T^N(Z)$. In addition, we introduce the subsets $T^N_0(Z)$ and $T^N_1(Z)$ of $T^N(Z)$ consisting of elements $\mathbf{a} := (\mathbf{a}^{(n)})_{n=0}^N \in T^N(Z)$ with $\mathbf{a}^{(0)} = 0$ and $\mathbf{a}^{(0)} = 1$, respectively.

In order to adapt the Lie group point of view on weakly geometric rough paths, we
observe that $T^N_1(Z)$ is a Lie group under $\otimes$, truncated at level $N$, with unit element $\mathbf{1} := (1,0,\ldots,0) \in T^N_1(Z)$. Moreover, for any $N \in \mathbb{N}$, we define the \emph{free step-$N$ nilpotent Lie algebra} as $\mathfrak{g}^N(Z) := \bigoplus_{n=0}^N L_n$, with homogeneous Lie polynomials $L_n \subseteq T^N(Z)$ of degree $n$ recursively defined by
\begin{equation}
	L_0 := \mathbf{0}, \quad\quad L_1 := Z, \quad\quad L_2 := [Z,L_1] = [Z,Z], \quad\quad L_3 := [Z,L_2] = [Z,[Z,Z]], \quad\quad \ldots,
\end{equation}
where $[\mathbf{a},\mathbf{b}] := \mathbf{a} \otimes \mathbf{b} - \mathbf{b} \otimes \mathbf{a}$ is the Lie bracket, with $[A,B] := \linspan(\lbrace [\mathbf{a},\mathbf{b}]: \mathbf{a} \in A, \, \mathbf{b} \in B \rbrace)$. Note that $L_n \subseteq Z^{\otimes n}$ is a vector subspace, ensuring that $\mathfrak{g}^N(Z) \subseteq T^N(Z)$ is a vector subspace. In addition, we define the exponential map as
\begin{equation}
	\label{eq:sig_exp}
	T^N_0(Z) \ni \mathbf{a} \quad \mapsto \quad \exp^N(\mathbf{a}) := \mathbf{1} + \sum_{n=1}^N \frac{1}{n!} \mathbf{a}^{\otimes n} \in T^N_1(Z),
\end{equation}
whose inverse is given by the logarithm
\begin{equation}
	\label{eq:sig_log}
	T^N_1(Z) \ni \mathbf{1} + \mathbf{b} \quad \mapsto \quad \log^N(\mathbf{1} + \mathbf{b}) := \sum_{n=1}^N \frac{(-1)^{n+1}}{n} \mathbf{b}^{\otimes n} \in T^N_0(Z).
\end{equation}
From this, we define the \emph{free step-N nilpotent Lie group} $G^N(Z) := \exp^N(\mathfrak{g}^N(Z))$, which we endow with the homogeneous norm $\Vert \mathbf{g} \Vert_{G^N(Z)} := \max_{n=1,\ldots,N} \Vert \mathbf{g}^{(n)} \Vert_{Z^{\otimes n}}^{1/n}$, inducing the homogeneous metric $d_{G^N(Z)}(\mathbf{g},\mathbf{h}) := \Vert \mathbf{g}^{-1} \otimes \mathbf{h} \Vert_{G^N(Z)}$, for $\mathbf{g},\mathbf{h} \in G^N(Z)$. Then, $G^N(Z)$ is a subgroup of $T^N_1(Z)$ and a $C^\infty_{loc}$-manifold with global chart $\log^N: G^N(Z) \rightarrow \mathfrak{g}^N(Z)$ over the model space $\mathfrak{g}^N(Z)$.

Moreover, the \emph{truncated signature at level $N$} of a path $X \in C^0([0,T];Z)$ of finite variation is defined by
\begin{equation}
	\begin{aligned}
		S^N(X)_{s,t} & := \left( 1, S^{(1)}(X)_{s,t}, \ldots, S^{(N)}(X)_{s,t} \right) \\
		& := \left( 1, \int_{s < u_1 < t} dX_{u_1}, \ldots, \int_{s < u_1 < \ldots < u_N < t} dX_{u_1} \otimes \cdots \otimes dX_{u_N} \right),
	\end{aligned}
\end{equation}
for $0 \leq s \leq t \leq T$, which takes values in $G^N(Z)$. In addition, the (entire) \emph{signature} of a path $X \in C^0([0,T];Z)$ of finite variation is defined by
\begin{equation}
	\begin{aligned}
		S(X)_{s,t} & := \left( 1, S^{(1)}(X)_{s,t}, S^{(2)}(X)_{s,t}, \ldots \right) \\
		& := \left( 1, \int_{s < u_1 < t} dX_{u_1}, \int_{s < u_1 < u_2 < t} dX_{u_1} \otimes dX_{u_2}, \ldots \right),
	\end{aligned}
\end{equation}
which takes values in the set of group-like elements
\begin{equation}
	G(Z) := \left\lbrace \mathbf{g} \in T((Z)): \big( \mathbf{g}^{(0)}, \ldots, \mathbf{g}^{(N)} \big) \in G^N(Z) \text{ for all } N \in \mathbb{N}_0 \right\rbrace.
\end{equation}
Furthermore, for any $m,n \in \mathbb{N}$, $e := e_1 \otimes \cdots \otimes e_m \in E^{\otimes m}$, and $\widetilde{e} := e_{m+1} \otimes \cdots \otimes e_{m+n} \in E^{\otimes n}$, we define the shuffle product
\begin{equation}
	\label{eq:sig_shuffle}
	e \shuffle \widetilde{e} := \sum_{\sigma \in \operatorname{Sh}(m,n)} e_{\sigma^{-1}(1)} \otimes \cdots \otimes e_{\sigma^{-1}(m+n)},
\end{equation}
where $\operatorname{Sh}(m,n)$ consists of shuffles $\sigma \in \mathcal{S}_{m+n}$ of $\lbrace 1,\ldots,m \rbrace$ and $\lbrace m+1,\ldots,m+n \rbrace$, i.e., $\sigma \in \mathcal{S}_{m+n}$ satisfying $\sigma(1) < \ldots < \sigma(m)$ and $\sigma(m+1) < \ldots < \sigma(m+n)$. In particular, for every $e := e_1 \otimes \cdots \otimes e_m \in E^{\otimes m}$, $\widetilde{e} := e_{m+1} \otimes \cdots \otimes e_{m+n} \in E^{\otimes n}$, and $\mathbf{g} \in G^N(Z)$ with $m+n \leq N$, we have
\begin{equation}
	\label{eq:sig_shuffle_prop}
	\langle \mathbf{g}^{(m)}, e \rangle_{Z^{\otimes m} \times E^{\otimes m}} \cdot \langle \mathbf{g}^{(n)}, \widetilde{e} \rangle_{Z^{\otimes n} \times E^{\otimes n}} = \langle \mathbf{g}^{(m+n)}, e \shuffle \widetilde{e} \rangle_{Z^{\otimes (m+n)} \times E^{\otimes (m+n)}},
\end{equation}
which is referred to as the \emph{shuffle product property} (see \cite[Theorem~2.15]{lyons07}).

\subsection{Manifold of weakly geometric $\alpha$-H\"older rough paths}
\label{sec:sig_rp}

We now introduce weakly geometric $\alpha$-H\"older rough paths with values in a dual Banach space $(Z,\Vert \cdot \Vert_Z)$ having predual $(E,\Vert \cdot \Vert_E)$, which can be seen as $\alpha$-H\"older continuous paths with values in $G^{\lfloor 1/\alpha \rfloor}(Z)$ (see also \cite{friz10,friz20}).

\begin{definition}
	For $\alpha \in (0,1]$, a continuous path $\mathbf{X}: [0,T] \rightarrow G^{\lfloor 1/\alpha \rfloor}(Z)$ of the form
	\begin{equation}
		[0,T] \ni t \quad \mapsto \quad \mathbf{X}_t := \big( 1,X_t,\mathbf{X}^{(2)}_t,\ldots,\mathbf{X}^{(\lfloor 1/\alpha \rfloor)}_t \big) \in G^{\lfloor 1/\alpha \rfloor}(Z)
	\end{equation}
	with $\mathbf{X}_0 := \mathbf{1} := (1,0,\ldots,0) \in G^{\lfloor 1/\alpha \rfloor}(Z)$ is called a \emph{weakly geometric $\alpha$-H\"older rough path} if
	\begin{equation}
		\Vert \mathbf{X} \Vert_\alpha := \sup_{0 \leq s < t \leq T} \frac{d_{G^{\lfloor 1/\alpha \rfloor}(Z)}(\mathbf{X}_s,\mathbf{X}_t)}{\vert s-t \vert^\alpha} := \sup_{0 \leq s < t \leq T} \frac{\max_{n=1,\ldots,\lfloor 1/\alpha \rfloor} \left\Vert (\mathbf{X}_s^{-1} \otimes \mathbf{X}_t)^{(n)} \right\Vert_{Z^{\otimes n}}^\frac{1}{n}}{\vert s-t \vert^\alpha} < \infty.
	\end{equation}
	We denote by $C^\alpha_T(Z) := C^\alpha_{\mathbf{1}}([0,T];G^{\lfloor 1/\alpha \rfloor}(Z))$ the space of weakly geometric $\alpha$-H\"older rough paths, which we equip with the $w^*$-uniform topology generated by the semi-metrics
	\begin{equation}
		d_{\infty,\mathbf{e}}(\mathbf{X},\mathbf{Y}) := \sup_{t \in [0,T]} d_{G^{\lfloor 1/\alpha \rfloor}(Z),\mathbf{e}}(\mathbf{X}_t,\mathbf{Y}_t) := \sup_{t \in [0,T]} \max_{n=1,\ldots,\lfloor 1/\alpha \rfloor} \big\vert \langle (\mathbf{X}_t^{-1} \otimes \mathbf{Y}_t)^{(n)}, \mathbf{e}_n \rangle_{Z^{\otimes n} \times E^{\otimes n}} \big\vert^\frac{1}{n}, 
	\end{equation}
	for $\mathbf{e} := (\mathbf{e}^{(1)},\ldots,\mathbf{e}^{(\lfloor 1/\alpha \rfloor)}) \in \bigoplus_{n=1}^{\lfloor 1/\alpha \rfloor} E^{\otimes n}$.
\end{definition}

Next, we define the truncated signature at level $N > \lfloor 1/\alpha \rfloor$ of a weakly geometric $\alpha$-H\"older rough path $\mathbf{X} \in C^\alpha_T(Z)$ as the unique Lyons extension yielding a path $S^N(\mathbf{X}): [0,T] \rightarrow G^N(Z)$ with finite $\alpha$-H\"older norm $\Vert \cdot \Vert_\alpha$ whose $n$-th component agrees with $\mathbf{X}^{(n)}$, for all $n = 0,\ldots,\lfloor 1/\alpha \rfloor$ (see \cite[Theorem~3.7]{lyons07} and \cite[Corollary~9.11~(ii)]{friz10}). By denoting the $n$-th signature component taking values in $Z^{\otimes n}$ by $S^{(n)}(\mathbf{X})$, the \emph{signature} of $\mathbf{X} \in C^\alpha_T(Z)$ is defined by
\begin{equation}
	[0,T] \ni t \quad \mapsto \quad S(\mathbf{X})_t := \big( 1, X_t, S^{(2)}(\mathbf{X})_{0,t}, S^{(3)}(\mathbf{X})_{0,t}, \ldots \big) \in G(Z).
\end{equation} 
Then, a linear function of the signature (at time $T$) is given as
\begin{equation}
	C^\alpha_T(Z) \ni \mathbf{X} \quad \mapsto \quad \ell(S(\mathbf{X})_T),
\end{equation}
where $\mathbf{a} \mapsto \ell(\mathbf{a}) := \sum_{n=0}^N \langle \mathbf{a}^{(n)}, e_n \rangle_{Z^{\otimes n} \times E^{\otimes n}}$, for some $N \in \mathbb{N}$ and $e_n \in E^{\otimes n}$, $n = 0,\ldots,N$.

Moreover, we use the bijection $\log^{\lfloor 1/\alpha \rfloor}: G^{\lfloor 1/\alpha \rfloor}(Z) \rightarrow \mathfrak{g}^{\lfloor 1/\alpha \rfloor}(Z)$ to observe that
\begin{equation}
	\label{eq:sig_chart}
	\phi_i:
	\begin{cases}
		C^\alpha_T(Z) \quad & \rightarrow \quad \mathfrak{C}^\alpha_T(Z) := C^\alpha_0([0,T];\mathfrak{g}^{\lfloor 1/\alpha \rfloor}(Z)) \\
		\mathbf{X} \quad & \mapsto \quad \log^{\lfloor 1/\alpha \rfloor}(\mathbf{X}) := \left( t \mapsto \log^{\lfloor 1/\alpha \rfloor}(\mathbf{X}_t) \right)
	\end{cases}
\end{equation}
is a bijection onto its image, whose inverse is given by
\begin{equation}
	\label{eq:sig_chart_inv}
	\phi_i^{-1}:
	\begin{cases}
		\phi_i(C^\alpha_T(Z)) \quad & \rightarrow \quad C^\alpha_T(Z) \\
		\mathbf{Y} \quad & \mapsto \quad \exp^{\lfloor 1/\alpha \rfloor}(\mathbf{Y}) := \left( t \mapsto \exp^{\lfloor 1/\alpha \rfloor}(\mathbf{Y}_t) \right)
	\end{cases}.
\end{equation}
Since $(Z,\Vert \cdot \Vert_Z)$ is a dual Banach space and $\mathfrak{g}^{\lfloor 1/\alpha \rfloor}(Z) = \bigoplus_{n=1}^{\lfloor 1/\alpha \rfloor} L_n$ with weak-$*$-closed $L_n \subseteq Z^{\otimes n}$, the free step-$\lfloor 1/\alpha \rfloor$ nilpotent Lie algebra $\mathfrak{g}^{\lfloor 1/\alpha \rfloor}(Z)$ has also a predual. Hence, we can equip $\mathfrak{C}^\alpha_T(Z)$ with the $w^*$-uniform topology $\tau_\infty$ generated by seminorms of the form
\begin{equation}
	p_{\mathbf{e}}(\mathbf{Y}) = \sup_{t \in [0,T]} \max_{n=1,\ldots,\lfloor 1/\alpha \rfloor} \left\vert \langle \mathbf{Y}_t^{(n)}, \mathbf{e}^{(n)} \rangle_{Z^{\otimes n} \times E^{\otimes n}} \right\vert,
\end{equation}
for all $\mathbf{e} := (\mathbf{e}^{(1)},\ldots,\mathbf{e}^{(\lfloor 1/\alpha \rfloor)}) \in \bigoplus_{n=1}^{\lfloor 1/\alpha \rfloor} E^{\otimes n}$.

Now, we observe that $(C^\alpha_T(Z),\tau_\infty)$ is a $C^\infty_{loc}$-manifold with global chart~\eqref{eq:sig_chart} over the model space $(\mathfrak{C}^\alpha_T(Z),\tau_\infty)$. This is in contrast to considering $C^\alpha_T(Z)$ as a submanifold of $C^\alpha([0,T];T^{\lfloor 1/\alpha \rfloor}(Z))$, which requires, like $G^{\lfloor 1/\alpha \rfloor}(Z) \hookrightarrow T^{\lfloor 1/\alpha \rfloor}(Z)$, infinitely many charts. In our case, the higher order tangent spaces at any point $\mathbf{X} \in C^\alpha_T(Z)$ are given by $T^j_\mathbf{X} C^\alpha_T(Z) \cong \mathfrak{C}^\alpha_T(Z)^j$, for all $j \in \mathbb{N}$, and the higher order tangent bundles are equal to
\begin{equation}
	T^j C^\alpha_T(Z) = \left\lbrace \left( \mathbf{X}, [c]^j_{\mathbf{X}} \right): \mathbf{X} \in C^\alpha_T(Z), \, [c]^j_\mathbf{X} \in T^j_\mathbf{X} C^\alpha_T(Z) \right\rbrace, \quad j \in \mathbb{N}.
\end{equation}
Furthermore, we fix some $k \in \mathbb{N}_0$ and define the collection $\Psi := (\psi_j)_{j=0,\ldots,k}$ of weight functions $\psi_j: T^j C^\alpha_T(Z) \rightarrow (0,\infty)$ by
\begin{equation}
	\label{eq:sig_weight}
	\psi_j(\mathbf{X},[c]^j_{\mathbf{X}}) := \exp\left( \beta_1 \max(j,1) \Vert \log^{\lfloor 1/\alpha \rfloor}(\mathbf{X}) \Vert_\alpha^{c_1} + \beta_2 \sum_{\ell=1}^j \Vert (\log^{\lfloor 1/\alpha \rfloor} \circ c)^{(\ell)}(0) \Vert_\alpha^{c_2} \right),
\end{equation}
for $(\mathbf{X},[c]^j_{\mathbf{X}}) \in T^j C^\alpha_T(Z)$, with $\beta_1,\beta_2 > 0$, $c_1 \geq \lfloor 1/\alpha \rfloor$, and $c_2 > 0$. Here, we use the chart $\phi_i := \log^{\lfloor 1/\alpha \rfloor}$ in the term $\Vert \log^{\lfloor 1/\alpha \rfloor}(\mathbf{X}) \Vert_\alpha$ instead of $\Vert \mathbf{X} \Vert_\alpha$ as in \cite[Section~5.2]{cuchiero23} considering the case without derivatives. This simplifies the collection $\Psi_i := (\psi_{i,j})_{j=0,\ldots,k}$ of push-forward weight functions $\psi_{i,j} := \psi_j \circ \Phi_i^{-j}: \phi_i(C^\alpha_T(Z)) \times \mathfrak{C}^\alpha_T(Z)^j \rightarrow (0,\infty)$ to
\begin{equation}
	\label{eq:sig_weight_pfwd}
	\psi_{i,j}(\mathbf{Y},\mathbf{V}_1,\ldots,\mathbf{V}_j) := \exp\left( \beta_1 \max(j,1) \Vert \mathbf{Y} \Vert_\alpha^{c_1} + \beta_2 \sum_{\ell=1}^j \Vert \mathbf{V}_\ell \Vert_\alpha^{c_2} \right),
\end{equation}
for $(\mathbf{Y},\mathbf{V}_1,\ldots,\mathbf{V}_j) \in \phi_i(C^\alpha_T(Z)) \times \mathfrak{C}^\alpha_T(Z)^j$. In the following lemma, we show that $(C^\alpha_T(Z),\Psi)$ is a weighted $C^k_{loc}$-manifold.

\begin{lemma}
	\label{lem:sig_w_mfd}
	Let $\alpha \in (0,1]$ and assume that the predual $(E,\Vert \cdot \Vert_E)$ of $(Z,\Vert \cdot \Vert_Z)$ has BAP. Then, $(C^\alpha_T(Z),\Psi)$ is a weighted $C^k_{loc}$-manifold with global chart \eqref{eq:sig_chart} over the model space $(\mathfrak{C}^\alpha_T(Z),\tau_\infty)$. 
\end{lemma}
\begin{proof}
	First, we show that $\phi_i(C^\alpha_T(Z)) \subseteq \mathfrak{C}^\alpha_T(Z)$. To this end, we observe for every fixed $\mathbf{X} \in C^\alpha_T(Z)$ and $n \in \mathbb{N}$ that
	\begin{equation}
		\label{eq:lem:sig_mfd:proof1}
		\begin{aligned}
			& \left\Vert \log^{\lfloor 1/\alpha \rfloor}(\mathbf{X}_t^{-1} \otimes \mathbf{X}_s)^{(n)} \right\Vert_{Z^{\otimes n}} \leq \sum_{m=1}^{\lfloor 1/\alpha \rfloor} \!\frac{1}{m} \sum_{k_1+\ldots+k_m=n}\! \left\Vert (\mathbf{X}_t^{-1} \otimes \mathbf{X}_s)^{(k_1)} \otimes \cdots \otimes (\mathbf{X}_t^{-1} \otimes \mathbf{X}_s)^{(k_m)} \right\Vert_{Z^{\otimes n}} \\
			& \leq c_n \max_{k_1+\ldots+k_m=n} \prod_{\ell=1}^m \left\Vert (\mathbf{X}_t^{-1} \!\otimes\! \mathbf{X}_s)^{(k_\ell)} \right\Vert_{Z^{\otimes k_\ell}} \! \leq c_n \max_{k_1+\ldots+k_m=n} \prod_{\ell=1}^m \left( \Vert \mathbf{X} \Vert_\alpha \vert s \!-\! t \vert^\alpha \right)^{k_\ell} \! = c_n \Vert \mathbf{X} \Vert_\alpha^n \vert s \!-\! t \vert^{n \alpha},
		\end{aligned}
	\end{equation}
	where $c_n > 0$ is a constant. Hence, by using the Baker-Campbell-Hausdorff formula (see, e.g., \cite[Lemma~7.24]{friz10}) with $\operatorname{BCH}(\mathbf{a},\mathbf{b}) := \mathbf{a} + \mathbf{b} + \frac{1}{2} [\mathbf{a},\mathbf{b}] + \frac{1}{12} [\mathbf{a},[\mathbf{a},\mathbf{b}]] - \frac{1}{12} [\mathbf{b},[\mathbf{a},\mathbf{b}]] + \ldots$ consisting of iterated Lie brackets of at least one $\mathbf{a}$ and $\mathbf{b}$, that $\Vert (\operatorname{BCH}(\mathbf{a},\mathbf{b})-\mathbf{a})^{(n)} \Vert_{Z^{\otimes n}}$ can be bounded via $\Vert [\mathbf{a},\mathbf{b}] \Vert_{Z^{\otimes (l+m)}} = \Vert \mathbf{a} \otimes \mathbf{b} - \mathbf{b} \otimes \mathbf{a} \Vert_{Z^{\otimes (l+m)}} \leq 2 \Vert \mathbf{a} \Vert_{Z^{\otimes l}} \Vert \mathbf{b} \Vert_{Z^{\otimes m}}$ into products of $\Vert \mathbf{a}^{(j)} \Vert_{Z^{\otimes j}}$ and $\Vert \mathbf{b}^{(k)} \Vert_{Z^{\otimes k}}$, with $j,k \geq 1$, and the inequality~\eqref{eq:lem:sig_mfd:proof1}, it holds that
	\begin{equation}
		\label{eq:lem:sig_mfd:proof2}
		\begin{aligned}
			& \Vert \log^{\lfloor 1/\alpha \rfloor}(\mathbf{X}) \Vert_\alpha = \sup_{0 \leq s < t \leq T} \frac{\max_{n=1,\ldots,\lfloor 1/\alpha \rfloor} \left\Vert (\log^{\lfloor 1/\alpha \rfloor}(\mathbf{X}_s) - \log^{\lfloor 1/\alpha \rfloor}(\mathbf{X}_t))^{(n)} \right\Vert_{Z^{\otimes n}}}{\vert s-t \vert^\alpha} \\
			& \leq \sup_{0 \leq s < t \leq T} \frac{\max_{n=1,\ldots,\lfloor 1/\alpha \rfloor} \left\Vert \big( \operatorname{BCH}(\log^{\lfloor 1/\alpha \rfloor}(\mathbf{X}_t),\log^{\lfloor 1/\alpha \rfloor}(\mathbf{X}_t^{-1} \otimes \mathbf{X}_s))-\log^{\lfloor 1/\alpha \rfloor}(\mathbf{X}_t) \big)^{(n)} \right\Vert_{Z^{\otimes n}}}{\vert s-t \vert^\alpha} \\
			& \leq C_2 \max_{n=1,\ldots,\lfloor 1/\alpha \rfloor \atop \vert\mathbf{j}\vert+\vert\mathbf{k}\vert=n} \left( \prod_{\ell=1}^p \left\Vert \log^{\lfloor 1/\alpha \rfloor}(\mathbf{X}_t)^{(j_\ell)} \right\Vert_{Z^{\otimes j_\ell}} \right) \left( \prod_{\ell=1}^q \sup_{0 \leq s < t \leq T} \frac{\left\Vert \log^{\lfloor 1/\alpha \rfloor}(\mathbf{X}_t^{-1} \otimes \mathbf{X}_s)^{(k_\ell)} \right\Vert_{Z^{\otimes k_\ell}} }{\vert s-t \vert^{k_\ell\alpha}} \right) \\
			& \leq C_3 \max_{n=1,\ldots,\lfloor 1/\alpha \rfloor \atop \vert\mathbf{j}\vert+\vert\mathbf{k}\vert=n} \left( \prod_{\ell=1}^p \left( c_{j_\ell} \Vert \mathbf{X} \Vert_\alpha^{j_\ell} \vert s - t \vert^{j_\ell \alpha} \right) \right) \left( \prod_{\ell=1}^q \left( c_{k_\ell} \Vert \mathbf{X} \Vert_\alpha^{k_\ell} \right) \right) \\
			& \leq C_4 \max_{n=1,\ldots,\lfloor 1/\alpha \rfloor} \Vert \mathbf{X} \Vert_\alpha^n < \infty,
		\end{aligned}
	\end{equation}
	where $C_1,\ldots,C_4 > 0$ are constants. This together with $\log^{\lfloor 1/\alpha \rfloor}(\mathbf{X}_0) = \log^{\lfloor 1/\alpha \rfloor}(\mathbf{1}) = \mathbf{0}$ proves that $\log^{\lfloor 1/\alpha \rfloor}(\mathbf{X}) \in \mathfrak{C}^\alpha_T(Z)$. In order to show that $\phi_i := \log^{\lfloor 1/\alpha \rfloor}: (C^\alpha_T(Z),\tau_\infty) \rightarrow (\mathfrak{C}^\alpha_T(Z),\tau_\infty)$ is continuous, we fix some $\mathbf{e} := (\mathbf{e}^{(1)},\ldots,\mathbf{e}^{(\lfloor 1/\alpha \rfloor)}) \in \bigoplus_{n=1}^{\lfloor 1/\alpha \rfloor} E^{\otimes n}$. Then, by using the Baker-Campbell-Hausdorff formula and similar arguments as in \eqref{eq:lem:sig_mfd:proof1}--\eqref{eq:lem:sig_mfd:proof2} but now with the $w^*$-seminorms of $Z^{\otimes n}$, there exists a finite subset $\mathcal{E} \subseteq \bigoplus_{n=1}^{\lfloor 1/\alpha \rfloor} E^{\otimes n}$ such that for every $\mathbf{X},\mathbf{Z} \in C^\alpha_T(Z)$, we conclude that
	\begin{equation}
		\label{eq:lem:sig_mfd:proof3}
		\begin{aligned}
			& p_\mathbf{e}\left( \log^{\lfloor 1/\alpha \rfloor}(\mathbf{X})-\log^{\lfloor 1/\alpha \rfloor}(\mathbf{Z}) \right) \\
			& = \sup_{t \in [0,T]} \max_{n=1,\ldots,\lfloor 1/\alpha \rfloor} \left\vert \langle (\log^{\lfloor 1/\alpha \rfloor}(\mathbf{X})-\log^{\lfloor 1/\alpha \rfloor}(\mathbf{Z}))^{(n)}, \mathbf{e}_n \rangle_{Z^{\otimes n} \times E^{\otimes n}} \right\vert \\
			& \leq \sup_{t \in [0,T]} \max_{n=1,\ldots,\lfloor 1/\alpha \rfloor} \left\vert \big\langle \big( \operatorname{BCH}(\log^{\lfloor 1/\alpha \rfloor}(\mathbf{Z}_t),\log^{\lfloor 1/\alpha \rfloor}(\mathbf{Z}_t^{-1} \otimes \mathbf{X}_t)) - \log^{\lfloor 1/\alpha \rfloor}(\mathbf{Z}_t) \big)^{(n)}, \mathbf{e}_n \rangle_{Z^{\otimes n} \times E^{\otimes n}} \right\vert \\
			& \leq C_5 \sup_{t \in [0,T]} \max_{n=1,\ldots,\lfloor 1/\alpha \rfloor \atop \vert\mathbf{j}\vert+\vert\mathbf{k}\vert = n} \left( \prod_{\ell=1}^p \left\vert \langle \log^{\lfloor 1/\alpha \rfloor}(\mathbf{Z}_t)^{(j_\ell)}, \mathbf{e}^{(\mathbf{j},\mathbf{k})}_{j,\ell} \rangle \right\vert \right) \left( \prod_{\ell=1}^q \left\vert \langle \log^{\lfloor 1/\alpha \rfloor}(\mathbf{Z}_t^{-1} \otimes \mathbf{X}_t)^{(k_\ell)}, \widetilde{\mathbf{e}}^{(\mathbf{j},\mathbf{k})}_{j,\ell} \rangle \right\vert \right) \\
			& \leq C_6 \sup_{t \in [0,T]} \max_{n=1,\ldots,\lfloor 1/\alpha \rfloor \atop \vert\mathbf{j}\vert+\vert\mathbf{k}\vert = n} \left( \prod_{\ell=1}^p \max_{\vert\mathbf{r}\vert=j_\ell} \prod_{m=1}^M \left\vert \langle \mathbf{Z}_t^{(r_m)}, \mathbf{e}^{(\mathbf{j},\mathbf{k},\mathbf{r})}_{l,m} \rangle \right\vert \right) \left( \prod_{\ell=1}^q \max_{\vert\mathbf{s}\vert=k_\ell} \left\vert \langle (\mathbf{Z}_t^{-1} \otimes \mathbf{X}_t)^{(s_m)}, \widetilde{\mathbf{e}}^{(\mathbf{j},\mathbf{k},\mathbf{s})}_{l,m} \rangle \right\vert \right) \\
			& \leq C_7 \left( \sup_{t \in [0,T]} \Vert \mathbf{Z}_t \Vert_{G^{\lfloor 1/\alpha \rfloor}(Z)} \right) \max_{k=1,\ldots\lfloor 1/\alpha \rfloor} \max_{\mathbf{e} \in \mathcal{E}} \left( \sup_{t \in [0,T]} \left\vert \langle (\mathbf{Z}_t^{-1} \otimes \mathbf{X}_t)^{(k)}, \mathbf{e}_k \rangle_{Z^{\otimes k} \times E^{\otimes k}} \right\vert^\frac{1}{k} \right)^k \\
			& \leq C_7 \left( \sup_{t \in [0,T]} \Vert \mathbf{Z}_t \Vert_{G^{\lfloor 1/\alpha \rfloor}(Z)} \right) \max_{k=1,\ldots\lfloor 1/\alpha \rfloor} \max_{\mathbf{e} \in \mathcal{E}} d_{\infty,\mathbf{e}}(\mathbf{Z},\mathbf{X})^k,
		\end{aligned}
	\end{equation}
	where $C_5,C_6,C_7 > 0$ are constants. This proves that $\phi_i := \log^{\lfloor 1/\alpha \rfloor}: (C^\alpha_T(Z),\tau_\infty) \rightarrow (\mathfrak{C}^\alpha_T(Z),\tau_\infty)$ is continuous at $\mathbf{Z} \in C^\alpha_T(Z)$. Conversely, in order to show that $\phi_i^{-1} := \exp^{\lfloor 1/\alpha \rfloor}: \! (\phi_i(C^\alpha_T(Z)),\tau_\infty) \!\rightarrow (C^\alpha_T(Z),\tau_\infty)$ is continuous, we use again the Baker-Campbell-Hausdorff formula and similar arguments as in \eqref{eq:lem:sig_mfd:proof3} to obtain that for every $\mathbf{e} := (\mathbf{e}^{(1)},\ldots,\mathbf{e}^{(\lfloor 1/\alpha \rfloor)}) \in \bigoplus_{n=1}^{\lfloor 1/\alpha \rfloor} E^{\otimes n}$ there exists a finite subset $\mathcal{E} \subseteq \bigoplus_{n=1}^{\lfloor 1/\alpha \rfloor} E^{\otimes n}$ such that for every $\mathbf{Y},\mathbf{Z} \in \phi_i(C^\alpha_T(Z))$ it holds that
	\begin{equation}
		\begin{aligned}
			& d_{\infty,\mathbf{e}}(\exp^{\lfloor 1/\alpha \rfloor}(\mathbf{Z}),\exp^{\lfloor 1/\alpha \rfloor}(\mathbf{Y})) \\
			& = \sup_{t \in [0,T]} \max_{n=1,\ldots,\lfloor 1/\alpha \rfloor} \left\vert \big\langle (\exp^{\lfloor 1/\alpha \rfloor}(\mathbf{Z}_t)^{-1} \otimes \exp^{\lfloor 1/\alpha \rfloor}(\mathbf{Y}_t))^{(n)}, \mathbf{e}_n \big\rangle_{Z^{\otimes n} \times E^{\otimes n}} \right\vert^\frac{1}{n} \\
			& = \sup_{t \in [0,T]} \max_{n=1,\ldots,\lfloor 1/\alpha \rfloor} \left\vert \big\langle \exp^{\lfloor 1/\alpha \rfloor}(\operatorname{BCH}(-\mathbf{Z}_t,\mathbf{Y}_t))^{(n)}, \mathbf{e}_n \big\rangle_{Z^{\otimes n} \times E^{\otimes n}} \right\vert^\frac{1}{n} \\
			& \leq C_8 \sup_{t \in [0,T]} \max_{n=1,\ldots,\lfloor 1/\alpha \rfloor} \left\vert \langle \operatorname{BCH}(-\mathbf{Z}_t,\mathbf{Y}_t)^{(n)}, \mathbf{e}_n \rangle_{Z^{\otimes n} \times E^{\otimes n}} \right\vert^\frac{1}{n} \\
			& \leq C_9 \left( \sup_{t \in [0,T]} \Vert \mathbf{Z}_t \Vert_{T^{\lfloor 1/\alpha \rfloor}(Z)} \right) \max_{n=1,\ldots\lfloor 1/\alpha \rfloor} \max_{\mathbf{e} \in \mathcal{E}} \left( \sup_{t \in [0,T]} \left\vert \langle (\mathbf{Y}_t - \mathbf{Z}_t)^{(n)}, \mathbf{e}_n \rangle_{Z^{\otimes n} \times E^{\otimes n}} \right\vert \right)^\frac{1}{n} \\
			& \leq C_9 \left( \sup_{t \in [0,T]} \Vert \mathbf{Z}_t \Vert_{T^{\lfloor 1/\alpha \rfloor}(Z)} \right) \max_{n=1,\ldots\lfloor 1/\alpha \rfloor} \max_{\mathbf{e} \in \mathcal{E}} p_{\mathbf{e}}(\mathbf{Y} - \mathbf{Z})^\frac{1}{n},
		\end{aligned}
	\end{equation}
	where $C_8, C_9 > 0$ are some constants. This proves that $\phi_i^{-1} := \exp^{\lfloor 1/\alpha \rfloor}: (\phi_i(C^\alpha_T(Z)),\tau_\infty) \rightarrow (C^\alpha_T(Z),\tau_\infty)$ is continuous at $\mathbf{Z} \in \phi_i(C^\alpha_T(Z))$. Hence, $\phi_i := \log^{\lfloor 1/\alpha \rfloor}: (C^\alpha_T(Z),\tau_\infty) \rightarrow (\mathfrak{C}^\alpha_T(Z),\tau_\infty)$ is a homeomorphism onto its image, which shows that $(C^\alpha_T(Z),\tau_\infty)$ is a $C^\infty_{loc}$-manifold.
	
	Finally, we show that the collection of push-forward weight functions $\Psi_i = (\psi_{i,j})_{j=0,\ldots,k}$ defined in \eqref{eq:sig_weight} is admissible. Since the identity $\Gamma: (\mathfrak{C}^\alpha_T(Z),\Vert \cdot \Vert_\alpha) \hookrightarrow (\mathfrak{C}^\alpha_T(Z),\tau_\infty)$ is a compact embedding (see \cite[Theorem~A.4]{cuchiero23}), we can follow the proof of Lemma~\ref{lem:w_dom}~\ref{lem:w_dom_cpt} to conclude that $\Psi_i := (\psi_{i,j})_{j=0,\ldots,k}$ is admissible, which shows that $(C^\alpha_T(Z),\Psi)$ is a weighted $C^k_{loc}$-manifold.
\end{proof}

In order to ensure point separation for the application of the weighted Nachbin theorem (Theorem~\ref{thm:w_nachbin_infdim}), we need to ensure that the signature (at terminal time) uniquely determines the path (up to so-called tree-like equivalences, see \cite{hambly10,boedihardjo16}). To this end, we define the subspace
\begin{equation}
	\widehat{C}^\alpha_T(Z) := \left\lbrace \widehat{\mathbf{X}} \in C^\alpha_T(\widehat{Z}): \widehat{X}_t = (t,X_t) \text{ for all } t \in [0,T] \text{ and some } \mathbf{X} \in C^\alpha_T(Z) \right\rbrace \subseteq C^\alpha_T(\widehat{Z}),
\end{equation}
equipped with the subspace topology of $(C^\alpha_T(\widehat{Z}),\tau_\infty)$, where running time is now added. Then,
\begin{equation}
	\label{eq:sig_time_ext}
	C^\alpha_T(Z) \ni \mathbf{X} \quad \mapsto \quad \widehat{\mathbf{X}} := \left( t \mapsto \exp^{\lfloor 1/\alpha \rfloor}(t \mathbf{a}_0) \otimes \iota(\mathbf{X}_t) \right) \in \widehat{C}^\alpha_T(Z)
\end{equation}
is a $C^\infty_{loc}$-diffeomorphism, where $\mathbf{a}_0 := (0,(1,0),0,\ldots,0) \in T^{\lfloor 1/\alpha \rfloor}_0(\widehat{Z})$ and where $\iota: G^{\lfloor 1/\alpha \rfloor}(Z) \hookrightarrow G^{\lfloor 1/\alpha \rfloor}(\widehat{Z})$ is the canonical embedding, where $\widehat{Z} := \mathbb{R} \oplus Z$ and $\widehat{E} := \mathbb{R} \oplus E$. Its inverse is given by
\begin{equation}
	\label{eq:sig_time_red}
	\widehat{C}^\alpha_T(Z) \ni \widehat{\mathbf{X}} \quad \mapsto \quad \mathbf{X} := \left( t \mapsto \pi\big( \exp^{\lfloor 1/\alpha \rfloor}(-t \mathbf{a}_0) \otimes \widehat{\mathbf{X}}_t \big) \right) \in C^\alpha_T(Z),
\end{equation}
where $\pi: G^{\lfloor 1/\alpha \rfloor}(\widehat{Z}) \rightarrow G^{\lfloor 1/\alpha \rfloor}(Z) $ is the canonical projection. Hence, we define the collection $\widehat{\Psi} := (\widehat{\psi}_j)_{j=0,\ldots,k}$ of weight functions $\widehat{\psi}_j: T^j \widehat{C}^\alpha_T(Z) \rightarrow (0,\infty)$ by
\begin{equation}
	\label{eq:sig_te_weight}
	\widehat{\psi}_j(\widehat{\mathbf{X}},[c]^j_{\widehat{\mathbf{X}}}) := \exp\left( \beta_1 \max(j,1) \Vert \log^{\lfloor 1/\alpha \rfloor}(\mathbf{X}) \Vert_\alpha^{c_1} + \beta_2 \sum_{\ell=1}^j \Vert (\widehat{\phi}_i \circ c)^{(\ell)}(0) \Vert_\alpha^{c_2} \right),
\end{equation}
for $(\widehat{\mathbf{X}},[c]^j_{\widehat{\mathbf{X}}}) \in T^j \widehat{C}^\alpha_T(Z)$, with $\beta_1,\beta_2 > 0$, $c_1 \geq \lfloor 1/\alpha \rfloor$, and $c_2 > 0$, where $\widehat{\phi}_i: \widehat{C}^\alpha_T(Z) \rightarrow \mathfrak{C}^\alpha_T(Z)$ is defined in \eqref{eq:sig_te_chart} below. Note that the corresponding push-forward weights $\widehat{\Psi}_i := (\widehat{\psi}_{i,j})_{j=0,\ldots,k}$ coincide with $\Psi_i := (\psi_{i,j})_{j=0,\ldots,k}$ defined in \eqref{eq:sig_weight_pfwd} since $\widehat{\psi}_{i,j} = \widehat{\psi}_j \circ \widehat{\Phi}_i^{-j} = \psi_j \circ \Phi_i^j = \psi_{i,j}$, for all $j = 0,\ldots,k$. Hence, $(\widehat{C}^\alpha_T(Z),\Psi)$ is also a weighted $C^k_{loc}$-manifold.

\begin{lemma}
	Let $\alpha \in (0,1]$ and assume that the predual $(E,\Vert \cdot \Vert_E)$ of $(Z,\Vert \cdot \Vert_Z)$ has BAP. Then, $(\widehat{C}^\alpha_T(Z),\widehat{\Psi})$ is a weighted $C^k_{loc}$-manifold over the model space $(\mathfrak{C}^\alpha_T(Z),\tau_\infty)$ with global chart
	\begin{equation}
		\label{eq:sig_te_chart}
		\widehat{C}^\alpha_T(Z) \ni \widehat{\mathbf{X}} \quad \mapsto \quad \widehat{\phi}_i(\widehat{\mathbf{X}}) := \log^{\lfloor 1/\alpha \rfloor}(\mathbf{X}) \in \mathfrak{C}^\alpha_T(Z)
	\end{equation}
	whose inverse is given by
	\begin{equation}
		\label{eq:sig_te_chart_inv}
		\widehat{\phi}_i(\widehat{C}^\alpha_T(Z)) \ni \mathbf{Y} \quad \mapsto \quad \widehat{\phi}_i^{-1}(\mathbf{Y}) := \reallywidehat{\exp^{\lfloor 1/\alpha \rfloor}(\mathbf{Y})} \in \widehat{C}^\alpha_T(Z).
	\end{equation}
\end{lemma}
\begin{proof}
	By using that \eqref{eq:sig_time_red} is a homeomorphism (with inverse \eqref{eq:sig_time_ext}) and $\phi_i = \log^{\lfloor 1/\alpha \rfloor}: (C^\alpha_T(Z),\tau_\infty) \rightarrow (\mathfrak{C}^\alpha_T(Z),\tau_\infty)$ in \eqref{eq:sig_chart} is a homeomorphism onto its image (with inverse $\phi_i^{-1} = \exp^{\lfloor 1/\alpha \rfloor}: (\phi_i(C^\alpha_T(Z)),\tau_\infty) \rightarrow (C^\alpha_T(Z),\tau_\infty)$ in \eqref{eq:sig_chart_inv}), we conclude that $\widehat{\phi}_i: (\widehat{C}^\alpha_T(Z),\tau_\infty) \rightarrow (\mathfrak{C}^\alpha_T(Z),\tau_\infty)$ in \eqref{eq:sig_te_chart} is also a homeomorphism onto its image (with inverse \eqref{eq:sig_te_chart_inv}), which shows that $(\widehat{C}^\alpha_T(Z),\Psi)$ is a $C^k_{loc}$-manifold over $(\mathfrak{C}^\alpha_T(Z),\tau_\infty)$. Moreover, by using that the push-forward weights $\widehat{\Psi}_i = (\widehat{\psi}_{i,j})_{j=0,\ldots,k}$ coincide with $\Psi_i = (\psi_{i,j})_{j=0,\ldots,k}$, we can follow the proof of Lemma~\ref{lem:sig_w_mfd} (invoking Lemma~\ref{lem:w_dom}~\ref{lem:w_dom_cpt}) to obtain that $(\widehat{C}^\alpha_T(Z),\widehat{\Psi})$ is a weighted $C^k_{loc}$-manifold.
\end{proof}

\subsection{Weighted universal approximation for differentiable functionals of rough paths}
\label{sec:sig_uat}

We now present the universal approximation theorem (UAT) for linear functions of the
signature, which can approximate any path space functional in $\mathcal{B}^k_{\widehat{\Psi}}(\widehat{C}^\alpha_T(Z))$ introduced in Section~\ref{sec:BPsik_maps}. In order to show their universality, we apply the weighted Nachbin theorem (Theorem~\ref{thm:w_nachbin_infdim}) over the infinite-dimensional weighted $C^k_{loc}$-manifold $(\widehat{C}^\alpha_T(Z),\widehat{\Psi})$ consisting of (time-extended) weakly $\alpha$-H\"older rough paths with values in a dual Banach space $(Z,\Vert \cdot \Vert_Z)$ having predual $(E,\Vert \cdot \Vert_E)$.

The application of the weighted Nachbin theorem (Theorem~\ref{thm:w_nachbin_infdim}) relies on the following properties of the signature. By using the Magnus expansion of the log-signature, we show that linear functions of the signature are $C^k_{\mathrm{loc}}$-maps on $(\widehat{C}^\alpha_T(Z),\tau_\infty)$ with appropriate growth conditions. Moreover, the time-extension ensures that the signature (at terminal time) uniquely determines the path (up to so-called tree-like equivalences, see \cite{hambly10,boedihardjo16}), which ensures point separation. Third, the shuffle product can be used to express any product of linear functions of the signature as another linear function of the signature, which asserts the algebra property.

\begin{theorem}[Universal approximation on $\mathcal{B}^k_{\widehat{\Psi}}(\widehat{C}^\alpha_T(Z))$]
	\label{thm:sig_uat}
	Let $\alpha \in (0,1]$ and assume that the predual $(E,\Vert \cdot \Vert_E)$ of $(Z,\Vert \cdot \Vert_Z)$ has BAP. Then, the linear span of the set
	\begin{equation}
		\left\lbrace \widehat{C}^\alpha_T(Z) \ni \widehat{\mathbf{X}} \mapsto \langle S^{(n)}(\widehat{\mathbf{X}})_T, \widehat{e}_n \rangle_{\widehat{Z}^{\otimes n} \times \widehat{E}^{\otimes n}} \in \mathbb{R}: \widehat{e}_n \in \widehat{E}^{\otimes n}, \, n \in \mathbb{N}_0 \right\rbrace
	\end{equation}
	is a dense subset of $\mathcal{B}^k_{\widehat{\Psi}}(\widehat{C}^\alpha_T(Z))$, i.e., for every $f \in \mathcal{B}^k_{\widehat{\Psi}}(\widehat{C}^\alpha_T(Z))$ and $\varepsilon > 0$ there exists some $N \in \mathbb{N}$ and a linear function $\mathbf{a} \mapsto \ell(\mathbf{a}) := \sum_{n=0}^N \langle \mathbf{a}^{(n)}, \widehat{e}_n \rangle_{\widehat{Z}^{\otimes n} \times \widehat{E}^{\otimes n}}$, with $\widehat{e}_n \in \widehat{E}^{\otimes n}$, such that
	\begin{equation}
		\max_{j=0,\ldots,k} \sup_{(\widehat{\mathbf{X}},[c]^j_{\widehat{\mathbf{X}}}) \in T^j \widehat{C}^\alpha_T(Z)} \frac{\left\vert (f \circ c)^{(j)}(0) - (\ell \circ S(\cdot)_T \circ c)^{(j)}(0) \right\vert}{\widehat{\psi}_j(\widehat{\mathbf{X}},[c]^j_{\widehat{\mathbf{X}}})} < \varepsilon.
	\end{equation}
\end{theorem}
\begin{proof}
	We aim to apply the weighted Nachbin theorem (Theorem~\ref{thm:w_nachbin_infdim}) to
	\begin{equation}
		\label{eq:thm:sig_uat:proof1}
		\mathcal{G} := \linspan\left(\left\lbrace \widehat{C}^\alpha_T(Z) \ni \widehat{\mathbf{X}} \mapsto \langle S^{(n)}(\widehat{\mathbf{X}})_T, \widehat{e}_n \rangle_{\widehat{Z}^{\otimes n} \times \widehat{E}^{\otimes n}} \in \mathbb{R}: \widehat{e}_n \in \widehat{E}^{\otimes n}, \, n \in \mathbb{N}_0 \right\rbrace\right).
	\end{equation}
	To this end, we need to show that $\mathcal{G} \subseteq \mathcal{B}^k_{\widehat{\Psi}}(\widehat{C}^\alpha_T(Z))$ is a subalgebra satisfying the conditions \ref{thm:w_nachbin_infdim:1}--\ref{thm:w_nachbin_infdim:2} of Theorem~\ref{thm:w_nachbin_infdim}, where
	\begin{equation}
		\label{eq:thm:sig_uat:proof2}
		\begin{aligned}
			\widetilde{\mathcal{G}} & := \linspan\Big(\Big\lbrace \widehat{C}^\alpha_T(Z) \ni \widehat{\mathbf{X}} \mapsto \big\langle S^{(n+k+1)}(\widehat{\mathbf{X}})_T, \left( (0,e_n) \shuffle (1,0)^{\otimes k} \right) \otimes (1,0) \big\rangle_{\widehat{Z}^{\otimes (n+k+1)},\widehat{E}^{\otimes (n+k+1)}} \in \mathbb{R}: \\
			& \quad\quad\quad\quad\quad\quad e_n \in E^{\otimes n}, n \in \lbrace 0,\ldots,\lfloor 1/\alpha \rfloor \rbrace, k \in \mathbb{N}_0 \Big\rbrace\Big) \subseteq \mathcal{G}
		\end{aligned}
	\end{equation}
	is a possible candidate for a strongly point separating and nowhere vanishing vector subspace of $\widehat{\Psi}$-moderate growth, with $\langle (t,z), (0,e_n) \rangle_{\widehat{Z} \times \widehat{E}} := \langle z,e_n \rangle_{Z \times E}$ and $\langle (t,z), (1,0) \rangle_{\widehat{Z} \times \widehat{E}} := t$.
	
	First, we show that the vector space $\mathcal{G}$ is contained in $\mathcal{B}^k_{\widehat{\Psi}}(\widehat{C}^\alpha_T(Z))$. By following \cite{schur91,hausdorff06}, we observe that the (truncated) log-signature $(L^N(\widehat{\mathbf{X}})_{s,t})_{0 \leq s \leq t \leq T} := (\log^N(S^N(\widehat{\mathbf{X}})_{s,t}))_{0 \leq s \leq t \leq T}$ at level $N \in \mathbb{N}$ of $\widehat{\mathbf{X}} \in \widehat{C}^\alpha_T(Z)$ satisfies the (backward) controlled rough differential equation (CRDE)
	\begin{equation}
		\begin{aligned}
			dL^N(\widehat{\mathbf{X}})_{s,t} & = H\big(\!\ad L^N(\widehat{\mathbf{X}})_{s,t}\big)(d\widehat{\mathbf{X}}_s), \qquad s \in [0,t], \\
			L^N(\widehat{\mathbf{X}})_{t,t} & = \mathbf{0}, 
		\end{aligned}
	\end{equation}
	in the Lie algebra $\mathfrak{g}^N(\widehat{Z})$, where $T^N_0(\widehat{Z}) \ni \mathbf{b} \mapsto (\ad \mathbf{a})(\mathbf{b}) := [\mathbf{a},\mathbf{b}] \in T^N_0(\widehat{Z})$, and where $H(z) := \frac{z}{e^z-1} =  \sum_{k=0}^\infty \frac{B_k}{k!} z^k$ with Bernoulli numbers $(B_k)_{k \in \mathbb{N}_0} := (1,-\frac{1}{2},\frac{1}{6},\ldots)$. Hence, by following \cite{magnus54,strichartz87,iserles99}, the log-signature $(L^N(\widehat{\mathbf{X}})_{s,t})_{0 \leq s \leq t \leq T}$ of $\widehat{\mathbf{X}} \in \widehat{C}^\alpha_T(Z)$ admits the Magnus expansion
	\begin{equation}
		\label{eq:thm:sig_uat:proof3}
		L^N(\widehat{\mathbf{X}})_{s,t} = \sum_{n=1}^\infty \sum_{\sigma \in \mathcal{S}_n} c_\sigma \int_{s < u_1 < \ldots < u_n < t} \Big[ d\widehat{\mathbf{X}}_{u_{\sigma(1)}}, \big[ \ldots , [d\widehat{\mathbf{X}}_{u_{\sigma(n-1)}}, d\widehat{\mathbf{X}}_{u_{\sigma(n)}}] \big] \ldots \Big],
	\end{equation}
	for some universal coefficients $(c_\sigma)_\sigma \subseteq \mathbb{R}$ ensuring that the series converges. Thus, by inserting
	\begin{equation}
		\begin{aligned}
			d\widehat{\mathbf{X}}_t & = d\left( \exp^{\lfloor 1/\alpha \rfloor}(t \mathbf{a}_0) \otimes \iota(\mathbf{X}_t) \right) \\
			& = d\exp^{\lfloor 1/\alpha \rfloor}(t \mathbf{a}_0) \otimes \exp^{\lfloor 1/\alpha \rfloor}(\iota(\mathbf{Y}_t)) + \exp^{\lfloor 1/\alpha \rfloor}(t \mathbf{a}_0) \otimes d\exp^{\lfloor 1/\alpha \rfloor}(\iota(\mathbf{Y}_t)) \\
			& = \exp^{\lfloor 1/\alpha \rfloor}(-t \mathbf{a}_0) \otimes \underbrace{\big( -G(\ad(t\mathbf{a}_0))(\mathbf{a}_0 dt) \big)}_{=-\mathbf{a}_0 dt} \otimes \exp^{\lfloor 1/\alpha \rfloor}(\iota(\mathbf{Y}_t)) \\
			& \quad\quad + \exp^{\lfloor 1/\alpha \rfloor}(t \mathbf{a}_0) \otimes \exp^{\lfloor 1/\alpha \rfloor}(-\iota(\mathbf{Y}_t)) \otimes \big( -G(\ad \iota(\mathbf{Y}_t))(\iota(d\mathbf{Y}_t)) \big)
		\end{aligned} 
	\end{equation}
	into \eqref{eq:thm:sig_uat:proof3}, where $G(z) := \frac{e^z-1}{z} = \sum_{k=0}^\infty \frac{z^k}{(k+1)!}$ (see, e.g., \cite[Lemma~7.23]{friz10}), it follows that $L^N(\widehat{\phi}_i^{-1}(\mathbf{Y}))_T := L^N(\widehat{\phi}_i^{-1}(\mathbf{Y}))_{0,T}$ is a finite universal linear combination of polynomial vector fields on the finite-step Lie algebra. Therefore, for every $\big( \widehat{\mathbf{X}} \mapsto g(\widehat{\mathbf{X}}) := \langle S^{(n)}(\widehat{\mathbf{X}})_T, \widehat{e}_n \rangle_{\widehat{Z}^{\otimes n} \times \widehat{E}^{\otimes n}} \big) \in \mathcal{G}$, with fixed $n = N \in \mathbb{N}$ and $\widehat{e}_n \in \widehat{E}^{\otimes n}$, we observe that
	\begin{equation}
		\begin{aligned}
			g_i(\mathbf{Y}) := \big( g \circ \widehat{\phi}_i^{-1} \big)(\mathbf{Y}) = \langle S^{(n)}(\widehat{\phi}_i^{-1}(\mathbf{Y}))_T, \widehat{e}_n \rangle_{\widehat{Z}^{\otimes n} \times \widehat{E}^{\otimes n}} = \langle \exp^n(L^n(\widehat{\mathbf{X}})_T)^{(n)}, \widehat{e}_n \rangle_{\widehat{Z}^{\otimes n} \times \widehat{E}^{\otimes n}},
		\end{aligned}
	\end{equation}
	is a finite universal linear combination of iterated rough integrals (of degree $n$ in $\mathbf{Y}$). Hence, by induction on $j = 1,\ldots,k$, the directional derivatives $\widehat{\phi}_i(\widehat{C}^\alpha_T(Z)) \times \mathfrak{C}^\alpha_T(Z)^j \ni (\mathbf{Y},\mathbf{V}_1,\ldots,\mathbf{V}_j) \mapsto d^j g_i(\mathbf{Y};\mathbf{V}_1,\ldots,\mathbf{V}_j) \in \mathbb{R}$ exist, are continuous on compact subsets of $\widehat{\phi}_i(\widehat{C}^\alpha_T(Z)) \times \mathfrak{C}^\alpha_T(Z)^j$, and form again universal linear combinations of iterated rough integrals (of degree $n-j$ in $\mathbf{Y}$ and of degree $1$ in each tangent direction $\mathbf{V}_1$, \ldots, $\mathbf{V}_j$). Thus, there exists some $\widetilde{e}_1,\ldots,\widetilde{e}_M \in E$ such that for every $j = 1,\ldots,k$ and $(\mathbf{Y},\mathbf{V}_1,\ldots,\mathbf{V}_j) \in \widehat{\phi}_i(\widehat{C}^\alpha_T(Z)) \times \mathfrak{C}^\alpha_T(Z)^j$ it holds that
	\begin{equation}
		\left\vert d^j g_i(\mathbf{Y};\mathbf{V}_1,\ldots,\mathbf{V}_j) \right\vert \leq \left( 1 + \Vert \mathbf{Y} \Vert_{\alpha,\widetilde{e}_{1:M}} \right)^{\max(n-j,0)} \prod_{\ell=1}^j \Vert \mathbf{V}_\ell \Vert_{\alpha,\widetilde{e}_{1:M}},
	\end{equation}
	where $\Vert \mathbf{Y} \Vert_{\alpha,\widetilde{e}_{1:M}} := \max_{m=1,\ldots,M} \Vert \mathbf{Y} \Vert_{\alpha,\widetilde{e}_m}$. Now, we recall from Remark~\ref{rem:hol_bap0} that $(\mathfrak{C}^\alpha_T(Z),\tau_\infty) := (C^\alpha_0([0,T];\mathfrak{g}^{\lfloor 1/\alpha \rfloor}(Z)),\tau_\infty)$ has AP with finite rank operators $(R_\gamma)_\gamma \subseteq (\mathfrak{C}^\alpha_T(Z),\tau_\infty)^* \otimes \mathfrak{C}^\alpha_T(Z)$ such that $\Vert R_\gamma(\mathbf{Y}) \Vert_{\alpha,\widetilde{e}_{1:M}} \leq C_{g_i} \Vert \mathbf{Y} \Vert_\alpha$, for all $\mathbf{Y} \in \mathfrak{C}^\alpha_T(Z)$ and some $C_{g_i} \geq 1$ (depending on $\widetilde{e}_1,\ldots,\widetilde{e}_M \in E$ and therefore on $g_i$). This implies for every $(\mathbf{Y},\mathbf{V}_1,\ldots,\mathbf{V}_j) \in \widehat{\phi}_i(\widehat{C}^\alpha_T(Z)) \times \mathfrak{C}^\alpha_T(Z)^j$ that
	\begin{equation}
		\label{eq:thm:sig_uat:proof4}
		\begin{aligned}
			\left\vert d^j g_i(R_\gamma(\mathbf{Y});R_\gamma(\mathbf{V}_1),\ldots,R_\gamma(\mathbf{V}_j)) \right\vert & \leq \left( 1 + \Vert R_\gamma(\mathbf{Y}) \Vert_{\alpha,\widetilde{e}_{1:M}} \right)^{\max(n-j,0)} \prod_{\ell=1}^j \Vert R_\gamma(\mathbf{V}_\ell) \Vert_{\alpha,\widetilde{e}_{1:M}} \\
			& \leq \left( 1 + C_{g_i} \Vert \mathbf{Y} \Vert_\alpha \right)^{\max(n-j,0)} \prod_{\ell=1}^j \left( C_{g_i} \Vert \mathbf{V}_\ell \Vert_\alpha \right) \\
			& \leq C_{g_i}^n \left( 1 + \Vert \mathbf{Y} \Vert_\alpha \right)^{\max(n-j,0)} \prod_{\ell=1}^j \Vert \mathbf{V}_\ell \Vert_\alpha.
		\end{aligned}
	\end{equation}
	Hence, it follows that
	\begin{equation}
		\label{eq:thm:sig_uat:proof5}
		\begin{aligned}
			& \lim_{R \rightarrow \infty} \sup_\gamma \max_{j=0,\ldots,k} \sup_{(\widehat{\phi}_i(\widehat{C}^\alpha_T(Z)) \times \mathfrak{C}^\alpha_T(Z)^j) \setminus K_{i,j,R}} \frac{\left\vert d^j g_i(R_\gamma(\mathbf{Y});R_\gamma(\mathbf{V}_1),\ldots,R_\gamma(\mathbf{V}_j)) \right\vert}{\psi_{i,j}(\mathbf{Y},\mathbf{V}_1,\ldots,\mathbf{V}_j)} \\
			& \leq C_{g_i}^n \lim_{R \rightarrow \infty} \max_{j=0,\ldots,k} \sup_{(\widehat{\phi}_i(\widehat{C}^\alpha_T(Z)) \times \mathfrak{C}^\alpha_T(Z)^j) \setminus K_{i,j,R}} \frac{\left( 1 + \Vert \mathbf{Y} \Vert_\alpha \right)^{\max(n-j,0)} \prod_{\ell=1}^j \Vert \mathbf{V}_\ell \Vert_\alpha}{\exp\left( \beta_1 \max(j,1) \Vert \mathbf{Y} \Vert_\alpha^{c_1} + \beta_2 \sum_{\ell=1}^j \Vert \mathbf{V}_\ell \Vert_\alpha^{c_2} \right)} \\
			& \leq C_{g_i}^n \lim_{R \rightarrow \infty} \max_{j=0,\ldots,k} \sup_{(\widehat{\phi}_i(\widehat{C}^\alpha_T(Z)) \times \mathfrak{C}^\alpha_T(Z)^j) \setminus K_{i,j,R}} \frac{\left( 1 + \Vert \mathbf{Y} \Vert_\alpha + \sum_{\ell=1}^j \Vert \mathbf{V}_\ell \Vert_\alpha \right)^n}{\exp\left( \beta_1 \max(j,1) \Vert \mathbf{Y} \Vert_\alpha^{c_1} + \beta_2 \sum_{\ell=1}^j \Vert \mathbf{V}_\ell \Vert_\alpha^{c_2} \right)} = 0,
		\end{aligned}
	\end{equation}
	where the supremum is taken over $(\mathbf{Y},\mathbf{V}_1,\ldots,\mathbf{V}_j) \in (\widehat{\phi}_i(\widehat{C}^\alpha_T(Z)) \times \mathfrak{C}^\alpha_T(Z)^j) \setminus K_{i,j,R}$. Thus, Proposition~\ref{prop:BPsik_equiv}~\ref{prop:BPsik_equiv:2} implies that $g_i \in \mathcal{B}^k_{\widehat{\Psi}_i}(\widehat{\phi}_i(\widehat{C}^\alpha_T(Z)))$, which ensures that $\mathcal{G} \subseteq \mathcal{B}^k_{\widehat{\Psi}}(\widehat{C}^\alpha_T(Z))$.
	
	Next, we observe that $\mathcal{G}$ is by the shuffle property \eqref{eq:sig_shuffle_prop} a subalgebra of $\mathcal{B}^k_{\widehat{\Psi}}(\widehat{C}^\alpha_T(Z))$, which also contains the constants (by choosing $n = 0$ in \eqref{eq:thm:sig_uat:proof1}). In order to show that $\mathcal{G}$ is strongly point separating and nowhere vanishing of $\widehat{\Psi}$-moderate growth, we claim that the vector subspace $\widetilde{\mathcal{G}} \subseteq \mathcal{G}$ defined in \eqref{eq:thm:sig_uat:proof2} satisfies the conditions \ref{M1}--\ref{M3} and \ref{M4'}--\ref{M5'}. For \ref{M1}, we fix some distinct $\widehat{\mathbf{X}},\widehat{\mathbf{Z}} \in \widehat{C}^\alpha_T(Z)$ and first assume (by contradiction) that for every fixed $k \in \mathbb{N}_0$, $n \in \lbrace 0,\ldots,\lfloor 1/\alpha \rfloor \rbrace$, and $e_n \in E^{\otimes n}$ it holds that 
	\begin{equation}
		\label{eq:thm:sig_uat:proof6}
		\begin{aligned}
			& \big\langle S^{(n+k+1)}(\widehat{\mathbf{X}})_T, \left( (0,e_n) \shuffle (1,0)^{\otimes k} \right) \otimes (1,0) \big\rangle_{\widehat{Z}^{\otimes (n+k+1)},\widehat{E}^{\otimes (n+k+1)}} \\
			& = \big\langle S^{(n+k+1)}(\widehat{\mathbf{Z}})_T, \left( (0,e_n) \shuffle (1,0)^{\otimes k} \right) \otimes (1,0) \big\rangle_{\widehat{Z}^{\otimes (n+k+1)},\widehat{E}^{\otimes (n+k+1)}}.
		\end{aligned}
	\end{equation}
	Then, by using \eqref{eq:sig_shuffle}--\eqref{eq:sig_shuffle_prop}, we observe for every $\widehat{\mathbf{W}} \in \widehat{C}^\alpha_T(Z)$ that
	\begin{equation}
		\label{eq:thm:sig_uat:proof7}
		\begin{aligned}
			& \big\langle S^{(n+k+1)}(\widehat{\mathbf{W}})_T, \left( (0,e_n) \shuffle (1,0)^{\otimes k} \right) \otimes (1,0) \big\rangle_{\widehat{Z}^{\otimes (n+k+1)},\widehat{E}^{\otimes (n+k+1)}} \\
			& = \int_0^T \langle S^{(n+k)}(\widehat{\mathbf{W}})_t, (0,e_n) \shuffle (1,0)^{\otimes k} \rangle_{\widehat{Z}^{\otimes (n+k)},\widehat{E}^{\otimes (n+k)}} \, dt \\
			& = \int_0^T \langle S^{(n)}(\widehat{\mathbf{W}})_t, (0,e_n) \rangle_{\widehat{Z}^{\otimes n} \times \widehat{E}^{\otimes n}} \, \langle S^{(k)}(\widehat{\mathbf{W}})_t, (1,0)^{\otimes k} \rangle_{\widehat{Z}^{\otimes k} \times \widehat{E}^{\otimes k}} \, dt \\
			& = \int_0^T \langle S^{(n)}(\mathbf{W})_t, e_n \rangle_{Z^{\otimes n} \times E^{\otimes n}} \frac{t^k}{k!} dt.
		\end{aligned}
	\end{equation}
	Hence, by combining \eqref{eq:thm:sig_uat:proof6} with \eqref{eq:thm:sig_uat:proof7}, it follows that
	\begin{equation}
		\int_0^T \langle S^{(n)}(\mathbf{X})_t - S^{(n)}(\mathbf{Z})_t, e_n \rangle_{Z^{\otimes n} \times E^{\otimes n}} \frac{t^k}{k!} dt = 0.
	\end{equation}
	Thus, by using that $\Pol(\mathbb{R})\vert_{[0,T]}$ is weakly dense in $L^1([0,T])$ with $L^1([0,T])^* \cong L^\infty([0,T]) \supseteq C^0([0,T])$, we have $\langle S^{(n)}(\mathbf{X})_t - S^{(n)}(\mathbf{Z})_t, e_n \rangle_{Z^{\otimes n} \times E^{\otimes n}} = 0$, for all $t \in [0,T]$. Since $E^* \cong Z$ is by the Hahn-Banach theorem point separating on $Z$, it follows that $\mathbf{X}_t = \mathbf{Z}_t$, for all $t \in [0,T]$. This however contradicts the assumption that $\widehat{\mathbf{X}},\widehat{\mathbf{Z}} \in \widehat{C}^\alpha_T(Z)$ are distinct, which shows that $\widetilde{\mathcal{G}}$ is point separating on $\widehat{C}^\alpha_T(Z)$. For \ref{M2}, we observe that the map $\widetilde{g}(\cdot) := \big\langle S^{(1)}(\cdot)_T, \big( (0,0) \shuffle (1,0)^{\otimes 0} \big) \otimes (1,0) \big\rangle_T$ satisfies $\widetilde{g}(\widehat{\mathbf{X}}) = T \neq 0$. For \ref{M3}, it suffices to show that $\widetilde{\mathcal{G}}_i := \lbrace \widetilde{g} \circ \widehat{\phi}_i^{-1}: \widetilde{g} \in \widetilde{\mathcal{G}} \rbrace$ has nowhere vanishing derivatives (as $\widehat{\phi}_i: \widehat{C}^\alpha_T(Z) \rightarrow \widehat{\phi}_i(\widehat{C}^\alpha_T(Z))$ is a $C^k_{loc}$-diffeomorphism). To this end, we fix some $\mathbf{Y} \in \widehat{\phi}_i(\widehat{C}^\alpha_T(Z))$ and $\mathbf{V} \in \mathfrak{C}^\alpha_T(Z) \setminus \lbrace 0 \rbrace$, whence there exists some $e_1 \in E$ and $t \in [0,T]$ such that $\langle \mathbf{V}_t, e_1 \rangle_{Z \times E} \neq 0$. Moreover, by using \eqref{eq:thm:sig_uat:proof6} (with $n = 1$) and the definition of $\exp^{\lfloor 1/\alpha \rfloor}$ in \eqref{eq:sig_exp}, it holds for every $k \in \mathbb{N}_0$ that
	\begin{equation}
		\label{eq:thm:sig_uat:proof8}
		\begin{aligned}
			& \big\langle S^{(k+2)}(\widehat{\exp^{\lfloor 1/\alpha \rfloor}(\mathbf{Y})})_T, \left( (0,e_1) \shuffle (1,0)^{\otimes k} \right) \otimes (1,0) \big\rangle_{\widehat{Z}^{\otimes (k+2)},\widehat{E}^{\otimes (k+2)}} \\
			& = \int_0^T \langle S^{(1)}(\exp^{\lfloor 1/\alpha \rfloor}(\mathbf{Y}))_t, e_1 \rangle_{Z \times E} \, \frac{t^k}{k!} dt = \int_0^T \langle \mathbf{Y}_t, e_1 \rangle_{Z \times E} \frac{t^k}{k!} dt.
		\end{aligned}
	\end{equation}
	Thus, by defining $\widetilde{g}_i \in \widetilde{\mathcal{G}}_i$ as the left-hand side of \eqref{eq:thm:sig_uat:proof8}, taking the directional derivative, and using again that polynomials are weakly dense in $L^1([0,T])$, there exists some $k \in \mathbb{N}_0$ such that
	\begin{equation}
		d\widetilde{g}_i(\mathbf{Y};\mathbf{V}) = \int_0^T \langle \mathbf{V}_t, e_1 \rangle_{Z \times E} \, \frac{t^k}{k!} dt \neq 0,
	\end{equation}
	which shows that $\widetilde{\mathcal{G}}_i$ has nowhere vanishing derivatives, and so does $\widetilde{\mathcal{G}}$. For \ref{M4'}, we use for every $\gamma$ that $T_\gamma(\widehat{\phi}_i(\widehat{C}^\alpha_T(Z)))$ is $m$-dimensional (with some $m \in \mathbb{N}$), on which $\widetilde{\mathcal{G}}_i := \big\lbrace \widetilde{g} \circ \widehat{\phi}_i^{-1}: \widetilde{g} \in \widetilde{\mathcal{G}} \big\rbrace$ is point separating and has nowhere vanishing derivatives, to obtain some $\widetilde{g}_1,\ldots,\widetilde{g}_m$ such that
	\begin{equation}
		\eta_{i,\gamma} := \big( \widetilde{g}_1 \circ \widehat{\phi}_i^{-1},\ldots,\widetilde{g}_m \circ \widehat{\phi}_i^{-1} \big)^\top\big\vert_{T_\gamma(\widehat{\phi}_i(\widehat{C}^\alpha_T(Z)))}: T_\gamma(\widehat{\phi}_i(\widehat{C}^\alpha_T(Z))) \rightarrow \mathbb{R}^m
	\end{equation}
	is an embedding, where the cutoff functions are obtained by a finite-dimensional smooth exhaustion argument. For \ref{M5'}, we fix $\widetilde{g}_i \in \widetilde{\mathcal{G}}_i$ and define $\lambda := \beta_1 (C_{g_i} \Lambda)^{-\lfloor 1/\alpha \rfloor}/2 > 0$ with $C_{g_i} \geq 1$ from above, where $(\widehat{\phi}_i(\widehat{C}^\alpha_T(Z)),\tau_\infty)$ has AP with finite rank operators $(T_\gamma)_\gamma$ satisfying \eqref{eq:cor:hol_bap} for some constant $\Lambda \geq 1$. Then, by using \eqref{eq:thm:sig_uat:proof4} and $c_1 \geq \lfloor 1/\alpha \rfloor$, it follows for every $\gamma$ that
	\begin{equation}
		\begin{aligned}
			& \lim_{R \rightarrow \infty} \max_{j=0,\ldots,k \atop \pi \in \mathscr{P}_j} \sup_{(\widetilde{\mathbf{Y}},\widetilde{\mathbf{V}}_1,\ldots,\widetilde{\mathbf{V}}_j) \in (T_\gamma(\widehat{\phi}_i(\widehat{C}^\alpha_T(Z))) \times T_\gamma(\mathfrak{C}^\alpha_T(Z))^j) \setminus K_{i,\gamma,j,R}} \frac{\exp\big( \lambda \vert g_i(\widetilde{\mathbf{Y}}) \vert \big) \vert d^\pi g_i(\widetilde{\mathbf{Y}};\widetilde{\mathbf{V}}_\pi) \vert}{\psi_{i,\gamma,j}(\widetilde{\mathbf{Y}},\widetilde{\mathbf{V}}_1,\ldots,\widetilde{\mathbf{V}}_j)} \\
			& \leq \!\lim_{R \rightarrow \infty} \max_{j=0,\ldots,k \atop \pi \in \mathscr{P}_j} \sup_{K_{i,j,R}^c} \frac{e^{\lambda (1+C_{g_i} \Lambda \Vert \mathbf{Y} \Vert_\alpha)^{\lfloor 1/\alpha \rfloor}} (1\!+\! C_{g_i} \Lambda \Vert \mathbf{Y} \Vert_\alpha)^{\max(\lfloor 1/\alpha \rfloor-\vert\pi\vert,0)} \prod_{\ell=1}^j (C_{g_i} \Lambda \Vert \mathbf{V}_\ell \Vert_\alpha)}{\exp\left( \beta_1 \max(j,1) \Vert \mathbf{Y} \Vert_\alpha^{c_1} + \beta_2 \sum_{\ell=1}^j \Vert \mathbf{V}_\ell \Vert_\alpha^{c_2} \right)} \\
			& \leq \! (C_{g_i} \Lambda)^{\lfloor 1/\alpha \rfloor} \lim_{R \rightarrow \infty} \max_{j=0,\ldots,k \atop \pi \in \mathscr{P}_j} \sup_{K_{i,j,R}^c} \frac{e^{\beta_1/2 (1\!+\!\Vert \mathbf{Y} \Vert_\alpha)^{\lfloor 1/\alpha \rfloor}} (1+\Vert \mathbf{Y} \Vert_\alpha)^{\max(\lfloor 1/\alpha \rfloor-\vert\pi\vert,0)} \prod_{\ell=1}^j \Vert \mathbf{V}_\ell \Vert_\alpha}{\exp\left( \beta_1 \max(j,1) \Vert \mathbf{Y} \Vert_\alpha^{c_1} + \beta_2 \sum_{\ell=1}^j \Vert \mathbf{V}_\ell \Vert_\alpha^{c_2} \right)} = 0,
		\end{aligned}
	\end{equation}
	where the supremum is taken over $(\mathbf{Y},\mathbf{V}_1,\ldots,\mathbf{V}_j) \in (\widehat{\phi}_i(\widehat{C}^\alpha_T(Z)) \times \mathfrak{C}^\alpha_T(Z)^j) \setminus K_{i,j,R}$. This shows that \ref{M5'} is satisfied.
	
	Finally, we show that condition~\ref{thm:w_nachbin_infdim:2} of Theorem~\ref{thm:w_nachbin_infdim} is satisfied. By following the proof of Corollary~\ref{cor:hol_bap}, we may assume that $R_\gamma \in (\mathfrak{C}^\alpha_T(Z),\tau_\infty)^* \otimes \mathfrak{C}^\alpha_T(Z)$ is of the form $\mathfrak{C}^\alpha_T(Z) \ni \mathbf{Y} \mapsto R_\gamma(\mathbf{Y}) := \sum_{j=1}^J \mathbf{Y}(t_j) h_j(\cdot) \in \mathfrak{C}^\alpha_T(Z)$. Hence, for every $g_i \in \mathcal{G}_i := \lbrace g \circ \widehat{\phi}_i^{-1}: g \in \mathcal{G} \rbrace$ and $\mathbf{Y} \in \widehat{\phi}_i(\widehat{C}^\alpha_T(Z))$, we observe that $(g_i \circ R_\gamma)(\mathbf{Y})$ depends only on products of linear functionals of $(\mathbf{Y}(t_j))_{j=1,\ldots,J}$, which can be approximated by elements from $\mathcal{G}_i$ with respect to $\Vert \cdot \Vert_{\mathcal{B}^k_{\widehat{\Psi}_i}(\widehat{\phi}_i(\widehat{C}^\alpha_T(Z)))}$. Now, we can apply Theorem~\ref{thm:w_nachbin_infdim} to conclude that $\mathcal{G}$ is a dense subset of $\mathcal{B}^k_{\widehat{\Psi}}(\widehat{C}^\alpha_T(Z))$.
\end{proof}

\begin{remark}
	Let us point out the following remarks concerning Theorem~\ref{thm:sig_uat}:
	\begin{enumerate}
		\item Theorem~\ref{thm:sig_uat} extends the universal approximation theorem of \cite[Theorem~5.4]{cuchiero23} for linear functions of the signature by including the approximation of the directional derivatives.
		\item A similar result could be obtained for weakly geometric $p$-variation rough paths by intersecting H\"older spaces with $p$-variation spaces (see \cite[Section~5.3]{cuchiero23}).
		\item Theorem~\ref{thm:sig_uat} could be generalized to the space of stopped $\alpha$-H\"older rough paths $\Lambda^\alpha_T(Z)$ given as the vector bundle
		\begin{equation}
			\Lambda^\alpha_T(Z) := \bigcup_{t \in (0,T)} \left\lbrace (t,\mathbf{X}_{[0,t]}): \widehat{\mathbf{X}}_{[0,t]} \in \widehat{C}^\alpha_t(Z) \right\rbrace.
		\end{equation}
		Then, similar universal approximation results as in Corollary~\ref{cor:naf_uat_full} (all directional derivatives) and Corollary~\ref{cor:naf_uat_horver} (only horizontal and vertical derivatives) can be shown, where the approximation holds uniformly in $t \in (0,T)$.
	\end{enumerate}
\end{remark}

\section{Numerical experiments}
\label{sec:numerics}

In this section, we illustrate in two examples\footnote{The experiments have been implemented in \texttt{Python} using the \texttt{TensorFlow} package on an HPC (high-performance computing) cluster of ETH Zurich. The code can be found under the following link: \url{https://github.com/psc25/GlobalUATDerivatives}.} how to learn path space functionals including their horizontal and (an approximation of the) vertical derivatives. More precisely, given a functional $f: \Lambda^{\alpha,1}_{T,\mathbb{R}} \rightarrow \mathbb{R}$, we use non-anticipative path-neural networks (Section~\ref{sec:naf}) and linear functions of the signature (Section~\ref{sec:sig}) to approximate the functional value $f(t,x)$, the horizontal derivative
\begin{equation}
	\mathcal{D} f(t,x) = \lim_{h \rightarrow 0^+} \frac{f(t+h,x^t)-f(t,x^t)}{h},
\end{equation}
and the vertical derivative
\begin{equation}
	\mathscr{D} f(t,x) := \mathscr{D}_{e_1} f(t,x) = \lim_{h \rightarrow 0} \frac{f(t,x^t+h\mathds{1}_{[t,T]})-f(t,x^t)}{h} \approx \lim_{h \rightarrow 0} \frac{f(t,x^t+hg_t)-f(t,x^t)}{h},
\end{equation}
where the latter is applied for linear functions of the signature (allowing only for continuous paths as input). Here, $g_t \in C^\alpha([0,T])$ is an approximation of $\mathds{1}_{[t,T]}$, e.g., given by
\begin{equation}
	g_t(s) :=
	\begin{cases}
		0 & \text{if } s \in [0,t], \\
		\frac{3}{h^2} (s-t)^2 - \frac{2}{h^3} (s-t)^3 & \text{if } s \in (t,\min(t+h,T)], \\
		1 & \text{if } s \in (\min(t+h,T),T],
	\end{cases}
\end{equation}
for some $h \in (0,T)$.

As input data we generate $M = 50 000$ sample paths of a one-dimensional Brownian motion $x(m) := (x(m)_t)_{t \in [0,T]}$, for $m = 1,\ldots,M$, with $T = 1$, which are discretized over $K = 101$ equidistant time points $(t_k)_{k=1,\ldots,K}$. Since the sample paths of Brownian motion are a.s.~$\alpha$-H\"older continuous, for all $\alpha \in (0,\frac{1}{2})$, we consider the weighted space of stopped $\alpha$-H\"older continuous paths $\Lambda^{\alpha,1}_{T,\mathbb{R}}$ introduced in Section~\ref{sec:naf}. On the other hand, every sample path of Brownian motion $x(m)$, $m = 1,\ldots,M$, can therefore be lifted to a weakly geometric $\alpha$-rough path $\mathbf{X}(m) \in C^\alpha_T(\mathbb{R})$, for $\alpha \in (0,1/2)$, from which we compute the time-extended signature $S\big(\widehat{\mathbf{X}(m)}\big)_t$, for $t \in [0,T]$.

Since we only consider two directional derivatives, we define for non-anticipative path-neural networks (PNNs) the weight function as in \eqref{eq:naf_w_horver}, i.e.,
\vspace{-0.05cm}
\begin{equation}
	\!\!\!\!\!\! \Lambda^{\alpha,1}_{T,\mathbb{R}} \ni (t,x) \quad \mapsto \quad \psi_{\mathrm{PNN}}(t,x) := \exp\left( \beta \Vert x \Vert_\alpha^c \right) \in (0,\infty),
	\vspace{-0.05cm}
\end{equation}
for some $\beta,c > 0$. Similarly, for linear functions of the signature, we omit the directional derivative terms in \eqref{eq:sig_weight} and define for $\beta > 0$ and $c \geq \lfloor 1/\alpha \rfloor$ the weight function
\vspace{-0.05cm}
\begin{equation}
	\quad\quad\quad\,\, \Lambda^{\alpha,1}_{T,\mathbb{R}} \ni (t,x) \quad \mapsto \quad \psi_{\mathrm{Sig}}(t,x) := \exp\left( \beta \Vert \log^{\lfloor 1/\alpha \rfloor}(S^2(x)) \Vert_\alpha^c \right) \in (0,\infty).
\end{equation}
In the first example, we consider the non-anticipative functional $f_1: \Lambda^{\alpha,1}_{T,\mathbb{R}} \rightarrow \mathbb{R}$, which is together with its horizontal and vertical derivatives for every $(t,x) \in \Lambda^{\alpha,1}_{T,\mathbb{R}}$ given by
\vspace{-0.05cm}
\begin{equation}
	\label{EqEx1}
	\begin{aligned}
		f_1(t,x) & = x_t \int_0^t x_s ds, \\
		\mathcal{D} f_1(t,x) & = x_t^2, \\
		\mathscr{D} f_1(t,x) & = \int_0^t x_s ds.
	\end{aligned}
\end{equation}
In the second example, we consider the non-anticipative functional $f_2: \Lambda^{\alpha,1}_{T,\mathbb{R}} \rightarrow \mathbb{R}$, which is together with its horizontal and vertical derivatives for every $(t,x) \in \Lambda^{\alpha,1}_{T,\mathbb{R}}$ given by
\vspace{-0.05cm}
\begin{equation}
	\label{EqEx2}
	\begin{aligned}
		\qquad\qquad\quad f_2(t,x) & = \int_0^t \max(x_s,0) ds - \frac{x_t^2}{2}, \\
		\qquad\qquad\quad \mathcal{D} f_2(t,x) & = \max(x_t,0), \\
		\qquad\qquad\quad \mathscr{D} f_2(t,x) & = -x_t.
	\end{aligned}
\end{equation}
We split up the data into $80\%/20\%$ for training and testing, respectively, and then apply the Adam algorithm (see \cite{kb15}) over $4000$ epochs with learning rate $10^{-5}$ and batchsize $500$ to minimize the weighted mean squared error
\vspace{-0.05cm}
\begin{equation}
	\label{eq:loss}
	\begin{aligned}
		& \frac{1}{MK} \sum_{m=1}^M \sum_{k=1}^K \Bigg[ \left( \frac{\left\vert f_i(t_k,x(m)) - g(t_k,x(m)) \right\vert}{\psi_{\mathrm{Method}}(t_k,x(m))} \right)^2 + \left( \frac{\left\vert \mathcal{D} f_i(t_k,x(m)) - \mathcal{D} g(t_k,x(m)) \right\vert}{\psi_{\mathrm{Method}}(t_k,x(m))} \right)^2 \\
		& \quad\quad\quad\quad\quad\quad\quad + \left( \frac{\left\vert \mathscr{D} f_i(t_k,x(m)) - \mathscr{D} g(t_k,x(m)) \right\vert}{\psi_{\mathrm{Method}}(t_k,x(m))} \right)^2 \Bigg]
	\end{aligned}
\end{equation}
for $i \in \lbrace 1,2 \rbrace$ and $\mathrm{Method} \in \lbrace \mathrm{PNN}, \mathrm{Sig} \rbrace$, where $g := \varphi \in \mathcal{PN}^{\widetilde{\rho},\rho,\mathcal{L}}_{\Lambda^{\alpha,1}_{T,\mathbb{R}}}$ for path-NNs (PNN) or $g(t,x) := \sum_{0 \leq \vert I \vert \leq N_{\mathrm{Sig}}} a_I \langle S(\widehat{x})_t, e_I \rangle$ for linear functions of the signature (Sig). In both cases, we compute an approximation of $\Vert x(m) \Vert_\alpha$ and $\Vert \log^{\lfloor 1/\alpha \rfloor}(S^2(x)) \Vert_\alpha$ used in $\psi_{\mathrm{PNN}}$ and $\psi_{\mathrm{Sig}}$, respectively (see the code). Moreover, we choose $\alpha = 0.4$, $\beta = 0.01$, $c = 2.0$, and $N_{\mathrm{Sig}} = 6$. For the PNNs, we consider $\varphi \in \mathcal{PN}^{\widetilde{\rho},\rho,\mathcal{L}}_{\Lambda^{\alpha,1}_{T,\mathbb{R}}}$ (see Definition~\ref{def:naf_nn}) with $N_{\mathrm{PNN}} = 30$ neurons, activation functions $\rho(s) = \widetilde{\rho}(s) = \tanh(s)$, and classical neural networks $(\phi_{n,1})_{n=1,\ldots,N_{\mathrm{PNN}}}$ with one hidden layer of $N_1 = 20$ neurons, where the time integral inside is approximated with a left Riemann sum.

Figures~\ref{fig:ex1} and \ref{fig:ex2} empirically demonstrate that the values of the functionals $f_1$ and $f_2$ together with their horizontal and vertical derivatives can be approximated both by non-anticipative path-neural networks (PNN) and by linear functions of the signature (Sig). The approximations of the PNNs (dotted lines) and linear functions of the signature (dash-dotted lines) are very accurate as they almost overlap the true values (solid lines).

Notice that the weighted mean squared error \eqref{eq:loss} reflects the weighted aspect of our universal approximation theorems (UATs), analogously to classical UATs over compact subsets, for which the unweighted mean squared error is applied. However, unlike classical UATs on compacta, our framework ensures the existence of an approximation beyond compact subsets, including the derivatives. This overcomes the limitation that, for a pre-specified compact training set (e.g., sample paths of Brownian motion), the test data may lie outside the chosen compactum.

\begin{figure}[htbp]
	\subfigure[Learning performance]{
		\includegraphics[height=4.8cm]{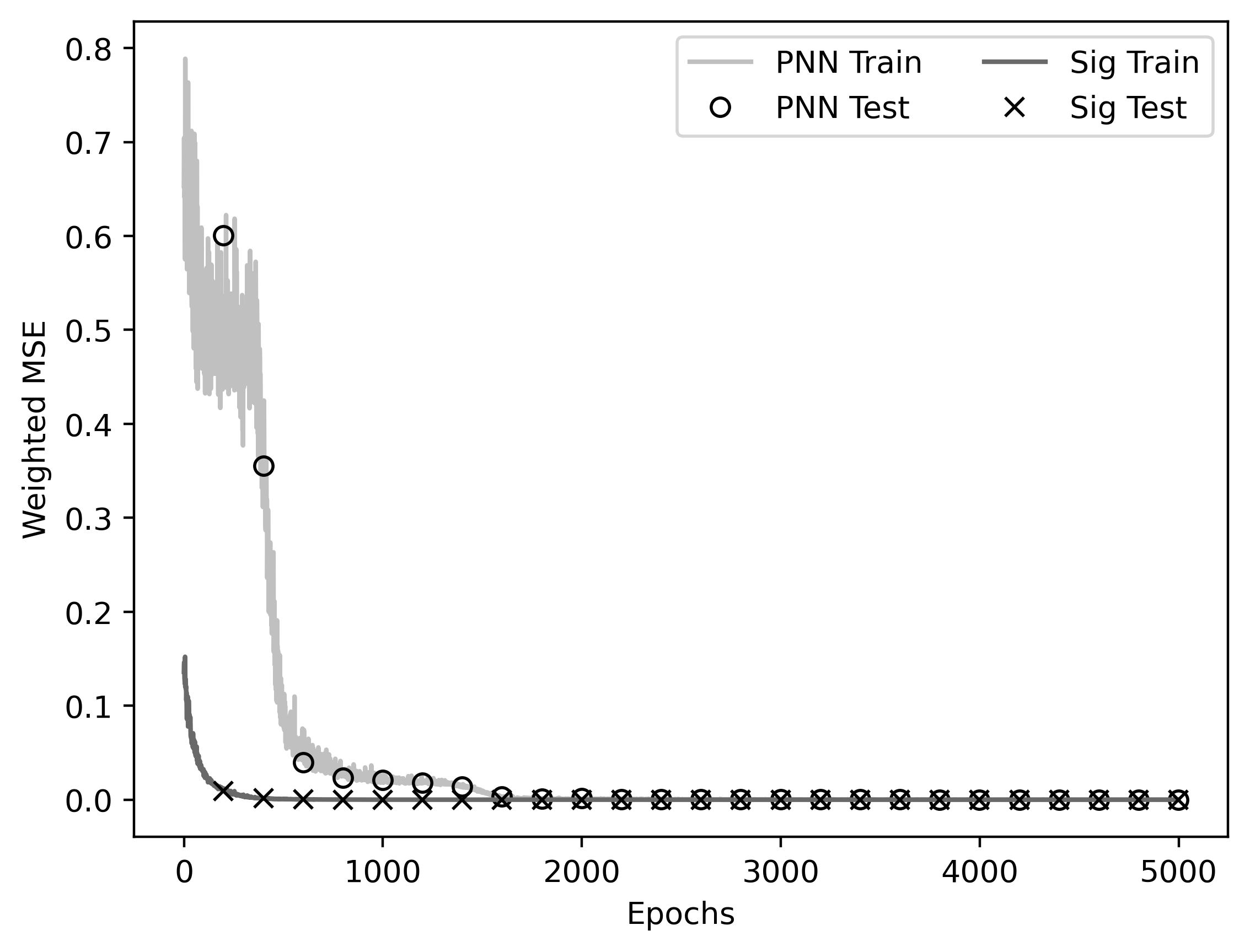}
	}
	\subfigure[$t \mapsto f_1(t,x)$ for three samples $x$ of test set]{
		\includegraphics[height=4.8cm]{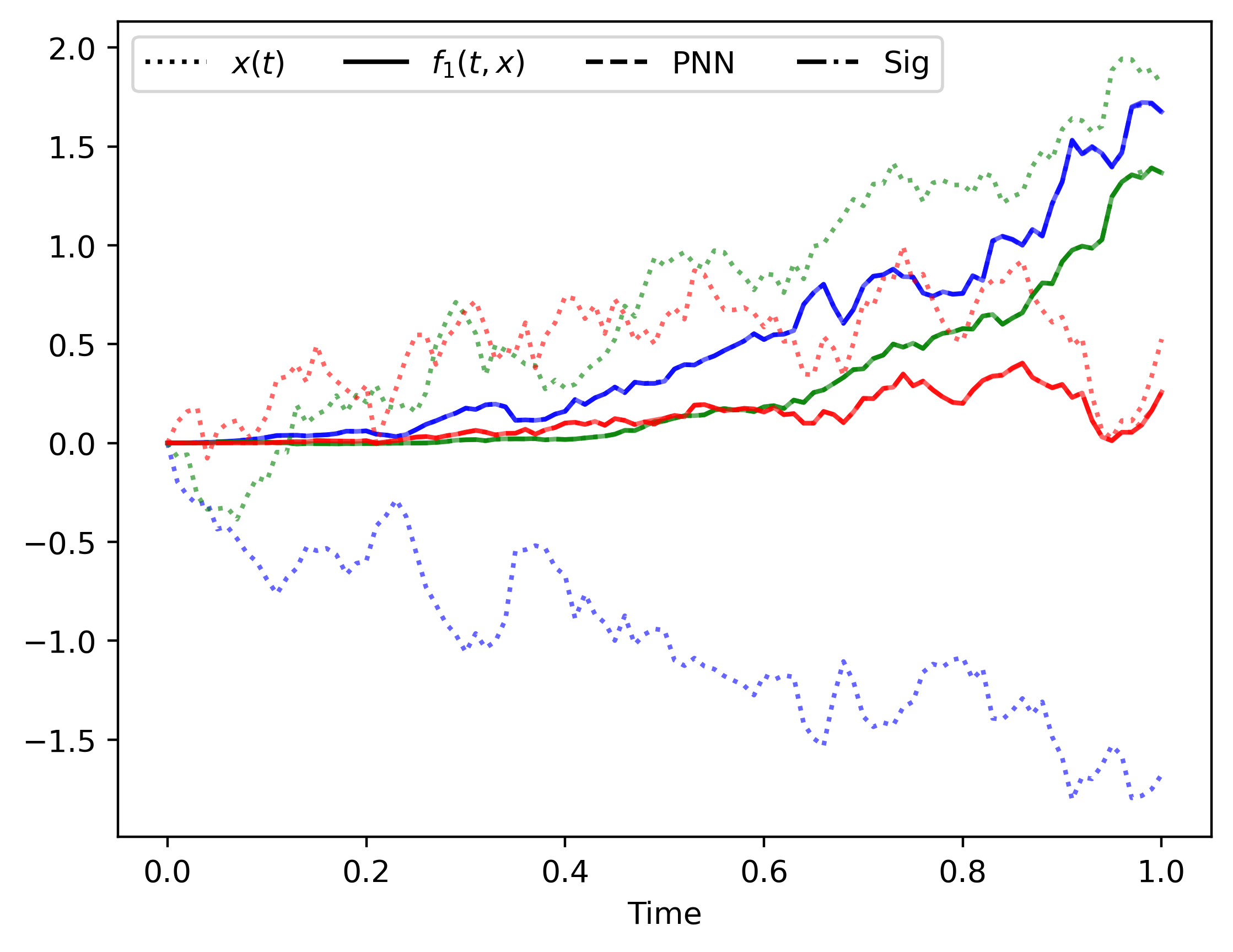}
	}
	\subfigure[$t \mapsto \mathcal{D} f_1(t,x)$ for three samples $x$ of test set]{
		\includegraphics[height=4.8cm]{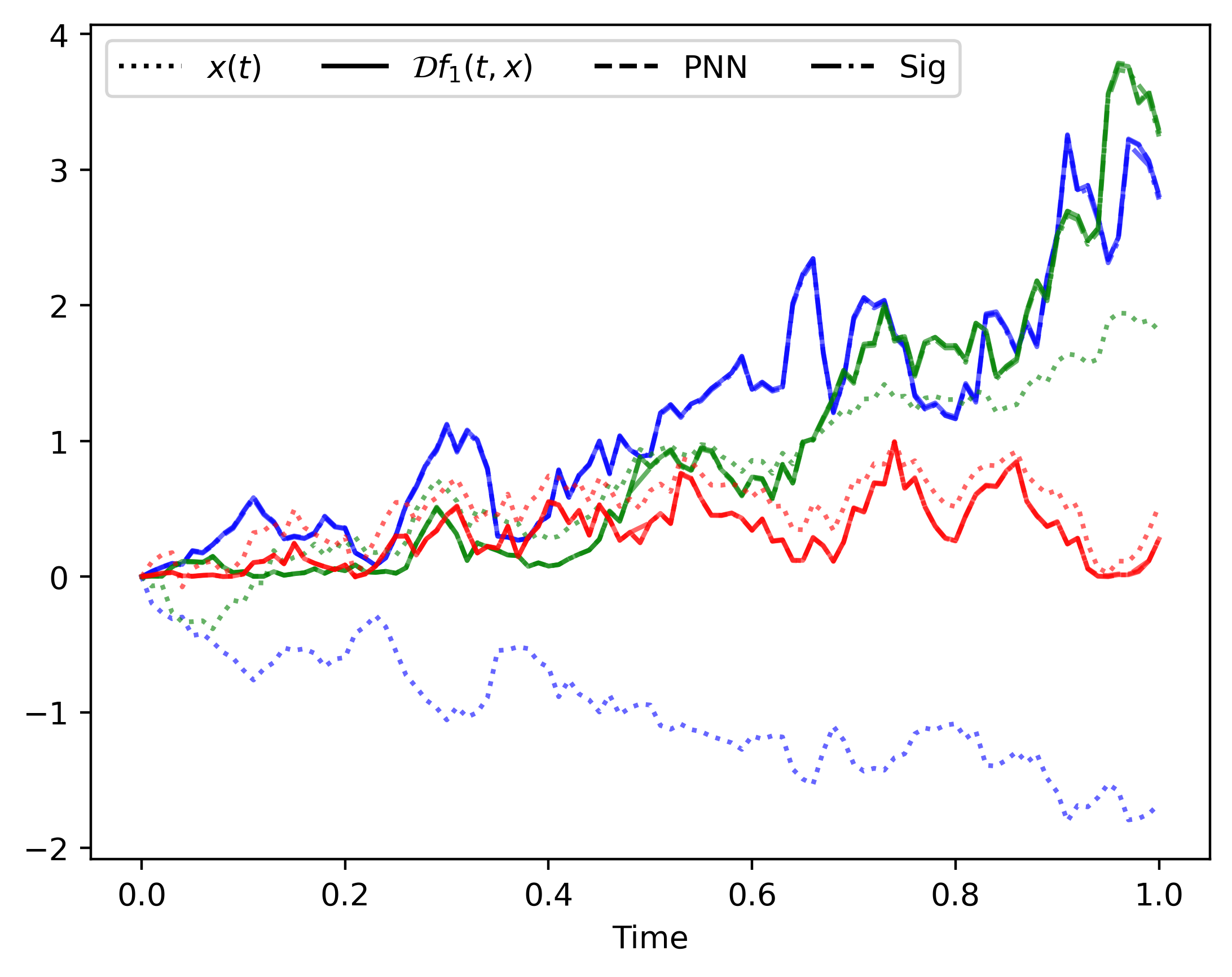}
	}
	\subfigure[$t \mapsto \mathscr{D} f_1(t,x)$ for three samples $x$ of test set]{
		\includegraphics[height=4.8cm]{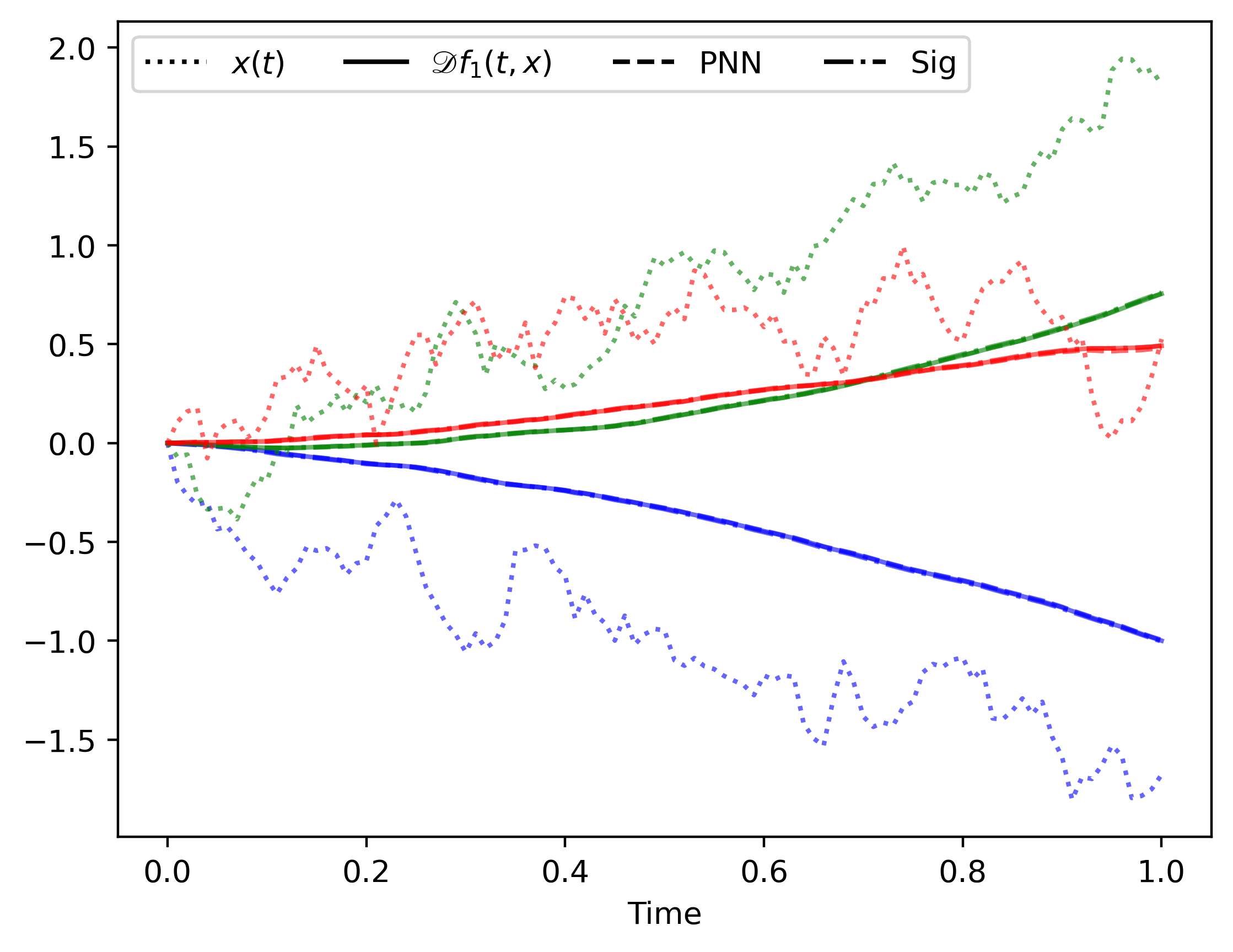}
	}
	\caption{\small Learning $f_1$ defined in \eqref{EqEx1} by path-NN $\varphi \in \mathcal{PN}^{\widetilde{\rho},\rho,\mathcal{L}}_{\Lambda^{\alpha,1}_{T,\mathbb{R}}}$ (label FNN) and linear function of the signature $\sum_{0 \leq \vert I \vert \leq N_{\mathrm{Sig}}} a_I \langle S(\widehat{\mathbf{X}})_t, e_I \rangle$ (label Sig). In (a), the weighted mean squared error \eqref{eq:loss} is evaluated on the training set in each epoch (solid line) as well as on the test set after every 200-th epoch (dots). In (b)--(d), three samples $x(m)$ of the test set are shown together with $f_1(\cdot,x(m))$, $\mathcal{D} f_1(\cdot,x(m))$, $\mathscr{D} f_1(\cdot,x(m))$ and their approximations.}
	\label{fig:ex1}
\end{figure}

\begin{figure}[htbp]
	\subfigure[Learning performance]{
		\includegraphics[height=4.8cm]{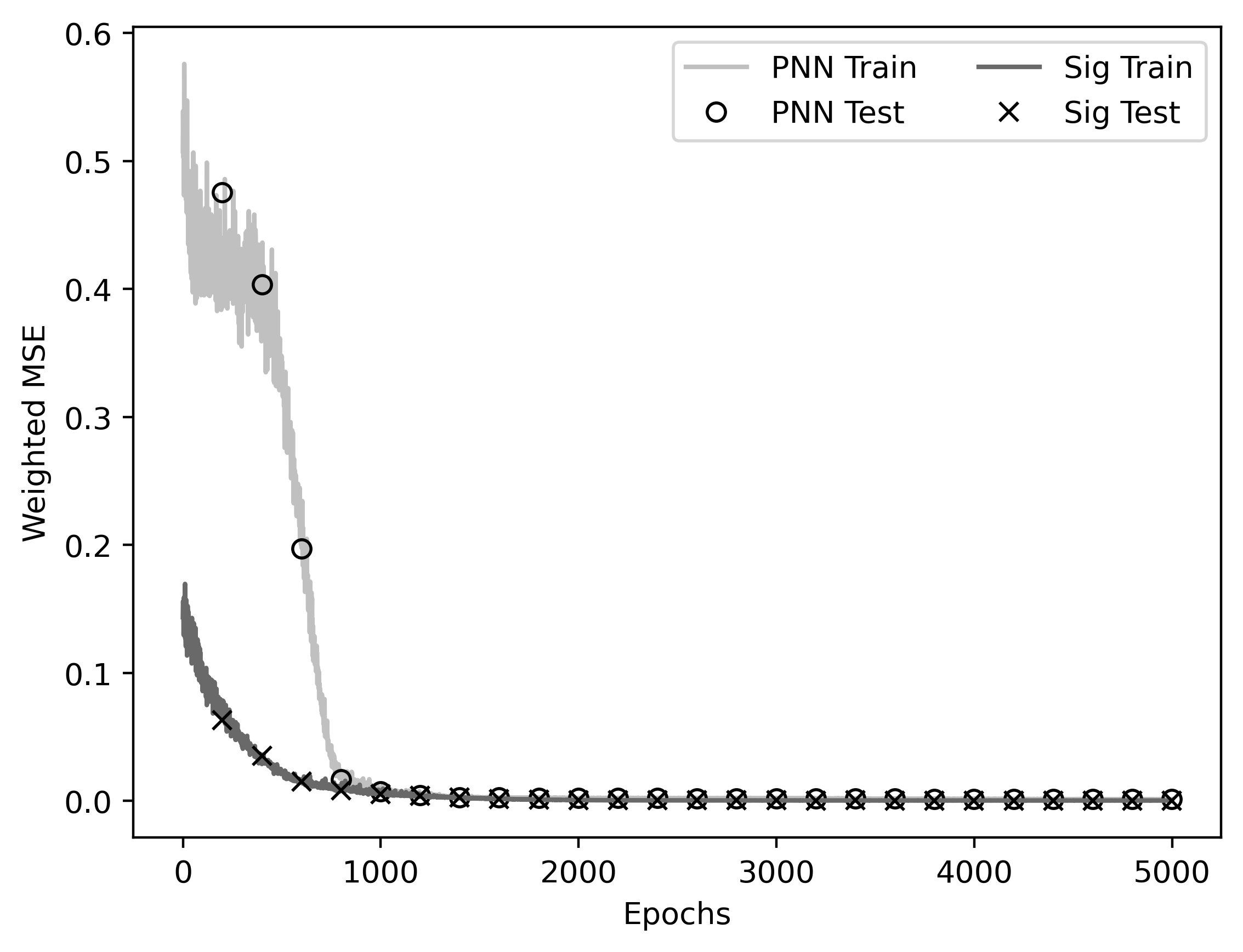}
	}
	\subfigure[$t \mapsto f_2(t,x)$ for three samples $x$ of test set]{
		\includegraphics[height=4.8cm]{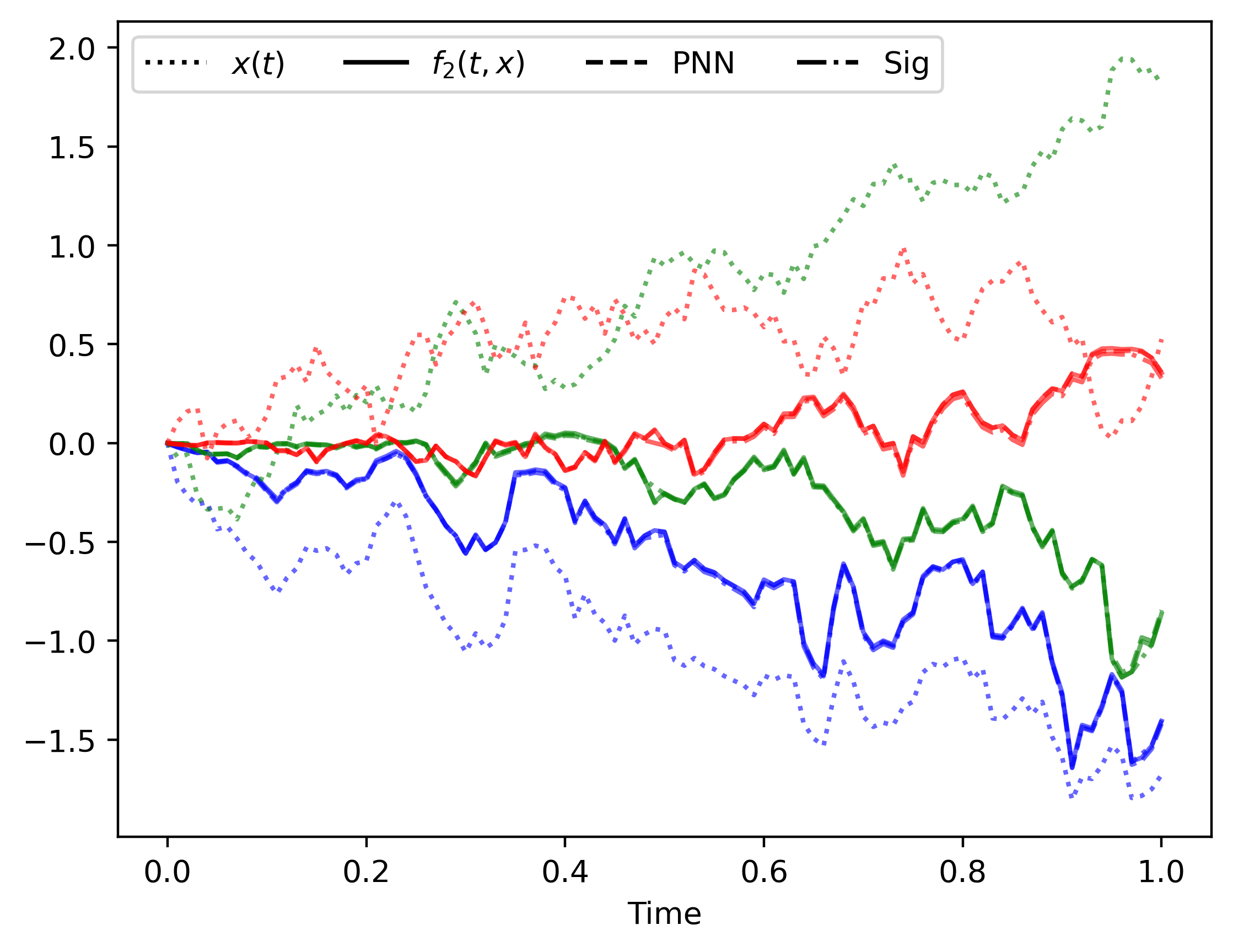}
	}
	\subfigure[$t \mapsto \mathcal{D} f_2(t,x)$ for three samples $x$ of test set]{
		\includegraphics[height=4.8cm]{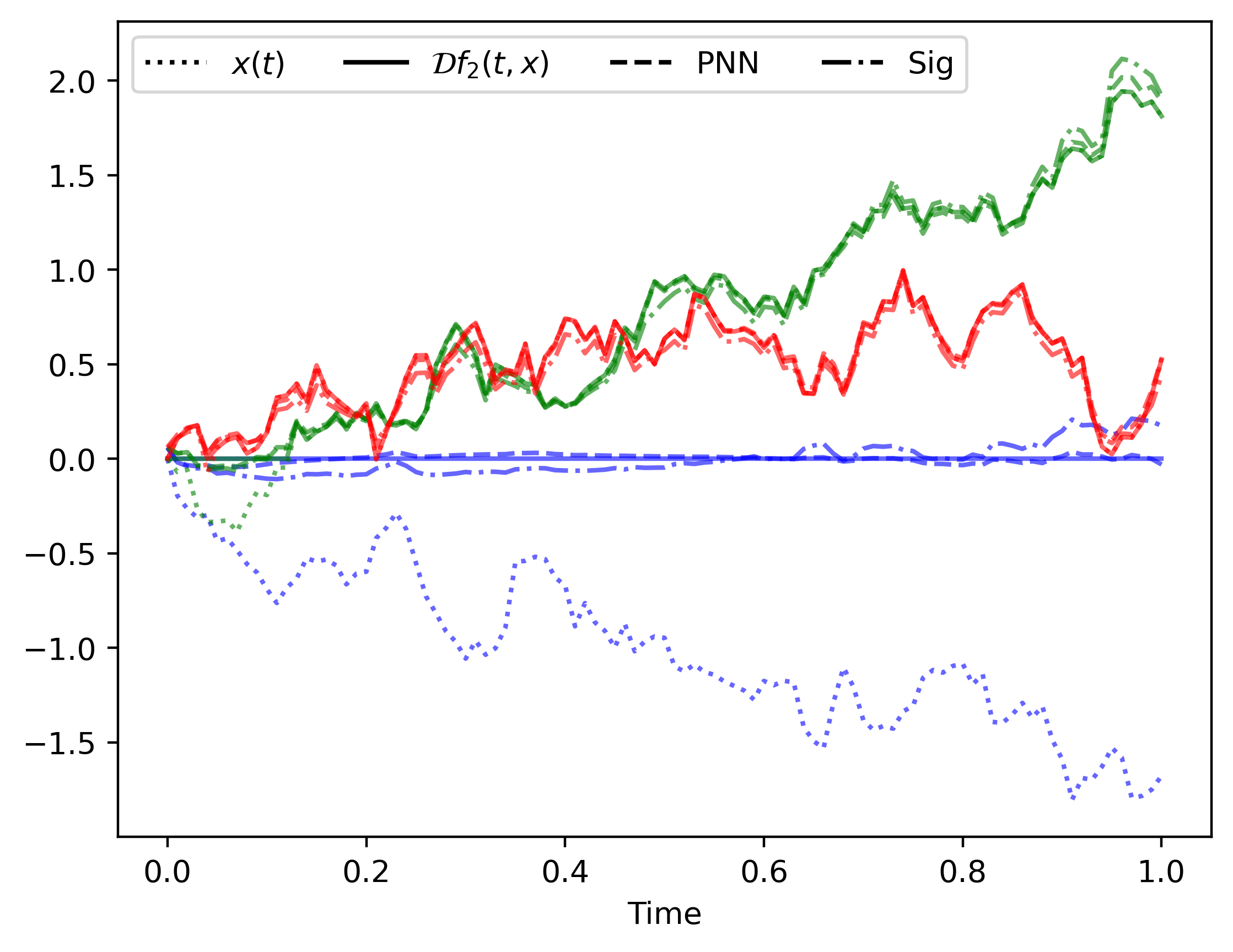}
	}
	\subfigure[$t \mapsto \mathscr{D} f_2(t,x)$ for three samples $x$ of test set]{
		\includegraphics[height=4.8cm]{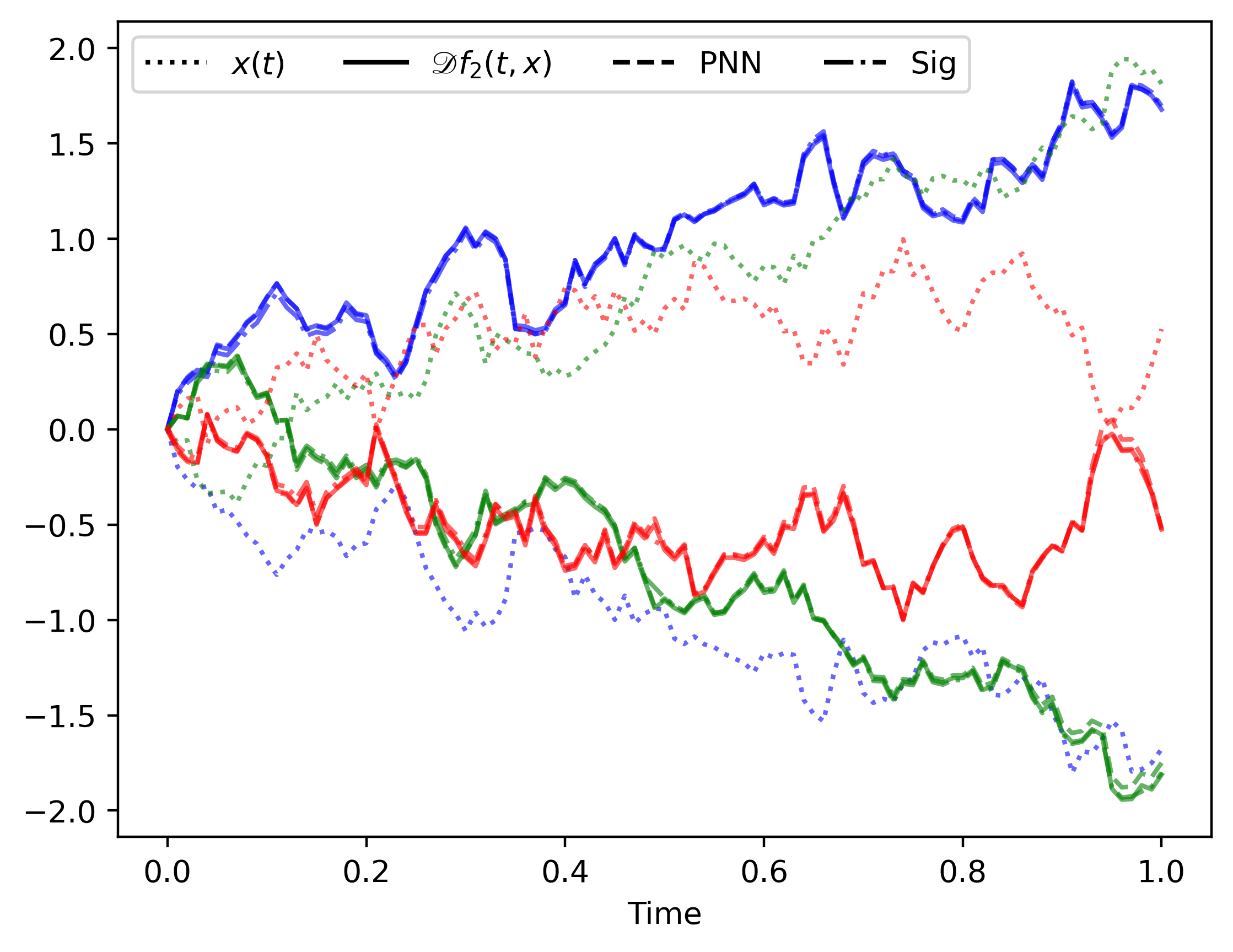}
	}
	\caption{\small Learning $f_2$ defined in \eqref{EqEx2} by path-NN $\varphi \in \mathcal{PN}^{\widetilde{\rho},\rho,\mathcal{L}}_{\Lambda^{\alpha,1}_{T,\mathbb{R}}}$ (label FNN) and linear function of the signature $\sum_{0 \leq \vert I \vert \leq N_{\mathrm{Sig}}} a_I \langle S(\widehat{\mathbf{X}})_t, e_I \rangle$ (label Sig). In (a), the weighted mean squared error \eqref{eq:loss} is evaluated on the training set in each epoch (solid line) as well as on the test after every 200-th epoch (dots). In (b)--(d), three samples $x(m)$ of the test set are shown together with $f_2(\cdot,x(m))$, $\mathcal{D} f_2(\cdot,x(m))$, $\mathscr{D} f_2(\cdot,x(m))$, and their approximations.}
	\label{fig:ex2}
\end{figure}

\pagebreak

\appendix

\section{\texorpdfstring{$\alpha$}{alpha}-H\"older Skorokhod space \texorpdfstring{$D^{\alpha,1}([0,T];Z)$}{Dalphar}}
\label{app:skorokhod}

In this section, we fix some $\alpha \in [0,1)$, $T > 0$, and a dual Banach space $(Z,\Vert \cdot \Vert_Z)$ with predual $(E,\Vert \cdot \Vert_E)$. Then, we first show that the $\alpha$-H\"older Skorokhod space $D^{\alpha,1}([0,T];Z) \subseteq D^0([0,T];Z)$ introduced in Section~\ref{sec:model_sp} is a Banach space, which is isometrically isomorphic to the direct sum of the $\alpha$-H\"older space $(C^\alpha([0,T];Z),\Vert \cdot \Vert_\alpha)$ and the Banach space $(\ell^1((0,T];Z),\Vert \cdot \Vert_{\ell^1})$ consisting of sequences $(z_t)_{t \in (0,T]} \subseteq Z$ with $\Vert z \Vert_{\ell^1} := \sum_{t \in (0,T]} \Vert z_t \Vert_Z < \infty$.

\begin{theorem}
	\label{thm:sk_banach}
	Let $\alpha \in [0,1)$. Then, $(D^{\alpha,1}([0,T];Z),\Vert \cdot \Vert_{\alpha,\ell^1})$ is a Banach space, which is isometrically isomorphic to $C^\alpha([0,T];Z) \oplus_\infty \ell^1((0,T];Z)$.
\end{theorem}
\begin{proof}
	First, we observe that the embedding
	\begin{equation}
		\label{eq:thm:sk_banach:proof1}
		(D^{\alpha,1}([0,T];Z),\Vert \cdot \Vert_{\alpha,\ell^1}) \ni x \quad \mapsto \quad \Xi(x) := (x^c,\Delta x) \in (C^\alpha([0,T];Z) \oplus_\infty \ell^1((0,T];Z),\Vert \cdot \Vert_{\oplus_\infty})
	\end{equation}
	is continuous, where $C^\alpha([0,T];Z) \oplus_\infty \ell^1((0,T];Z)$ is a Banach space under the norm $\Vert (y,z) \Vert_{\oplus_\infty} := \max(\Vert y \Vert_\alpha,\Vert z \Vert_{\ell^1})$, see, e.g., \cite[Section~II.B.20]{wojtaszczyk96}. Moreover, the linear mapping
	\begin{equation}
		(C^\alpha([0,T];Z) \oplus_\infty \ell^1((0,T];Z),\Vert \cdot \Vert_{\oplus_\infty}) \ni (y,z) \mapsto \left( t \mapsto y(t) \!+\!\! \sum_{s \in (0,t]} \!\! z_s \right) \in (D^{\alpha,1}([0,T];Z),\Vert \cdot \Vert_{\alpha,\ell^1})
	\end{equation}
	is well-defined, continuous, and an inverse of $\Xi$, which shows that $\Xi$ is an isometric isomorphism. Hence, $(D^{\alpha,1}([0,T];Z),\Vert \cdot \Vert_{\alpha,\ell^1})$ is a Banach space as an isometrically isomorphic image of the Banach space $C^\alpha([0,T];Z) \oplus_\infty \ell^1((0,T];Z)$. 
\end{proof}

In addition, we use the preduals of the Banach spaces $(C^\alpha([0,T];Z),\Vert \cdot \Vert_\alpha)$ and $(\ell^1((0,T];Z),\Vert \cdot \Vert_{\ell^1})$ to show that $(D^{\alpha,1}([0,T];Z),\Vert \cdot \Vert_{\alpha,\ell^1})$ is a dual Banach space.

\begin{theorem}
	\label{thm:sk_dual}
	Let $\alpha \in (0,1)$. Then, $(D^{\alpha,1}([0,T];Z),\Vert \cdot \Vert_{\alpha,\ell^1})$ is a dual Banach space. Moreover, its weak-$*$-topology $\tau_{w^*}$ coincides on every $\Vert \cdot \Vert_{\alpha,\ell^1}$-bounded subset of $C^\alpha([0,T];Z) \subseteq D^{\alpha,1}([0,T];Z)$ with the $w^*$-uniform topology $\tau_\infty$.
\end{theorem}
\begin{proof}
	By using the linear isomorphism $C^\alpha([0,T];Z) \ni x \mapsto (x(0),x-x(0)) \in Z \oplus_\infty C^\alpha_0([0,T];Z)$ and that $C^\alpha_0([0,T];Z) \cong L(\textrm{\AE}([0,T],d_\alpha);Z)$ (see \cite[Theorem~3.6]{weaver99}, where $\textrm{\AE}([0,T],d_\alpha)$ denotes the Arens-Eells space defined in \cite[Definition~3.2]{weaver99} over the snow-flaked metric space $([0,T],d_\alpha)$, with $d_\alpha(s,t) = \vert s-t \vert^\alpha$), we observe that
	\begin{equation}
		C^\alpha([0,T];Z) \cong Z \oplus_\infty C^\alpha_0([0,T];Z) \cong Z \oplus_\infty L(\textrm{\AE}([0,T],d_\alpha);Z).
	\end{equation}
	Hence, by combining this with Theorem~\ref{thm:sk_banach}, it follows that
	\begin{equation}
		D^{\alpha,1}([0,T];Z) \cong C^\alpha([0,T];Z) \oplus_\infty \ell^1((0,T];Z) \cong Z \oplus_\infty L(\textrm{\AE}([0,T],d_\alpha);Z) \oplus_\infty \ell^1((0,T];Z).
	\end{equation}
	Thus, by using that $L(\textrm{\AE}([0,T],d_\alpha);Z) \cong (\textrm{\AE}([0,T],d_\alpha) \widehat{\otimes}_\pi E)^*$ (see \cite[Theorem~2.9]{ryan02}, where $\widehat{\otimes}_\pi$ denotes the completed projective tensor product) and that $c_0((0,T];E)^* \cong \ell^1((0,T];Z)$ (where $c_0((0,T];E) := \lbrace (e_t)_{t \in (0,T]}: \#\lbrace t \in (0,T]: \Vert e_t \Vert_E \geq \varepsilon \rbrace < \infty \text{ for all } \varepsilon > 0 \rbrace$ is a Banach space under the norm $\Vert e \Vert_\infty := \sup_{t \in (0,T]} \Vert e_t \Vert_E$), we can apply \cite[Section~II.B.21]{wojtaszczyk96} to conclude that
	\begin{equation}
		\label{eq:thm:sk_dual:proof1}
		\begin{aligned}
			D^{\alpha,1}([0,T];Z) & \cong Z \oplus_\infty L(\textrm{\AE}([0,T],d_\alpha);Z) \oplus_\infty \ell^1((0,T];Z) \\
			& \cong \left( E \oplus_1 \left( \textrm{\AE}([0,T],d_\alpha) \widehat{\otimes}_\pi E \right) \oplus_1 c_0((0,T];E) \right)^*
		\end{aligned}
	\end{equation}
	is a dual Banach space, where the predual $G := E \oplus_1 \left( \textrm{\AE}([0,T],d_\alpha) \widehat{\otimes}_\pi E \right) \oplus_1 c_0((0,T];E)$ is equipped with the norm $\Vert (e,T,y) \Vert_{\oplus_1} := \Vert e \Vert_E + \Vert T \Vert_{\textrm{\AE}([0,T],d_\alpha) \widehat{\otimes}_\pi E} + \Vert y \Vert_\infty$.
	
	Finally, we show for every fixed $\Vert \cdot \Vert_{\alpha,\ell^1}$-bounded subset $B \subseteq C^\alpha([0,T];Z) \subseteq D^{\alpha,1}([0,T];Z)$ that $\tau_{w^* \vert B} := \lbrace U \cap B: U \in \tau_{w^*} \rbrace = \lbrace U \cap B: U \in \tau_\infty \rbrace := \tau_{\infty \vert B}$. For $\tau_{w^* \vert B} \subseteq \tau_{\infty \vert B}$, we fix some $x \in B$ and a set $\left\lbrace y \in B: \max_{m=1,\ldots,M} \left\vert \langle x-y, g_m \rangle_{D^{\alpha,1}([0,T];Z) \times G} \right\vert < \delta \right\rbrace$ of the $x$-neighborhood basis of $(B,\tau_{w^* \vert B})$, where $\delta > 0$ and $g_1,\ldots,g_M \in G := F \oplus_1 c_0((0,T];E)$ with $F := E \oplus_1 \left( \textrm{\AE}([0,T],d_\alpha) \widehat{\otimes}_\pi E \right)$. Then, by using the canonical projection $\pi: G \rightarrow F$ and that the weak-$*$-topology of $C^\alpha([0,T];Z)$ coincides with $\tau_\infty$ on the $\Vert \cdot \Vert_\alpha$-bounded subset $B$ (see \cite[Theorem~A.5]{cuchiero23}), there exist $e_1,\ldots,e_N \in E$ and $\varepsilon > 0$ such that
	\begin{equation}
		\begin{aligned}
			\Big\lbrace y \!\in\! B: \max_m \left\vert \langle x-y, g_m \rangle_{D^{\alpha,1}([0,T];Z) \times G} \right\vert < \delta \Big\rbrace & = \Big\lbrace y \!\in\! B: \max_m \left\vert \langle x-y, \pi(g_m) \rangle_{C^\alpha([0,T];Z) \times F} \right\vert < \delta \Big\rbrace \\
			& \subseteq \bigg\lbrace y \!\in\! B: \max_n \sup_{t \in [0,T]} \left\vert \langle (x-y)(t), e_n \rangle_{Z \times E} \right\vert < \varepsilon \bigg\rbrace.
		\end{aligned}
	\end{equation}
	Since the set on the right-hand side belongs to the $x$-neighborhood basis of $(B,\tau_{\infty \vert B})$, we obtain that $\tau_{w^* \vert B} \subseteq \tau_{\infty \vert B}$. Conversely, for $\tau_{w^* \vert B} \supseteq \tau_{\infty \vert B}$, we fix again some $x \in B$ and a set $\big\lbrace y \in B: \max_{m=1,\ldots,M} \sup_{t \in [0,T]} \left\vert \langle (x-y)(t), e_m \rangle_{Z \times E} \right\vert < \delta \big\rbrace$ of the $x$-neighborhood basis of $(B,\tau_{\infty \vert B})$, where $e_1,\ldots,e_M \in E$ and $\delta > 0$. Then, by using the canonical embedding $\iota: F \rightarrow G$ and again that $\tau_\infty$ coincides on the $\Vert \cdot \Vert_\alpha$-bounded subset $B$ with the weak-$*$-topology of $C^\alpha([0,T];Z)$ (see \cite[Theorem~A.5]{cuchiero23}), there exist some $f_1,\ldots,f_N \in F$ and $\varepsilon > 0$ such that
	\begin{equation}
		\begin{aligned}
			\left\lbrace y \!\in\! B: \max_m \sup_{t \in [0,T]} \left\vert \langle (x \!-\! y)(t), e_m \rangle_{Z \times E} \right\vert < \delta \right\rbrace & \subseteq \left\lbrace y \!\in\! B: \max_n \left\vert \langle x \!-\! y, f_n \rangle_{C^\alpha([0,T];Z) \times F} \right\vert < \varepsilon \right\rbrace \\
			& = \left\lbrace y \!\in\! B: \max_n \left\vert \langle x \!-\! y, \iota(f_n) \rangle_{D^{\alpha,1}([0,T];Z) \times G} \right\vert < \varepsilon \right\rbrace.
		\end{aligned}
	\end{equation}
	Since the set on the right-hand side belongs to the $x$-neighborhood basis of $(B,\tau_{w^*})$, we obtain that $\tau_{w^* \vert B} \supseteq \tau_{\infty \vert B}$, which shows that $\tau_{w^* \vert B} = \tau_{\infty \vert B}$.
\end{proof}

\section{BAP of \texorpdfstring{$(C^\alpha(S;Z),\tau_\infty)$}{Calpha} and \texorpdfstring{$(D^{\alpha,1}([0,T];Z),\tau_{w^*})$}{Dalphar}}
\label{app:ap}

In this section, we first show when $(C^\alpha(S;Z),\tau_\infty)$ has the bounded approximation property (BAP), where $\alpha \in [0,1)$, $(S,d_S)$ is a compact metric space, and $(Z,\Vert \cdot \Vert_Z)$ is a dual Banach space equipped with the weak-$*$-topology $\tau_{w^*}$. To this end, we start with the case $\alpha = 0$.

\begin{theorem}
	\label{thm:C0_bap}
	Let $(S,d_S)$ be a compact metric space and let $(Z,\Vert \cdot \Vert_Z)$ be a dual Banach space with predual $(E,\Vert \cdot \Vert_E)$ having BAP. Then, $(C^0(S;Z),\tau_\infty)$ has $\Vert \cdot \Vert_\infty$-BAP.
\end{theorem}
\begin{proof}
	Since $(E,\Vert \cdot \Vert_E)$ has BAP, there exists some $\lambda \geq 1$ and $(Q_\gamma)_\gamma \subseteq E^* \otimes E$ with $\Vert Q_\gamma \Vert_{L(E;E)} \leq \lambda$, for all $\gamma$, such that for every relatively compact subset $\widetilde{K} \subseteq E$ it holds that
	\begin{equation}
		\label{eq:thm:C0_bap:proof1}
		\lim_\gamma \sup_{e \in \widetilde{K}} \Vert e - Q_\gamma(e) \Vert_E = 0.
	\end{equation}
	Now, we fix some $\varepsilon > 0$, a relatively compact subset $K$ of $(C^0(S;Z),\tau_\infty)$, and some $e_1,\ldots,e_N \in E$ defining the seminorms $\big( z \mapsto p_Z(z) := \max_{n=1,\ldots,N} \vert \langle z, e_n \rangle_{Z \times E} \vert \big) \in \mathfrak{P}_{(Z,\tau_{w^*})}$ and $\big( x \mapsto p_{C^0(S;Z)}(x) := \sup_{s \in S} p_Z(x(s)) \big) \in \mathfrak{P}_{(C^0(S;Z),\tau_\infty)}$. Then, by applying the vector-valued Arzel\`a-Ascoli theorem in \cite[Theorem~43.15]{willard04}, the relatively compact set $K$ is equicontinuous (with respect to $p_Z$). Hence, there exists 
	$\delta \in (0,1)$ such that for every $s,t \in S$ and $x \in K$ it holds that
	\begin{equation}
		\label{eq:thm:C0_bap:proof2}
		d_S(s,t) < \delta \quad \Longrightarrow \quad p_Z(x(t)-x(s)) < \frac{\varepsilon}{2}.
	\end{equation}
	For $0 < r < \frac{\delta}{2}$, we now use that $(S,d_S)$ is compact and thus totally bounded to obtain a finite maximal $r$-separated set of points $(t_j)_{j=1,\ldots,J} \subset S$, i.e., $d_S(t_i,t_j) \geq r$, for all $i \neq j$, such that $\bigcup_{j=1}^J B_r(t_j) = S$. In addition, there exists a partition of unity $(g_j)_{j=1,\ldots,J}$ subordinate to $(B_{2r}(t_j))_{j=1,\ldots,J}$, i.e., $\supp(g_j) \subseteq B_{2r}(t_j)$, $0 \leq g_j \leq 1$, and $\sum_{j=1}^J g_j = 1$. Furthermore, since $B := \lbrace x(s): x \in K, \, s \in S \rbrace$ is bounded in $(Z,\tau_{w^*})$, which implies that $C := 1 + \sup_{z \in B} \Vert z \Vert < \infty$ by the uniform boundedness principle, we use \eqref{eq:thm:C0_bap:proof1} to obtain some $\gamma$ satisfying
	\begin{equation}
		\max_{n=1,\ldots,N} \Vert e_n - Q_\gamma e_n \Vert_E < \frac{\varepsilon}{2C}.
	\end{equation}
	Next, we define $\big( x \mapsto T_{\varepsilon,K,e_{1:N},\gamma}(x) := \sum_{j=1}^J g_j(\cdot) Q_\gamma^*(x(t_j)) \big) \in (C^0(S;Z),\tau_\infty)^* \otimes C^0(S;Z)$. Then, by using that $s \in \supp(g_j)$ implies $d_S(s,t_j) < r < \frac{\delta}{2} \leq \delta$ and therefore $p_Z(x(s)-x(t_j)) < \frac{\varepsilon}{2}$ by \eqref{eq:thm:C0_bap:proof2}, and that $p_Z\left( x(t_j) - Q_\gamma^*(x(t_j)) \right) = \max_{n=1,\ldots,N} \vert \langle e_n, x(t_j) - Q_\gamma^*(x(t_j)) \rangle \vert = \max_{n=1,\ldots,N} \vert \langle e_n - Q_\gamma(e_n), x(t_j) \rangle \vert \leq \max_{n=1,\ldots,N} \Vert e_n - Q_\gamma(e_n) \Vert_E \Vert x(t_j) \Vert_Z < \frac{\varepsilon}{2C} C = \frac{\varepsilon}{2}$, we have
	\begin{equation}
		\begin{aligned}
			& \sup_{x \in K} p_{C^0(S;Z)}\left( x-T_{\varepsilon,K,e_{1:N},\gamma}(x) \right) = \sup_{x \in K} \sup_{s \in S} p_Z\left( x(s) - T_{\varepsilon,K,e_{1:N},\gamma}(x)(s) \right) \\
			& = \sup_{x \in K} \sup_{s \in S} p_Z\left( \sum_{j=1}^J g_j(s) \left( x(s) - Q_\gamma^*(x(t_j)) \right) \right) \\
			& \leq \sup_{x \in K} \sup_{s \in S} \sum_{j=1}^J g_j(s) p_Z\left( x(s) - Q_\gamma^*(x(t_j)) \right) \\
			& \leq \sup_{x \in K} \sup_{s \in S} \sum_{j=1}^J g_j(s) p_Z\left( x(s) - x(t_j) \right) + \sup_{x \in K} \sup_{s \in S} \sum_{j=1}^J g_j(s) p_Z\left( x(t_j) - Q_\gamma^*(x(t_j)) \right) \\
			& < \frac{\varepsilon}{2} + \frac{\varepsilon}{2} = \varepsilon.
		\end{aligned}
	\end{equation}
	Hence, the net $(T_{\varepsilon,K,e_{1:N},\gamma})_{\varepsilon,K,e_{1:N},\gamma} \subseteq (C^0(S;Z),\tau_\infty)^* \otimes C^0(S;Z)$ converges to the identity $\id_{C^0(S;Z)}: C^0(S;Z) \rightarrow C^0(S;Z)$, uniformly on each relatively compact subset of $(C^0(S;Z),\tau_\infty)$, showing that $(C^0(S;Z),\tau_\infty)$ has AP. Moreover, for every $x \in C^0(S;Z)$ and $\widetilde{e}_1,\ldots,\widetilde{e}_M \in E$ defining the seminorms $\big( z \mapsto \widetilde{p}_Z(z) := \max_{n=1,\ldots,M} \vert \langle \widetilde{e}_n, z \rangle_{Z \times E} \vert \big) \in \mathfrak{P}_{(Z,\tau_{w^*})}$ and $\big( x \mapsto \widetilde{p}_{C^0(S;Z)}(x) := \sup_{s \in S} \widetilde{p}_Z(x(s)) \big) \in \mathfrak{P}_{(C^0(S;Z),\tau_\infty)}$, we observe that
	\begin{equation}
		\begin{aligned}
			\widetilde{p}_{C^0(S;Z)}\left( T_{\varepsilon,K,e_{1:N},\gamma}(x) \right) & = \sup_{s \in S} \widetilde{p}_Z\left( \sum_{j=1}^J g_j(s) Q_\gamma^*(x(t_j)) \right) \\
			& \leq \sup_{s \in S} \sum_{j=1}^J g_j(s) \max_{n=1,\ldots,M} \left\vert \langle x(t_j), Q_\gamma(\widetilde{e}_n) \rangle_{Z \times E} \right\vert \\
			& \leq \Big( \max_{n=1,\ldots,M} \Vert Q_\gamma \widetilde{e}_n \Vert_E \Big) \Vert x \Vert_\infty \\
			& \leq \Big( \lambda \max_{n=1,\ldots,M} \Vert \widetilde{e}_n \Vert_E \Big) \Vert x \Vert_\infty,
		\end{aligned}
	\end{equation}
	which proves that $(C^0(S;Z),\tau_\infty)$ has $\Vert \cdot \Vert_\infty$-BAP.
\end{proof}

For the BAP of $(C^\alpha(S;Z),\tau_\infty)$ with $\alpha \in (0,1)$, we impose the following condition on $(S,d_S)$ to obtain a specific partition of unity $(g_j)_{j=1,\ldots,J}$. Here, a metric space $(S,d_S)$ is called \emph{doubling} if there exists a doubling constant $M > 0$ such that for every $s \in S$ and $r > 0$ the open ball $B_r(s)$ can be covered with $M$ open balls of radius $r/2$. Moreover, $\Lip(S)$ denotes the vector space of Lipschitz continuous functions $g: S \rightarrow \mathbb{R}$ with $\vert g \vert_1 := \sup_{s,t \in S, \, s \neq t} \frac{\vert g(s) - g(t) \vert}{d_S(s,t)} < \infty$.

\begin{lemma}
	\label{lem:partition_of_unity}
	Let $\alpha \in (0,1)$, let $(S,d_S)$ be a compact doubling metric space. Then, for every $\varepsilon > 0$, relatively compact subset $K$ of $(C^0(S;Z),\tau_\infty)$, and $e_1,\ldots,e_N \in E$, the partition of unity $(g_j)_{j=1,\ldots,J}$ in the proof of Theorem~\ref{thm:C0_bap} can be chosen to satisfy $g_1,\ldots,g_J \in \Lip(S)$ with $\supp(g_j) \subseteq B_{2r}(t_j)$, $\vert g_j \vert_1 \leq 20 C r^{-1}$, and $\#\lbrace j: s \in \supp(g_j) \rbrace \leq C$, where $r := \frac{\delta}{4}$ and $C > 0$ is a universal constant independent of $\varepsilon$, $K$, $e_1,\ldots,e_N$, $\delta$, and $r$.
\end{lemma}
\begin{proof}
	Let $(t_j)_{j=1,\ldots,J} \subset S$ be the finite maximal $r$-separated set of points from the proof of Theorem~\ref{thm:C0_bap}, i.e., $d_S(t_i,t_j) \geq r$, for all $i \neq j$. Then, for every $j = 1,\ldots,J$, we define the function
	\begin{equation}
		S \ni s \quad \mapsto \quad h_j(s) := \max\left( \frac{3}{2} - r^{-1} d_S(s,t_j), 0 \right) \in [0,\tfrac{3}{2}],
	\end{equation} 
	which satisfies $\supp(h_j) \subseteq B_{2r}(t_j)$ and $\vert h_j \vert_1 \leq r^{-1}$, for all $j = 1,\ldots,J$. Since $\bigcup_{j=1}^J B_r(t_j) = S$, there exists for every $s \in S$ some $j = 1,\ldots,J$ with $d_S(s,t_j) < r$, which implies that $h_j(s) \geq \frac{1}{2}$. Hence, by using the function $S \ni s \mapsto H(s) := \sum_{j=1}^J h_j(s) \in [0,\infty)$, which satisfies $H(s) \geq \frac{1}{2}$, for all $s \in S$, we can define for every $j = 1,\ldots,J$ the function
	\begin{equation}
		S \ni s \quad \mapsto \quad g_j(s) := \frac{h_j(s)}{H(s)} \in [0,\infty),
	\end{equation}
	which satisfies $\supp(g_j) \subseteq \supp(h_j) \subseteq B_{2r}(t_j)$, $0 \leq g_j \leq 1$, and $\sum_{j=1}^J g_j = 1$. Thus, by using that $(t_j)_{j=1,\ldots,J} \subset S$ are $r$-separated, there exists a constant $C > 0$ (depending only on the doubling constant $M > 0$) such that every ball of radius $5r/2$ contains at most $C$ of the points $t_j$, which implies that $\#\lbrace j: s \in \supp(h_j) \rbrace = \#\lbrace j: d_S(s,t_j) < 2r \rbrace \leq C$ and therefore $\frac{1}{2} \leq H(s) \leq \frac{3C}{2}$. Thus, if $d_S(s,t) < r$, we use that $\vert H(s) - H(t) \vert \leq \sum_{j=1, \, h_j(s) \neq h_j(t)} \vert h_j(s)-h_j(t) \vert \leq C r^{-1} d_S(s,t)$ with $h_j(s) \neq h_j(t)$ implying $h_j(s) \neq 0$ or $h_j(t) \neq 0$ and therefore $t_j \in B_{5r/2}(s)$, and if $d_S(s,t) \geq r$, we insert that $\vert H(s) - H(t) \vert \leq \vert H(s) \vert + \vert H(t) \vert \leq 2 \frac{3}{2} C \leq 3 C r^{-1} d_S(s,t)$ to conclude for every $s,t \in S$ that
	\begin{equation}
		\begin{aligned}
			\vert g_j(s) - g_j(t) \vert & = \left\vert \frac{h_j(s)}{H(s)} - \frac{h_j(t)}{H(t)} \right\vert \leq \frac{\vert h_j(s) - h_j(t) \vert}{H(s)} + \vert h_j(t) \vert \frac{\vert H(s) - H(t) \vert}{H(s) H(t)} \\
			& \leq 2 \vert h_j(s) - h_j(t) \vert + \frac{3}{2} 4 \vert H(s) - H(t) \vert \\
			& \leq 2 r^{-1} d_S(s,t) + 18 C r^{-1} d_S(s,t) \\
			& \leq 20 C r^{-1} d_S(s,t),
		\end{aligned}
	\end{equation}
	which proves that $\vert g_j \vert_1 \leq 20 C r^{-1}$.
\end{proof}

\begin{corollary}
	\label{cor:hol_bap}
	Let $\alpha \in (0,1)$, let $(S,d_S)$ be a compact doubling metric space with designated origin $0 \in S$, and let $(Z,\Vert \cdot \Vert_Z)$ be a dual Banach space with predual $(E,\Vert \cdot \Vert_E)$ having BAP. Then, $(C^\alpha(S;Z),\tau_\infty)$ has AP with finite rank operators $(T_\vartheta)_\vartheta \subseteq (C^\alpha(S;Z),\tau_\infty)^* \otimes C^\alpha(S;Z)$ and for every $\widetilde{e}_1,\ldots,\widetilde{e}_M \in E$ there exists some $\Lambda > 0$ such that for every $\vartheta$ and $x \in C^\alpha(S;Z)$ it holds that
	\begin{equation}
		\label{eq:cor:hol_bap}
		\Vert T_\vartheta(x) \Vert_{\alpha,\widetilde{e}_{1:M}} \leq \Lambda \Vert x \Vert_\alpha,
	\end{equation}
	where $\Vert x \Vert_{\alpha,\widetilde{e}_{1:M}} := \max_{m=1,\ldots,M} \Vert x \Vert_{\alpha,\widetilde{e}_m}$.
\end{corollary}
\begin{proof}
	By using the embedding $C^\alpha(S;Z) \hookrightarrow C^0(S;Z)$, we adapt the proof of Theorem~\ref{thm:C0_bap} with finite rank operators $\big( x \mapsto T_{\varepsilon,K,e_{1:N},\gamma}(x) := \sum_{j=1}^J g_j(\cdot) Q_\gamma^*(x(t_j)) \big) \in (C^\alpha(S;Z),\tau_\infty)^* \otimes C^\alpha(S;Z)$, where $(Q_\gamma)_\gamma \subseteq E^* \otimes E$ with $\Vert Q_\gamma \Vert_{L(E;E)} \leq \lambda$. Note that, by Lemma~\ref{lem:partition_of_unity}, we may choose the partition of unity $(g_j)_{j=1,\ldots,J}$ to satisfy $\supp(g_j) \subseteq B_{2r}(t_j)$, $\vert g_j \vert_1 \leq 20 C r^{-1}$, and $\#\lbrace j: s \in \supp(g_j) \rbrace \leq C$, where $r = \frac{\delta}{4}$. Hence, by the proof of Theorem~\ref{thm:C0_bap}, it follows for every $\varepsilon > 0$, $e_1,\ldots,e_N \in E$ defining the seminorm $\big( x \mapsto p_{C^\alpha(S;Z)}(x) := \sup_{s \in S} p_Z(x(s)) \big) \in \mathfrak{P}_{(C^\alpha(S;Z),\tau_\infty)}$, and relatively compact subset $K$ of $(C^\alpha(S;Z),\tau_\infty)$ that there exists some $\gamma$ such that
	\begin{equation}
		\sup_{x \in K} p_{C^\alpha(S;Z)}\left( x - T_{\varepsilon,K,e_{1:N},\gamma}(x) \right) < \varepsilon,
	\end{equation}
	which shows that $(C^\alpha(S;Z),\tau_\infty)$ has AP. Moreover, for every fixed $x \in C^\alpha(S;Z)$ and $\widetilde{e}_1,\ldots,\widetilde{e}_M \in E$, we use that $0 \in \supp(g_j) \subseteq B_{2r}(t_j)$ implies $d_S(t_j,0) < 2r \leq 1$ to obtain that
	\begin{equation}
		\label{eq:cor:hol_bap:proof1}
		\begin{aligned}
			& \left\vert \langle T_{\varepsilon,K,e_{1:N},\gamma}(x)(0), \widetilde{e}_m \rangle_{Z \times E} \right\vert \leq \left\vert \sum_{j=1}^J g_j(0) \langle Q_\gamma^*(x(t_j)), \widetilde{e}_m \rangle_{Z \times E} \right\vert \\
			& \leq \vert \langle x(0), Q_\gamma(\widetilde{e}_m) \rangle_{Z \times E} \vert + \left\vert \sum_{j=1}^J g_j(0) \langle x(t_j)-x(0), Q_\gamma(\widetilde{e}_m) \rangle_{Z \times E} \right\vert  \\
			& \leq \Vert Q_\gamma \Vert_{L(E;E)} \Vert \widetilde{e}_m \Vert_E \Vert x(0) \Vert_Z + \Vert Q_\gamma \Vert_{L(E;E)} \Vert \widetilde{e}_m \Vert_E \vert x \vert_\alpha \sum_{j=1 \atop g_j(0) \neq 0}^J g_j(0) d_S(t_j,0)^\alpha  \\
			& \leq \lambda \Vert \widetilde{e}_m \Vert_E \left( \Vert x(0) \Vert_Z + \vert x \vert_\alpha \right).
		\end{aligned}
	\end{equation}
	In addition, for $s,t \in S$ with $d_S(s,t) < r$, we add and subtract $\langle x(s), Q_\gamma(\widetilde{e}_m) \rangle_{Z \times E}$, use that $g_j(s) \neq g_j(t)$ implies $s \in \supp(g_j) \subseteq B_{2r}(t_j)$ or $t \in \supp(g_j) \subseteq B_{2r}(t_j)$, ensuring that $d_S(t_j,s) \leq 2 r$ (if $s \in \supp(g_j)$) or $d_S(t_j,s) \leq d_S(t_j,t) + d_S(t,s) < 2 r + r = 3r$ (if $t \in \supp(g_j)$), and that $\sum_{j=1, \, g_j(s) \neq g_j(t)}^J d_S(t_j,s)^\alpha \vert g_j(s) - g_j(t) \vert \leq C \cdot (3r)^\alpha \vert g_j \vert_1 d_S(s,t) \leq 40 \cdot 3^\alpha r^{\alpha-1} C^2 d_S(s,t)^\alpha \leq 120 C^2 d_S(s,t)^\alpha$ to deduce that
	\begin{equation}
		\label{eq:cor:hol_bap:proof2}
		\begin{aligned}
			& \left\vert \langle T_{\varepsilon,K,e_{1:N},\gamma}(x)(s) - T_{\varepsilon,K,e_{1:N},\gamma}(x)(t), \widetilde{e}_m \rangle_{Z \times E} \right\vert \leq \left\vert \sum_{j=1}^J (g_j(s) - g_j(t)) \langle Q_\gamma^*(x(t_j)), \widetilde{e}_m \rangle_{Z \times E} \right\vert \\
			& \leq \sum_{j=1 \atop g_j(s) \neq g_j(t)}^J \vert g_j(s) - g_j(t) \vert \vert \langle x(t_j)-x(s), Q_\gamma(\widetilde{e}_m) \rangle_{Z \times E} \vert \\
			& \leq \Vert Q_\gamma \Vert_{L(E;E)} \Vert \widetilde{e}_m \Vert_E \vert x \vert_\alpha \sum_{j=1 \atop g_j(s) \neq g_j(t)}^J d_S(t_j,s)^\alpha \vert g_j(s) - g_j(t) \vert \\
			& \leq 120 C^2 \lambda \Vert \widetilde{e}_m \Vert_E \vert x \vert_\alpha d_S(s,t)^\alpha.
		\end{aligned}	
	\end{equation}
	Furthermore, for $s,t \in S$ with $d_S(s,t) \geq r$, we use that $s \in \supp(g_j) \subseteq B_{2r}(t_j)$ implies $d_S(t_j,s)^\alpha \leq (2r)^\alpha \leq 2 d_S(s,t)^\alpha$ to conclude that
	\begin{equation}
		\label{eq:cor:hol_bap:proof3}
		\begin{aligned}
			& \left\vert \langle T_{\varepsilon,K,e_{1:N},\gamma}(x)(s) - T_{\varepsilon,K,e_{1:N},\gamma}(x)(t), \widetilde{e}_m \rangle_{Z \times E} \right\vert \leq \left\vert \sum_{j=1}^J (g_j(s) - g_j(t)) \langle Q_\gamma^*(x(t_j)), \widetilde{e}_m \rangle_{Z \times E} \right\vert \\
			& \leq \left\vert \langle x(s) - x(t), Q_\gamma(\widetilde{e}_m) \rangle_{Z \times E} \right\vert + \left\vert \sum_{j=1}^J g_j(s) \langle x(t_j)-x(s), Q_\gamma(\widetilde{e}_m) \rangle_{Z \times E} \right\vert \\
			& \quad\quad + \left\vert \sum_{j=1}^J g_j(t) \langle x(t_j)-x(t), Q_\gamma(\widetilde{e}_m) \rangle_{Z \times E} \right\vert \\
			& \leq \Vert Q_\gamma \Vert_{L(E;E)} \Vert \widetilde{e}_m \Vert_E \left( \Vert x(s) - x(t) \Vert_Z + \vert x \vert_\alpha \left( \sum_{j=1}^J g_j(s) d_S(t_j,s)^\alpha + \sum_{j=1}^J g_j(t) d_S(t_j,t)^\alpha \right) \right) \\
			& \leq \lambda \Vert \widetilde{e}_m \Vert_E \left( \vert x \vert_\alpha d_S(s,t)^\alpha + 2 \vert x \vert_\alpha d_S(s,t)^\alpha + 2  \vert x \vert_\alpha d_S(s,t)^\alpha \right) \\
			& \leq 5 \lambda \Vert \widetilde{e}_m \Vert_E \vert x \vert_\alpha d_S(s,t)^\alpha.
		\end{aligned}		
	\end{equation}
	Thus, by using \eqref{eq:cor:hol_bap:proof1}, \eqref{eq:cor:hol_bap:proof2}, and \eqref{eq:cor:hol_bap:proof3}, it follows that
	\begin{equation}
		\begin{aligned}
			& \Vert T_{\varepsilon,K,e_{1:N},\gamma}(x) \Vert_{\alpha,\widetilde{e}_{1:M}} \\
			& = \max_{m=1,\ldots,M} \left( \left\vert \langle T_{\varepsilon,K,e_{1:N},\gamma}(x)(0), \widetilde{e}_m \rangle_{Z \times E} \right\vert + \sup_{s,t \in S \atop s \neq t} \frac{\left\vert \langle T_{\varepsilon,K,e_{1:N},\gamma}(x)(s) - T_{\varepsilon,K,e_{1:N},\gamma}(x)(t), \widetilde{e}_m \rangle_{Z \times E} \right\vert}{d_S(s,t)^\alpha} \right) \\
			& \leq \left( 121 C^2 \lambda \max_{m=1,\ldots,M} \Vert \widetilde{e}_m \Vert_E \right) \Vert x \Vert_\alpha,
		\end{aligned}
	\end{equation}
	which proves \eqref{eq:cor:hol_bap}.
\end{proof}

\begin{remark}
	\label{rem:hol_bap0}
	Let $\alpha \in (0,1)$, let $(S,d_S)$ be a compact doubling metric space with designated origin $0 \in S$, and let $(Z,\Vert \cdot \Vert_Z)$ be a dual Banach space with predual $(E,\Vert \cdot \Vert_E)$ having BAP. Then, for the vector space $C^\alpha_0(S;Z)$ of $\alpha$-H\"older continuous functions $x: S \rightarrow Z$ preserving the origin, i.e., $x(0) = 0$, it is possible to choose the partition of unity $(g_j)_{j=1,\ldots,J}$ to satisfy $g_j(0)=0$, which ensures that $(C^\alpha_0(S;Z),\tau_\infty)$ has AP with \eqref{eq:cor:hol_bap}.
\end{remark}

Moreover, we follow the proof of Lemma~\ref{lem:ap_dual} to give conditions when $(D^{\alpha,1}([0,T];Z),\tau_{w^*})$ has the bounded approximation property (BAP), where $\alpha \in (0,1)$, $T > 0$, and $(Z,\Vert \cdot \Vert_Z)$ is a dual Banach space equipped with the weak-$*$-topology $\tau_{w^*}$.

\begin{theorem}
	\label{thm:D0_bap}
	Let $\alpha \in (0,1)$ and assume that the predual $(E,\Vert \cdot \Vert_E)$ of $(Z,\Vert \cdot \Vert_Z)$ has BAP. Then, $(D^{\alpha,1}([0,T];Z),\tau_{w^*})$ has $\Vert \cdot \Vert_{\alpha,\ell^1}$-BAP.
\end{theorem}
\begin{proof}
	First, we recall from \eqref{eq:thm:sk_dual:proof1} that a predual of $D^{\alpha,1}([0,T];Z)$ is given by
	\begin{equation}
		G := E \oplus_1 \left( \textrm{\AE}([0,T],d_\alpha) \widehat{\otimes}_\pi E \right) \oplus_1 c_0((0,T];E).
	\end{equation}
	Now, we observe that $(E,\Vert \cdot \Vert_E)$ has BAP by assumption. Moreover, $(\textrm{\AE}([0,T],d_\alpha),\Vert \cdot \Vert_{\textrm{\AE}([0,T],d_\alpha)})$ has BAP by \cite[Corollary~2.2]{lancien13}, which is preserved under the completed projective tensor product $\widehat{\otimes}_\pi$ (see \cite[Section~4.1]{ryan02}). In addition, $(c_0((0,T];E),\Vert \cdot \Vert_\infty)$ has BAP with finite rank operators $\big( y \mapsto T_{F,\gamma}(y) := (\mathds{1}_F(s) Q_\gamma(y_s))_{s \in (0,T]} \big) \in c_0((0,T];E)^* \otimes c_0((0,T];E)$, where $F \subseteq (0,T]$ is finite and $(Q_\gamma)_\gamma$ is the BAP-net of $(E,\Vert \cdot \Vert_E)$. Hence, by using that finite $\ell^1$-sums preserve norm-BAP, the predual $(G,\Vert \cdot \Vert_{\oplus_1})$ has BAP. Finally, we can use adjoints as in the proof of Lemma~\ref{lem:ap_dual} to conclude that $(D^{\alpha,1}([0,T];Z),\tau_{w^*})$ has $\Vert \cdot \Vert_{\alpha,\ell^1}$-BAP.
\end{proof}

\section{Proof of results in Section~\ref{sec:w_dom}}

\subsection{Proof of Proposition~\ref{prop:BPsik_equiv}}
\label{app:BPsik_equiv}

\begin{proof}[Proof of Proposition~\ref{prop:BPsik_equiv}]
	For \ref{prop:BPsik_equiv:1}, let $f \in \mathcal{B}^k_\Psi(U;Y)$ and fix some $\varepsilon > 0$. Then, by definition of $\mathcal{B}^k_\Psi(U;Y)$, there exists some $g \in C^k_b(U;Y)$ with $\Vert d^j g(u;v_1,\ldots,v_j) \Vert_Y \leq C_g p_X(v_1) \cdots p_X(v_j)$ for all $j = 0,\ldots,k$, $u \in U$, $v_1,\ldots,v_j \in X$, and some $C_g \geq 0$ and $p_X \in \mathfrak{P}_{(X,\tau_X)}$, such that
	\begin{equation}
		\label{eq:prop:BPsik_equiv:proof1}
		\Vert f-g \Vert_{\mathcal{B}^k_\Psi(U;Y)} = \max_{j=0,\ldots,k} \sup_{(u,v_1,\ldots,v_j) \in U \times X^j} \frac{\Vert d^j f(u;v_1,\ldots,v_j) - d^j g(u;v_1,\ldots,v_j) \Vert_Y}{\psi_j(u,v_1,\ldots,v_j)} < \frac{\varepsilon}{2}.
	\end{equation}
	Moreover, by using \eqref{eq:ass:w_growth}, there exists some $R > 0$ such that
	\begin{equation}
		\label{eq:prop:BPsik_equiv:proof2}
		\max_{j=0,\ldots,k} \sup_{(u,v_1,\ldots,v_j) \in (U \times X^j) \setminus K_{j,R}} \frac{p_X(v_1) \cdots p_X(v_j)}{\psi_j(u,v_1,\ldots,v_j)} < \frac{\varepsilon}{2 \max(C_g,1)}.
	\end{equation}
	Hence, by using the inequalities \eqref{eq:prop:BPsik_equiv:proof1}, \eqref{eq:ass:w_growth}, and \eqref{eq:prop:BPsik_equiv:proof2}, it follows that
	\begin{equation}
		\begin{aligned}
			& \max_{j=0,\ldots,k} \sup_{(u,v_1,\ldots,v_j) \in (U \times X^j) \setminus K_{j,R}} \frac{\Vert d^j f(u;v_1,\ldots,v_j) \Vert_Y}{\psi_j(u,v_1,\ldots,v_j)} \\
			& \leq \max_{j=0,\ldots,k} \sup_{(U \times X^j) \setminus K_{j,R}} \frac{\Vert d^j f(u;v_1,\ldots,v_j) - d^j g(u;v_1,\ldots,v_j) \Vert_Y}{\psi_j(u,v_1,\ldots,v_j)} \\
			& \quad\quad + \max_{j=0,\ldots,k} \sup_{(u,v_1,\ldots,v_j) \in (U \times X^j) \setminus K_{j,R}} \frac{\Vert d^j g(u;v_1,\ldots,v_j) \Vert_Y}{\psi_j(u,v_1,\ldots,v_j)} \\
			& \leq \Vert f-g \Vert_{\mathcal{B}^k_\Psi(U;Y)} + C_g \max_{j=0,\ldots,k} \sup_{(u,v_1,\ldots,v_j) \in (U \times X^j) \setminus K_{j,R}} \frac{p_X(v_1) \cdots p_X(v_j)}{\psi_j(u,v_1,\ldots,v_j)} \\
			& < \frac{\varepsilon}{2} + C_g \frac{\varepsilon}{2 \max(C_g,1)} \leq \varepsilon.
		\end{aligned}
	\end{equation} 
	Since $\varepsilon > 0$ was chosen arbitrarily small, we obtain \eqref{eq:prop:BPsik_equiv:1}. Moreover, with $g \in C^k_b(U;Y)$ as above, we observe that for every fixed $j = 0,\ldots,k$ and $R > 0$ it holds that
	\begin{equation}
		\begin{aligned}
			& \sup_{(u,v_1,\ldots,v_j) \in K_{j,R}} \Vert d^j f(u;v_1,\ldots,v_j) - d^j g(u;v_1,\ldots,v_j) \Vert_Y \\
			& \leq R \sup_{(u,v_1,\ldots,v_j) \in K_{j,R}} \frac{\Vert d^j f(u;v_1,\ldots,v_j) - d^j g(u;v_1,\ldots,v_j) \Vert_Y}{\psi_j(u,v_1,\ldots,v_j)} \\
			& \leq R \Vert f-g \Vert_{\mathcal{B}^k_\Psi(U;Y)} < R \frac{\varepsilon}{2}.
		\end{aligned}
	\end{equation}
	This implies that $K_{j,R} \ni (u,v_1,\ldots,v_j) \mapsto d^j f(u;v_1,\ldots,v_j) \in Y$ is continuous as uniform limit of the continuous mappings $K_{j,R} \ni (u,v_1,\ldots,v_j) \mapsto d^j g(u;v_1,\ldots,v_j) \in Y$, showing $f \in C^k_{loc}(U;Y)$.
	
	For \ref{prop:BPsik_equiv:2}, we first consider the case when $(X,\tau_X)$ is not locally compact and $(U,\tau_X)$ has AP with net of finite rank operators $(T_\gamma)_\gamma \subseteq X^* \otimes X$ satisfying $T_\gamma(U) \subseteq U$ and \eqref{eq:prop:BPsik_equiv:3}. Moreover, we define the constant $C_{\inf} := \left( \min_{j=0,\ldots,k} \inf_{(u,v_1,\ldots,v_j) \in U \times X^j} \psi_j(u,v_1,\ldots,v_j) \right)^{-1} > 0$ and fix some $\varepsilon > 0$. Then, by using \eqref{eq:prop:BPsik_equiv:2}--\eqref{eq:prop:BPsik_equiv:3}, there exists some $R > 0$ such that
	\begin{align}
		\label{eq:prop:BPsik_equiv:proof3a}
		\max_{j=0,\ldots,k} \sup_{(u,v_1,\ldots,v_j) \in (U \times X^j) \setminus K_{j,R}} \frac{\Vert d^j f(u;v_1,\ldots,v_j) \Vert_Y}{\psi_j(u,v_1,\ldots,v_j)} & < \frac{\varepsilon}{3}, \\
		\label{eq:prop:BPsik_equiv:proof3b}
		\sup_\gamma \max_{j=0,\ldots,k \atop \mathcal{L} \subseteq \lbrace 1,\ldots,j \rbrace} \sup_{(u,v_1,\ldots,v_j) \in (U \times X^j) \setminus K_{j,R}} \frac{\Vert d^{\vert\mathcal{L}\vert} f(T_\gamma(u);(T_\gamma(v_\ell))_{\ell \in \mathcal{L}}) \Vert_Y}{\psi_\mathcal{L}(u,v_\mathcal{L})} & < \frac{\varepsilon}{3 \cdot 2^k C_\Psi}.
	\end{align}
	Now, we use that $K_{j,R} \ni (u,v_1,\ldots,v_j) \mapsto d^j f(u;v_1,\ldots,v_j) \in Y$ is continuous, thus uniformly continuous on the compact set $K_{j,R}$, to conclude that there exists an open $0$-neighborhood $V_{j,0} \times V_{j,1} \times \ldots \times V_{j,j}$ of $(X \times X^j,\tau_X \times \tau_X^j)$, with $V_{j,0}, \ldots, V_{j,j} \in \tau_X$, such that for every $(u,v_1,\ldots,v_j) \in K_{j,R}$ and $(\widetilde{u},\widetilde{v}_1,\ldots,\widetilde{v}_j) \in U \times X^j$ with
	\begin{equation}
		\label{eq:prop:BPsik_equiv:proof4}
		(u,v_1,\ldots,v_j) - (\widetilde{u},\widetilde{v}_1,\ldots,\widetilde{v}_j) \,\, \in \,\, V_{j,0} \times V_{j,1} \times \ldots \times V_{j,j},
	\end{equation}
	it holds that
	\begin{equation}
		\label{eq:prop:BPsik_equiv:proof5}
		\Vert d^j f(u;v_1,\ldots,v_j) - d^j f(\widetilde{u};\widetilde{v}_1,\ldots,\widetilde{v}_j) \Vert_Y < \frac{\varepsilon}{3C_{\inf}}.
	\end{equation}
	From this, we define the set $K := \bigcup_{j=0}^k \bigcup_{\ell=0}^j \pi_\ell(K_{j,R})$ that is compact as a finite union of compact images under the continuous projection $U \times X^j \ni (v_0,v_1,\ldots,v_j) \mapsto \pi_\ell(v_0,v_1,\ldots,v_j) := v_\ell \in X$. Then, by using that $(U,\tau_X)$ has AP, there exists some $T_\gamma \in X^* \otimes X$ such that for every $x \in K$ we have $x-T_\gamma(x) \in \bigcap_{j=0}^k \bigcap_{\ell=0}^j V_{j,\ell}$. Hence, for every $(u,v_1,\ldots,v_j) \in K_{j,R}$, it holds that
	\begin{equation}
		(u,v_1,\ldots,v_j) - (T_\gamma(u),T_\gamma(v_1),\ldots,T_\gamma(v_j)) \,\, \in \,\, V_{j,0} \times V_{j,1} \times \ldots \times V_{j,j}.
	\end{equation}
	Thus, by combining this with \eqref{eq:prop:BPsik_equiv:proof4} as well as using the chain rule and \eqref{eq:prop:BPsik_equiv:proof5}, it follows for every $j = 0,\ldots,k$ and $(u,v_1,\ldots,v_j) \in K_{j,R}$ that
	\begin{equation}
		\label{eq:prop:BPsik_equiv:proof6}
		\begin{aligned}
			& \Vert d^j f(u;v_1,\ldots,v_j) - d^j (f \circ T_\gamma)(u;v_1,\ldots,v_j) \Vert_Y \\
			& = \Vert d^j f(u;v_1,\ldots,v_j) - d^j f(T_\gamma(u);T_\gamma(v_1),\ldots,T_\gamma(v_j)) \Vert_Y < \frac{\varepsilon}{3C_{\inf}}.
		\end{aligned}
	\end{equation}
	Now, for every $j = 0,\ldots,k$ and $R > 0$, we observe that $\prod_{\ell=0}^j T_\gamma(\pi_\ell(K_{j,R})) \ni (u,v_1,\ldots,v_j) \mapsto d^j f(u;v_1,\ldots,v_j) \in Y$ is continuous as restriction of the continuous map $K_{j,R} \ni (u,v_1,\ldots,v_j) \mapsto d^j f(u;v_1,\ldots,v_j) \in Y$, where $\prod_{\ell=0}^j T_\gamma(\pi_\ell(K_{j,R}))$ is compact in $T_\gamma(U) \times T_\gamma(X)^j$. Hence, by using that the finite-dimensional spaces $T_\gamma(U)$ and $T_\gamma(X)$ are locally compact, it follows from \cite[Lemma~46.3+46.4]{munkres14} that $T_\gamma(U) \times T_\gamma(X)^j \ni (u,v_1,\ldots,v_j) \mapsto d^j f(u;v_1,\ldots,v_j) \in Y$ is globally continuous. Since $T_\gamma \in X^* \otimes X$ is also continuous, we conclude that $U \times X^j \ni (u,v_1,\ldots,v_j) \mapsto d^j (f \circ T_\gamma)(u;v_1,\ldots,v_j) = d^j f(T_\gamma(u);T_\gamma(v_1),\ldots,T_\gamma(v_j)) \in Y$ is continuous, which shows that $f \circ T_\gamma \in C^k(U;Y)$. Moreover, by using again that $T_\gamma(U)$ is locally compact, there exists some $h \in C^\infty_c(T_\gamma(X))$ such that
	\begin{equation}
		\label{eq:prop:BPsik_equiv:proof7}
		\begin{cases}
			h(u) = 1, & \text{for all } u \in \bigcup_{j=0}^k T_\gamma(\pi_0(K_{j,R})), \\[4pt]
			0 \leq h(u) \leq 1, & \text{for all } u \in T_\gamma(U), \\[6pt]
			\left\vert d^j (h \circ T_\gamma)(u;v_1,\ldots,v_j) \right\vert \leq \psi_j(u,v_1,\ldots,v_j), & 
			\begin{matrix*}[l]
				\text{for all } j = 1,\ldots,k \text{ and} \\
				(u,v_1,\ldots,v_j) \in U \times X^j,
			\end{matrix*}
		\end{cases}
	\end{equation}
	From this, we define the map $U \ni u \mapsto g(u) := (h \circ T_\gamma)(u) (f \circ T_\gamma)(u) \in Y$. Then, by using the Leibniz product rule, the monotonicity of $\Psi := (\psi_j)_{j=0,\ldots,k}$, the last property of \eqref{eq:prop:BPsik_equiv:proof7}, the chain rule, and \eqref{eq:prop:BPsik_equiv:proof3b}, it holds for every $j = 0,\ldots,k$ and $(u,v_1,\ldots,v_j) \in (U \times X^j) \setminus K_{j,R}$ that
	\begin{equation}
		\label{eq:prop:BPsik_equiv:proof8}
		\begin{aligned}
			\frac{\Vert d^j g(u;v_1,\ldots,v_j) \Vert_Y}{\psi_j(u,v_1,\ldots,v_j)} & = \frac{\big\Vert \sum_{\mathcal{L} \subseteq \lbrace 1,\ldots,j \rbrace} d^{j-\vert\mathcal{L}\vert} (h \circ T_\gamma)(u;v_{\mathcal{L}^c}) d^{\vert\mathcal{L}\vert} (f \circ T_\gamma)(u;v_\mathcal{L}) \big\Vert_Y}{\psi_j(u,v_1,\ldots,v_j)} \\
			& \leq C_\Psi \frac{\sum_{\mathcal{L} \subseteq \lbrace 1,\ldots,j \rbrace} \vert d^{j-\vert\mathcal{L}\vert} (h \circ T_\gamma)(u;v_{\mathcal{L}^c}) \vert \Vert d^{\vert\mathcal{L}\vert} (f \circ T_\gamma)(u;v_\mathcal{L}) \Vert_Y}{\psi_{j-\vert\mathcal{L}\vert}(u,v_{\mathcal{L}^c}) \psi_\mathcal{L}(u,v_\mathcal{L})} \\
			& \leq 2^k C_\Psi \max_{\mathcal{L} \subseteq \lbrace 1,\ldots,j \rbrace} \frac{\Vert d^{\vert\mathcal{L}\vert} (f \circ T_\gamma)(u;v_\mathcal{L}) \Vert_Y}{\psi_\mathcal{L}(u,v_\mathcal{L})} \\
			& \leq 2^k C_\Psi \max_{\mathcal{L} \subseteq \lbrace 1,\ldots,j \rbrace} \frac{\Vert d^{\vert\mathcal{L}\vert} f(T_\gamma(u);(T_\gamma(v_\ell))_{\ell \in \mathcal{L}}) \Vert_Y}{\psi_\mathcal{L}(u,v_\mathcal{L})} \\
			& < 2^k C_\Psi \frac{\varepsilon}{3 \cdot 2^k C_\Psi} = \frac{\varepsilon}{3}.
		\end{aligned}
	\end{equation}
	Hence, by using \eqref{eq:prop:BPsik_equiv:proof3a}, \eqref{eq:prop:BPsik_equiv:proof6}, and \eqref{eq:prop:BPsik_equiv:proof8}, we conclude for $g \in C^k_b(U;Y)$ that
	\begin{equation}
		\begin{aligned}
			& \Vert f - g \Vert_{\mathcal{B}^k_\Psi(U;Y)} = \max_{j=0,\ldots,k} \sup_{(u,v_1,\ldots,v_j) \in U \times X^j} \frac{\left\Vert d^j f(u;v_1,\ldots,v_j) - d^j g(u;v_1,\ldots,v_j) \right\Vert_Y}{\psi_j(u,v_1,\ldots,v_j)} \\
			& \leq \max_{j=0,\ldots,k} \sup_{(u,v_1,\ldots,v_j) \in (U \times X^j) \setminus K_{j,R}} \frac{\left\Vert d^j f(u;v_1,\ldots,v_j) \right\Vert_Y}{\psi_j(u,v_1,\ldots,v_j)} \\
			& \quad\quad + C_{\inf} \max_{j=0,\ldots,k} \sup_{(u,v_1,\ldots,v_j) \in K_{j,R}} \left\Vert d^j f(u;v_1,\ldots,v_j) - d^j g(u;v_1,\ldots,v_j) \right\Vert_Y \\
			& \quad\quad + \max_{j=0,\ldots,k} \sup_{(u,v_1,\ldots,v_j) \in (U \times X^j) \setminus K_{j,R}} \frac{\left\Vert d^j g(u;v_1,\ldots,v_j) \right\Vert_Y}{\psi_j(u,v_1,\ldots,v_j)} \\
			& < \frac{\varepsilon}{3} + C_{\inf} \frac{\varepsilon}{3C_{\inf}} + \frac{\varepsilon}{3} = \varepsilon.
		\end{aligned}
	\end{equation}
	Since $\varepsilon > 0$ was chosen arbitrarily and $\mathcal{B}^k_\Psi(U;Y)$ is defined as the closure of $C^k_b(U;Y)$ with respect to $\Vert \cdot \Vert_{\mathcal{B}^k_\Psi(U;Y)}$, we conclude that $f \in \mathcal{B}^k_\Psi(U;Y)$. In the other case, if $(X,\tau_X)$ is locally compact, we do not need to concatenate with finite rank operators and can directly obtain some $h \in C^\infty_c(X)$ satisfying \eqref{eq:prop:BPsik_equiv:proof7}, i.e., we can replace $T_\gamma: X \rightarrow X$ with the identity $\id_X: X \rightarrow X$.
\end{proof}

\section{Proof of results in Section~\ref{sec:nachbin}}

\subsection{Proof of Lemma~\ref{lem:real_anal}}
\label{app:real_anal}

For more background on Banach space-valued real-analytic functions, we refer to \cite[Chapter~IX]{dieudonne69}.

\begin{proof}[Proof of Lemma~\ref{lem:real_anal}]
	Let $V := \lbrace z \in \mathbb{C}: \vert \imag(z) \vert < \lambda/2 \rbrace$, fix some distinct $z_0,z_1 \in \mathbb{C}$, and define the function $\mathbb{R} \ni s \mapsto h_{z_0}(s) := s \cos(z_0 s) \in \mathbb{C}$ satisfying $h_{z_0}^{(j)}(s) = s \cos^{(j)}(z_0 s) z_0^j + j \cos^{(j-1)}(z_0 s) z_0^{j-1}$. Then, by splitting into cosine change and monomial change, by applying Taylor’s theorem to the holomorphic functions $\cos: \mathbb{C} \rightarrow \mathbb{C}$ and $\mathbb{C} \ni z \mapsto z^j \in \mathbb{C}$ (see also \cite[Equation~3.1]{cuchiero23}), and by using that $\vert s \vert^2 = \frac{8}{\lambda^2} \frac{(\frac{\lambda}{2} \vert s \vert)^2}{2!} \leq \frac{8}{\lambda^2} \sum_{n=1}^\infty \frac{(\frac{\lambda}{2} \vert s \vert)^n}{n!} \leq \frac{8}{\lambda^2} \exp\left( \frac{\lambda}{2} \vert s \vert \right)$ together with $1 \leq \exp(\lambda \vert s \vert)$, it follows for every $j = 0,\ldots,k$ and $s \in \mathbb{R}$ that
	\begin{equation}
		\begin{aligned}
			& \left\vert \frac{\cos^{(j)}\left( z_1 s \right) z_1^j - \cos^{(j)}\left( z_0 s \right) z_0^j}{z_1 - z_0} - h_{z_0}^{(j)}(s) \right\vert \\
			& \leq \left\vert s z_1^j \int_0^1 \left( \cos^{(j)}((z_0+t(z_1-z_0))s) - \cos^{(j)}(z_0 s) \right) dt \right\vert + \left\vert s z_1^j \cos^{(j)}(z_0 s) - s z_0^j \cos^{(j)}(z_0 s) \right\vert \\
			& \quad\quad + \left\vert j \cos^{(j-1)}(z_0 s) \int_0^1 \left( (z_0+t(z_1-z_0))^{j-1} - z_0^{j-1} \right) dt \right\vert \\
			& \leq \vert s \vert \vert z_1 \vert^j \int_0^1 \left\vert \cos^{(j)}((z_0+t(z_1-z_0))s) - \cos^{(j)}(z_0 s) \right\vert dt + \vert s \vert \left\vert z_1^j - z_0^j \right\vert \\
			& \quad\quad + j \int_0^1 \left\vert (z_0+t(z_1-z_0))^{j-1} - z_0^{j-1} \right\vert dt \\
			& \leq \vert s \vert \vert z_1 \vert^j \! \int_0^1 \! \vert s (z_1 \!-\! z_0) \vert e^{\vert s \vert \max(\vert\imag(z_0+t(z_1-z_0))\vert,\vert\imag(z_0)\vert)} dt \!+\! j \vert z_0 \vert^{j-1} \vert z_1 \!-\! z_0 \vert \!+\! j (j-1) \vert z_0 \vert^{j-2} \vert z_1 \!-\! z_0 \vert \\
			& \leq (1 + \vert z_0 \vert + \vert z_1 \vert)^k \left( \vert s \vert^2 e^{\frac{\lambda}{2} \vert s \vert} \vert z_1-z_0 \vert + 2 k^2 \vert z_1 - z_0 \vert \right) \\
			& \leq (1 + \vert z_0 \vert + \vert z_1 \vert)^k \left( \frac{8}{\lambda^2} + 2k^2 \right) \vert z_1-z_0 \vert e^{\lambda \vert s \vert}.
		\end{aligned}
	\end{equation}
	Hence, the Fa\`a di Bruno formula implies for every $j = 0,\ldots,k$ and $(u,v_1,\ldots,v_j) \in U \times (\mathbb{R}^d)^j$ that
	\begin{equation}
		\begin{aligned}
			& \left\vert d^j \left( \frac{\cos\left( z_1 \widetilde{g}(\cdot) \right) - \cos\left( z_0 \widetilde{g}(\cdot) \right)}{z_1 - z_0} \right)(u;v_1,\ldots,v_j) - d^j \left( \widetilde{g}(\cdot) \cos'\left( z_0 \widetilde{g}(\cdot) \right) \right)(u;v_1,\ldots,v_j) \right\vert \\
			& \leq \left\vert \sum_{\pi \in \mathscr{P}_j} \frac{\cos^{(\vert\pi\vert)}\left( z_1 \widetilde{g}(u) \right) z_1^{\vert\pi\vert} - \cos^{(\vert\pi\vert)}\left( z_0 \widetilde{g}(u) \right) z_0^{\vert\pi\vert}}{z_1 - z_0} d^\pi \widetilde{g}(u;v_\pi) - \sum_{\pi \in \mathscr{P}_j} h_{z_0}^{(\vert\pi\vert)}(\widetilde{g}(u)) d^\pi \widetilde{g}(u;v_\pi) \right\vert \\
			& \leq \sum_{\pi \in \mathscr{P}_j} \left\vert \frac{\cos^{(\vert\pi\vert)}\left( z_1 \widetilde{g}(u) \right) z_1^{\vert\pi\vert} - \cos^{(\vert\pi\vert)}\left( z_0 \widetilde{g}(u) \right) z_0^{\vert\pi\vert}}{z_1 - z_0} - h_{z_0}^{(\vert\pi\vert)}(\widetilde{g}(u)) \right\vert \left\vert d^\pi \widetilde{g}(u;v_\pi) \right\vert \\
			& \leq k! k^2 (1 + \vert z_0 \vert + \vert z_1 \vert)^k \frac{8 \vert z_1 - z_0 \vert}{\lambda^2} e^{\lambda \vert g(u) \vert} \max_{\pi \in \mathscr{P}_j} \left\vert d^\pi \widetilde{g}(u;v_\pi) \right\vert.
		\end{aligned}
	\end{equation}
	Thus, by using \ref{M2} and the continuity of $U \times (\mathbb{R}^d)^j \!\ni\! (u,v_1,\ldots,v_j) \mapsto \exp\left( \lambda \vert g(u) \vert \right) \left\vert d^\pi \widetilde{g}(u;v_\pi) \right\vert$ on the compact pre-images $K_{j,R}$, for all $R > 0$, we conclude that
	\begin{equation}
		\begin{aligned}
			& \left\Vert \frac{\cos\left( z_1 \widetilde{g}(\cdot) \right) - \cos\left( z_0 \widetilde{g}(\cdot) \right)}{z_1 - z_0} - \widetilde{g}(\cdot) \cos'\left( z_0 \widetilde{g}(\cdot) \right) \right\Vert_{\mathcal{B}^k_\Psi(U)} \\
			& \leq k! (1 + \vert z_0 \vert + \vert z_1 \vert)^k \left( \frac{8}{\lambda^2} + 2k^2 \right) \vert z_1 - z_0 \vert \underbrace{\max_{j=0,\ldots,k \atop \pi \in \mathscr{P}_j} \sup_{(u,v_1,\ldots,v_j) \in U \times (\mathbb{R}^d)^j} \frac{\exp\left( \lambda \vert g(u) \vert \right) \left\vert d^\pi \widetilde{g}(u;v_\pi) \right\vert}{\psi_j(u,v_1,\ldots,v_j)}}_{< \infty} \\
			& \quad\quad \overset{z_1 \rightarrow z_0}{\longrightarrow} \quad 0,
		\end{aligned}
	\end{equation}
	which shows that $V \ni z \mapsto \cos(z \widetilde{g}(\cdot)) \in \mathcal{B}^k_\Psi(U)$ is holomorphic, implying that the mapping $\mathbb{R} \ni t \mapsto \cos(t \widetilde{g}(\cdot)) \in \mathcal{B}^k_\Psi(U)$ is real-analytic. By a similar argument, $V \ni z \mapsto \sin(z \widetilde{g}(\cdot)) \in \mathcal{B}^k_\Psi(U)$ is holomorphic, which ensures that $\mathbb{R} \ni t \mapsto \sin(t \widetilde{g}(\cdot)) \in \mathcal{B}^k_\Psi(U)$ is real-analytic.
\end{proof}

\subsection{Proof of Lemma~\ref{lem:inp_bap}}
\label{app:lem:inp_bap}

\begin{proof}[Proof of Lemma~\ref{lem:inp_bap}]
	Since $f \in C^k_b(U;Y)$, there exists some $C_f \geq 0$ and $p_X \in \mathfrak{P}_{(X,\tau_X)}$ such that for every $j = 0,\ldots,k$ and $(u,v_1,\ldots,v_j) \in U \times X^j$ it holds that
	\begin{equation}
		\label{eq:lem:inp_bap:proof1}
		\left\Vert d^j f(u;v_1,\ldots,v_j) \right\Vert_Y \leq C_f p_X(v_1) \cdots p_X(v_j).
	\end{equation}
	Moreover, by using that $(U,\tau_X)$ has $\mathfrak{Q}_X$-BAP, there exists a net of finite rank operators $(T_\gamma)_\gamma \subseteq X^* \otimes X$ with $T_\gamma(U) \subseteq U$ approximating the identity $\id_X: X \rightarrow X$ uniformly on each relatively compact subset of $(X,\tau_X)$ such that for every $p_X \in \mathfrak{P}_{(X,\tau_X)}$ there exists some $\lambda > 0$ and $q_X \in \mathfrak{Q}_X$ satisfying for every $\gamma$ and $v \in X$ that
	\begin{equation}
		\label{eq:lem:inp_bap:proof2a}
		p_X(T_\gamma(v)) \leq \lambda q_X(v).
	\end{equation}
	In addition, by using \eqref{eq:ass:w_growth}, there exists some $R > 0$ such that
	\begin{equation}
		\label{eq:lem:inp_bap:proof2}
		\begin{aligned}
			\max_{j=0,\ldots,k} \sup_{(u,v_1,\ldots,v_j) \in (U \times X^j) \setminus K_{j,R}} \frac{p_X(v_1) \cdots p_X(v_j)}{\psi_j(u,v_1,\ldots,v_j)} & < \frac{\varepsilon}{3 (1+C_f)}, \\
			\max_{j=0,\ldots,k} \sup_{(u,v_1,\ldots,v_j) \in (U \times X^j) \setminus K_{j,R}} \frac{q_X(v_1) \cdots q_X(v_j)}{\psi_j(u,v_1,\ldots,v_j)} & < \frac{\varepsilon}{3 \lambda^k (1+C_f)}.
		\end{aligned}
	\end{equation}
	Furthermore, we observe that $K_{j,R} \ni (u,v_1,\ldots,v_j) \mapsto d^j f(u;v_1,\ldots,v_j) \in Y$ is continuous, thus uniformly continuous on the compact set $K_{j,R}$, to conclude that there exists an open $0$-neighborhood $V_{j,0} \times V_{j,1} \times \ldots \times V_{j,j}$ of $(X \times X^j,\tau_X \times \tau_X^j)$, with $V_{j,0}, \ldots, V_{j,j} \in \tau_X$, such that for every $(u,v_1,\ldots,v_j) \in K_{j,R}$ and $(\widetilde{u},\widetilde{v}_1,\ldots,\widetilde{v}_j) \in U \times X^j$ with
	\begin{equation}
		\label{eq:lem:inp_bap:proof3}
		(u,v_1,\ldots,v_j) - (\widetilde{u},\widetilde{v}_1,\ldots,\widetilde{v}_j) \,\, \in \,\, V_{j,0} \times V_{j,1} \times \ldots \times V_{j,j},
	\end{equation}
	it holds that
	\begin{equation}
		\label{eq:lem:inp_bap:proof4}
		\left\Vert d^j f(u;v_1,\ldots,v_j) - d^j f(\widetilde{u};\widetilde{v}_1,\ldots,\widetilde{v}_j) \right\Vert_Y < \frac{\varepsilon}{3C_{\inf}}.
	\end{equation}
	From this, we define the set $K := \bigcup_{j=0}^k \bigcup_{\ell=0}^j \pi_\ell(K_{j,R})$, which is compact as a finite union of compact images under the continuous projection $U \times X^j \ni (v_0,\ldots,v_\ell) \mapsto \pi_\ell(v_0,v_1,\ldots,v_j) := v_\ell \in X$. Hence, there exists some $\gamma$ such that for every $(u,v_1,\ldots,v_j) \in K_{j,R}$, we have
	\begin{equation}
		(u,v_1,\ldots,v_j) - (T_\gamma(u),T_\gamma(v_1),\ldots,T_\gamma(v_j)) \,\, \in \,\, V_{j,0} \times V_{j,1} \times \ldots \times V_{j,j}.
	\end{equation}
	Thus, by combining this with \eqref{eq:lem:inp_bap:proof3}, we conclude from the chain rule and \eqref{eq:lem:inp_bap:proof4} that
	\begin{equation}
		\label{eq:lem:inp_bap:proof5}
		\begin{aligned}
			& \left\Vert d^j f(u;v_1,\ldots,v_j) - d^j (f \circ T_\gamma)(u;v_1,\ldots,v_j) \right\Vert_Y \\
			& = \left\Vert d^j f(u;v_1,\ldots,v_j) - d^j f(T_\gamma(u);T_\gamma(v_1),\ldots,T_\gamma(v_j)) \right\Vert_Y < \frac{\varepsilon}{3 C_{\inf}}.
		\end{aligned}
	\end{equation}
	Moreover, by using \eqref{eq:lem:inp_bap:proof1}--\eqref{eq:lem:inp_bap:proof2a}, we observe for every $(u,v_1,\ldots,v_j) \in U \times X^j$ that
	\begin{equation}
		\label{eq:lem:inp_bap:proof6}
		\begin{aligned}
			\left\Vert d^j (f \circ T_\gamma)(u;v_1,\ldots,v_j) \right\Vert_Y & = \left\Vert d^j f(T_\gamma(u);T_\gamma(v_1),\ldots,T_\gamma(v_j)) \right\Vert_Y \\
			& \leq C_f p_X(T_\gamma(v_1)) \cdots p_X(T_\gamma(v_j)) \\
			& \leq C_f \lambda^j q_X(v_1) \cdots q_X(v_j).
		\end{aligned}
	\end{equation}
	Finally, by combining the inequalities \eqref{eq:lem:inp_bap:proof1}, \eqref{eq:lem:inp_bap:proof2}, \eqref{eq:lem:inp_bap:proof5}, and \eqref{eq:lem:inp_bap:proof6}, it follows that
	\begin{equation}
		\begin{aligned}
			& \Vert f - f \circ T_\gamma \Vert_{\mathcal{B}^k_\Psi(U;Y)} = \max_{j=0,\ldots,k} \sup_{(u,v_1,\ldots,v_j) \in U \times X^j} \frac{\left\Vert d^j f(u;v_1,\ldots,v_j) - d^j (f \circ T_\gamma)(u;v_1,\ldots,v_j) \right\Vert_Y}{\psi_j(u,v_1,\ldots,v_j)} \\
			& \leq \max_{j=0,\ldots,k} \sup_{(u,v_1,\ldots,v_j) \in (U \times X^j) \setminus K_{j,R}} \frac{\left\Vert d^j f(u;v_1,\ldots,v_j) \right\Vert_Y}{\psi_j(u,v_1,\ldots,v_j)} \\
			& \quad\quad + C_{\inf} \max_{j=0,\ldots,k} \sup_{(u,v_1,\ldots,v_j) \in K_{j,R}} \left\Vert d^j f(u;v_1,\ldots,v_j) - d^j (f \circ T_\gamma)(u;v_1,\ldots,v_j) \right\Vert_Y \\
			& \quad\quad + \max_{j=0,\ldots,k} \sup_{(u,v_1,\ldots,v_j) \in (U \times X^j) \setminus K_{j,R}} \frac{\left\Vert d^j (f \circ T_\gamma)(u;v_1,\ldots,v_j) \right\Vert_Y}{\psi_j(u,v_1,\ldots,v_j)} \\
			& < C_f \max_{j=0,\ldots,k} \sup_{(u,v_1,\ldots,v_j) \in (U \times X^j) \setminus K_{j,R}} \frac{p_X(v_1) \cdots p_X(v_j)}{\psi_j(u,v_1,\ldots,v_j)} + C_{\inf} \frac{\varepsilon}{3 C_{\inf}} \\
			& \quad\quad + C_f \max_{j=0,\ldots,k} \sup_{(u,v_1,\ldots,v_j) \in (U \times X^j) \setminus K_{j,R}} \frac{q_X(v_1) \cdots q_X(v_j)}{\psi_j(u,v_1,\ldots,v_j)} \\
			& < C_f \frac{\varepsilon}{3 (1+C_f)} + C_{\inf} \frac{\varepsilon}{3 C_{\inf}} + C_f \lambda^k \frac{\varepsilon}{3 \lambda^k (1+C_f)} \leq \varepsilon.
		\end{aligned}
	\end{equation}
	Since $\varepsilon > 0$ was chosen arbitrarily, we obtain the conclusion.
\end{proof}

\subsection{Proof of Lemma~\ref{lem:out_bap}}
\label{app:lem:out_bap}

\begin{proof}[Proof of Lemma~\ref{lem:out_bap}]
	Let $(Y,\Vert \cdot \Vert_Y)$ have BAP (with constant $\lambda \in [1,\infty)$) and fix some $\varepsilon > 0$. Then, by Lemma~\ref{prop:BPsik_equiv}~\ref{prop:BPsik_equiv:1}, there exists some $R > 0$ such that
	\begin{equation}
		\label{eq:lem:out_bap:proof1}
		\max_{j=0,\ldots,k} \sup_{(u,v_1,\ldots,v_j) \in (U \times X^j) \setminus K_{j,R}} \frac{\Vert d^j f(u;v_1,\ldots,v_j) \Vert_Y}{\psi_j(u,v_1,\ldots,v_j)} < \frac{\varepsilon}{2(1+\lambda)}.
	\end{equation}
	Moreover, we define the constant $C_{\inf} := \big( \min_{j=0,\ldots,k} \inf_{(u,v_1,\ldots,v_j) \in U \times X^j} \psi_j(u,v_1,\ldots,v_j) \big)^{-1} > 0$ and the set $K := \bigcup_{j=0}^k \left\lbrace d^j f(u;v_1,\ldots,v_j): (u,v_1,\ldots,v_j) \in K_{j,R} \right\rbrace$, which is compact as a finite union of continuous images of compact sets. Then, by using that $(Y,\Vert \cdot \Vert_Y)$ has BAP (with constant $\lambda \in [1,\infty)$), there exists some $\big( y \mapsto T(y) := \sum_{n=1}^N \ell_n(y) y_n \big) \in Y^* \otimes Y$, with $\ell_1,\ldots,\ell_N \in Y^*$ and $y_1,\ldots,y_N \in Y$, satisfying $\Vert T \Vert_{L(Y;Y)} \leq \lambda$ such that
	\begin{equation}
		\sup_{y \in K} \Vert y - T(y) \Vert_Y < \frac{\varepsilon}{2 C_{\inf}}.
	\end{equation}
	This together with the chain rule implies that
	\begin{equation}
		\label{eq:lem:out_bap:proof3}
		\begin{aligned}
			& \max_{j=0,\ldots,k} \sup_{(u,v_1,\ldots,v_j) \in K_{j,R}} \left\Vert d^j f(u;v_1,\ldots,v_j) - d^j (T \circ f)(u;v_1,\ldots,v_j) \right\Vert_Y \\
			& = \max_{j=0,\ldots,k} \sup_{(u,v_1,\ldots,v_j) \in K_{j,R}} \left\Vert d^j f(u;v_1,\ldots,v_j) - T\left( d^j f(u;v_1,\ldots,v_j) \right) \right\Vert_Y \\
			& \leq \sup_{y \in K} \Vert y - T(y) \Vert_Y < \frac{\varepsilon}{2 C_{\inf}}.
		\end{aligned}
	\end{equation}
	In addition, $\Vert T \Vert_{L(Y;Y)} \leq \lambda$ ensures for every $j = 0,\ldots,k$ and $(u,v_1,\ldots,v_j) \in U \times X^j$ that
	\begin{equation}
		\label{eq:lem:out_bap:proof4}
		\begin{aligned}
			\left\Vert d^j (T \circ f)(u;v_1,\ldots,v_j) \right\Vert_Y & = \left\Vert T\left( d^j f(u;v_1,\ldots,v_j) \right) \right\Vert_Y \\
			& \leq \Vert T \Vert_{L(Y;Y)} \left\Vert d^j f(u;v_1,\ldots,v_j) \right\Vert_Y \\
			& \leq \lambda \left\Vert d^j f(u;v_1,\ldots,v_j) \right\Vert_Y.
		\end{aligned}
	\end{equation}
	Furthermore, we claim for every fixed $n = 1,\ldots,N$ that $\ell_n \circ f \in \mathcal{B}_\Psi^k(U)$. Indeed, by definition, $f \in \mathcal{B}_\Psi^k(U;Y)$ can be approximated by a sequence $(g_m)_{m \in \mathbb{N}} \subseteq C^k_b(U;Y)$ with respect to $\Vert \cdot \Vert_{\mathcal{B}_\Psi^k(U;Y)}$, whence $(\ell_n \circ g_m)_{m \in \mathbb{N}} \subseteq C^k_b(U)$ approximates the function $\ell_n \circ f: U \rightarrow \mathbb{R}$ with respect to $\Vert \cdot \Vert_{\mathcal{B}_\Psi^k(U)}$, ensuring that $\ell_n \circ f \in \mathcal{B}_\Psi^k(U)$. Thus, \eqref{eq:lem:out_bap:proof1}--\eqref{eq:lem:out_bap:proof4} imply that
	\begin{equation}
		\begin{aligned}
			& \Vert f - T \circ f \Vert_{\mathcal{B}^k_\Psi(U;Y)} = \max_{j=0,\ldots,k} \sup_{(u,v_1,\ldots,v_j) \in U \times X^j} \frac{\left\Vert d^j f(u;v_1,\ldots,v_j) - d^j (T \circ f)(u;v_1,\ldots,v_j) \right\Vert_Y}{\psi_j(u,v_1,\ldots,v_j)} \\
			& \leq \max_{j=0,\ldots,k} \sup_{(u,v_1,\ldots,v_j) \in (U \times X^j) \setminus K_{j,R}} \frac{\left\Vert d^j f(u;v_1,\ldots,v_j) \right\Vert_Y}{\psi_j(u,v_1,\ldots,v_j)} \\
			& \quad\quad + C_{\inf} \max_{j=0,\ldots,k} \sup_{(u,v_1,\ldots,v_j) \in K_{j,R}} \left\Vert d^j f(u;v_1,\ldots,v_j) - d^j (T \circ f)(u;v_1,\ldots,v_j) \right\Vert_Y \\
			& \quad\quad + \max_{j=0,\ldots,k} \sup_{(u,v_1,\ldots,v_j) \in (U \times X^j) \setminus K_{j,R}} \frac{\left\Vert d^j (T \circ f)(u;v_1,\ldots,v_j) \right\Vert_Y}{\psi_j(u,v_1,\ldots,v_j)} \\
			& < (1+\lambda) \max_{j=0,\ldots,k} \sup_{(u,v_1,\ldots,v_j) \in (U \times X^j) \setminus K_{j,R}} \frac{\left\Vert d^j f(u;v_1,\ldots,v_j) \right\Vert_Y}{\psi_j(u,v_1,\ldots,v_j)} + C_{\inf} \frac{\varepsilon}{2 C_{\inf}} < \varepsilon.
		\end{aligned}
	\end{equation}
	Since $f \in \mathcal{B}^k_\Psi(U;Y)$ and $\varepsilon > 0$ were chosen arbitrarily, we obtain the conclusion.
\end{proof}

\section{Proof of results in Section~\ref{sec:w_uat}}
\label{app:w_uat}

\subsection{Auxiliary lemma for the proof of Theorem~\ref{thm:w_uat_findim}}
\label{app:thm:w_uat_findim}

\begin{lemma}
	\label{lem:NN_well_def}
	Let $(U,\Psi)$ be a weighted domain and let $(Y,\Vert \cdot \Vert_Y)$ be a Banach space. Moreover, for $c \in (0,\infty)$, let $\rho \in \mathscr{B}^k_c(\mathbb{R})$ and assume that $\mathcal{A} \subseteq \mathcal{B}^k_\Psi(U)$ satisfies for every $a \in \mathcal{A}$ that
	\begin{equation}
		\label{eq:lem:NN_well_def:1}
		C_a := \max_{j=0,\ldots,k \atop \pi \in \mathscr{P}_j} \sup_{(u,v_1,\ldots,v_j) \in U \times X^j} \frac{\left( 1 + \left\vert a(u) \right\vert \right)^c \left\vert d^\pi a(u;v_\pi) \right\vert}{\psi_j(u,v_1,\ldots,v_j)} < \infty.
	\end{equation}
	In addition, let $\mathcal{L} \subseteq Y$ be a vector subspace. Then, $\mathcal{NN}^{\mathcal{A},\rho,\mathcal{L}}_{U,Y} \subseteq \mathcal{B}^k_\Psi(U;Y)$.
\end{lemma}
\begin{proof}
	Since $\mathcal{NN}^{\mathcal{A},\rho,\mathcal{L}}_{U,Y}$ is defined as the linear span of maps of the form $U \ni u \mapsto y \rho(a(u)+b) \in Y$, with $a \in \mathcal{A}$ and $y \in \mathcal{L}$, and the mapping $\mathcal{B}^k_\Psi(U) \ni g \mapsto y \cdot g \in \mathcal{B}^k_\Psi(U;Y)$ is well-defined and continuous, it suffices to prove that $\rho \circ (a(\cdot)+b) \in \mathcal{B}^k_\Psi(U)$. To this end, we fix some $a \in \mathcal{A}$, $b \in \mathbb{R}$, and $\varepsilon \in (0,1)$. Then, by using that $\mathscr{B}^k_c(\mathbb{R})$ is defined as the closure of $C^k_b(\mathbb{R})$ with respect to $\Vert \cdot \Vert_{\mathscr{B}^k_c(\mathbb{R})}$, there exists some $\widetilde{\rho} \in C^k_b(\mathbb{R})$ such that
	\begin{equation}
		\label{eq:lem:NN_well_def:proof1}
		\Vert \rho - \widetilde{\rho} \Vert_{\mathscr{B}^k_c(\mathbb{R})} = \max_{j=0,\ldots,k} \sup_{z \in \mathbb{R}} \frac{\left\vert \rho^{(j)}(z) - \widetilde{\rho}^{(j)}(z) \right\vert}{(1+\vert z \vert)^c} < \frac{\varepsilon}{2 k! C_a},
	\end{equation}
	where $C_a > 0$ is defined in \eqref{eq:lem:NN_well_def:1}. Moreover, for $\Vert \widetilde{\rho} \Vert_{C^k_b(\mathbb{R})} := \max_{j=0,\ldots,k} \sup_{s \in \mathbb{R}} \vert \widetilde{\rho}^{(j)}(s) \vert < \infty$, we use that $a \in \mathcal{A} \subseteq \mathcal{B}^k_\Psi(U)$ to obtain from the definition of $\mathcal{B}^k_\Psi(U)$ some $\widetilde{a} \in C^k_b(U)$ such that
	\begin{equation}
		\begin{aligned}
			\Vert a - \widetilde{a} \Vert_{\mathcal{B}^k_\Psi(U)} & = \max_{j=0,\ldots,k} \sup_{(u,v_1,\ldots,v_j) \in U \times X^j} \frac{\left\vert d^j a(u;v_1,\ldots,v_j) - d^j \widetilde{a}(u;v_1,\ldots,v_j) \right\vert}{\psi_j(u,v_1,\ldots,v_j)} \\
			& < \frac{\varepsilon}{2 k \cdot k! C_\Psi^k \big( 1 + \Vert a \Vert_{\mathcal{B}^k_\Psi(U)} \big)^k \big( 1 + \Vert \widetilde{\rho} \Vert_{C^k_b(\mathbb{R})} \big)},
		\end{aligned}
	\end{equation}
	which implies that $\Vert \widetilde{a} \Vert_{\mathcal{B}^k_\Psi(U)} \leq 1 + \Vert a \Vert_{\mathcal{B}^k_\Psi(U)}$. Hence, by using a telescoping sum together with the monotonicity of $\Psi := (\psi_j)_{j=0,\ldots,k}$, it follows that
	\begin{equation}
		\label{eq:lem:NN_well_def:proof2}
		\begin{aligned}
			& \max_{j=0,\ldots,k \atop \pi \in \mathscr{P}_j} \sup_{(u,v_1,\ldots,v_j) \in U \times X^j} \frac{\left\vert d^\pi a(u;v_\pi) - d^\pi \widetilde{a}(u;v_\pi) \right\vert}{\psi_j(u,v_1,\ldots,v_j)} \\
			& \leq C_\Psi^k \!\! \max_{j=0,\ldots,k \atop \pi \in \mathscr{P}_j} \! \sup_{(u,v_1,\ldots,v_j) \in U \times X^j} \!\! \frac{\sum_{r=1}^{\vert\pi\vert} \big\vert \prod_{\ell=1}^{r-1} \! d^{\pi_\ell} a(u;v_{\pi_\ell}) \!\left( d^{\pi_r} a(u;v_{\pi_r}) \!-\! d^{\pi_r} \widetilde{a}(u;v_{\pi_r}) \right) \! \prod_{\ell=r}^{\vert\pi\vert} \! d^{\pi_\ell} \widetilde{a}(u;v_{\pi_\ell}) \big\vert}{\psi_{\pi_1}(u,v_{\pi_1}) \cdots \psi_{\pi_{\vert\pi\vert}}(u,v_{\pi_{\vert\pi\vert}})} \\
			& \leq k C_\Psi^k (1+\Vert a \Vert_{\mathcal{B}^k_\Psi(U)})^k \max_{j=0,\ldots,k \atop \pi \in \mathscr{P}_j} \max_{r=1,\ldots,\vert\pi\vert} \sup_{(u,v_1,\ldots,v_j) \in U \times X^j} \frac{\left\vert d^{\pi_r} a(u;v_{\pi_r}) - d^{\pi_r} \widetilde{a}(u;v_{\pi_r}) \right\vert}{\psi_{\pi_r}(u,v_{\pi_r})} \\
			& \leq k C_\Psi^k (1+\Vert a \Vert_{\mathcal{B}^k_\Psi(U)})^k \frac{\varepsilon}{2 k \cdot k! C_\Psi^k \big( 1 + \Vert a \Vert_{\mathcal{B}^k_\Psi(U)} \big)^k \big( 1 + \Vert \widetilde{\rho} \Vert_{C^k_b(\mathbb{R})} \big)} = \frac{\varepsilon}{2 k! \big( 1 + \Vert \widetilde{\rho} \Vert_{C^k_b(\mathbb{R})} \big)}.
		\end{aligned} 
	\end{equation}
	Thus, the Fa\`a di Bruno formula and \eqref{eq:lem:NN_well_def:proof1}--\eqref{eq:lem:NN_well_def:proof2} imply for $\widetilde{\rho} \circ (\widetilde{a}(\cdot)+b) \in C^k_b(U)$ that
	\begin{equation}
		\begin{aligned}
			& \Vert \rho \circ (a(\cdot)+b) - \widetilde{\rho} \circ (\widetilde{a}(\cdot)+b) \Vert_{\mathcal{B}^k_\Psi(U)} \\
			& \leq \Vert \rho \circ (a(\cdot)+b) - \widetilde{\rho} \circ (a(\cdot)+b) \Vert_{\mathcal{B}^k_\Psi(U)} + \Vert \widetilde{\rho} \circ (a(\cdot)+b) - \widetilde{\rho} \circ (\widetilde{a}(\cdot)+b) \Vert_{\mathcal{B}^k_\Psi(U)} \\
			& = \max_{j=0,\ldots,k} \sup_{(u,v_1,\ldots,v_j) \in U \times X^j} \frac{\left\vert d^j (\rho \circ (a(\cdot)+b))(u;v_1,\ldots,v_j) - d^j (\widetilde{\rho} \circ (a(\cdot)+b))(u;v_1,\ldots,v_j) \right\vert}{\psi_j(u,v_1,\ldots,v_j)} \\
			& \quad + \max_{j=0,\ldots,k} \sup_{(u,v_1,\ldots,v_j) \in U \times X^j} \frac{\left\vert d^j (\widetilde{\rho} \circ (a(\cdot)+b))(u;v_1,\ldots,v_j) - d^j (\widetilde{\rho} \circ (\widetilde{a}(\cdot)+b))(u;v_1,\ldots,v_j) \right\vert}{\psi_j(u,v_1,\ldots,v_j)} \\
			& \leq \max_{j=0,\ldots,k} \sup_{(u,v_1,\ldots,v_j) \in U \times X^j} \frac{\sum_{\pi \in \mathscr{P}_j} \left\vert \rho^{(\vert\pi\vert)}(a(u)+b) d^\pi a(u;v_\pi) - \widetilde{\rho}^{(\vert\pi\vert)}(a(u)+b) d^\pi a(u;v_\pi) \right\vert}{\psi_j(u,v_1,\ldots,v_j)} \\
			& \quad + \max_{j=0,\ldots,k} \sup_{(u,v_1,\ldots,v_j) \in U \times X^j} \frac{\sum_{\pi \in \mathscr{P}_j} \left\vert \widetilde{\rho}^{(\vert\pi\vert)}(a(u)+b) d^\pi a(u;v_\pi) - \widetilde{\rho}^{(\vert\pi\vert)}(\widetilde{a}(u)+b) d^\pi \widetilde{a}(u;v_\pi) \right\vert}{\psi_j(u,v_1,\ldots,v_j)} \\
			& \leq C_a k! \max_{j=0,\ldots,k \atop \pi \in \mathscr{P}_j} \sup_{(u,v_1,\ldots,v_j) \in U \times X^j} \frac{\left\vert \rho^{(\vert\pi\vert)}(a(u)+b) - \widetilde{\rho}^{(\vert\pi\vert)}(a(u)+b) \right\vert}{(1+\vert a(u)+b \vert)^c} \\
			& \quad + k! \Vert \widetilde{\rho} \Vert_{C^k_b(\mathbb{R})} \max_{j=0,\ldots,k \atop \pi \in \mathscr{P}_j} \frac{\left\vert d^\pi a(u;v_\pi) - d^\pi \widetilde{a}(u;v_\pi) \right\vert}{\psi_j(u,v_1,\ldots,v_j)} \\
			& < C_a k! \frac{\varepsilon}{2 C_a k!} + k! \Vert \widetilde{\rho} \Vert_{C^k_b(\mathbb{R})} \frac{\varepsilon}{2 k! \big( 1 + \Vert \widetilde{\rho} \Vert_{C^k_b(\mathbb{R})} \big)} \leq \varepsilon.
		\end{aligned}
	\end{equation}
	Since $\varepsilon \in (0,1)$ was chosen arbitrarily and $\mathcal{B}^k_\Psi(U)$ is defined as the closure of $C^k_b(U)$ with respect to $\Vert \cdot \Vert_{\mathcal{B}^k_\Psi(U)}$, this shows that $\rho \circ (a(\cdot)+b) \in \mathcal{B}^k_\Psi(U)$.
\end{proof}

\section{Proof of results in Section~\ref{sec:naf}}
\label{app:naf}

\subsection{Proof of Lemma~\ref{lem:naf_mfd_addfam}}
\label{app:lem:naf_mfd_addfam}

\begin{proof}[Proof of Lemma~\ref{lem:naf_mfd_addfam}]
	First, we show that $(\Lambda^{\alpha,1}_{T,Z},\Psi)$ is a weighted $C^k_{loc}$-manifold. To this end, we conclude from Theorem~\ref{thm:sk_dual} that $(\mathbb{R} \times D^{\alpha,1}([0,T];Z),\Vert \cdot \Vert_{\mathbb{R} \times D^{\alpha,1}([0,T];Z)})$ is a dual Banach space with some predual $(G,\Vert \cdot \Vert_G)$, where $\Vert (t,x) \Vert_{\mathbb{R} \times D^{\alpha,1}([0,T];Z)} := \vert t \vert + \Vert x \Vert_{\alpha,\ell^1}$. Thus, $(\phi_i(\Lambda^{\alpha,1}_{T,Z}),\Psi_i)$ is by Lemma~\ref{lem:w_dom}~\ref{lem:w_dom_dual} a weighted domain, whence $(\Lambda^{\alpha,1}_{T,Z},\Psi)$ is by definition a weighted $C^k_{loc}$-manifold.
	
	Now, we prove that $(\phi_i(\Lambda^{\alpha,1}_{T,Z}),\tau_\mathbb{R} \times \tau_{w^*})$ has $\Vert \cdot \Vert_{\mathbb{R} \times D^{\alpha,1}([0,T];Z)}$-BAP. Indeed, as $(D^{\alpha,1}([0,T];Z),\Vert \cdot \Vert_{\alpha,\ell^1})$ is by Theorem~\ref{thm:sk_dual} a dual Banach space with some predual $(G,\Vert \cdot \Vert_G)$, we can apply Theorem~\ref{thm:D0_bap} to conclude that $(D^{\alpha,1}([0,T];Z),\tau_{w^*})$ has $\Vert \cdot \Vert_{\alpha,\ell^1}$-BAP with finite rank operators $(Q_\gamma^*)_\gamma \subseteq (D^{\alpha,1}([0,T];Z),\tau_{w^*})^* \otimes D^{\alpha,1}([0,T];Z)$, where $(Q_\gamma)_\gamma \subseteq G^* \otimes G$ is the BAP net of $(G,\Vert \cdot \Vert_G)$ with $\Vert Q_\gamma \Vert_{L(G;G)} \leq C_G$, for some $C_G \geq 1$. From this, we define the finite rank operators $(R_\gamma)_\gamma \subseteq (\mathbb{R} \times D^{\alpha,1}([0,T];Z),\tau_\mathbb{R} \times \tau_{w^*})^* \otimes (\mathbb{R} \times D^{\alpha,1}([0,T];Z))$ by
	\begin{equation}
		\mathbb{R} \times D^{\alpha,1}([0,T];Z) \ni (t,x) \quad \mapsto \quad R_\gamma(t,x) := \left( t, Q_\gamma^*(x) \right) \in \mathbb{R} \times D^{\alpha,1}([0,T];Z).
	\end{equation}
	Then, for every $\big( (t,x) \mapsto p_{\mathbb{R} \times D^{\alpha,1}([0,T];Z)}(t,x) := \vert t \vert + p_{D^{\alpha,1}([0,T];Z)}(x) \big) \in \mathfrak{P}_{(\mathbb{R} \times D^{\alpha,1}([0,T];Z),\tau_\mathbb{R} \times \tau_{w^*})}$ with $p_{D^{\alpha,1}([0,T];Z)} \in \mathfrak{P}_{(D^{\alpha,1}([0,T];Z),\tau_{w^*})}$, we conclude for every relatively compact subset $K$ of $(\mathbb{R} \times D^{\alpha,1}([0,T];Z),\tau_\mathbb{R} \times \tau_{w^*})$ that
	\begin{equation}
		\begin{aligned}
			\lim_\gamma \sup_{(t,x) \in K} p_{\mathbb{R} \times D^{\alpha,1}([0,T];Z)}\left( (t,x) - R_\gamma(t,x) \right) & = \lim_\gamma \sup_{(t,x) \in K} p_{\mathbb{R} \times D^{\alpha,1}([0,T];Z)}\left( (t,x) - \left( t, Q_\gamma^*(x) \right) \right) \\
			& = \lim_\gamma \sup_{(t,x) \in K} \left( \vert t - t \vert + p_{D^{\alpha,1}([0,T];Z)}\left( x - Q_\gamma^*(x) \right) \right) \\
			& = \lim_\gamma \sup_{(t,x) \in K} p_{D^{\alpha,1}([0,T];Z)}\left( x - Q_\gamma^*(x) \right) = 0,
		\end{aligned}
	\end{equation}
	which shows that $(\phi_i(\Lambda^{\alpha,1}_{T,Z}),\tau_\mathbb{R} \times \tau_{w^*})$ has AP. In addition, it holds for every $g_1,\ldots,g_N \in G$, $(t,x) \in \mathbb{R} \times D^{\alpha,1}([0,T];Z)$, and $\big( (t,x) \mapsto \widetilde{p}_{\mathbb{R} \times D^{\alpha,1}([0,T];Z)}(t,x) := \vert t \vert + \max_{n=1,\ldots,N} \vert \langle x, g_n \rangle_{D^{\alpha,1}([0,T];Z) \times G} \vert \big) \in \mathfrak{P}_{(\mathbb{R} \times D^{\alpha,1}([0,T];Z),\tau_\mathbb{R} \times \tau_{w^*})}$ that
	\begin{equation}
		\label{eq:lem:naf_mfd_addfam:proof0}
		\begin{aligned}
			\widetilde{p}_{\mathbb{R} \times D^{\alpha,1}([0,T];Z)}(R_\gamma(t,x)) & = \vert t \vert + \max_{n=1,\ldots,N} \vert \langle Q_\gamma^*(x), g_n \rangle_{D^{\alpha,1}([0,T];Z) \times G} \vert \\
			& = \vert t \vert + \max_{n=1,\ldots,N} \vert \langle x, Q_\gamma(g_n) \rangle_{D^{\alpha,1}([0,T];Z) \times G} \vert \\
			& \leq \vert t \vert + \Vert x \Vert_{\alpha,\ell^1} \Vert Q_\gamma \Vert_{L(E;E)} \max_{n=1,\ldots,N} \Vert g_n \Vert_G \\
			& = \left( 1 + C_G \max_{n=1,\ldots,N} \Vert g_n \Vert_G \right) \Vert (t,x) \Vert_{\mathbb{R} \times D^{\alpha,1}([0,T];Z)},
		\end{aligned}
	\end{equation}
	which proves that $(\phi_i(\Lambda^{\alpha,1}_{T,Z}),\tau_\mathbb{R} \times \tau_{w^*})$ has $\Vert \cdot \Vert_{\mathbb{R} \times D^{\alpha,1}([0,T];Z)}$-BAP.
	
	In order to show that $\mathcal{A} \subseteq \mathcal{B}^k_\Psi(\Lambda^{\alpha,1}_{T,Z})$ is well-defined, we aim to apply Proposition~\ref{prop:BPsik_equiv}~\ref{prop:BPsik_equiv:2}. To this end, we fix some $\lambda \in \mathbb{R}$ as well as $\widetilde{\varphi} \in \mathcal{NN}^{\mathbb{R},\widetilde{\rho},E}_{\mathbb{R},E}$ and consider the map $\Lambda^{\alpha,1}_{T,Z} \ni (t,x) \mapsto a(t,x) \!:=\! \lambda t + \int_0^T \langle \overline{x}^t_s, \widetilde{\varphi}(s) \rangle_{Z \times E} \, ds \in \mathbb{R}$. Then, for every $j = 0,\ldots,k$ and $\big( (t,\overline{x}^t),(s_1,v_1),\ldots,(s_j,v_j) \big) \in \phi_i(\Lambda^{\alpha,1}_{T,Z}) \times (\mathbb{R} \times D^{\alpha,1}([0,T];Z))^j$, we have
	\begin{equation}
		\label{eq:lem:naf_mfd_addfam:proof1}
		d^j \left( a \circ \phi_i^{-1} \right)((t,x);(s_1,v_1),\ldots,(s_j,v_j)) =
		\begin{cases}
			\lambda t + \int_0^T \langle \overline{x}^t_s, \widetilde{\varphi}(s) \rangle_{Z \times E} \, ds, & \text{if } j = 0, \\
			\lambda s_1 + \int_0^T \langle v_{1,s}, \widetilde{\varphi}(s) \rangle_{Z \times E} \, ds, & \text{if } j = 1, \\
			0, & \text{if } j \geq 2.
		\end{cases}
	\end{equation}
	which shows that $a \circ \phi_i^{-1} \in C^k(\phi_i(\Lambda^{\alpha,1}_{T,Z}))$. Moreover, by using a similar argument as in \eqref{eq:lem:naf_mfd_addfam:proof0} and the constant $C := 1 + \vert\lambda\vert + \max(1,T)^\alpha C_G \int_0^T \Vert \widetilde{\varphi}(s) \Vert_E ds > 0$, it holds that
	\begin{equation}
		\label{eq:lem:naf_mfd_addfam:proof2}
		\begin{aligned}
			& \left\vert d^j \left( a \circ \phi_i^{-1} \right)(R_\gamma(t,\overline{x}^t);R_\gamma(s_1,v_1),\ldots,R_\gamma(s_j,v_j)) \right\vert \\
			& \leq \vert \lambda \vert (\vert t \vert + \vert s_1 \vert) + \int_0^T \left\vert \langle Q_\gamma^*(\overline{x}^t)_s, \widetilde{\varphi}(s) \rangle_{Z \times E} \right\vert ds + \int_0^T \left\vert \langle Q_\gamma^*(v_1)_s, \widetilde{\varphi}(s) \rangle_{Z \times E} \right\vert ds \\
			& \leq \vert \lambda \vert (\vert t \vert + \vert s_1 \vert) + \left( \int_0^T \Vert \widetilde{\varphi}(s) \Vert_E ds \right) \max(1,T)^\alpha \left( \Vert Q_\gamma^*(\overline{x}^t) \Vert_{\alpha,\ell^1} + \Vert Q_\gamma^*(v_1) \Vert_{\alpha,\ell^1} \right) \\
			& \leq \vert \lambda \vert (\vert t \vert + \vert s_1 \vert) + \max(1,T)^\alpha C_G \left( \int_0^T \Vert \widetilde{\varphi}(s) \Vert_E ds \right) \left( \Vert \overline{x}^t \Vert_{\alpha,\ell^1} + \Vert v_1 \Vert_{\alpha,\ell^1} \right) \\
			& \leq C \left( \vert t \vert + \Vert \overline{x}^t \Vert_\alpha + \vert s_1 \vert + \Vert v_1 \Vert_\alpha \right) \\
			& \leq C \left( \max(j,1) \Vert (t,\overline{x}^t) \Vert_{\mathbb{R} \times D^{\alpha,1}([0,T];Z)} + \sum_{\ell=1}^j \Vert (s_\ell,v_\ell) \Vert_{\mathbb{R} \times D^{\alpha,1}([0,T];Z)} \right)
		\end{aligned}
	\end{equation}
	Therefore, by using this together with the assumption $\lim_{r \rightarrow \infty} \frac{r^k}{\eta(r)} = 0$, we obtain that
	\begin{equation}
		\begin{aligned}
			& \lim_{R \rightarrow \infty} \max_{j=0,\ldots,k \atop \mathcal{L} \subseteq \lbrace 1,\ldots,j \rbrace} \sup_{\big((t,\overline{x}^t), (s_1,v_1),\ldots,(s_j,v_j)\big) \in K_{ij,R}^c} \frac{\left\vert d^{\vert\mathcal{L}\vert} \left( a \circ \phi_i^{-1} \right)(R_\gamma(t,\overline{x}^t);(R_\gamma(s_\ell,v_\ell))_{\ell \in \mathcal{L}}) \right\vert}{\psi_{i,\vert\mathcal{L}\vert}\big( (t,\overline{x}^t), (s_\ell,v_\ell)_{\ell \in \mathcal{L}} \big)} \\
			& \leq C \lim_{R \rightarrow \infty} \max_{j=0,\ldots,k \atop \mathcal{L} \subseteq \lbrace 1,\ldots,j \rbrace} \sup_{K_{i,j,R}^c} \frac{\max(\vert\mathcal{L}\vert,1) \Vert (t,\overline{x}^t) \Vert_{\mathbb{R} \times D^{\alpha,1}([0,T];Z)} + \sum_{\ell \in \mathcal{L}} \Vert (s_\ell,v_\ell) \Vert_{\mathbb{R} \times D^{\alpha,1}([0,T];Z)}}{\eta\left( \max(\vert\mathcal{L}\vert,1) \Vert (t,\overline{x}^t) \Vert_{\mathbb{R} \times D^{\alpha,1}([0,T];Z)} + \sum_{\ell\in\mathcal{L}} \Vert (s_\ell,v_\ell) \Vert_{\mathbb{R} \times D^{\alpha,1}([0,T];Z)} \right)} \\
			& = C \lim_{r \rightarrow \infty} \frac{r}{\eta(r)} = 0,
		\end{aligned}
	\end{equation}
	with supremum taken over $\big( (t,\overline{x}^t), (s_1,v_1),\ldots,(s_j,v_j) \big) \in (\phi_i(\Lambda^{\alpha,1}_{T,Z}) \times (\mathbb{R} \times D^{\alpha,1}([0,T];Z))^j) \setminus K_{i,j,R}$. Hence, Proposition~\ref{prop:BPsik_equiv}~\ref{prop:BPsik_equiv:2} implies $a \circ \phi_i^{-1} \in \mathcal{B}^k_{\Psi_i}(\phi_i(\Lambda^{\alpha,1}_{T,Z}))$, showing that $\mathcal{A} \subseteq \mathcal{B}^k_\Psi(\Lambda^{\alpha,1}_{T,Z})$.
	
	Finally, we prove that $\mathcal{A}$ is an additive family on $\Lambda^{\alpha,1}_{T,Z}$ by showing that $\mathcal{A}_i = \left\lbrace a \circ \phi_i^{-1}: a \in \mathcal{A} \right\rbrace$ is an additive family on $\phi_i(\Lambda^{\alpha,1}_{T,Z})$. For \ref{A1}, we observe that $\mathcal{A}_i$ is a vector space and therefore closed under addition. For \ref{A2}, let $(t_1,\overline{x}_1^{t_1}), (t_2,\overline{x}_2^{t_2}) \in \phi_i(\Lambda^{\alpha,1}_{T,Z})$ be two distinct points. If $t_1 \neq t_2$, the map $\big( (t,\overline{x}^t) \mapsto (a \circ \phi_i^{-1})(t,\overline{x}^t) := t \big) \in \mathcal{A}_i$ satisfies $a(t_1,\overline{x}_1^{t_1}) = t_1 \neq t_2 = a(t_2,\overline{x}_2^{t_2})$. Otherwise, if $t_1 = t_2 =: t$ and $\overline{x}_1^t = \overline{x}_2^t$ on $[0,t)$ but $\overline{x}_{1,t}^t \neq \overline{x}_{2,t}^t$, there exists some $e \in E$ such that $\langle \overline{x}_{1,t}^t, e \rangle_{Z \times E} \neq \langle \overline{x}_{2,t}^t, e \rangle_{Z \times E}$. Thus, there exists some $\widetilde{\varphi}_0 \in \mathcal{NN}^{\mathbb{R},\widetilde{\rho},\mathbb{R}}_{\mathbb{R},\mathbb{R}}$ with $\int_t^T \widetilde{\varphi}_0(s) ds \neq 0$, whence $\big( (t,\overline{x}^t) \mapsto (a \circ \phi_i^{-1})(t,\overline{x}^t) := \int_0^T \langle \overline{x}^t_s, \widetilde{\varphi}_0(s) e \rangle_{Z \times E} ds \big) \in \mathcal{A}_i$ satisfies
	\begin{equation}
		\label{eq:lem:naf_mfd_addfam:proof3}
		\begin{aligned}
			& a(t_1,\overline{x}_1^{t_1}) = \int_0^T \langle \overline{x}^t_{1,s}, \widetilde{\varphi}_0(s) e \rangle_{Z \times E} ds = \int_0^t\langle \overline{x}^t_{1,s}, \widetilde{\varphi}_0(s) e \rangle_{Z \times E} ds + \Big\langle \overline{x}^t_{1,t}, e \int_t^T \widetilde{\varphi}_0(s) ds \Big\rangle_{Z \times E} \\
			& \neq \int_0^t \langle \overline{x}^t_{2,s}, \widetilde{\varphi}_0(s) e \rangle_{Z \times E} ds + \Big\langle \overline{x}^t_{2,t}, e \int_t^T \widetilde{\varphi}_0(s) ds \Big\rangle_{Z \times E} = \int_0^T \langle \overline{x}^t_{2,s}, \widetilde{\varphi}_0(s) e \rangle_{Z \times E} ds = a(t_2,\overline{x}_2^{t_2}).
		\end{aligned}
	\end{equation}
	Otherwise, if $t_1 = t_2 =: t$ and $\overline{x}_{1,t}^t = \overline{x}_{2,t}^t$ but $\overline{x}_1^t \neq \overline{x}_2^t$ on $[0,t)$, there exists some $e \in E$ such that $s \mapsto \langle \overline{x}_{1,s}^t, e \rangle_{Z \times E}$ differs from $s \mapsto \langle \overline{x}_{2,s}^t, e \rangle_{Z \times E}$ on $[0,t)$. Thus, by using that $\mathcal{NN}^{\mathbb{R},\widetilde{\rho},\mathbb{R}}_{\mathbb{R},\mathbb{R}}$ is weakly dense in $L^1([0,t])$ (see \cite[p.~31]{cuchiero23}), there exists some $\widetilde{\varphi}_0 \in \mathcal{NN}^{\mathbb{R},\widetilde{\rho},\mathbb{R}}_{\mathbb{R},\mathbb{R}}$ with $\int_0^t \langle \overline{x}_{1,s}^t, \widetilde{\varphi}_0(s) e \rangle_{Z \times E} ds \neq \int_0^t \langle \overline{x}_{2,s}^t, \widetilde{\varphi}_0(s) e \rangle_{Z \times E} ds$, whence $\big( (t,\overline{x}^t) \mapsto (a \circ \phi_i^{-1})(t,\overline{x}^t) := \int_0^T \langle \overline{x}^t_s, \widetilde{\varphi}_0(s) e \rangle_{Z \times E} ds \big) \in \mathcal{A}_i$ also satisfies \eqref{eq:lem:naf_mfd_addfam:proof3}, which shows that $\mathcal{A}_i$ is point separating on $\phi_i(\Lambda^{\alpha,1}_{T,Z})$. For \ref{A3}, we fix some $(t,\overline{x}^t) \in \phi_i(\Lambda^{\alpha,1}_{T,Z})$ and $(s,v) \in \mathbb{R} \times D^{\alpha,1}([0,T];Z)$. Then, by applying the Hahn-Banach theorem, there exists some $\lambda \in \mathbb{R}$, $\widetilde{\varphi}_0 \in \mathcal{NN}^{\mathbb{R},\widetilde{\rho},\mathbb{R}}_{\mathbb{R},\mathbb{R}}$, and $e \in E$ such that $\lambda s + \int_0^T \langle v_s, \widetilde{\varphi}_0(s) e \rangle_{Z \times E} ds \neq 0$. Hence, the map $\big( (t,\overline{x}^t) \mapsto (a \circ \phi_i^{-1})(t,\overline{x}^t) := \lambda t + \int_0^T \langle \overline{x}^t_s, \widetilde{\varphi}_0(s) e \rangle_{Z \times E} ds \big) \in \mathcal{A}_i$ satisfies
	\begin{equation}
		\begin{aligned}
			d\left( a \circ \phi_i^{-1} \right)((t,\overline{x}^t);(s,v)) & = \lambda s + \int_0^T \langle v_s, \widetilde{\varphi}(s) \rangle_{Z \times E} \, ds = \lambda s + \int_0^T \langle v_s, \widetilde{\varphi}_0(s) e \rangle_{Z \times E} ds \neq 0,
		\end{aligned}			
	\end{equation}
	which shows that $\mathcal{A}_i$ has nowhere vanishing derivatives. For \ref{A4'}, we use for every $\gamma$ that $T_\gamma(\phi_i(\Lambda^{\alpha,1}_{T,Z}))$ is $m$-dimensional (with some $m \in \mathbb{N}$), on which $\mathcal{A}_i$ is point separating and has nowhere vanishing derivatives, to obtain some $a_1,\ldots,a_m \in \mathcal{A}$ such that
	\begin{equation}
		\eta_{i,\gamma} := \big( a_1 \circ \phi_i^{-1},\ldots,a_m \circ \phi_i^{-1} \big)^\top\big\vert_{T_\gamma(\phi_i(\Lambda^{\alpha,1}_{T,Z}))}: T_\gamma(\phi_i(\Lambda^{\alpha,1}_{T,Z})) \rightarrow \mathbb{R}^m
	\end{equation}
	is an embedding, where $a_1,\ldots,a_m \in \mathcal{A}$ have uniformly bounded derivatives ensuring that the corresponding limit in \ref{A4} is finite (see also Remark~\ref{rem:psi_mod_growth}).
\end{proof}

\subsection{Proof of Corollary~\ref{cor:naf_uat_full}}
\label{app:cor:naf_uat_full}

\begin{proof}[Proof of Corollary~\ref{cor:naf_uat_full}]
	We aim to apply Theorem~\ref{thm:w_uat_infdim} to obtain that $\mathcal{PN}^{\widetilde{\rho},\rho,\mathcal{L}}_{\Lambda^{\alpha,1}_{T,Z},Y} = \mathcal{NN}^{\mathcal{A},\rho,\mathcal{L}}_{\Lambda^{\alpha,1}_{T,Z},Y}$ is a dense subset of $\mathcal{B}^k_\Psi(\Lambda^{\alpha,1}_{T,Z};Y)$. To this end, we fix some $a \in \mathcal{A}$ of the form $\Lambda^{\alpha,1}_{T,Z} \ni (t,x) \mapsto a(t,x) := \lambda t + \int_0^T \langle \overline{x}^t_s, \widetilde{\varphi}(s) \rangle_{Z \times E} \, ds \in \mathbb{R}$, with some $\lambda \in \mathbb{R}$ and $\widetilde{\varphi} \in \mathcal{NN}^{\mathbb{R},\widetilde{\rho},E}_{\mathbb{R},E}$, and first show that the constant $C_{a,i} > 0$ defined in \eqref{eq:thm:w_uat_infdim1} is finite. Indeed, by using \eqref{eq:lem:naf_mfd_addfam:proof1}--\eqref{eq:lem:naf_mfd_addfam:proof2}, we observe that
	\begin{equation}
		\begin{aligned}
			C_{a,i} & := \max_{j=0,\ldots,k \atop \pi \in \mathscr{P}_j} \sup_{(u,v_1,\ldots,v_j) \in \phi_i(\Lambda^{\alpha,1}_{T,Z}) \times (\mathbb{R} \times D^{\alpha,1}([0,T];Z))^j} \frac{\left( 1 + \left\vert \left( a \circ \phi_i^{-1} \right)(u) \right\vert \right)^c \left\vert d^\pi \left( a \circ \phi_i^{-1} \right)(u;v_{\pi}) \right\vert}{\psi_{i,j}(u,v_{1:j})} \\
			& \leq \max_{j=0,\ldots,k} \sup_{(u,v_1,\ldots,v_j)} \frac{\left( 1 + \max(j,1) \Vert u \Vert_{\mathbb{R} \times D^{\alpha,1}([0,T];Z)} + \sum_{\ell=1}^j \Vert v_\ell \Vert_{\mathbb{R} \times D^{\alpha,1}([0,T];Z)} \right)^{\max(j,c)}}{\eta\left( \max(j,1) \Vert u \Vert_{\mathbb{R} \times D^{\alpha,1}([0,T];Z)} + \sum_{\ell=1}^j \Vert v_\ell \Vert_{\mathbb{R} \times D^{\alpha,1}([0,T];Z)} \right)} \\
			& \leq \sup_{r \in (0,\infty)} \frac{(1+r)^{k\max(1,c)}}{\eta(r)} < \infty.
		\end{aligned}
	\end{equation}
	Next, we show \eqref{eq:thm:w_uat_infdim2}, i.e., for every $a_i := a \circ \phi_i^{-1} \in \mathcal{A}_i := \left\lbrace a \circ \phi_i^{-1}: a \in \mathcal{A} \right\rbrace$, $b \in \mathbb{R}$, and every finite rank operator $(T_\gamma)_{g_{1:N}} \subseteq (\mathbb{R} \times D^{\alpha,1}([0,T];Z),\tau_\mathbb{R} \times \tau_{w^*})^* \otimes (\mathbb{R} \times D^{\alpha,1}([0,T];Z))$ from Lemma~\ref{lem:naf_mfd_addfam}, we prove that the composition $\rho((a_i \circ T_\gamma)(\cdot) + b)$ belongs to the closure of $\mathcal{NN}^{\mathcal{A}_i,\rho,\mathbb{R}}_{\phi_i(\Lambda^{\alpha,1}_{T,Z}),\mathbb{R}}$ with respect to $\Vert \cdot \Vert_{\mathcal{B}^k_{\Psi_i}(\phi_i(\Lambda^{\alpha,1}_{T,Z}))}$, where $D^{\alpha,1}([0,T];Z) \ni x \mapsto T_\gamma(x) := \sum_{n=1}^N \langle x, g_n \rangle_{D^{\alpha,1}([0,T];Z) \times G} \, x_n \in D^{\alpha,1}([0,T];Z)$, for some $g_1,\ldots,g_N \in G$ and $x_1,\ldots,x_N \in D^{\alpha,1}([0,T];Z)$, where $(G,\Vert \cdot \Vert_G)$ is a predual of $(D^{\alpha,1}([0,T];Z),\Vert \cdot \Vert_{\alpha,\ell^1})$. To this end, we fix some $\lambda \in \mathbb{R}$, $\widetilde{\varphi} \in \mathcal{NN}^{\mathbb{R},\widetilde{\rho},E}_{\mathbb{R},E}$, $b \in \mathbb{R}$, $g_1,\ldots,g_N \in G$, and $\varepsilon \in (0,1)$. Then, for every fixed $n = 1,\ldots,N$, we apply the weighted UAT in \cite[Theorem~4.13]{cuchiero23} for $\mathcal{B}^0_\Psi$-maps without derivatives (onto $(D^{\alpha,1}([0,T];Z),\tau_{w^*})$ equipped with the weight function $(1+\Vert \cdot \Vert_{\alpha,\ell^1})^c$) to obtain some $\widetilde{\varphi}_n \in \mathcal{NN}^{\mathbb{R},\widetilde{\rho},E}_{\mathbb{R},E}$ satisfying
	\begin{equation}
		\label{eq:cor:naf_uat_full:proof1}
		\sup_{x \in D^{\alpha,1}([0,T];Z)} \frac{\left\vert \langle x, g_n \rangle_{D^{\alpha,1}([0,T];Z) \times G} - \int_0^T \langle x_s, \widetilde{\varphi}_n(s) \rangle_{Z \times E} \, ds \right\vert}{(1+\Vert x \Vert_{\alpha,\ell^1})^c} < \frac{\varepsilon}{C_1 C_\eta C_\rho k N \max(1,\vert r_n \vert)},
	\end{equation}
	where $C_1 := \Big(1 + \sup_{u \in \mathbb{R} \times D^{\alpha,1}([0,T];Z)} \frac{\vert (a_i \circ T_\gamma)(u) \vert}{(1+\Vert u \Vert_{\mathbb{R} \times D^{\alpha,1}([0,T];Z)})^c} \Big)^k \geq 1$, $C_\eta := 1 + \sup_{r \in (0,\infty)} \frac{r^{k\max(1,c)}}{\eta(r)} \geq 1$, $C_\rho := 1 + \max_{j=1,\ldots,k+1} \sup_{s \in \mathbb{R}} \vert \rho^{(j)}(s) \vert \geq 1$, and $r_n := \int_0^T \vert \langle x_{n,s}, \widetilde{\varphi}(s) \rangle_{Z \times E} \vert ds \geq 0$. From this, we define $D^{\alpha,1}([0,T];Z) \ni x \mapsto Q_\gamma(x) := \sum_{n=1}^N x_n \int_0^T \langle x_s, \widetilde{\varphi}_n(s) \rangle_{Z \times E} \, ds \in D^{\alpha,1}([0,T];Z)$ and $\mathbb{R} \times D^{\alpha,1}([0,T];Z) \ni (t,x) \mapsto R_\gamma(t,x) := (t,Q_\gamma(x)) \in \mathbb{R} \times D^{\alpha,1}([0,T];Z)$. Hence, for every $(t,\overline{x}^t) \in \phi_i(\Lambda^{\alpha,1}_{T,Z})$, we observe that
	\begin{equation}
		\begin{aligned}
			(a_i \circ R_\gamma)(t,\overline{x}^t) & = \lambda t + \int_0^T \langle Q_\gamma(\overline{x}^t)_s, \widetilde{\varphi}(s) \rangle_{Z \times E} \, ds \\
			& = \lambda t + \sum_{n=1}^N \left( \int_0^T \langle \overline{x}^t_s, \widetilde{\varphi}_n(s) \rangle_{Z \times E} \, ds \right) \left( \int_0^T \langle x_{n,s}, \widetilde{\varphi}(s) \rangle_{Z \times E} \, ds \right),
		\end{aligned}
	\end{equation}
	showing that $a_i \circ R_\gamma \in \mathcal{A}_i$. Moreover, \eqref{eq:cor:naf_uat_full:proof1} ensures for every $(t,x) \in \mathbb{R} \times D^{\alpha,1}([0,T];Z)$ that
	\begin{equation}
		\begin{aligned}
			& \left\vert (a_i \circ T_\gamma)(t,x) - (a_i \circ R_\gamma)(t,x) \right\vert \\
			& = \Bigg\vert \left( \lambda t + \sum_{n=1}^N \langle x, g_n \rangle_{D^{\alpha,1}([0,T];Z) \times G} \int_0^T \langle x_{n,s}, \widetilde{\varphi}(s) \rangle_{Z \times E} \, ds \right) \\
			& \quad\quad - \left( \lambda t + \sum_{n=1}^N \left( \int_0^T \langle x_s, \widetilde{\varphi}_n(s) \rangle_{Z \times E} \, ds \right) \left( \int_0^T \langle x_{n,s}, \widetilde{\varphi}(s) \rangle_{Z \times E} \, ds \right) \right) \Bigg\vert \\
			& \leq \sum_{n=1}^N \vert r_n \vert \left\vert \langle x, g_n \rangle_{D^{\alpha,1}([0,T];Z) \times G} - \int_0^T \langle x_s, \widetilde{\varphi}_n(s) \rangle_{Z \times E} \, ds \right\vert \\
			& \leq \sum_{n=1}^N \vert r_n \vert \frac{\varepsilon}{C_1 C_\eta C_\rho k N \max(1,\vert r_n \vert)} \left( 1 + \Vert x \Vert_{\alpha,\ell^1} \right)^c \\
			& \leq \frac{\varepsilon}{C_1 C_\eta C_\rho k} \left( 1 + \Vert x \Vert_{\alpha,\ell^1} \right)^c,
		\end{aligned}
	\end{equation}
	which implies that $\vert (a_i \circ R_\gamma)(t,x) \vert \leq \vert (a_i \circ T_\gamma)(t,x) \vert + \left( 1 + \Vert (t,x) \Vert_{\mathbb{R} \times D^{\alpha,1}([0,T];Z)} \right)^c$. Thus, by using a telescoping sum, it holds for every $j = 1,\ldots,k$ and $v_1,\ldots,v_j \in \mathbb{R} \times D^{\alpha,1}([0,T];Z)$ that
	\begin{equation}
		\begin{aligned}
			& \left\vert \prod_{\ell=1}^j (a_i \circ T_\gamma)(v_\ell) - \prod_{\ell=1}^j (a_i \circ R_\gamma)(v_\ell) \right\vert \\
			& \leq \sum_{\ell=1}^j \left( \prod_{m=1}^{\ell-1} \vert (a_i \circ T_\gamma)(v_m) \vert \right) \vert (a_i \circ T_\gamma)(v_\ell) - (a_i \circ R_\gamma)(v_\ell) \vert \left( \prod_{m=\ell+1}^j \vert (a_i \circ R_\gamma)(v_m) \vert \right) \\
			& \leq j \frac{\varepsilon}{k C_1 C_\eta C_\rho} \prod_{\ell=1}^j \left( \left\vert (a_i \circ T_\gamma)(v_\ell) \right\vert + \left( 1 + \Vert v_\ell \Vert_{\mathbb{R} \times D^{\alpha,1}([0,T];Z)} \right)^c \right) \\
			& \leq \frac{\varepsilon}{C_1 C_\eta C_\rho} C_1 \prod_{\ell=1}^j \left( 1 + \Vert v_\ell \Vert_{\mathbb{R} \times D^{\alpha,1}([0,T];Z)} \right)^c \\
			& = \frac{\varepsilon}{C_\eta C_\rho} \left( 1 + \sum_{\ell=1}^j \Vert v_\ell \Vert_{\mathbb{R} \times D^{\alpha,1}([0,T];Z)} \right)^{cj}.
		\end{aligned}
	\end{equation}
	Therefore, by using the chain rule, it follows for every $j = 1,\ldots,k$ and $(u,v_1,\ldots,v_j) \in \phi_i(\Lambda^{\alpha,1}_{T,Z}) \times (\mathbb{R} \times D^{\alpha,1}([0,T];Z))^j$ that
	\begin{equation}
		\begin{aligned}
			& \left\vert d^j \rho\left( (a_i \circ T_\gamma)(\cdot)+b \right)(u;v_1,\ldots,v_j) - d^j \rho\left( (a_i \circ R_\gamma)(\cdot)+b \right)(u;v_1,\ldots,v_j) \right\vert \\
			& \leq \left\vert \rho^{(\vert j \vert)}((a_i \circ T_\gamma)(u)+b) \prod_{\ell=1}^j (a_i \circ T_\gamma)(v_\ell) - \rho^{(\vert j \vert)}((a_i \circ R_\gamma)(u)+b) \prod_{\ell=1}^j (a_i \circ R_\gamma)(v_\ell) \right\vert \\
			& \leq \left\vert \rho^{(\vert j \vert)}((a_i \circ T_\gamma)(u)+b) - \rho^{(\vert j \vert)}((a_i \circ R_\gamma)(u)+b) \right\vert \prod_{\ell=1}^j \left\vert (a_i \circ T_\gamma)(v_\ell) \right\vert \\
			& \quad\quad + \left\vert \rho^{(\vert j \vert)}((a_i \circ R_\gamma)(u)+b) \right\vert \left\vert \prod_{\ell=1}^j (a_i \circ T_\gamma)(v_\ell) - \prod_{\ell=1}^j (a_i \circ R_\gamma)(v_\ell) \right\vert \\
			& \leq C_\rho \left\vert (a_i \circ T_\gamma)(u) - (a_i \circ R_\gamma)(u) \right\vert C_1 + C_\rho \left\vert \prod_{\ell=1}^j (a_i \circ T_\gamma)(v_\ell) - \prod_{\ell=1}^j (a_i \circ R_\gamma)(v_\ell) \right\vert \\
			& \leq C_1 C_\rho \frac{\varepsilon}{C_1 C_\eta C_\rho k} \left( 1 + \Vert u \Vert_{\mathbb{R} \times D^{\alpha,1}([0,T];Z)} \right)^c + C_\rho \frac{\varepsilon}{C_\eta C_\rho} \prod_{\ell=1}^j \left( 1 + \Vert v_\ell \Vert_{\mathbb{R} \times D^{\alpha,1}([0,T];Z)} \right)^c \\
			& \leq \frac{\varepsilon}{C_\eta} \left( 1 + \max(j,1) \Vert u \Vert_{\mathbb{R} \times D^{\alpha,1}([0,T];Z)} + \sum_{\ell=1}^j \Vert v_\ell \Vert_{\mathbb{R} \times D^{\alpha,1}([0,T];Z)} \right)^{cj}.
		\end{aligned}
	\end{equation}
	Hence, we conclude that
	\begin{equation}
		\begin{aligned}
			& \left\Vert \rho\left( (a_i \circ T_\gamma)(\cdot) + b \right) - \rho\left( (a_i \circ R_\gamma)(\cdot) + b \right) \right\Vert_{\mathcal{B}^k_{\Psi_i}(\phi_i(\Lambda^{\alpha,1}_{T,Z}))} \\
			& = \max_{j=0,\ldots,k} \sup_{(u,v_{1:j})} \frac{\left\vert d^j \rho\left( (a_i \circ T_\gamma)(\cdot) + b \right)(u;v_{1:j}) - d^j \rho\left( (a_i \circ R_\gamma)(\cdot) + b \right)(u;v_{1:j}) \right\vert}{\psi_{i,j}(u,v_{1:j})} \\
			& \leq \frac{\varepsilon}{C_\eta} \max_{j=0,\ldots,k} \sup_{(u,v_{1:j})} \frac{\left( 1 + \max(j,1) \Vert u \Vert_{\mathbb{R} \times D^{\alpha,1}([0,T];Z)} + \sum_{\ell=1}^j \Vert v_\ell \Vert_{\mathbb{R} \times D^{\alpha,1}([0,T];Z)} \right)^{cj}}{\eta\left( \max(j,1) \Vert u \Vert_{\mathbb{R} \times D^{\alpha,1}([0,T];Z)} + \sum_{\ell=1}^j \Vert v_\ell \Vert_{\mathbb{R} \times D^{\alpha,1}([0,T];Z)} \right)} \\
			& \leq \frac{\varepsilon}{C_\eta} \max_{j=0,\ldots,k} \sup_{r \in (0,\infty)} \frac{(1+r)^{j\max(1,c)}}{\eta(r)} = \frac{\varepsilon}{C_\eta} C_\eta = \varepsilon,
		\end{aligned}
	\end{equation}
	where the supremum is taken over $(u,v_{1:j}) \in \phi_i(\Lambda^{\alpha,1}_{T,Z}) \times (\mathbb{R} \times D^{\alpha,1}([0,T];Z))^j$. Since $\varepsilon \in (0,1)$ was chosen arbitrarily, this shows that $\rho((a_i \circ T_\gamma)(\cdot)+b)$ belongs to the closure of $\mathcal{NN}^{\mathcal{A}_i,\rho,\mathbb{R}}_{\phi_i(\Lambda^{\alpha,1}_{T,Z}),Y}$ with respect to $\Vert \cdot \Vert_{\mathcal{B}^k_{\Psi_i}(\phi_i(\Lambda^{\alpha,1}_{T,Z}))}$. Finally, we can apply Theorem~\ref{thm:w_uat_infdim} to obtain that $\mathcal{PN}^{\widetilde{\rho},\rho,\mathcal{L}}_{\Lambda^{\alpha,1}_{T,Z},Y} = \mathcal{NN}^{\mathcal{A},\rho,\mathcal{L}}_{\Lambda^{\alpha,1}_{T,Z},Y}$ is a dense subset of $\mathcal{B}^k_\Psi(\Lambda^{\alpha,1}_{T,Z};Y)$.
\end{proof}

\subsection{Proof of Corollary~\ref{cor:naf_uat_horver}}
\label{app:cor:naf_uat_horver}

\begin{proof}[Proof of Corollary~\ref{cor:naf_uat_horver}]
	Compared to Corollary~\ref{cor:naf_uat_full}, we now restrict ourselves to maps that have derivatives only in certain directions, i.e., the horizontal derivatives $\mathcal{D}^j f: \Lambda^{\alpha,1}_{T,\mathbb{R}^d} \rightarrow Y$, $j = 1,\ldots,k$, and vertical derivatives $\mathscr{D}_\beta f: \Lambda^{\alpha,1}_{T,\mathbb{R}^d} \rightarrow Y$, $\beta \in \mathbb{N}^d_{0,\ell}$, which are represented on the model space $\phi_i(\Lambda^{\alpha,1}_{T,\mathbb{R}^d}) \subseteq \mathbb{R} \times D^{\alpha,1}([0,T];\mathbb{R}^d)$ by the directional derivatives
	\begin{equation}
		d^j \left( f \circ \phi_i^{-1} \right)((t,\overline{x}^t);(1,0),\ldots,(1,0)) \quad\quad \text{and} \quad d^{\vert\beta\vert} \left( f \circ \phi_i^{-1} \right)((t,\overline{x}^t);(0,e_1 \mathds{1}_{[t,T]}),\ldots,(0,e_d \mathds{1}_{[t,T]})),
	\end{equation}
	respectively, where the $i$-th unit vector $e_i \in \mathbb{R}^d$ appears $\beta_i$-times. Hence, in order to show that $\mathcal{PN}^{\widetilde{\rho},\rho,\mathcal{L}}_{\Lambda^{\alpha,1}_{T,\mathbb{R}^d},Y}$ is dense in $\mathbb{B}^{k,\ell}_\psi(\Lambda^{\alpha,1}_{T,\mathbb{R}^d};Y)$, we apply Nachbin's theorem not to the full jet space, but to the restricted jet space generated by the horizontal and vertical directions. To this end, we only need that the additive family $\mathcal{A}$ has nowhere vanishing derivatives in the directions of interest, i.e., $(1,0) \in \mathbb{R} \times D^{\alpha,1}([0,T];\mathbb{R}^d)$ and $(0,e_i \mathds{1}_{[t,T]}) \in \mathbb{R} \times D^{\alpha,1}([0,T];\mathbb{R}^d)$, to obtain the conclusion.
\end{proof}

\noindent\textbf{Acknowledgments.} P.~Schmocker gratefully acknowledges financial support by the FinsureTech Hub of ETH Zurich. J.~Teichmann gratefully acknowledges financial support by ETH Foundation.

\bibliographystyle{abbrv}
\bibliography{mybib}

\end{document}